\documentclass[numbers=enddot,12pt,final,onecolumn,notitlepage]{scrartcl}%
\usepackage[headsepline,footsepline,manualmark]{scrlayer-scrpage}
\usepackage[all,cmtip]{xy}
\usepackage{amsfonts}
\usepackage{amssymb}
\usepackage{amsmath}
\usepackage{amsthm}
\usepackage{framed}
\usepackage{comment}
\usepackage{color}
\usepackage[breaklinks=true]{hyperref}
\usepackage[sc]{mathpazo}
\usepackage[T1]{fontenc}
\usepackage{tikz}
\usepackage{needspace}
\usepackage{tabls}
\providecommand{\U}[1]{\protect\rule{.1in}{.1in}}
\usetikzlibrary{arrows}
\newcounter{exer}

\theoremstyle{definition}
\newtheorem{theo}{Theorem}[section]
\newenvironment{theorem}[1][]
{\begin{theo}[#1]\begin{leftbar}}
{\end{leftbar}\end{theo}}
\newtheorem{lem}[theo]{Lemma}
\newenvironment{lemma}[1][]
{\begin{lem}[#1]\begin{leftbar}}
{\end{leftbar}\end{lem}}
\newtheorem{prop}[theo]{Proposition}
\newenvironment{proposition}[1][]
{\begin{prop}[#1]\begin{leftbar}}
{\end{leftbar}\end{prop}}
\newtheorem{defi}[theo]{Definition}
\newenvironment{definition}[1][]
{\begin{defi}[#1]\begin{leftbar}}
{\end{leftbar}\end{defi}}
\newtheorem{remk}[theo]{Remark}
\newenvironment{remark}[1][]
{\begin{remk}[#1]\begin{leftbar}}
{\end{leftbar}\end{remk}}
\newtheorem{coro}[theo]{Corollary}
\newenvironment{corollary}[1][]
{\begin{coro}[#1]\begin{leftbar}}
{\end{leftbar}\end{coro}}
\newtheorem{conv}[theo]{Convention}
\newenvironment{convention}[1][]
{\begin{conv}[#1]\begin{leftbar}}
{\end{leftbar}\end{conv}}
\newtheorem{quest}[theo]{Question}
\newenvironment{question}[1][]
{\begin{quest}[#1]\begin{leftbar}}
{\end{leftbar}\end{quest}}
\newtheorem{warn}[theo]{Warning}

\newtheorem{conj}[theo]{Conjecture}

\newtheorem{exam}[theo]{Example}
\newenvironment{example}[1][]
{\begin{exam}[#1]\begin{leftbar}}
{\end{leftbar}\end{exam}}
\newtheorem{exmp}[exer]{Exercise}

\newenvironment{statement}{\begin{quote}}{\end{quote}}

\newcommand{\silentsubsection}{\subsection}

\let\sumnonlimits\sum
\let\prodnonlimits\prod
\let\cupnonlimits\bigcup
\let\capnonlimits\bigcap
\renewcommand{\sum}{\sumnonlimits\limits}
\renewcommand{\prod}{\prodnonlimits\limits}
\renewcommand{\bigcup}{\cupnonlimits\limits}
\renewcommand{\bigcap}{\capnonlimits\limits}
\newenvironment{verlong}{}{}
\newenvironment{vershort}{}{}
\newenvironment{noncompile}{}{}
\excludecomment{verlong}
\includecomment{vershort}
\excludecomment{noncompile}

\ihead{Pelletier--Ressayre hidden symmetry}
\ohead{page \thepage}
\begin{document}

\title{The Pelletier--Ressayre hidden symmetry for Littlewood--Richardson coefficients}
\author{Darij Grinberg}
\date{
1 January 2022
\\
(Extended version of a paper published in:
\href{https://doi.org/10.5070/C61055382}{Combinatorial Theory \textbf{1}
(2021), \#16}.)}
\maketitle

\begin{abstract}
\textbf{Abstract.} We prove an identity for Littlewood--Richardson
coefficients conjectured by Pelletier and Ressayre. The proof relies on a
novel birational involution defined over any semifield.


\textbf{MSC2010 Mathematics Subject Classifications:} 05E05.

\end{abstract}
\tableofcontents

\section*{***}

One of the central concepts in the theory of symmetric functions are the
\textit{Littlewood--Richardson coefficients} $c_{\mu,\nu}^{\lambda}$: the
coefficients when a product $s_{\mu}s_{\nu}$ of two Schur functions is
expanded back in the Schur basis $\left(  s_{\lambda}\right)  _{\lambda
\in\operatorname*{Par}}$. Various properties of these coefficients have been
found, among them combinatorial interpretations, vanishing results, bounds and
symmetries (i.e., equalities between $c_{\mu,\nu}^{\lambda}$ for different
$\lambda,\mu,\nu$). A recent overview of the latter can be found in
\cite{BriRos20}.

\begin{noncompile}
(See also \cite[Exercises 2.7.11(c), 2.9.15(c) and 2.9.16(c)]{GriRei} for
algebraic proofs of the classical symmetries.)
\end{noncompile}

In \cite{PelRes20}, Pelletier and Ressayre conjectured a further symmetry of
Littlewood--Richardson coefficients. Unless the classical ones, it is a
partial symmetry (i.e., it does not cover every Littlewood--Richardson
coefficient); it is furthermore much less simple to state, to the extent that
Pelletier and Ressayre have conjectured its existence while leaving open the
question which exact coefficients it matches up. In this paper, we answer this
question and prove the conjecture thus concretized.

The conjecture, in its original (unconcrete) form, can be stated as follows:
Let $n\geq2$, and consider the set $\operatorname*{Par}\left[  n\right]  $ of
all partitions having length $\leq n$. Let $a$ and $b$ be two nonnegative
integers, and define the two partitions $\alpha=\left(  a+b,a^{n-2}\right)  $
and $\beta=\left(  a+b,b^{n-2}\right)  $ (where $c^{n-2}$ means
$\underbrace{c,c,\ldots,c}_{n-2\text{ times}}$, as usual in partition
combinatorics). Fix another partition $\mu\in\operatorname*{Par}\left[
n\right]  $. Then, the families $\left(  c_{\alpha,\mu}^{\omega}\right)
_{\omega\in\operatorname*{Par}\left[  n\right]  }$ and $\left(  c_{\beta,\mu
}^{\omega}\right)  _{\omega\in\operatorname*{Par}\left[  n\right]  }$ of
Littlewood--Richardson coefficients seem to be identical up to permutation.
(We can restate this in terms of Schur polynomials in the $n$ variables
$x_{1},x_{2},\ldots,x_{n}$; this then becomes the claim that the products
$s_{\alpha}\left(  x_{1},x_{2},\ldots,x_{n}\right)  \cdot s_{\mu}\left(
x_{1},x_{2},\ldots,x_{n}\right)  $ and $s_{\beta}\left(  x_{1},x_{2}%
,\ldots,x_{n}\right)  \cdot s_{\mu}\left(  x_{1},x_{2},\ldots,x_{n}\right)  $,
when expanded in the basis of Schur polynomials, have the same multiset of coefficients.)

Pelletier and Ressayre have proved this conjecture for $n=3$ (see
\cite[Corollary 2]{PelRes20}) and in some further cases. We shall prove it in
full generality, and construct what is essentially a bijection $\varphi
:\operatorname*{Par}\left[  n\right]  \rightarrow\operatorname*{Par}\left[
n\right]  $ that makes it explicit (i.e., that satisfies $c_{\alpha,\mu
}^{\omega}=c_{\beta,\mu}^{\varphi\left(  \omega\right)  }$ for each $\omega
\in\operatorname*{Par}\left[  n\right]  $). To be fully precise, $\varphi$
will not be a bijection $\operatorname*{Par}\left[  n\right]  \rightarrow
\operatorname*{Par}\left[  n\right]  $, but rather a bijection from
$\mathbb{Z}^{n}$ to $\mathbb{Z}^{n}$, and it will satisfy $c_{\alpha,\mu
}^{\omega}=c_{\beta,\mu}^{\varphi\left(  \omega\right)  }$ with the
understanding that $c_{\alpha,\mu}^{\omega}=c_{\beta,\mu}^{\omega}=0$ when
$\omega\notin\operatorname*{Par}\left[  n\right]  $. (Here,
$\operatorname*{Par}\left[  n\right]  $ is understood to be a subset of
$\mathbb{Z}^{n}$ by identifying each partition $\lambda\in\operatorname*{Par}%
\left[  n\right]  $ with the $n$-tuple $\left(  \lambda_{1},\lambda_{2}%
,\ldots,\lambda_{n}\right)  $.)

We will define this bijection $\varphi$ by explicit (if somewhat intricate)
formulas that \textquotedblleft mingle\textquotedblright\ the entries of the
partition it is being applied to with those of $\mu$ (as well as $a$ and $b$)
using the $\min$ and $+$ operators. These formulas are best understood in the
\textit{birational picture}, in which these $\min$ and $+$ operators are
generalized to the addition and the multiplication of an arbitrary semifield.
(Our proof does not require this generality, but the birational picture has
the advantage of greater familiarity and better notational support. It also
reveals a connection with  a known birational map known as a \textquotedblleft
birational R-matrix\textquotedblright\ (see Section \ref{subsect.fin.Rmat} for
details), which throws some light on the otherwise rather mysterious bijection.)

Another ingredient of our proof is an explicit formula for $s_{\alpha}\left(
x_{1},x_{2},\ldots,x_{n}\right)  $ for the above-mentioned partition $\alpha$.

\subsubsection*{Acknowledgments}

I thank Sunita Chepuri, Grigori Olshanski, Pasha Pylyavskyy and Tom Roby for
interesting and helpful conversations.

Parts of this paper were written at the Mathematisches Forschungsinstitut
Oberwolfach, where I was staying as a Leibniz fellow in Summer 2020. The
SageMath computer algebra system \cite{SageMath} has been used in discovering
some of the results.

\subsection*{Remark on alternative versions}

\begin{vershort}
This paper also has a detailed version \cite{verlong}, which includes some
proofs that have been omitted from the present version (mostly basic
properties of symmetric functions).
\end{vershort}

\begin{verlong}
You are reading the detailed version of this paper. For the standard version
(which is shorter by virtue of omitting some straightforward or well-known
proofs), see \cite{vershort}.
\end{verlong}

An older version of this paper appeared as
\href{https://publications.mfo.de/handle/mfo/3773}{Oberwolfach Preprint
OWP-2020-18}.

\section{\label{sect.not}Notations}

We will use the following notations (most of which are also used in
\cite[\S 2.1]{GriRei}):

\begin{itemize}
\item We let $\mathbb{N}=\left\{  0,1,2,\ldots\right\}  $.

\item We fix a commutative ring $\mathbf{k}$; we will use this $\mathbf{k}$ as
the base ring in what follows.

\item A \emph{weak composition} means an infinite sequence of nonnegative
integers that contains only finitely many nonzero entries (i.e., a sequence
$\left(  \alpha_{1},\alpha_{2},\alpha_{3},\ldots\right)  \in\mathbb{N}%
^{\infty}$ such that all but finitely many $i\in\left\{  1,2,3,\ldots\right\}
$ satisfy $\alpha_{i}=0$).

\item We let $\operatorname*{WC}$ denote the set of all weak compositions.

\item For any weak composition $\alpha$ and any positive integer $i$, we let
$\alpha_{i}$ denote the $i$-th entry of $\alpha$ (so that $\alpha=\left(
\alpha_{1},\alpha_{2},\alpha_{3},\ldots\right)  $). More generally, we use
this notation whenever $\alpha$ is an infinite sequence of any kind of objects.

\item The \emph{size} $\left\vert \alpha\right\vert $ of a weak composition
$\alpha$ is defined to be $\alpha_{1}+\alpha_{2}+\alpha_{3}+\cdots
\in\mathbb{N}$.

\item A \emph{partition} means a weak composition whose entries weakly
decrease (i.e., a weak composition $\alpha$ satisfying $\alpha_{1}\geq
\alpha_{2}\geq\alpha_{3}\geq\cdots$).

\item We let $\operatorname*{Par}$ denote the set of all partitions.

\item We will sometimes omit trailing zeroes from partitions: i.e., a
partition $\lambda=\left(  \lambda_{1},\lambda_{2},\lambda_{3},\ldots\right)
$ will be identified with the $k$-tuple $\left(  \lambda_{1},\lambda
_{2},\ldots,\lambda_{k}\right)  $ whenever $k\in\mathbb{N}$ satisfies
$\lambda_{k+1}=\lambda_{k+2}=\lambda_{k+3}=\cdots=0$. For example, $\left(
3,2,1,0,0,0,\ldots\right)  =\left(  3,2,1\right)  =\left(  3,2,1,0\right)  $.

As a consequence of this, an $n$-tuple $\left(  \lambda_{1},\lambda_{2}%
,\ldots,\lambda_{n}\right)  \in\mathbb{Z}^{n}$ (for any given $n\in\mathbb{N}%
$) is a partition if and only if it satisfies $\lambda_{1}\geq\lambda_{2}%
\geq\cdots\geq\lambda_{n}\geq0$.

\item A \emph{part} of a partition $\lambda$ means a nonzero entry of
$\lambda$. For example, the parts of the partition $\left(  3,1,1\right)
=\left(  3,1,1,0,0,0,\ldots\right)  $ are $3,1,1$.

\item The \emph{length} of a partition $\lambda$ means the smallest
$k\in\mathbb{N}$ such that $\lambda_{k+1}=\lambda_{k+2}=\lambda_{k+3}%
=\cdots=0$. Equivalently, the length of a partition $\lambda$ is the number of
parts of $\lambda$ (counted with multiplicity). This length is denoted by
$\ell\left(  \lambda\right)  $. For example, $\ell\left(  \left(
4,2,0,0\right)  \right)  =\ell\left(  \left(  4,2\right)  \right)  =2$ and
$\ell\left(  \left(  5,1,1\right)  \right)  =3$.

\item We will use the notation $m^{k}$ for \textquotedblleft%
$\underbrace{m,m,\ldots,m}_{k\text{ times}}$\textquotedblright\ in partitions
and tuples (whenever $m\in\mathbb{N}$ and $k\in\mathbb{N}$). (For example,
$\left(  2,1^{4}\right)  =\left(  2,1,1,1,1\right)  $.)

\item We let $\Lambda$ denote the ring of symmetric functions in infinitely
many variables $x_{1},x_{2},x_{3},\ldots$ over $\mathbf{k}$. This is a subring
of the ring $\mathbf{k}\left[  \left[  x_{1},x_{2},x_{3},\ldots\right]
\right]  $ of formal power series. To be more specific, $\Lambda$ consists of
all power series in $\mathbf{k}\left[  \left[  x_{1},x_{2},x_{3}%
,\ldots\right]  \right]  $ that are symmetric (i.e., invariant under
permutations of the variables) and of bounded degree (see \cite[\S 2.1]%
{GriRei} for the precise meaning of this).

\item A \emph{monomial} shall mean a formal expression of the form
$x_{1}^{\alpha_{1}}x_{2}^{\alpha_{2}}x_{3}^{\alpha_{3}}\cdots$ with $\alpha
\in\operatorname*{WC}$. Formal power series are formal infinite $\mathbf{k}%
$-linear combinations of such monomials.

\item For any weak composition $\alpha$, we let $\mathbf{x}^{\alpha}$ denote
the monomial $x_{1}^{\alpha_{1}}x_{2}^{\alpha_{2}}x_{3}^{\alpha_{3}}\cdots$.

\item The \emph{degree} of a monomial $\mathbf{x}^{\alpha}$ is defined to be
$\left\vert \alpha\right\vert $.
\end{itemize}

We shall use the symmetric functions $h_{n}$ and $s_{\lambda}$ in $\Lambda$ as
defined in \cite[Sections 2.1 and 2.2]{GriRei}. Let us briefly recall how they
are defined:

\begin{itemize}
\item For each $n\in\mathbb{Z}$, we define the \emph{complete homogeneous
symmetric function }$h_{n}\in\Lambda$ by%
\[
h_{n}=\sum_{i_{1}\leq i_{2}\leq\cdots\leq i_{n}}x_{i_{1}}x_{i_{2}}\cdots
x_{i_{n}}=\sum_{\substack{\alpha\in\operatorname*{WC};\\\left\vert
\alpha\right\vert =n}}\mathbf{x}^{\alpha}.
\]
Thus, $h_{0}=1$ and $h_{n}=0$ for all $n<0$.

\item For each partition $\lambda$, we define the \emph{Schur function
}$s_{\lambda}\in\Lambda$ by%
\[
s_{\lambda}=\sum\mathbf{x}_{T},
\]
where the sum ranges over all semistandard tableaux $T$ of shape $\lambda$,
and where $\mathbf{x}_{T}$ denotes the monomial obtained by multiplying the
$x_{i}$ for all entries $i$ of $T$. We refer the reader to \cite[Definition
2.2.1]{GriRei} or to \cite[\S 7.10]{Stanley-EC2} for the details of this
definition and further descriptions of the Schur functions.

The family $\left(  s_{\lambda}\right)  _{\lambda\in\operatorname*{Par}}$ is a
basis of the $\mathbf{k}$-module $\Lambda$, and is known as the \emph{Schur
basis}. It is easy to see that each $n\in\mathbb{N}$ satisfies $s_{\left(
n\right)  }=h_{n}$.

\item We shall use the \emph{Littlewood--Richardson coefficients }$c_{\mu,\nu
}^{\lambda}$ (for $\lambda,\mu,\nu\in\operatorname*{Par}$), as defined in
\cite[Definition 2.5.8]{GriRei}, in \cite[\S 7.15]{Stanley-EC2} or in
\cite[Chapter 10]{Egge19}. One of their defining properties is the following
fact (see, e.g., \cite[(2.5.6)]{GriRei} or \cite[(7.64)]{Stanley-EC2} or
\cite[(10.1)]{Egge19}): Any two partitions $\mu,\nu\in\operatorname*{Par}$
satisfy%
\begin{equation}
s_{\mu}s_{\nu}=\sum_{\lambda\in\operatorname*{Par}}c_{\mu,\nu}^{\lambda
}s_{\lambda}. \label{eq.lrcoeff.def}%
\end{equation}

\end{itemize}

\section{\label{sect.mainthm}The theorem}

\begin{convention}
\label{conv.main}\ \ 

\begin{enumerate}
\item[\textbf{(a)}] For the rest of this paper, we fix a positive integer $n$.

\item[\textbf{(b)}] Let $\operatorname*{Par}\left[  n\right]  $ be the set of
all partitions having length $\leq n$. In other words,%
\begin{align*}
\operatorname*{Par}\left[  n\right]   &  =\left\{  \lambda\in
\operatorname*{Par}\ \mid\ \lambda=\left(  \lambda_{1},\lambda_{2}%
,\ldots,\lambda_{n}\right)  \right\}  =\operatorname*{Par}\cap\mathbb{N}^{n}\\
&  =\left\{  \left(  \lambda_{1},\lambda_{2},\ldots,\lambda_{n}\right)
\in\mathbb{Z}^{n}\ \mid\ \lambda_{1}\geq\lambda_{2}\geq\cdots\geq\lambda
_{n}\geq0\right\}
\end{align*}
(where we are using our convention that trailing zeroes can be omitted from
partitions, so that a partition of length $\leq n$ can always be identified
with an $n$-tuple).

\item[\textbf{(c)}] A family $\left(  u_{i}\right)  _{i\in\mathbb{Z}}$ of
objects (e.g., of numbers) is said to be $n$\emph{-periodic} if each
$j\in\mathbb{Z}$ satisfies $u_{j}=u_{j+n}$. Equivalently, a family $\left(
u_{i}\right)  _{i\in\mathbb{Z}}$ of objects is $n$-periodic if and only if it
has the property that%
\[
\left(  u_{j}=u_{j^{\prime}}\text{ whenever }j\text{ and }j^{\prime}\text{ are
two integers satisfying }j\equiv j^{\prime}\operatorname{mod}n\right)  .
\]
Thus, an $n$-periodic family $\left(  u_{i}\right)  _{i\in\mathbb{Z}}$ is
uniquely determined by the $n$ entries $u_{1},u_{2},\ldots,u_{n}$ (because for
any integer $j$, we have $u_{j}=u_{j^{\prime}}$, where $j^{\prime}$ is the
unique element of $\left\{  1,2,\ldots,n\right\}  $ that is congruent to $j$
modulo $n$).
\end{enumerate}
\end{convention}

\begin{example}
If $n=3$, then both partitions $\left(  3,2\right)  $ and $\left(
3,2,2\right)  $ belong to $\operatorname*{Par}\left[  n\right]  $, while the
partition $\left(  3,2,2,2\right)  $ does not. The $n$-tuples $\left(
4,2,1\right)  $ and $\left(  3,3,0\right)  $ are partitions, while the
$n$-tuples $\left(  1,0,-1\right)  $ and $\left(  2,0,1\right)  $ are not.

If $\zeta$ is an $n$-th root of unity, then the family $\left(  \zeta
^{i}\right)  _{i\in\mathbb{Z}}$ of complex numbers is $n$-periodic.
\end{example}

We can now state our main theorem, which is a concretization of
\cite[Conjecture 1]{PelRes20}:

\begin{theorem}
\label{thm.main}Assume that $n\geq2$. Let $a,b\in\mathbb{N}$.

Define the two partitions $\alpha=\left(  a+b,a^{n-2}\right)  $ and
$\beta=\left(  a+b,b^{n-2}\right)  $.

Fix any partition $\mu\in\operatorname*{Par}\left[  n\right]  $.

Define a map $\varphi:\mathbb{Z}^{n}\rightarrow\mathbb{Z}^{n}$ as follows:

Let $\omega\in\mathbb{Z}^{n}$. Define an $n$-tuple $\nu=\left(  \nu_{1}%
,\nu_{2},\ldots,\nu_{n}\right)  \in\mathbb{Z}^{n}$ by
\[
\nu_{i}=\omega_{i}-a\ \ \ \ \ \ \ \ \ \ \text{for each }i\in\left\{
1,2,\ldots,n\right\}  ,
\]
where $\omega_{i}$ means the $i$-th entry of $\omega$.

For each $i\in\mathbb{Z}$, we let $i\#$ denote the unique element of $\left\{
1,2,\ldots,n\right\}  $ congruent to $i$ modulo $n$.

For each $j\in\mathbb{Z}$, set%
\begin{align*}
\tau_{j}  &  =\min\left\{  \left(  \nu_{\left(  j+1\right)  \#}+\nu_{\left(
j+2\right)  \#}+\cdots+\nu_{\left(  j+k\right)  \#}\right)  \right. \\
&  \ \ \ \ \ \ \ \ \ \ \ \ \ \ \ \ \ \ \ \ \left.  +\left(  \mu_{\left(
j+k+1\right)  \#}+\mu_{\left(  j+k+2\right)  \#}+\cdots+\mu_{\left(
j+n-1\right)  \#}\right)  \right. \\
&  \ \ \ \ \ \ \ \ \ \ \ \ \ \ \ \ \ \ \ \ \left.  \ \mid\ k\in\left\{
0,1,\ldots,n-1\right\}
\vphantom{\left(\nu_{\left(j+1\right)\#}\right)}\right\}  .
\end{align*}

Define an $n$-tuple $\eta=\left(  \eta_{1},\eta_{2},\ldots,\eta_{n}\right)
\in\mathbb{Z}^{n}$ by setting%
\[
\eta_{i}=\mu_{i\#}+\left(  \mu_{\left(  i-1\right)  \#}+\tau_{\left(
i-1\right)  \#}\right)  -\left(  \nu_{\left(  i+1\right)  \#}+\tau_{\left(
i+1\right)  \#}\right)  \ \ \ \ \ \ \ \ \ \ \text{for each }i\in\left\{
1,2,\ldots,n\right\}  .
\]

Let $\varphi\left(  \omega\right)  $ be the $n$-tuple $\left(  \eta_{1}%
+b,\eta_{2}+b,\ldots,\eta_{n}+b\right)  \in\mathbb{Z}^{n}$. Thus, we have
defined a map $\varphi:\mathbb{Z}^{n}\rightarrow\mathbb{Z}^{n}$.

Then:

\begin{itemize}
\item[\textbf{(a)}] The map $\varphi$ is a bijection.

\item[\textbf{(b)}] We have%
\[
c_{\alpha,\mu}^{\omega}=c_{\beta,\mu}^{\varphi\left(  \omega\right)
}\ \ \ \ \ \ \ \ \ \ \text{for each }\omega\in\mathbb{Z}^{n}.
\]
Here, we are using the convention that every $n$-tuple $\omega\in
\mathbb{Z}^{n}$ that is not a partition satisfies $c_{\alpha,\mu}^{\omega}=0$
and $c_{\beta,\mu}^{\omega}=0$.
\end{itemize}
\end{theorem}

This theorem will be proved at the end of this paper, after we have shown
several (often seemingly unrelated, yet eventually useful) results.

\begin{example}
Let $n=4$ and $a=1$ and $b=4$. The partitions $\alpha$ and $\beta$ defined in
Theorem \ref{thm.main} then take the forms $\alpha=\left(  1+4,1^{2}\right)
=\left(  5,1,1\right)  $ and $\beta=\left(  1+4,4^{2}\right)  =\left(
5,4,4\right)  $.

Let $\mu\in\operatorname*{Par}\left[  n\right]  $ be the partition $\left(
2,1\right)  =\left(  2,1,0,0\right)  $. Let $\omega\in\operatorname*{Par}%
\left[  n\right]  $ be the partition $\left(  5,3,2\right)  =\left(
5,3,2,0\right)  $. We shall compute the $n$-tuple $\varphi\left(
\omega\right)  $ defined in Theorem \ref{thm.main}.

Indeed, the $n$-tuple $\nu$ from Theorem \ref{thm.main} is
\[
\nu=\left(  \omega_{1}-a,\omega_{2}-a,\omega_{3}-a,\omega_{4}-a\right)
=\left(  5-1,3-1,2-1,0-1\right)  =\left(  4,2,1,-1\right)  .
\]
The integers $i\#$ from Theorem \ref{thm.main} form an $n$-periodic family%
\[
\left(  i\#\right)  _{i\in\mathbb{Z}}=\left(  \ldots
,0\#,1\#,2\#,3\#,4\#,5\#,6\#,7\#,\ldots\right)  =\left(  \ldots
,4,1,2,3,4,1,2,3,\ldots\right)  .
\]
The integers $\tau_{j}$ (for $j\in\mathbb{Z}$) from Theorem \ref{thm.main} are
given by%
\begin{align*}
\tau_{1}  &  =\min\left\{  \left(  \nu_{2\#}+\nu_{3\#}+\cdots+\nu_{\left(
k+1\right)  \#}\right)  +\left(  \mu_{\left(  k+2\right)  \#}+\mu_{\left(
k+3\right)  \#}+\cdots+\mu_{4\#}\right)  \right. \\
&  \ \ \ \ \ \ \ \ \ \ \ \ \ \ \ \ \ \ \ \ \left.  \ \mid\ k\in\left\{
0,1,2,3\right\}  \right\} \\
&  =\min\left\{  \mu_{2\#}+\mu_{3\#}+\mu_{4\#},\ \ \ \nu_{2\#}+\mu_{3\#}%
+\mu_{4\#},\ \ \ \nu_{2\#}+\nu_{3\#}+\mu_{4\#},\ \ \ \nu_{2\#}+\nu_{3\#}%
+\nu_{4\#}\right\} \\
&  =\min\left\{  \mu_{2}+\mu_{3}+\mu_{4},\ \ \ \nu_{2}+\mu_{3}+\mu
_{4},\ \ \ \nu_{2}+\nu_{3}+\mu_{4},\ \ \ \nu_{2}+\nu_{3}+\nu_{4}\right\} \\
&  =\min\left\{  1+0+0,\ \ \ 2+0+0,\ \ \ 2+1+0,\ \ \ 2+1+\left(  -1\right)
\right\} \\
&  =\min\left\{  1,2,3,2\right\}  =1
\end{align*}
and%
\begin{align*}
\tau_{2}  &  =\min\left\{  \left(  \nu_{3\#}+\nu_{4\#}+\cdots+\nu_{\left(
k+2\right)  \#}\right)  +\left(  \mu_{\left(  k+3\right)  \#}+\mu_{\left(
k+4\right)  \#}+\cdots+\mu_{5\#}\right)  \right. \\
&  \ \ \ \ \ \ \ \ \ \ \ \ \ \ \ \ \ \ \ \ \left.  \ \mid\ k\in\left\{
0,1,2,3\right\}  \right\} \\
&  =\min\left\{  \mu_{3\#}+\mu_{4\#}+\mu_{5\#},\ \ \ \nu_{3\#}+\mu_{4\#}%
+\mu_{5\#},\ \ \ \nu_{3\#}+\nu_{4\#}+\mu_{5\#},\ \ \ \nu_{3\#}+\nu_{4\#}%
+\nu_{5\#}\right\} \\
&  =\min\left\{  \mu_{3}+\mu_{4}+\mu_{1},\ \ \ \nu_{3}+\mu_{4}+\mu
_{1},\ \ \ \nu_{3}+\nu_{4}+\mu_{1},\ \ \ \nu_{3}+\nu_{4}+\nu_{1}\right\} \\
&  =\min\left\{  0+0+2,\ \ \ 1+0+2,\ \ \ 1+\left(  -1\right)
+2,\ \ \ 1+\left(  -1\right)  +4\right\} \\
&  =\min\left\{  2,3,2,4\right\}  =2
\end{align*}
and%
\begin{align*}
\tau_{3}  &  =\min\left\{  \left(  \nu_{4\#}+\nu_{5\#}+\cdots+\nu_{\left(
k+3\right)  \#}\right)  +\left(  \mu_{\left(  k+4\right)  \#}+\mu_{\left(
k+5\right)  \#}+\cdots+\mu_{6\#}\right)  \right. \\
&  \ \ \ \ \ \ \ \ \ \ \ \ \ \ \ \ \ \ \ \ \left.  \ \mid\ k\in\left\{
0,1,2,3\right\}  \right\} \\
&  =\min\left\{  \mu_{4\#}+\mu_{5\#}+\mu_{6\#},\ \ \ \nu_{4\#}+\mu_{5\#}%
+\mu_{6\#},\ \ \ \nu_{4\#}+\nu_{5\#}+\mu_{6\#},\ \ \ \nu_{4\#}+\nu_{5\#}%
+\nu_{6\#}\right\} \\
&  =\min\left\{  \mu_{4}+\mu_{1}+\mu_{2},\ \ \ \nu_{4}+\mu_{1}+\mu
_{2},\ \ \ \nu_{4}+\nu_{1}+\mu_{2},\ \ \ \nu_{4}+\nu_{1}+\nu_{2}\right\} \\
&  =\min\left\{  0+2+1,\ \ \ \left(  -1\right)  +2+1,\ \ \ \left(  -1\right)
+4+1,\ \ \ \left(  -1\right)  +4+2\right\} \\
&  =\min\left\{  3,2,4,5\right\}  =2
\end{align*}
and%
\begin{align*}
\tau_{4}  &  =\min\left\{  \left(  \nu_{5\#}+\nu_{6\#}+\cdots+\nu_{\left(
k+4\right)  \#}\right)  +\left(  \mu_{\left(  k+5\right)  \#}+\mu_{\left(
k+6\right)  \#}+\cdots+\mu_{7\#}\right)  \right. \\
&  \ \ \ \ \ \ \ \ \ \ \ \ \ \ \ \ \ \ \ \ \left.  \ \mid\ k\in\left\{
0,1,2,3\right\}  \right\} \\
&  =\min\left\{  \mu_{5\#}+\mu_{6\#}+\mu_{7\#},\ \ \ \nu_{5\#}+\mu_{6\#}%
+\mu_{7\#},\ \ \ \nu_{5\#}+\nu_{6\#}+\mu_{7\#},\ \ \ \nu_{5\#}+\nu_{6\#}%
+\nu_{7\#}\right\} \\
&  =\min\left\{  \mu_{1}+\mu_{2}+\mu_{3},\ \ \ \nu_{1}+\mu_{2}+\mu
_{3},\ \ \ \nu_{1}+\nu_{2}+\mu_{3},\ \ \ \nu_{1}+\nu_{2}+\nu_{3}\right\} \\
&  =\min\left\{  2+1+0,\ \ \ 4+1+0,\ \ \ 4+2+0,\ \ \ 4+2+1\right\} \\
&  =\min\left\{  3,5,6,7\right\}  =3
\end{align*}
and%
\[
\tau_{j}=\tau_{j^{\prime}}\ \ \ \ \ \ \ \ \ \ \text{whenever }j\equiv
j^{\prime}\operatorname{mod}4
\]
(the latter equality follows from the $n$-periodicity of the family $\left(
i\#\right)  _{i\in\mathbb{Z}}$). Thus, the $n$-tuple $\eta=\left(  \eta
_{1},\eta_{2},\ldots,\eta_{n}\right)  $ from Theorem \ref{thm.main} is given
by%
\[
\eta_{1}=\underbrace{\mu_{1\#}}_{=\mu_{1}=2}+\left(  \underbrace{\mu_{0\#}%
}_{=\mu_{4}=0}+\underbrace{\tau_{0\#}}_{=\tau_{4}=3}\right)  -\left(
\underbrace{\nu_{2\#}}_{=\nu_{2}=2}+\underbrace{\tau_{2\#}}_{=\tau_{2}%
=2}\right)  =2+\left(  0+3\right)  -\left(  2+2\right)  =1
\]
and%
\[
\eta_{2}=\underbrace{\mu_{2\#}}_{=\mu_{2}=1}+\left(  \underbrace{\mu_{1\#}%
}_{=\mu_{1}=2}+\underbrace{\tau_{1\#}}_{=\tau_{1}=1}\right)  -\left(
\underbrace{\nu_{3\#}}_{=\nu_{3}=1}+\underbrace{\tau_{3\#}}_{=\tau_{3}%
=2}\right)  =1+\left(  2+1\right)  -\left(  1+2\right)  =1
\]
and%
\[
\eta_{3}=\underbrace{\mu_{3\#}}_{=\mu_{3}=0}+\left(  \underbrace{\mu_{2\#}%
}_{=\mu_{2}=1}+\underbrace{\tau_{2\#}}_{=\tau_{2}=2}\right)  -\left(
\underbrace{\nu_{4\#}}_{=\nu_{4}=-1}+\underbrace{\tau_{4\#}}_{=\tau_{4}%
=3}\right)  =0+\left(  1+2\right)  -\left(  \left(  -1\right)  +3\right)  =1
\]
and%
\[
\eta_{4}=\underbrace{\mu_{4\#}}_{=\mu_{4}=0}+\left(  \underbrace{\mu_{3\#}%
}_{=\mu_{3}=0}+\underbrace{\tau_{3\#}}_{=\tau_{3}=2}\right)  -\left(
\underbrace{\nu_{5\#}}_{=\nu_{1}=4}+\underbrace{\tau_{5\#}}_{=\tau_{1}%
=1}\right)  =0+\left(  0+2\right)  -\left(  4+1\right)  =-3,
\]
so $\eta=\left(  1,1,1,-3\right)  $. Hence, $\varphi\left(  \omega\right)
=\left(  1+b,1+b,1+b,-3+b\right)  =\left(  5,5,5,1\right)  $ (since $b=4$).
This is a partition. Theorem \ref{thm.main} \textbf{(b)} now yields
$c_{\alpha,\mu}^{\omega}=c_{\beta,\mu}^{\varphi\left(  \omega\right)  }$, that
is, $c_{\left(  5,1,1\right)  ,\left(  2,1\right)  }^{\left(  5,3,2\right)
}=c_{\left(  5,4,4\right)  ,\left(  2,1\right)  }^{\left(  5,5,5,1\right)  }$.
And indeed, this equality holds (both of its sides being equal to $1$).
\end{example}

\begin{question}
Can the bijection $\varphi$ in Theorem \ref{thm.main} be defined in a more
\textquotedblleft intuitive\textquotedblright\ way, similar to (e.g.)
jeu-de-taquin or the RSK correspondence? (Of course, there is no tableau being
transformed here, just a partition, but this should make this construction easier.)
\end{question}

\section{\label{sect.bir}A birational involution}

The leading role in our proof of Theorem \ref{thm.main} will be played by a
certain piecewise-linear involution (which is similar to the bijection
$\varphi$ in Theorem \ref{thm.main}, but without the shifting by $-a$ and
$b$). For the sake of convenience, we prefer to study this involution in a
more general setting, in which the operations $\min$, $+$ and $-$ are replaced
by the structure operations $+$, $\cdot$ and $/$ of a semifield. This kind of
generalization is called \emph{detropicalization} (or \emph{birational
lifting}, or \emph{tropicalization} in the older combinatorial literature);
see, e.g., \cite{Kirill01}, \cite{NoumiYamada}, \cite[Sections 5 and
9]{EinPro13} or \cite[\S 4.2]{Roby15} for examples of this procedure (although
our use of it will be conceptually simpler).

\subsection{Semifields}

We recall some basic definitions from basic abstract algebra (mostly to avoid
confusion arising from slight terminological differences):

\begin{itemize}
\item A \emph{semigroup} means a pair $\left(  S,\ast\right)  $, where $S$ is
a set and where $\ast$ is an associative binary operation on $S$. We do not
require this operation $\ast$ to have a neutral element. We usually write the
operation $\ast$ infix (i.e., we write $a\ast b$ instead of $\ast\left(
a,b\right)  $ when $a,b\in S$).

\item A semigroup $\left(  S,\ast\right)  $ is said to be \emph{abelian} if
the operation $\ast$ is commutative (i.e., we have $a\ast b=b\ast a$ for all
$a,b\in S$).

\item A \emph{monoid} means a triple $\left(  S,\ast,e\right)  $, where
$\left(  S,\ast\right)  $ is a semigroup and where $e$ is a neutral element
for the operation $\ast$ (that is, $e$ is an element of $S$ that satisfies
$e\ast a=a\ast e=a$ for each $a\in S$). Usually, the monoid $\left(
S,\ast,e\right)  $ is equated with the semigroup $\left(  S,\ast\right)  $
because the neutral element is uniquely determined by $S$ and $\ast$.

\item If $\left(  S,\ast,e\right)  $ is a monoid and $a$ is an element of $S$,
then an \emph{inverse} of $a$ (with respect to $\ast$) means an element $b$ of
$S$ satisfying $a\ast b=b\ast a=e$. Such an inverse of $a$ is always unique
when it exists.

\item A \emph{group} means a monoid $\left(  S,\ast,e\right)  $ such that each
element of $S$ has an inverse (with respect to $\ast$).
\end{itemize}

We next recall the definition of a semifield (more precisely, the one we will
be using, as there are many competing ones):

\begin{definition}
\label{def.semifield}A \emph{semifield} means a set $\mathbb{K}$ endowed with

\begin{itemize}
\item two binary operations called \textquotedblleft
addition\textquotedblright\ and \textquotedblleft
multiplication\textquotedblright, and denoted by $+$ and $\cdot$,
respectively, and both written infix (i.e., we write $a+b$ and $a\cdot b$
instead of $+\left(  a,b\right)  $ and $\cdot\left(  a,b\right)  $), and

\item an element called \textquotedblleft unity\textquotedblright\ and denoted
by $1$
\end{itemize}

\noindent such that $\left(  \mathbb{K},+\right)  $ is an abelian semigroup
and $\left(  \mathbb{K},\cdot,1\right)  $ is an abelian group, and such that
the following axiom is satisfied:

\begin{itemize}
\item \textit{Distributivity:} We have $a\cdot\left(  b+c\right)  =\left(
a\cdot b\right)  +\left(  a\cdot c\right)  $ and $\left(  a+b\right)  \cdot
c=\left(  a\cdot c\right)  +\left(  b\cdot c\right)  $ for all $a\in
\mathbb{K}$, $b\in\mathbb{K}$ and $c\in\mathbb{K}$.
\end{itemize}
\end{definition}

Thus, a semifield is similar to a field, except that it has no additive
inverses and no zero element, but, on the other hand, has multiplicative
inverses for all its elements (not just the nonzero ones).

\begin{example}
\label{exa.semifield.Qplus}Let $\mathbb{Q}_{+}$ be the set of all positive
rational numbers. Then, $\mathbb{Q}_{+}$ (endowed with its standard addition
and multiplication and the number $1$) is a semifield.
\end{example}

\begin{example}
\label{exa.semifield.mintrop}Let $\left(  \mathbb{A},\ast,e\right)  $ be any
totally ordered abelian group (whose operation is $\ast$ and whose neutral
element is $e$). Then, $\mathbb{A}$ becomes a semifield if we endow it with
the \textquotedblleft addition\textquotedblright\ $\min$ (that is, we set
$a+b:=\min\left\{  a,b\right\}  $ for all $a,b\in\mathbb{A}$), the
\textquotedblleft multiplication\textquotedblright\ $\ast$ (that is, we set
$a\cdot b:=a\ast b$ for all $a,b\in\mathbb{A}$), and the \textquotedblleft
unity\textquotedblright\ $e$. This semifield $\left(  \mathbb{A},\min
,\ast,e\right)  $ is called the \emph{min tropical semifield} of $\left(
\mathbb{A},\ast,e\right)  $.
\end{example}

\begin{convention}
\label{conv.semifield.notations}All conventions that are typically used for
fields will be used for semifields as well, to the extent they apply. Specifically:

\begin{itemize}
\item If $\mathbb{K}$ is a semifield, and if $a,b\in\mathbb{K}$, then $a\cdot
b$ shall be abbreviated by $ab$.

\item We shall use the standard \textquotedblleft PEMDAS\textquotedblright%
\ convention that multiplication-like operations have higher precedence than
addition-like operations; thus, e.g., the expression \textquotedblleft%
$ab+ac$\textquotedblright\ must be understood as \textquotedblleft$\left(
ab\right)  +\left(  ac\right)  $\textquotedblright\ (and not, for example, as
\textquotedblleft$a\left(  b+a\right)  c$\textquotedblright).

\item If $\mathbb{K}$ is a semifield, then the inverse of any element
$b\in\mathbb{K}$ in the abelian group $\left(  \mathbb{K},\cdot,1\right)  $
will be denoted by $b^{-1}$. Note that this inverse is always defined (unlike
when $\mathbb{K}$ is a field).

\item If $\mathbb{K}$ is a semifield, and if $a,b\in\mathbb{K}$, then the
product $ab^{-1}$ will be denoted by $a/b$ and by $\dfrac{a}{b}$. Note that
this is always defined (unlike when $\mathbb{K}$ is a field).

\item Finite products $\prod_{i\in I}a_{i}$ of elements of a semifield are
defined in the same way as in commutative rings. The same applies to finite
sums $\sum_{i\in I}a_{i}$ as long as they are nonempty (i.e., as long as
$I\neq\varnothing$). The empty sum is not defined in a semifield, since there
is no zero element.
\end{itemize}
\end{convention}

\subsection{The birational involution}

For the rest of Section \ref{sect.bir}, we agree to the following two conventions:

\begin{convention}
\label{conv.bir.1}We fix a positive integer $n$ and a semifield $\mathbb{K}$.
We also fix an $n$-tuple $u\in\mathbb{K}^{n}$.
\end{convention}

\begin{convention}
\label{conv.bir.peri}If $a\in\mathbb{K}^{n}$ is an $n$-tuple, and if
$i\in\mathbb{Z}$, then $a_{i}$ shall denote the $i^{\prime}$-th entry of $a$,
where $i^{\prime}$ is the unique element of $\left\{  1,2,\ldots,n\right\}  $
satisfying $i^{\prime}\equiv i\operatorname{mod}n$. Thus, each $n$-tuple
$a\in\mathbb{K}^{n}$ satisfies $a=\left(  a_{1},a_{2},\ldots,a_{n}\right)  $
and $a_{i}=a_{i+n}$ for each $i\in\mathbb{Z}$. Therefore, if $a\in
\mathbb{K}^{n}$ is any $n$-tuple, then the family $\left(  a_{i}\right)
_{i\in\mathbb{Z}}$ is $n$-periodic.
\end{convention}

We shall soon use the letter $x$ for an $n$-tuple in $\mathbb{K}^{n}$; thus,
$x_{1},x_{2},\ldots,x_{n}$ will be the entries of this $n$-tuple. This has
nothing to do with the indeterminates $x_{1},x_{2},x_{3},\ldots$ from Section
\ref{sect.not} (that unfortunately use the same letters); we actually
\textbf{forget all conventions from Section \ref{sect.not}} (apart from
$\mathbb{N}=\left\{  0,1,2,\ldots\right\}  $) for the entire Section
\ref{sect.bir}.

The following is obvious:

\begin{lemma}
\label{lem.aprod}If $a\in\mathbb{K}^{n}$ is any $n$-tuple, then $a_{k+1}%
a_{k+2}\cdots a_{k+n}=a_{1}a_{2}\cdots a_{n}$ for each $k\in\mathbb{Z}$.
\end{lemma}

\begin{verlong}
\begin{proof}
[Proof of Lemma \ref{lem.aprod}.]Let $a\in\mathbb{K}^{n}$ be an $n$-tuple. For
each $k\in\mathbb{Z}$, we set $b_{k}=a_{k+1}a_{k+2}\cdots a_{k+n}$. Now, it is
easy to see that%
\begin{equation}
b_{p}=b_{p+1}\ \ \ \ \ \ \ \ \ \ \text{for each }p\in\mathbb{Z}
\label{pf.lem.aprod.2}%
\end{equation}
\footnote{\textit{Proof of (\ref{pf.lem.aprod.2}):} Let $p\in\mathbb{Z}$.
Recall that the family $\left(  a_{i}\right)  _{i\in\mathbb{Z}}$ is
$n$-periodic (by Convention \ref{conv.bir.peri}). In other words, we have%
\[
a_{j}=a_{j^{\prime}}\text{ whenever }j\text{ and }j^{\prime}\text{ are two
integers satisfying }j\equiv j^{\prime}\operatorname{mod}n.
\]
Applying this to $j=p+1$ and $j^{\prime}=p+n+1$, we obtain $a_{p+1}=a_{p+n+1}$
(since $p+1\equiv p+n+1\operatorname{mod}n$).
\par
The definition of $b_{p}$ yields
\[
b_{p}=a_{p+1}a_{p+2}\cdots a_{p+n}=\underbrace{a_{p+1}}_{=a_{p+n+1}}\left(
a_{p+2}a_{p+3}\cdots a_{p+n}\right)  =a_{p+n+1}\left(  a_{p+2}a_{p+3}\cdots
a_{p+n}\right)  .
\]
The definition of $b_{p+1}$ yields%
\begin{align*}
b_{p+1}  &  =a_{\left(  p+1\right)  +1}a_{\left(  p+1\right)  +2}\cdots
a_{\left(  p+1\right)  +n}=a_{p+2}a_{p+3}\cdots a_{p+n+1}\\
&  =\left(  a_{p+2}a_{p+3}\cdots a_{p+n}\right)  a_{p+n+1}=a_{p+n+1}\left(
a_{p+2}a_{p+3}\cdots a_{p+n}\right)  .
\end{align*}
Comparing these two equalities, we obtain $b_{p}=b_{p+1}$. This proves
(\ref{pf.lem.aprod.2}).}. In other words,%
\[
\cdots=b_{-2}=b_{-1}=b_{0}=b_{1}=b_{2}=\cdots.
\]
In other words, all of the elements $\ldots,b_{-2},b_{-1},b_{0},b_{1}%
,b_{2},\ldots$ are equal. Hence, $b_{k}=b_{0}$ for each $k\in\mathbb{Z}$.
Thus, for each $k\in\mathbb{Z}$, we have%
\begin{align*}
a_{k+1}a_{k+2}\cdots a_{k+n}  &  =b_{k}\ \ \ \ \ \ \ \ \ \ \left(  \text{since
}b_{k}=a_{k+1}a_{k+2}\cdots a_{k+n}\right) \\
&  =b_{0}=a_{0+1}a_{0+2}\cdots a_{0+n}\ \ \ \ \ \ \ \ \ \ \left(  \text{by the
definition of }b_{0}\right) \\
&  =a_{1}a_{2}\cdots a_{n}.
\end{align*}
This proves Lemma \ref{lem.aprod}.
\end{proof}
\end{verlong}

\begin{definition}
\label{def.f}We define a map $\mathbf{f}_{u}:\mathbb{K}^{n}\rightarrow
\mathbb{K}^{n}$ as follows:

Let $x\in\mathbb{K}^{n}$ be an $n$-tuple. For each $j\in\mathbb{Z}$ and
$r\in\mathbb{N}$, define an element $t_{r,j}\in\mathbb{K}$ by
\[
t_{r,j}=\sum_{k=0}^{r}\underbrace{x_{j+1}x_{j+2}\cdots x_{j+k}}_{=\prod
_{i=1}^{k}x_{j+i}}\cdot\underbrace{u_{j+k+1}u_{j+k+2}\cdots u_{j+r}}%
_{=\prod_{i=k+1}^{r}u_{j+i}}.
\]
Define $y\in\mathbb{K}^{n}$ by setting%
\[
y_{i}=u_{i}\cdot\dfrac{u_{i-1}t_{n-1,i-1}}{x_{i+1}t_{n-1,i+1}}%
\ \ \ \ \ \ \ \ \ \ \text{for each }i\in\left\{  1,2,\ldots,n\right\}  .
\]
Set $\mathbf{f}_{u}\left(  x\right)  =y$.
\end{definition}

\begin{example}
\label{exa.f.4}Set $n=4$ for this example. Let $x\in\mathbb{K}^{n}$ be an
$n$-tuple; thus, $x=\left(  x_{1},x_{2},x_{3},x_{4}\right)  $. Let us see what
the definition of $\mathbf{f}_{u}\left(  x\right)  $ in Definition \ref{def.f}
boils down to in this case.

Let us first compute the elements $t_{n-1,j}=t_{3,j}$ from Definition
\ref{def.f}. The definition of $t_{3,0}$ yields%
\begin{align*}
t_{3,0}  &  =\sum_{k=0}^{3}x_{0+1}x_{0+2}\cdots x_{0+k}\cdot u_{0+k+1}%
u_{0+k+2}\cdots u_{0+3}\\
&  =\sum_{k=0}^{3}x_{1}x_{2}\cdots x_{k}\cdot u_{k+1}u_{k+2}\cdots u_{3}\\
&  =u_{1}u_{2}u_{3}+x_{1}u_{2}u_{3}+x_{1}x_{2}u_{3}+x_{1}x_{2}x_{3}.
\end{align*}
Similarly,%
\begin{align*}
t_{3,1}  &  =u_{2}u_{3}u_{4}+x_{2}u_{3}u_{4}+x_{2}x_{3}u_{4}+x_{2}x_{3}%
x_{4};\\
t_{3,2}  &  =u_{3}u_{4}u_{5}+x_{3}u_{4}u_{5}+x_{3}x_{4}u_{5}+x_{3}x_{4}x_{5}\\
&  =u_{3}u_{4}u_{1}+x_{3}u_{4}u_{1}+x_{3}x_{4}u_{1}+x_{3}x_{4}x_{1}\\
&  \ \ \ \ \ \ \ \ \ \ \left(  \text{since }u_{5}=u_{1}\text{ and }x_{5}%
=x_{1}\right)  ;\\
t_{3,3}  &  =u_{4}u_{5}u_{6}+x_{4}u_{5}u_{6}+x_{4}x_{5}u_{6}+x_{4}x_{5}x_{6}\\
&  =u_{4}u_{1}u_{2}+x_{4}u_{1}u_{2}+x_{4}x_{1}u_{2}+x_{4}x_{1}x_{2}\\
&  \ \ \ \ \ \ \ \ \ \ \left(  \text{since }u_{5}=u_{1}\text{ and }x_{5}%
=x_{1}\text{ and }u_{6}=u_{2}\text{ and }x_{6}=x_{2}\right)  .
\end{align*}
We don't need to compute any further $t_{3,j}$'s, since we can easily see that%
\begin{equation}
t_{3,j}=t_{3,j^{\prime}}\text{ for any integers }j\text{ and }j^{\prime}\text{
satisfying }j\equiv j^{\prime}\operatorname{mod}4. \label{eq.exa.f.4.3}%
\end{equation}
Thus, in particular, $t_{3,4}=t_{3,0}$ and $t_{3,5}=t_{3,1}$.

Now, let us compute the $4$-tuple $y\in\mathbb{K}^{n}=\mathbb{K}^{4}$ from
Definition \ref{def.f}. By its definition, we have%
\begin{align*}
y_{1}  &  =u_{1}\cdot\dfrac{u_{1-1}t_{3,1-1}}{x_{1+1}t_{3,1+1}}=u_{1}%
\cdot\dfrac{u_{0}t_{3,0}}{x_{2}t_{3,2}}=u_{1}\cdot\dfrac{u_{4}t_{3,0}}%
{x_{2}t_{3,2}}\\
&  \ \ \ \ \ \ \ \ \ \ \left(  \text{since }u_{0}=u_{4}\right) \\
&  =u_{1}\cdot\dfrac{u_{4}\left(  u_{1}u_{2}u_{3}+x_{1}u_{2}u_{3}+x_{1}%
x_{2}u_{3}+x_{1}x_{2}x_{3}\right)  }{x_{2}\left(  u_{3}u_{4}u_{1}+x_{3}%
u_{4}u_{1}+x_{3}x_{4}u_{1}+x_{3}x_{4}x_{1}\right)  }%
\end{align*}
(by our formulas for $t_{3,0}$ and $t_{3,2}$). Similar computations lead to%
\begin{align*}
y_{2}  &  =u_{2}\cdot\dfrac{u_{1}\left(  u_{2}u_{3}u_{4}+x_{2}u_{3}u_{4}%
+x_{2}x_{3}u_{4}+x_{2}x_{3}x_{4}\right)  }{x_{3}\left(  u_{4}u_{1}u_{2}%
+x_{4}u_{1}u_{2}+x_{4}x_{1}u_{2}+x_{4}x_{1}x_{2}\right)  };\\
y_{3}  &  =u_{3}\cdot\dfrac{u_{2}\left(  u_{3}u_{4}u_{1}+x_{3}u_{4}u_{1}%
+x_{3}x_{4}u_{1}+x_{3}x_{4}x_{1}\right)  }{x_{4}\left(  u_{1}u_{2}u_{3}%
+x_{1}u_{2}u_{3}+x_{1}x_{2}u_{3}+x_{1}x_{2}x_{3}\right)  };\\
y_{4}  &  =u_{4}\cdot\dfrac{u_{3}\left(  u_{4}u_{1}u_{2}+x_{4}u_{1}u_{2}%
+x_{4}x_{1}u_{2}+x_{4}x_{1}x_{2}\right)  }{x_{1}\left(  u_{2}u_{3}u_{4}%
+x_{2}u_{3}u_{4}+x_{2}x_{3}u_{4}+x_{2}x_{3}x_{4}\right)  }.
\end{align*}
Of course, knowing one of these four equalities is enough; the expression for
$y_{i+1}$ is obtained from the expression for $y_{i}$ by shifting all indices
(other than the \textquotedblleft$3$\textquotedblright s that were originally
\textquotedblleft$n-1$\textquotedblright s) forward by $1$.
\end{example}

\begin{remark}
Instead of assuming $\mathbb{K}$ to be a semifield, we could have assumed that
$\mathbb{K}$ is an infinite field. In that case, the $\mathbf{f}_{u}$ in
Definition \ref{def.f} would be a birational map instead of a map in the usual
sense of this word, since the denominators $x_{i+1}t_{n-1,i+1}$ in the
definition of $y$ can be zero. Everything we say below about $\mathbf{f}_{u}$
would nevertheless still hold on the level of birational maps.
\end{remark}

The map $\mathbf{f}_{u}$ we just defined has the following properties:

\begin{theorem}
\label{thm.f.full}\ \ 

\begin{enumerate}
\item[\textbf{(a)}] The map $\mathbf{f}_{u}$ is an involution (i.e., we have
$\mathbf{f}_{u}\circ\mathbf{f}_{u}=\operatorname*{id}$).

\item[\textbf{(b)}] Let $x\in\mathbb{K}^{n}$ and $y\in\mathbb{K}^{n}$ be such
that $y=\mathbf{f}_{u}\left(  x\right)  $. Then,%
\[
y_{1}y_{2}\cdots y_{n}\cdot x_{1}x_{2}\cdots x_{n}=\left(  u_{1}u_{2}\cdots
u_{n}\right)  ^{2}.
\]

\item[\textbf{(c)}] Let $x\in\mathbb{K}^{n}$ and $y\in\mathbb{K}^{n}$ be such
that $y=\mathbf{f}_{u}\left(  x\right)  $. Then,%
\[
\left(  u_{i}+x_{i}\right)  \left(  \dfrac{1}{u_{i+1}}+\dfrac{1}{x_{i+1}%
}\right)  =\left(  u_{i}+y_{i}\right)  \left(  \dfrac{1}{u_{i+1}}+\dfrac
{1}{y_{i+1}}\right)
\]
for each $i\in\mathbb{Z}$.

\item[\textbf{(d)}] Let $x\in\mathbb{K}^{n}$ and $y\in\mathbb{K}^{n}$ be such
that $y=\mathbf{f}_{u}\left(  x\right)  $. Then,%
\[
\prod_{i=1}^{n}\dfrac{u_{i}+x_{i}}{x_{i}}=\prod_{i=1}^{n}\dfrac{u_{i}+y_{i}%
}{u_{i}}.
\]

\end{enumerate}
\end{theorem}

Theorem \ref{thm.f.full} will be crucial for us; but before we can prove it,
we will need a few lemmas.

\begin{lemma}
\label{lem.f.steps}Let $x\in\mathbb{K}^{n}$ be an $n$-tuple. Let $t_{r,j}$ and
$y$ be as in Definition \ref{def.f}. Then:

\begin{enumerate}
\item[\textbf{(a)}] We have $t_{r,j}=t_{r,j^{\prime}}$ for any $r\in
\mathbb{N}$ and any two integers $j$ and $j^{\prime}$ satisfying $j\equiv
j^{\prime}\operatorname{mod}n$. In other words, for each $r\in\mathbb{N}$, the
family $\left(  t_{r,j}\right)  _{j\in\mathbb{Z}}$ is $n$-periodic.

\item[\textbf{(b)}] We have $t_{0,j}=1$ for each $j\in\mathbb{Z}$.

\item[\textbf{(c)}] For each $r\in\mathbb{N}$ and $j\in\mathbb{Z}$, we have%
\[
x_{j}t_{r,j}+u_{j}u_{j+1}\cdots u_{j+r}=t_{r+1,j-1}.
\]

\item[\textbf{(d)}] For each $r\in\mathbb{N}$ and $j\in\mathbb{Z}$, we have%
\[
u_{j+r+1}t_{r,j}+x_{j+1}x_{j+2}\cdots x_{j+r+1}=t_{r+1,j}.
\]

\item[\textbf{(e)}] For each $a\in\mathbb{Z}$ and $b\in\mathbb{Z}$, we have%
\[
x_{a}t_{n-1,a}+u_{b-1}t_{n-1,b-1}=x_{b}t_{n-1,b}+u_{a-1}t_{n-1,a-1}.
\]

\item[\textbf{(f)}] For each $i\in\mathbb{Z}$, we have%
\[
x_{i+1}t_{n-1,i+1}+u_{i-1}t_{n-1,i-1}=\left(  x_{i}+u_{i}\right)  t_{n-1,i}.
\]

\item[\textbf{(g)}] For each $j\in\mathbb{Z}$ and each positive integer $q$,
we have%
\[
t_{n-1,j+q+1}\cdot x_{j+2}x_{j+3}\cdots x_{j+q+1}+u_{j}t_{n-1,j}%
t_{q-1,j+1}=t_{n-1,j+1}t_{q,j}.
\]

\item[\textbf{(h)}] For each $i\in\mathbb{Z}$, we have%
\[
y_{i}=u_{i}\cdot\dfrac{u_{i-1}t_{n-1,i-1}}{x_{i+1}t_{n-1,i+1}}.
\]

\end{enumerate}

Now, for each $j\in\mathbb{Z}$ and $r\in\mathbb{N}$, let us define an element
$t_{r,j}^{\prime}\in\mathbb{K}$ by
\[
t_{r,j}^{\prime}=\sum_{k=0}^{r}\underbrace{y_{j+1}y_{j+2}\cdots y_{j+k}%
}_{=\prod_{i=1}^{k}y_{j+i}}\cdot\underbrace{u_{j+k+1}u_{j+k+2}\cdots u_{j+r}%
}_{=\prod_{i=k+1}^{r}u_{j+i}}.
\]
(This is precisely how $t_{r,j}$ was defined, except that we are using $y$ in
place of $x$ now.) Then:

\begin{enumerate}
\item[\textbf{(i)}] For each $j\in\mathbb{Z}$ and $q\in\mathbb{N}$, we have%
\[
\dfrac{t_{q,j}^{\prime}}{u_{j+1}u_{j+2}\cdots u_{j+q}}=\dfrac{t_{n-1,j+1}%
}{t_{n-1,j+q+1}}\cdot\dfrac{t_{q,j}}{x_{j+2}x_{j+3}\cdots x_{j+q+1}}.
\]

\item[\textbf{(j)}] For each $j\in\mathbb{Z}$, we have%
\[
\dfrac{t_{n-1,j}^{\prime}u_{j}}{u_{1}u_{2}\cdots u_{n}}=\dfrac{t_{n-1,j+1}%
x_{j+1}}{x_{1}x_{2}\cdots x_{n}}.
\]

\item[\textbf{(k)}] For each $i\in\mathbb{Z}$, we have%
\[
x_{i}=u_{i}\cdot\dfrac{u_{i-1}t_{n-1,i-1}^{\prime}}{y_{i+1}t_{n-1,i+1}%
^{\prime}}.
\]

\end{enumerate}
\end{lemma}

\begin{proof}
[Proof of Lemma \ref{lem.f.steps}.]The proof of this lemma is long but
unsophisticated: Each part follows by rather straightforward computations
(and, in the cases of parts \textbf{(g)} and \textbf{(i)}, an induction on
$q$) from the previously proven parts. We shall show the details, but a
computationally inclined reader may have a better time reconstructing them
independently.\footnote{We note that the hardest parts of the proof -- namely,
the proofs of parts \textbf{(g)}, \textbf{(i)}, \textbf{(j)} and \textbf{(k)}
-- can be sidestepped entirely, as these parts will only be used in the proof
of Theorem \ref{thm.f.full} \textbf{(a)}, but we will give an alternative
proof of Theorem \ref{thm.f.full} \textbf{(a)} later on (in Remark
\ref{rmk.f.invol-by-trick}), which avoids using them.}

\begin{vershort}
\textbf{(a)} Let $r\in\mathbb{N}$. The definition of $t_{r,j}$ shows that
$t_{r,j}=t_{r,j+n}$ for each $j\in\mathbb{Z}$ (since each $i\in\mathbb{Z}$
satisfies $u_{i}=u_{i+n}$ and $x_{i}=x_{i+n}$). Thus, the family $\left(
t_{r,j}\right)  _{j\in\mathbb{Z}}$ is $n$-periodic. This yields the claim of
Lemma \ref{lem.f.steps} \textbf{(a)}.
\end{vershort}

\begin{verlong}
Let $r\in\mathbb{N}$. Let $j\in\mathbb{Z}$. We shall show that $t_{r,j}%
=t_{r,j+n}$.

Indeed, Convention \ref{conv.bir.peri} yields that the family $\left(
u_{i}\right)  _{i\in\mathbb{Z}}$ is $n$-periodic. Thus, $u_{i}=u_{i+n}$ for
each $i\in\mathbb{Z}$. In other words,%
\begin{equation}
u_{i+n}=u_{i}\ \ \ \ \ \ \ \ \ \ \text{for each }i\in\mathbb{Z}.
\label{pf.lem.f.steps.a.uper}%
\end{equation}
Likewise,
\begin{equation}
x_{i+n}=x_{i}\ \ \ \ \ \ \ \ \ \ \text{for each }i\in\mathbb{Z}.
\label{pf.lem.f.steps.a.xper}%
\end{equation}

Now, the definition of $t_{r,j}$ yields%
\begin{align*}
t_{r,j}  &  =\sum_{k=0}^{r}\underbrace{x_{j+1}x_{j+2}\cdots x_{j+k}}%
_{=\prod_{i=1}^{k}x_{j+i}}\cdot\underbrace{u_{j+k+1}u_{j+k+2}\cdots u_{j+r}%
}_{=\prod_{i=k+1}^{r}u_{j+i}}\\
&  =\sum_{k=0}^{r}\left(  \prod_{i=1}^{k}x_{j+i}\right)  \cdot\left(
\prod_{i=k+1}^{r}u_{j+i}\right)  .
\end{align*}
The same argument (applied to $j+n$ instead of $j$) yields%
\begin{align*}
t_{r,j+n}  &  =\sum_{k=0}^{r}\left(  \prod_{i=1}^{k}\underbrace{x_{j+n+i}%
}_{\substack{=x_{\left(  j+i\right)  +n}\\=x_{j+i}\\\text{(by
(\ref{pf.lem.f.steps.a.xper}),}\\\text{applied to }j+i\text{ instead of
}i\text{)}}}\right)  \cdot\left(  \prod_{i=k+1}^{r}\underbrace{u_{j+n+i}%
}_{\substack{=u_{\left(  j+i\right)  +n}\\=u_{j+i}\\\text{(by
(\ref{pf.lem.f.steps.a.uper}),}\\\text{applied to }j+i\text{ instead of
}i\text{)}}}\right) \\
&  =\sum_{k=0}^{r}\left(  \prod_{i=1}^{k}x_{j+i}\right)  \cdot\left(
\prod_{i=k+1}^{r}u_{j+i}\right)  .
\end{align*}
Comparing these two equalities, we obtain $t_{r,j}=t_{r,j+n}$.

Now, forget that we fixed $j$. We thus have shown that $t_{r,j}=t_{r,j+n}$ for
each $j\in\mathbb{Z}$. In other words, the family $\left(  t_{r,j}\right)
_{j\in\mathbb{Z}}$ is $n$-periodic. In other words, we have $t_{r,j}%
=t_{r,j^{\prime}}$ for any two integers $j$ and $j^{\prime}$ satisfying
$j\equiv j^{\prime}\operatorname{mod}n$. This proves Lemma \ref{lem.f.steps}
\textbf{(a)}.
\end{verlong}

\begin{vershort}
\textbf{(b)} Trivial consequence of the definition of $t_{0,j}$.
\end{vershort}

\begin{verlong}
\textbf{(b)} Let $j\in\mathbb{Z}$. The definition of $t_{0,j}$ yields%
\begin{align*}
t_{0,j}  &  =\sum_{k=0}^{0}x_{j+1}x_{j+2}\cdots x_{j+k}\cdot u_{j+k+1}%
u_{j+k+2}\cdots u_{j+0}\\
&  =\underbrace{x_{j+1}x_{j+2}\cdots x_{j+0}}_{=\left(  \text{empty
product}\right)  =1}\cdot\underbrace{u_{j+0+1}u_{j+0+2}\cdots u_{j+0}%
}_{=\left(  \text{empty product}\right)  =1}\\
&  =1.
\end{align*}
This proves Lemma \ref{lem.f.steps} \textbf{(b)}.
\end{verlong}

\begin{vershort}
\textbf{(c)} Let $r\in\mathbb{N}$ and $j\in\mathbb{Z}$. Then, the definition
of $t_{r+1,j-1}$ yields%
\begin{align*}
&  t_{r+1,j-1}\\
&  =\sum_{k=0}^{r+1}x_{j}x_{j+1}\cdots x_{j+k-1}\cdot u_{j+k}u_{j+k+1}\cdots
u_{j+r}\\
&  =\underbrace{x_{j}x_{j+1}\cdots x_{j-1}}_{=\left(  \text{empty
product}\right)  =1}\cdot u_{j}u_{j+1}\cdots u_{j+r}+\sum_{k=1}^{r+1}%
\underbrace{x_{j}x_{j+1}\cdots x_{j+k-1}}_{=x_{j}\cdot x_{j+1}x_{j+2}\cdots
x_{j+k-1}}\cdot u_{j+k}u_{j+k+1}\cdots u_{j+r}\\
&  \ \ \ \ \ \ \ \ \ \ \left(  \text{here, we have split off the addend for
}k=0\text{ from the sum}\right) \\
&  =u_{j}u_{j+1}\cdots u_{j+r}+\underbrace{\sum_{k=1}^{r+1}x_{j}\cdot
x_{j+1}x_{j+2}\cdots x_{j+k-1}\cdot u_{j+k}u_{j+k+1}\cdots u_{j+r}}%
_{=x_{j}\sum_{k=1}^{r+1}x_{j+1}x_{j+2}\cdots x_{j+k-1}\cdot u_{j+k}%
u_{j+k+1}\cdots u_{j+r}}\\
&  =u_{j}u_{j+1}\cdots u_{j+r}+x_{j}\underbrace{\sum_{k=1}^{r+1}x_{j+1}%
x_{j+2}\cdots x_{j+k-1}\cdot u_{j+k}u_{j+k+1}\cdots u_{j+r}}_{\substack{=\sum
_{k=0}^{r}x_{j+1}x_{j+2}\cdots x_{j+k}\cdot u_{j+k+1}u_{j+k+2}\cdots
u_{j+r}\\\text{(here, we have substituted }k\text{ for }k-1\text{ in the
sum)}}}\\
&  =u_{j}u_{j+1}\cdots u_{j+r}+x_{j}\underbrace{\sum_{k=0}^{r}x_{j+1}%
x_{j+2}\cdots x_{j+k}\cdot u_{j+k+1}u_{j+k+2}\cdots u_{j+r}}%
_{\substack{=t_{r,j}\\\text{(by the definition of }t_{r,j}\text{)}}}\\
&  =u_{j}u_{j+1}\cdots u_{j+r}+x_{j}t_{r,j}=x_{j}t_{r,j}+u_{j}u_{j+1}\cdots
u_{j+r}.
\end{align*}
This proves Lemma \ref{lem.f.steps} \textbf{(c)}.
\end{vershort}

\begin{verlong}
\textbf{(c)} Let $r\in\mathbb{N}$ and $j\in\mathbb{Z}$. Then, the definition
of $t_{r+1,j-1}$ yields%
\begin{align*}
&  t_{r+1,j-1}\\
&  =\sum_{k=0}^{r+1}\underbrace{x_{\left(  j-1\right)  +1}x_{\left(
j-1\right)  +2}\cdots x_{\left(  j-1\right)  +k}}_{=x_{j}x_{j+1}\cdots
x_{j+k-1}}\cdot\underbrace{u_{\left(  j-1\right)  +\left(  k+1\right)
}u_{\left(  j-1\right)  +\left(  k+2\right)  }\cdots u_{\left(  j-1\right)
+\left(  r+1\right)  }}_{=u_{j+k}u_{j+k+1}\cdots u_{j+r}}\\
&  =\sum_{k=0}^{r+1}x_{j}x_{j+1}\cdots x_{j+k-1}\cdot u_{j+k}u_{j+k+1}\cdots
u_{j+r}\\
&  =\underbrace{x_{j}x_{j+1}\cdots x_{j+0-1}}_{=\left(  \text{empty
product}\right)  =1}\cdot\underbrace{u_{j+0}u_{j+0+1}\cdots u_{j+r}}%
_{=u_{j}u_{j+1}\cdots u_{j+r}}+\sum_{k=1}^{r+1}\underbrace{x_{j}x_{j+1}\cdots
x_{j+k-1}}_{=x_{j}\cdot x_{j+1}x_{j+2}\cdots x_{j+k-1}}\cdot u_{j+k}%
u_{j+k+1}\cdots u_{j+r}\\
&  \ \ \ \ \ \ \ \ \ \ \left(  \text{here, we have split off the addend for
}k=0\text{ from the sum}\right) \\
&  =u_{j}u_{j+1}\cdots u_{j+r}+\underbrace{\sum_{k=1}^{r+1}x_{j}\cdot
x_{j+1}x_{j+2}\cdots x_{j+k-1}\cdot u_{j+k}u_{j+k+1}\cdots u_{j+r}}%
_{=x_{j}\sum_{k=1}^{r+1}x_{j+1}x_{j+2}\cdots x_{j+k-1}\cdot u_{j+k}%
u_{j+k+1}\cdots u_{j+r}}\\
&  =u_{j}u_{j+1}\cdots u_{j+r}+x_{j}\underbrace{\sum_{k=1}^{r+1}x_{j+1}%
x_{j+2}\cdots x_{j+k-1}\cdot u_{j+k}u_{j+k+1}\cdots u_{j+r}}_{\substack{=\sum
_{k=0}^{r}x_{j+1}x_{j+2}\cdots x_{j+k}\cdot u_{j+k+1}u_{j+k+2}\cdots
u_{j+r}\\\text{(here, we have substituted }k\text{ for }k-1\text{ in the
sum)}}}\\
&  =u_{j}u_{j+1}\cdots u_{j+r}+x_{j}\underbrace{\sum_{k=0}^{r}x_{j+1}%
x_{j+2}\cdots x_{j+k}\cdot u_{j+k+1}u_{j+k+2}\cdots u_{j+r}}%
_{\substack{=t_{r,j}\\\text{(by the definition of }t_{r,j}\text{)}}}\\
&  =u_{j}u_{j+1}\cdots u_{j+r}+x_{j}t_{r,j}=x_{j}t_{r,j}+u_{j}u_{j+1}\cdots
u_{j+r}.
\end{align*}
This proves Lemma \ref{lem.f.steps} \textbf{(c)}.
\end{verlong}

\textbf{(d)} Let $r\in\mathbb{N}$ and $j\in\mathbb{Z}$. Then, the definition
of $t_{r+1,j}$ yields%
\begin{align*}
t_{r+1,j}  &  =\sum_{k=0}^{r+1}x_{j+1}x_{j+2}\cdots x_{j+k}\cdot
u_{j+k+1}u_{j+k+2}\cdots u_{j+r+1}\\
&  =\sum_{k=0}^{r}x_{j+1}x_{j+2}\cdots x_{j+k}\cdot\underbrace{u_{j+k+1}%
u_{j+k+2}\cdots u_{j+r+1}}_{=u_{j+k+1}u_{j+k+2}\cdots u_{j+r}\cdot u_{j+r+1}%
}\\
&  \ \ \ \ \ \ \ \ \ \ +x_{j+1}x_{j+2}\cdots x_{j+r+1}\cdot
\underbrace{u_{j+\left(  r+1\right)  +1}u_{j+\left(  r+1\right)  +2}\cdots
u_{j+r+1}}_{=\left(  \text{empty product}\right)  =1}\\
&  \ \ \ \ \ \ \ \ \ \ \left(  \text{here, we have split off the addend for
}k=r+1\text{ from the sum}\right) \\
&  =\underbrace{\sum_{k=0}^{r}x_{j+1}x_{j+2}\cdots x_{j+k}\cdot u_{j+k+1}%
u_{j+k+2}\cdots u_{j+r}\cdot u_{j+r+1}}_{=u_{j+r+1}\sum_{k=0}^{r}%
x_{j+1}x_{j+2}\cdots x_{j+k}\cdot u_{j+k+1}u_{j+k+2}\cdots u_{j+r}}%
+x_{j+1}x_{j+2}\cdots x_{j+r+1}\\
&  =u_{j+r+1}\underbrace{\sum_{k=0}^{r}x_{j+1}x_{j+2}\cdots x_{j+k}\cdot
u_{j+k+1}u_{j+k+2}\cdots u_{j+r}}_{\substack{=t_{r,j}\\\text{(by the
definition of }t_{r,j}\text{)}}}+x_{j+1}x_{j+2}\cdots x_{j+r+1}\\
&  =u_{j+r+1}t_{r,j}+x_{j+1}x_{j+2}\cdots x_{j+r+1}.
\end{align*}
This proves Lemma \ref{lem.f.steps} \textbf{(d)}.

\textbf{(e)} We WLOG assume that $n\neq1$, since otherwise the claim is easy
to check by hand. Thus, $n\geq2$, so that $n-2\in\mathbb{N}$.

Let $a\in\mathbb{Z}$ and $b\in\mathbb{Z}$. Then, Lemma \ref{lem.f.steps}
\textbf{(c)} (applied to $r=n-2$ and $j=a$) yields%
\[
x_{a}t_{n-2,a}+u_{a}u_{a+1}\cdots u_{a+n-2}=t_{\left(  n-2\right)
+1,a-1}=t_{n-1,a-1}%
\]
(since $\left(  n-2\right)  +1=n-1$). Multiplying both sides of this equality
by $u_{a-1}$, we obtain%
\[
u_{a-1}\left(  x_{a}t_{n-2,a}+u_{a}u_{a+1}\cdots u_{a+n-2}\right)
=u_{a-1}t_{n-1,a-1}.
\]
Hence,%
\begin{align}
u_{a-1}t_{n-1,a-1}  &  =u_{a-1}\left(  x_{a}t_{n-2,a}+u_{a}u_{a+1}\cdots
u_{a+n-2}\right) \nonumber\\
&  =\underbrace{u_{a-1}x_{a}}_{=x_{a}u_{a-1}}t_{n-2,a}+\underbrace{u_{a-1}%
\cdot u_{a}u_{a+1}\cdots u_{a+n-2}}_{\substack{=u_{a-1}u_{a}\cdots
u_{a+n-2}\\=u_{\left(  a-2\right)  +1}u_{\left(  a-2\right)  +2}\cdots
u_{\left(  a-2\right)  +n}\\=u_{1}u_{2}\cdots u_{n}\\\text{(by Lemma
\ref{lem.aprod})}}}\nonumber\\
&  =x_{a}u_{a-1}t_{n-2,a}+u_{1}u_{2}\cdots u_{n}. \label{pf.lem.f.steps.e.1}%
\end{align}

Also, Lemma \ref{lem.f.steps} \textbf{(d)} (applied to $r=n-2$ and $j=b$)
yields%
\[
u_{b+\left(  n-2\right)  +1}t_{n-2,b}+x_{b+1}x_{b+2}\cdots x_{b+\left(
n-2\right)  +1}=t_{\left(  n-2\right)  +1,b}.
\]
In view of $\left(  n-2\right)  +1=n-1$, this rewrites as%
\[
u_{b+n-1}t_{n-2,b}+x_{b+1}x_{b+2}\cdots x_{b+n-1}=t_{n-1,b}.
\]
Multiplying both sides of this equality by $x_{b}$, we obtain%
\[
x_{b}\left(  u_{b+n-1}t_{n-2,b}+x_{b+1}x_{b+2}\cdots x_{b+n-1}\right)
=x_{b}t_{n-1,b},
\]
so that%
\begin{align*}
x_{b}t_{n-1,b}  &  =x_{b}\left(  u_{b+n-1}t_{n-2,b}+x_{b+1}x_{b+2}\cdots
x_{b+n-1}\right) \\
&  =x_{b}\underbrace{u_{b+n-1}}_{=u_{\left(  b-1\right)  +n}=u_{b-1}}%
t_{n-2,b}+\underbrace{x_{b}\cdot x_{b+1}x_{b+2}\cdots x_{b+n-1}}%
_{\substack{=x_{b}x_{b+1}\cdots x_{b+n-1}\\=x_{\left(  b-1\right)
+1}x_{\left(  b-1\right)  +2}\cdots x_{\left(  b-1\right)  +n}\\=x_{1}%
x_{2}\cdots x_{n}\\\text{(by Lemma \ref{lem.aprod})}}}\\
&  =x_{b}u_{b-1}t_{n-2,b}+x_{1}x_{2}\cdots x_{n}.
\end{align*}
Adding (\ref{pf.lem.f.steps.e.1}) to this equality, we obtain%
\begin{align}
&  x_{b}t_{n-1,b}+u_{a-1}t_{n-1,a-1}\nonumber\\
&  =x_{b}u_{b-1}t_{n-2,b}+x_{1}x_{2}\cdots x_{n}+x_{a}u_{a-1}t_{n-2,a}%
+u_{1}u_{2}\cdots u_{n}\label{pf.lem.f.steps.e.5}\\
&  =x_{a}u_{a-1}t_{n-2,a}+x_{1}x_{2}\cdots x_{n}+x_{b}u_{b-1}t_{n-2,b}%
+u_{1}u_{2}\cdots u_{n}.\nonumber
\end{align}
The same argument (applied to $b$ and $a$ instead of $a$ and $b$) yields%
\begin{align*}
&  x_{a}t_{n-1,a}+u_{b-1}t_{n-1,b-1}\\
&  =x_{b}u_{b-1}t_{n-2,b}+x_{1}x_{2}\cdots x_{n}+x_{a}u_{a-1}t_{n-2,a}%
+u_{1}u_{2}\cdots u_{n}.
\end{align*}
Comparing this with (\ref{pf.lem.f.steps.e.5}), we obtain $x_{a}%
t_{n-1,a}+u_{b-1}t_{n-1,b-1}=x_{b}t_{n-1,b}+u_{a-1}t_{n-1,a-1}$. This proves
Lemma \ref{lem.f.steps} \textbf{(e)}.

\textbf{(f)} Applying Lemma \ref{lem.f.steps} \textbf{(e)} to $a=i+1$ and
$b=i$, we obtain%
\begin{align*}
x_{i+1}t_{n-1,i+1}+u_{i-1}t_{n-1,i-1}  &  =x_{i}t_{n-1,i}%
+\underbrace{u_{i+1-1}}_{=u_{i}}\underbrace{t_{n-1,i+1-1}}_{=t_{n-1,i}}\\
&  =x_{i}t_{n-1,i}+u_{i}t_{n-1,i}=\left(  x_{i}+u_{i}\right)  t_{n-1,i}.
\end{align*}
This proves Lemma \ref{lem.f.steps} \textbf{(f)}.

\textbf{(g)} We shall prove Lemma \ref{lem.f.steps} \textbf{(g)} by induction
on $q$:

\textit{Induction base:} Let us show that Lemma \ref{lem.f.steps} \textbf{(g)}
holds for $q=1$.

\begin{vershort}
Indeed, let $j\in\mathbb{Z}$. The definition of $t_{1,j}$ yields%
\[
t_{1,j}=\sum_{k=0}^{1}x_{j+1}x_{j+2}\cdots x_{j+k}\cdot u_{j+k+1}%
u_{j+k+2}\cdots u_{j+1}=x_{j+1}+u_{j+1}.
\]
Hence,%
\[
t_{n-1,j+1}\underbrace{t_{1,j}}_{=x_{j+1}+u_{j+1}}=t_{n-1,j+1}\left(
x_{j+1}+u_{j+1}\right)  =\left(  x_{j+1}+u_{j+1}\right)  t_{n-1,j+1}.
\]
Comparing this with%
\begin{align*}
&  \underbrace{t_{n-1,j+1+1}}_{=t_{n-1,j+2}}\cdot\underbrace{x_{j+2}%
x_{j+3}\cdots x_{j+1+1}}_{\substack{=x_{j+2}x_{j+3}\cdots x_{j+2}\\=x_{j+2}%
}}+u_{j}t_{n-1,j}\underbrace{t_{1-1,j+1}}_{\substack{=t_{0,j+1}=1\\\text{(by
Lemma \ref{lem.f.steps} \textbf{(b)})}}}\\
&  =t_{n-1,j+2}x_{j+2}+u_{j}t_{n-1,j}=x_{j+2}t_{n-1,j+2}+u_{j}t_{n-1,j}%
=\left(  x_{j+1}+u_{j+1}\right)  t_{n-1,j+1}\\
&  \ \ \ \ \ \ \ \ \ \ \left(  \text{by Lemma \ref{lem.f.steps} \textbf{(f)},
applied to }i=j+1\right)  ,
\end{align*}
we obtain
\[
t_{n-1,j+1+1}\cdot x_{j+2}x_{j+3}\cdots x_{j+1+1}+u_{j}t_{n-1,j}%
t_{1-1,j+1}=t_{n-1,j+1}t_{1,j}.
\]

\end{vershort}

\begin{verlong}
Indeed, let $j\in\mathbb{Z}$. The definition of $t_{1,j}$ yields%
\begin{align*}
t_{1,j}  &  =\sum_{k=0}^{1}x_{j+1}x_{j+2}\cdots x_{j+k}\cdot u_{j+k+1}%
u_{j+k+2}\cdots u_{j+1}\\
&  =\underbrace{x_{j+1}x_{j+2}\cdots x_{j+0}}_{=\left(  \text{empty
product}\right)  =1}\cdot\underbrace{u_{j+0+1}u_{j+0+2}\cdots u_{j+1}%
}_{=u_{j+1}}+\underbrace{x_{j+1}x_{j+2}\cdots x_{j+1}}_{=x_{j+1}}%
\cdot\underbrace{u_{j+1+1}u_{j+1+2}\cdots u_{j+1}}_{=\left(  \text{empty
product}\right)  =1}\\
&  =u_{j+1}+x_{j+1}=x_{j+1}+u_{j+1}.
\end{align*}
Hence,%
\[
t_{n-1,j+1}\underbrace{t_{1,j}}_{=x_{j+1}+u_{j+1}}=t_{n-1,j+1}\left(
x_{j+1}+u_{j+1}\right)  =\left(  x_{j+1}+u_{j+1}\right)  t_{n-1,j+1}.
\]
Comparing this with%
\begin{align*}
&  \underbrace{t_{n-1,j+1+1}}_{=t_{n-1,j+2}}\cdot\underbrace{x_{j+2}%
x_{j+3}\cdots x_{j+1+1}}_{\substack{=x_{j+2}x_{j+3}\cdots x_{j+2}\\=x_{j+2}%
}}+u_{j}t_{n-1,j}\underbrace{t_{1-1,j+1}}_{\substack{=t_{0,j+1}=1\\\text{(by
Lemma \ref{lem.f.steps} \textbf{(b)},}\\\text{applied to }j+1\text{ instead of
}j\text{)}}}\\
&  =t_{n-1,j+2}x_{j+2}+u_{j}t_{n-1,j}=\underbrace{x_{j+2}t_{n-1,j+2}%
}_{\substack{=x_{\left(  j+1\right)  +1}t_{n-1,\left(  j+1\right)
+1}\\\text{(since }j+2=\left(  j+1\right)  +1\text{)}}}+\underbrace{u_{j}%
t_{n-1,j}}_{\substack{=u_{\left(  j+1\right)  -1}t_{n-1,\left(  j+1\right)
-1}\\\text{(since }j=\left(  j+1\right)  -1\text{)}}}\\
&  =x_{\left(  j+1\right)  +1}t_{n-1,\left(  j+1\right)  +1}+u_{\left(
j+1\right)  -1}t_{n-1,\left(  j+1\right)  -1}=\left(  x_{j+1}+u_{j+1}\right)
t_{n-1,j+1}\\
&  \ \ \ \ \ \ \ \ \ \ \left(  \text{by Lemma \ref{lem.f.steps} \textbf{(f)},
applied to }i=j+1\right)  ,
\end{align*}
we obtain
\[
t_{n-1,j+1+1}\cdot x_{j+2}x_{j+3}\cdots x_{j+1+1}+u_{j}t_{n-1,j}%
t_{1-1,j+1}=t_{n-1,j+1}t_{1,j}.
\]

\end{verlong}

Now, forget that we fixed $j$. We thus have proved that%
\[
t_{n-1,j+1+1}\cdot x_{j+2}x_{j+3}\cdots x_{j+1+1}+u_{j}t_{n-1,j}%
t_{1-1,j+1}=t_{n-1,j+1}t_{1,j}%
\]
for each $j\in\mathbb{Z}$. In other words, Lemma \ref{lem.f.steps}
\textbf{(g)} holds for $q=1$. This completes the induction base.\footnote{We
could have simplified this part of the proof by taking $q=0$ as induction base
instead. But this would have required extending the semifield $\mathbb{K}$ to
a semiring $\mathbb{K}\sqcup\left\{  0\right\}  $ by adjoining a zero (since
$t_{-1,j}$ would be an empty sum). It is not hard to do this, but we prefer
computations to technicalities.}

\textit{Induction step:} Fix a positive integer $p$. Assume (as induction
hypothesis) that Lemma \ref{lem.f.steps} \textbf{(g)} holds for $q=p$. We must
now show that Lemma \ref{lem.f.steps} \textbf{(g)} holds for $q=p+1$.

We have assumed that Lemma \ref{lem.f.steps} \textbf{(g)} holds for $q=p$. In
other words, each $j\in\mathbb{Z}$ satisfies%
\begin{equation}
t_{n-1,j+p+1}\cdot x_{j+2}x_{j+3}\cdots x_{j+p+1}+u_{j}t_{n-1,j}%
t_{p-1,j+1}=t_{n-1,j+1}t_{p,j}. \label{pf.lem.f.steps.g.IH}%
\end{equation}

\begin{vershort}
Now, let $j\in\mathbb{Z}$ be arbitrary. Then, Lemma \ref{lem.f.steps}
\textbf{(d)} (applied to $r=p$) yields%
\[
u_{j+p+1}t_{p,j}+x_{j+1}x_{j+2}\cdots x_{j+p+1}=t_{p+1,j}.
\]
Multiplying both sides of this equality by $t_{n-1,j+1}$, we obtain%
\[
t_{n-1,j+1}\left(  u_{j+p+1}t_{p,j}+x_{j+1}x_{j+2}\cdots x_{j+p+1}\right)
=t_{n-1,j+1}t_{p+1,j}.
\]
Hence,%
\begin{align*}
&  t_{n-1,j+1}t_{p+1,j}\\
&  =t_{n-1,j+1}\left(  u_{j+p+1}t_{p,j}+x_{j+1}x_{j+2}\cdots x_{j+p+1}\right)
\\
&  =\underbrace{t_{n-1,j+1}u_{j+p+1}}_{=u_{j+p+1}t_{n-1,j+1}}t_{p,j}%
+t_{n-1,j+1}\cdot x_{j+1}x_{j+2}\cdots x_{j+p+1}\\
&  =u_{j+p+1}\underbrace{t_{n-1,j+1}t_{p,j}}_{\substack{=t_{n-1,j+p+1}\cdot
x_{j+2}x_{j+3}\cdots x_{j+p+1}+u_{j}t_{n-1,j}t_{p-1,j+1}\\\text{(by
(\ref{pf.lem.f.steps.g.IH}))}}}+t_{n-1,j+1}\cdot\underbrace{x_{j+1}%
x_{j+2}\cdots x_{j+p+1}}_{=x_{j+1}\cdot x_{j+2}x_{j+3}\cdots x_{j+p+1}}\\
&  =u_{j+p+1}\left(  t_{n-1,j+p+1}\cdot x_{j+2}x_{j+3}\cdots x_{j+p+1}%
+u_{j}t_{n-1,j}t_{p-1,j+1}\right) \\
&  \ \ \ \ \ \ \ \ \ \ +t_{n-1,j+1}\cdot x_{j+1}\cdot x_{j+2}x_{j+3}\cdots
x_{j+p+1}\\
&  =u_{j+p+1}t_{n-1,j+p+1}\cdot x_{j+2}x_{j+3}\cdots x_{j+p+1}+u_{j+p+1}%
u_{j}t_{n-1,j}t_{p-1,j+1}\\
&  \ \ \ \ \ \ \ \ \ \ +t_{n-1,j+1}\cdot x_{j+1}\cdot x_{j+2}x_{j+3}\cdots
x_{j+p+1}\\
&  =\underbrace{\left(  u_{j+p+1}t_{n-1,j+p+1}+t_{n-1,j+1}\cdot x_{j+1}%
\right)  }_{\substack{=x_{j+1}t_{n-1,j+1}+u_{j+p+1}t_{n-1,j+p+1}}}\cdot
x_{j+2}x_{j+3}\cdots x_{j+p+1}+u_{j+p+1}u_{j}t_{n-1,j}t_{p-1,j+1}\\
&  =\underbrace{\left(  x_{j+1}t_{n-1,j+1}+u_{j+p+1}t_{n-1,j+p+1}\right)
}_{\substack{=x_{j+p+2}t_{n-1,j+p+2}+u_{j}t_{n-1,j}\\\text{(by Lemma
\ref{lem.f.steps} \textbf{(e)}, applied to }a=j+1\text{ and }b=j+p+2\text{)}%
}}\cdot x_{j+2}x_{j+3}\cdots x_{j+p+1}+u_{j+p+1}u_{j}t_{n-1,j}t_{p-1,j+1}\\
&  =\left(  x_{j+p+2}t_{n-1,j+p+2}+u_{j}t_{n-1,j}\right)  \cdot x_{j+2}%
x_{j+3}\cdots x_{j+p+1}+u_{j+p+1}u_{j}t_{n-1,j}t_{p-1,j+1}\\
&  =x_{j+p+2}t_{n-1,j+p+2}\cdot x_{j+2}x_{j+3}\cdots x_{j+p+1}+u_{j}%
t_{n-1,j}\cdot x_{j+2}x_{j+3}\cdots x_{j+p+1}\\
&  \ \ \ \ \ \ \ \ \ \ +u_{j+p+1}u_{j}t_{n-1,j}t_{p-1,j+1}%
\end{align*}%
\begin{align*}
&  =u_{j}t_{n-1,j}\underbrace{\left(  u_{j+p+1}t_{p-1,j+1}+x_{j+2}%
x_{j+3}\cdots x_{j+p+1}\right)  }_{\substack{=t_{p,j+1}\\\text{(by Lemma
\ref{lem.f.steps} \textbf{(d)}, applied to }j+1\text{ and }p-1\text{ instead
of }j\text{ and }r\text{)}}}\\
&  \ \ \ \ \ \ \ \ \ \ +\underbrace{x_{j+p+2}t_{n-1,j+p+2}\cdot x_{j+2}%
x_{j+3}\cdots x_{j+p+1}}_{=t_{n-1,j+p+2}\cdot x_{j+2}x_{j+3}\cdots
x_{j+p+1}\cdot x_{j+p+2}}\\
&  =u_{j}t_{n-1,j}t_{p,j+1}+t_{n-1,j+p+2}\cdot\underbrace{x_{j+2}x_{j+3}\cdots
x_{j+p+1}\cdot x_{j+p+2}}_{=x_{j+2}x_{j+3}\cdots x_{j+p+2}}\\
&  =u_{j}t_{n-1,j}t_{p,j+1}+t_{n-1,j+p+2}\cdot x_{j+2}x_{j+3}\cdots
x_{j+p+2}\\
&  =t_{n-1,j+p+2}\cdot x_{j+2}x_{j+3}\cdots x_{j+p+2}+u_{j}t_{n-1,j}t_{p,j+1}.
\end{align*}
In other words,%
\[
t_{n-1,j+p+2}\cdot x_{j+2}x_{j+3}\cdots x_{j+p+2}+u_{j}t_{n-1,j}%
t_{p,j+1}=t_{n-1,j+1}t_{p+1,j}.
\]

\end{vershort}

\begin{verlong}
Now, let $j\in\mathbb{Z}$ be arbitrary. Then, Lemma \ref{lem.f.steps}
\textbf{(d)} (applied to $r=p$) yields%
\[
u_{j+p+1}t_{p,j}+x_{j+1}x_{j+2}\cdots x_{j+p+1}=t_{p+1,j}.
\]
Multiplying both sides of this equality by $t_{n-1,j+1}$, we obtain%
\[
t_{n-1,j+1}\left(  u_{j+p+1}t_{p,j}+x_{j+1}x_{j+2}\cdots x_{j+p+1}\right)
=t_{n-1,j+1}t_{p+1,j}.
\]
Hence,%
\begin{align*}
&  t_{n-1,j+1}t_{p+1,j}\\
&  =t_{n-1,j+1}\left(  u_{j+p+1}t_{p,j}+x_{j+1}x_{j+2}\cdots x_{j+p+1}\right)
\\
&  =\underbrace{t_{n-1,j+1}u_{j+p+1}}_{=u_{j+p+1}t_{n-1,j+1}}t_{p,j}%
+t_{n-1,j+1}\cdot x_{j+1}x_{j+2}\cdots x_{j+p+1}\\
&  =u_{j+p+1}\underbrace{t_{n-1,j+1}t_{p,j}}_{\substack{=t_{n-1,j+p+1}\cdot
x_{j+2}x_{j+3}\cdots x_{j+p+1}+u_{j}t_{n-1,j}t_{p-1,j+1}\\\text{(by
(\ref{pf.lem.f.steps.g.IH}))}}}+t_{n-1,j+1}\cdot\underbrace{x_{j+1}%
x_{j+2}\cdots x_{j+p+1}}_{=x_{j+1}\cdot x_{j+2}x_{j+3}\cdots x_{j+p+1}}\\
&  =u_{j+p+1}\left(  t_{n-1,j+p+1}\cdot x_{j+2}x_{j+3}\cdots x_{j+p+1}%
+u_{j}t_{n-1,j}t_{p-1,j+1}\right) \\
&  \ \ \ \ \ \ \ \ \ \ +t_{n-1,j+1}\cdot x_{j+1}\cdot x_{j+2}x_{j+3}\cdots
x_{j+p+1}\\
&  =u_{j+p+1}t_{n-1,j+p+1}\cdot x_{j+2}x_{j+3}\cdots x_{j+p+1}+u_{j+p+1}%
u_{j}t_{n-1,j}t_{p-1,j+1}\\
&  \ \ \ \ \ \ \ \ \ \ +t_{n-1,j+1}\cdot x_{j+1}\cdot x_{j+2}x_{j+3}\cdots
x_{j+p+1}\\
&  =\underbrace{\left(  u_{j+p+1}t_{n-1,j+p+1}+t_{n-1,j+1}\cdot x_{j+1}%
\right)  }_{\substack{=x_{j+1}t_{n-1,j+1}+u_{j+p+1}t_{n-1,j+p+1}%
\\=x_{j+1}t_{n-1,j+1}+u_{\left(  j+p+2\right)  -1}t_{n-1,\left(  j+p+2\right)
-1}\\\text{(since }j+p+1=\left(  j+p+2\right)  -1\text{)}}}\cdot
x_{j+2}x_{j+3}\cdots x_{j+p+1}+u_{j+p+1}u_{j}t_{n-1,j}t_{p-1,j+1}\\
&  =\underbrace{\left(  x_{j+1}t_{n-1,j+1}+u_{\left(  j+p+2\right)
-1}t_{n-1,\left(  j+p+2\right)  -1}\right)  }_{\substack{=x_{j+p+2}%
t_{n-1,j+p+2}+u_{\left(  j+1\right)  -1}t_{n-1,\left(  j+1\right)
-1}\\\text{(by Lemma \ref{lem.f.steps} \textbf{(e)}, applied to }a=j+1\text{
and }b=j+p+2\text{)}}}\cdot x_{j+2}x_{j+3}\cdots x_{j+p+1}+u_{j+p+1}%
u_{j}t_{n-1,j}t_{p-1,j+1}\\
&  =\left(  x_{j+p+2}t_{n-1,j+p+2}+\underbrace{u_{\left(  j+1\right)  -1}%
}_{=u_{j}}\underbrace{t_{n-1,\left(  j+1\right)  -1}}_{=t_{n-1,j}}\right)
\cdot x_{j+2}x_{j+3}\cdots x_{j+p+1}+u_{j+p+1}u_{j}t_{n-1,j}t_{p-1,j+1}\\
&  =\left(  x_{j+p+2}t_{n-1,j+p+2}+u_{j}t_{n-1,j}\right)  \cdot x_{j+2}%
x_{j+3}\cdots x_{j+p+1}+u_{j+p+1}u_{j}t_{n-1,j}t_{p-1,j+1}\\
&  =x_{j+p+2}t_{n-1,j+p+2}\cdot x_{j+2}x_{j+3}\cdots x_{j+p+1}+u_{j}%
t_{n-1,j}\cdot x_{j+2}x_{j+3}\cdots x_{j+p+1}\\
&  \ \ \ \ \ \ \ \ \ \ +u_{j+p+1}u_{j}t_{n-1,j}t_{p-1,j+1}%
\end{align*}%
\begin{align*}
&  =u_{j}t_{n-1,j}\left(  \underbrace{u_{j+p+1}}_{=u_{\left(  j+1\right)
+\left(  p-1\right)  +1}}t_{p-1,j+1}+\underbrace{x_{j+2}x_{j+3}\cdots
x_{j+p+1}}_{=x_{\left(  j+1\right)  +1}x_{\left(  j+1\right)  +2}\cdots
x_{\left(  j+1\right)  +\left(  p-1\right)  +1}}\right) \\
&  \ \ \ \ \ \ \ \ \ \ +\underbrace{x_{j+p+2}t_{n-1,j+p+2}\cdot x_{j+2}%
x_{j+3}\cdots x_{j+p+1}}_{=t_{n-1,j+p+2}\cdot x_{j+2}x_{j+3}\cdots
x_{j+p+1}\cdot x_{j+p+2}}\\
&  =u_{j}t_{n-1,j}\underbrace{\left(  u_{\left(  j+1\right)  +\left(
p-1\right)  +1}t_{p-1,j+1}+x_{\left(  j+1\right)  +1}x_{\left(  j+1\right)
+2}\cdots x_{\left(  j+1\right)  +\left(  p-1\right)  +1}\right)
}_{\substack{=t_{\left(  p-1\right)  +1,j+1}\\\text{(by Lemma
\ref{lem.f.steps} \textbf{(d)}, applied to }j+1\text{ and }p-1\text{ instead
of }j\text{ and }r\text{)}}}\\
&  \ \ \ \ \ \ \ \ \ \ +t_{n-1,j+p+2}\cdot\underbrace{x_{j+2}x_{j+3}\cdots
x_{j+p+1}\cdot x_{j+p+2}}_{=x_{j+2}x_{j+3}\cdots x_{j+p+2}}\\
&  =u_{j}t_{n-1,j}\underbrace{t_{\left(  p-1\right)  +1,j+1}}_{=t_{\left(
p+1\right)  -1,j+1}}+\underbrace{t_{n-1,j+p+2}}_{=t_{n-1,j+\left(  p+1\right)
+1}}\cdot\underbrace{x_{j+2}x_{j+3}\cdots x_{j+p+2}}_{=x_{j+2}x_{j+3}\cdots
x_{j+\left(  p+1\right)  +1}}\\
&  =u_{j}t_{n-1,j}t_{\left(  p+1\right)  -1,j+1}+t_{n-1,j+\left(  p+1\right)
+1}\cdot x_{j+2}x_{j+3}\cdots x_{j+\left(  p+1\right)  +1}\\
&  =t_{n-1,j+\left(  p+1\right)  +1}\cdot x_{j+2}x_{j+3}\cdots x_{j+\left(
p+1\right)  +1}+u_{j}t_{n-1,j}t_{\left(  p+1\right)  -1,j+1}.
\end{align*}
In other words,%
\[
t_{n-1,j+\left(  p+1\right)  +1}\cdot x_{j+2}x_{j+3}\cdots x_{j+\left(
p+1\right)  +1}+u_{j}t_{n-1,j}t_{\left(  p+1\right)  -1,j+1}=t_{n-1,j+1}%
t_{p+1,j}.
\]

\end{verlong}

\begin{vershort}
Now, forget that we fixed $j$. We thus have proved that each $j\in\mathbb{Z}$
satisfies%
\[
t_{n-1,j+p+2}\cdot x_{j+2}x_{j+3}\cdots x_{j+p+2}+u_{j}t_{n-1,j}%
t_{p,j+1}=t_{n-1,j+1}t_{p+1,j}.
\]
In other words, Lemma \ref{lem.f.steps} \textbf{(g)} holds for $q=p+1$. This
completes the induction step. Hence, Lemma \ref{lem.f.steps} \textbf{(g)} is
proved by induction.
\end{vershort}

\begin{verlong}
Now, forget that we fixed $j$. We thus have proved that each $j\in\mathbb{Z}$
satisfies%
\[
t_{n-1,j+\left(  p+1\right)  +1}\cdot x_{j+2}x_{j+3}\cdots x_{j+\left(
p+1\right)  +1}+u_{j}t_{n-1,j}t_{\left(  p+1\right)  -1,j+1}=t_{n-1,j+1}%
t_{p+1,j}.
\]
In other words, Lemma \ref{lem.f.steps} \textbf{(g)} holds for $q=p+1$. This
completes the induction step. Hence, Lemma \ref{lem.f.steps} \textbf{(g)} is
proved by induction.
\end{verlong}

\begin{vershort}
\textbf{(h)} For $i\in\left\{  1,2,\ldots,n\right\}  $, the claim of Lemma
\ref{lem.f.steps} \textbf{(h)} follows from the definition of $y$. Thus, it
also holds for all $i\in\mathbb{Z}$, since any integers $j$ and $j^{\prime}$
satisfying $j\equiv j^{\prime}\operatorname{mod}n$ satisfy $x_{j}%
=x_{j^{\prime}}$ and $u_{j}=u_{j^{\prime}}$ and $t_{n-1,j}=t_{n-1,j^{\prime}}$
(by Lemma \ref{lem.f.steps} \textbf{(a)}) and $y_{j}=y_{j^{\prime}}$. Thus,
Lemma \ref{lem.f.steps} \textbf{(h)} is proved.
\end{vershort}

\begin{verlong}
\textbf{(h)} In short, this follows from the definition of $y$ because
\textquotedblleft everything is $n$-periodic\textquotedblright. Here is the
argument in detail:

Convention \ref{conv.bir.peri} entails that the families $\left(
u_{i}\right)  _{i\in\mathbb{Z}}$, $\left(  x_{i}\right)  _{i\in\mathbb{Z}}$
and $\left(  y_{i}\right)  _{i\in\mathbb{Z}}$ are $n$-periodic. In other
words, the families $\left(  u_{j}\right)  _{j\in\mathbb{Z}}$, $\left(
x_{j}\right)  _{j\in\mathbb{Z}}$ and $\left(  y_{j}\right)  _{j\in\mathbb{Z}}$
are $n$-periodic (here, we have renamed the index $i$ as $j$).

Let $i\in\mathbb{Z}$. Let $i^{\prime}$ be the unique element of $\left\{
1,2,\ldots,n\right\}  $ that is congruent to $i$ modulo $n$. Thus, $i^{\prime
}\equiv i\operatorname{mod}n$. Hence, $u_{i^{\prime}}=u_{i}$ (since the family
$\left(  u_{j}\right)  _{j\in\mathbb{Z}}$ is $n$-periodic) and $y_{i^{\prime}%
}=y_{i}$ (since the family $\left(  y_{j}\right)  _{j\in\mathbb{Z}}$ is
$n$-periodic). Also, from $i^{\prime}\equiv i\operatorname{mod}n$, we obtain
$i^{\prime}+1\equiv i+1\operatorname{mod}n$, so that $x_{i^{\prime}+1}%
=x_{i+1}$ (since the family $\left(  x_{j}\right)  _{j\in\mathbb{Z}}$ is
$n$-periodic) and $t_{n-1,i^{\prime}+1}=t_{n-1,i+1}$ (by Lemma
\ref{lem.f.steps} \textbf{(a)}, applied to $r=n-1$, $j=i^{\prime}+1$ and
$j^{\prime}=i+1$). Also, from $i^{\prime}\equiv i\operatorname{mod}n$, we
obtain $i^{\prime}-1\equiv i-1\operatorname{mod}n$, so that $u_{i^{\prime}%
-1}=u_{i-1}$ (since the family $\left(  u_{j}\right)  _{j\in\mathbb{Z}}$ is
$n$-periodic) and $t_{n-1,i^{\prime}-1}=t_{n-1,i-1}$ (by Lemma
\ref{lem.f.steps} \textbf{(a)}, applied to $r=n-1$, $j=i^{\prime}-1$ and
$j^{\prime}=i-1$). Now, recall that $y_{i^{\prime}}=y_{i}$; hence,%
\begin{align*}
y_{i}  &  =y_{i^{\prime}}=u_{i^{\prime}}\cdot\dfrac{u_{i^{\prime}%
-1}t_{n-1,i^{\prime}-1}}{x_{i^{\prime}+1}t_{n-1,i^{\prime}+1}}%
\ \ \ \ \ \ \ \ \ \ \left(  \text{by the definition of }y\text{, since
}i^{\prime}\in\left\{  1,2,\ldots,n\right\}  \right) \\
&  =u_{i}\cdot\dfrac{u_{i-1}t_{n-1,i-1}}{x_{i+1}t_{n-1,i+1}}\\
&  \ \ \ \ \ \ \ \ \ \ \left(
\begin{array}
[c]{c}%
\text{since }u_{i^{\prime}}=u_{i}\text{ and }u_{i^{\prime}-1}=u_{i-1}\text{
and }t_{n-1,i^{\prime}-1}=t_{n-1,i-1}\\
\text{and }x_{i^{\prime}+1}=x_{i+1}\text{ and }t_{n-1,i^{\prime}%
+1}=t_{n-1,i+1}%
\end{array}
\right)  .
\end{align*}
Thus, Lemma \ref{lem.f.steps} \textbf{(h)} is proved.
\end{verlong}

\textbf{(i)} We shall prove Lemma \ref{lem.f.steps} \textbf{(i)} by induction
on $q$:

\begin{vershort}
\textit{Induction base:} For each $j\in\mathbb{Z}$, we have%
\[
\dfrac{t_{0,j}^{\prime}}{u_{j+1}u_{j+2}\cdots u_{j}}=\dfrac{t_{n-1,j+1}%
}{t_{n-1,j+1}}\cdot\dfrac{t_{0,j}}{x_{j+2}x_{j+3}\cdots x_{j+1}}%
\]
\footnote{\textit{Proof.} Let $j\in\mathbb{Z}$. Lemma \ref{lem.f.steps}
\textbf{(b)} yields $t_{0,j}=1$. Similarly, $t_{0,j}^{\prime}=1$. From this
equality, and from $u_{j+1}u_{j+2}\cdots u_{j}=\left(  \text{empty
product}\right)  =1$, we obtain%
\[
\dfrac{t_{0,j}^{\prime}}{u_{j+1}u_{j+2}\cdots u_{j}}=\dfrac{1}{1}.
\]
Comparing this with%
\begin{align*}
\dfrac{t_{n-1,j+1}}{t_{n-1,j+1}}\cdot\dfrac{t_{0,j}}{x_{j+2}x_{j+3}\cdots
x_{j+1}}  &  =\dfrac{t_{0,j}}{x_{j+2}x_{j+3}\cdots x_{j+1}}=\dfrac{1}{1}\\
&  \ \ \ \ \ \ \ \ \ \ \left(  \text{since }t_{0,j}=1\text{ and }%
x_{j+2}x_{j+3}\cdots x_{j+1}=\left(  \text{empty product}\right)  =1\right)  ,
\end{align*}
we obtain $\dfrac{t_{0,j}^{\prime}}{u_{j+1}u_{j+2}\cdots u_{j}}=\dfrac
{t_{n-1,j+1}}{t_{n-1,j+1}}\cdot\dfrac{t_{0,j}}{x_{j+2}x_{j+3}\cdots x_{j+1}}$.
Qed.}. In other words, Lemma \ref{lem.f.steps} \textbf{(i)} holds for $q=0$.
This completes the induction base.
\end{vershort}

\begin{verlong}
\textit{Induction base:} For each $j\in\mathbb{Z}$, we have%
\[
\dfrac{t_{0,j}^{\prime}}{u_{j+1}u_{j+2}\cdots u_{j+0}}=\dfrac{t_{n-1,j+1}%
}{t_{n-1,j+0+1}}\cdot\dfrac{t_{0,j}}{x_{j+2}x_{j+3}\cdots x_{j+0+1}}%
\]
\footnote{\textit{Proof.} Let $j\in\mathbb{Z}$. Lemma \ref{lem.f.steps}
\textbf{(b)} yields $t_{0,j}=1$. The same argument (applied to $y$ and
$t_{r,j}^{\prime}$ instead of $x$ and $t_{r,j}$) yields $t_{0,j}^{\prime}=1$.
From this equality, and from $u_{j+1}u_{j+2}\cdots u_{j+0}=\left(  \text{empty
product}\right)  =1$, we obtain%
\[
\dfrac{t_{0,j}^{\prime}}{u_{j+1}u_{j+2}\cdots u_{j+0}}=\dfrac{1}{1}=1.
\]
Comparing this with%
\begin{align*}
&  \dfrac{t_{n-1,j+1}}{t_{n-1,j+0+1}}\cdot\dfrac{t_{0,j}}{x_{j+2}x_{j+3}\cdots
x_{j+0+1}}\\
&  =\dfrac{t_{n-1,j+1}}{t_{n-1,j+0+1}}\cdot\dfrac{1}{1}\\
&  \ \ \ \ \ \ \ \ \ \ \left(  \text{since }t_{0,j}=1\text{ and }%
x_{j+2}x_{j+3}\cdots x_{j+0+1}=\left(  \text{empty product}\right)  =1\right)
\\
&  =\dfrac{t_{n-1,j+1}}{t_{n-1,j+0+1}}=\dfrac{t_{n-1,j+1}}{t_{n-1,j+1}}=1,
\end{align*}
we obtain $\dfrac{t_{0,j}^{\prime}}{u_{j+1}u_{j+2}\cdots u_{j+0}}%
=\dfrac{t_{n-1,j+1}}{t_{n-1,j+0+1}}\cdot\dfrac{t_{0,j}}{x_{j+2}x_{j+3}\cdots
x_{j+0+1}}$. Qed.}. In other words, Lemma \ref{lem.f.steps} \textbf{(i)} holds
for $q=0$. This completes the induction base.
\end{verlong}

\textit{Induction step:} Fix $r\in\mathbb{N}$. Assume (as induction
hypothesis) that Lemma \ref{lem.f.steps} \textbf{(i)} holds for $q=r$. We must
now show that Lemma \ref{lem.f.steps} \textbf{(i)} holds for $q=r+1$.

We have assumed that Lemma \ref{lem.f.steps} \textbf{(i)} holds for $q=r$. In
other words, each $j\in\mathbb{Z}$ satisfies%
\begin{equation}
\dfrac{t_{r,j}^{\prime}}{u_{j+1}u_{j+2}\cdots u_{j+r}}=\dfrac{t_{n-1,j+1}%
}{t_{n-1,j+r+1}}\cdot\dfrac{t_{r,j}}{x_{j+2}x_{j+3}\cdots x_{j+r+1}}.
\label{pf.lem.f.steps.i.IH}%
\end{equation}

\begin{vershort}
Now, let $j\in\mathbb{Z}$ be arbitrary. Then, (\ref{pf.lem.f.steps.i.IH})
(applied to $j+1$ instead of $j$) yields%
\begin{equation}
\dfrac{t_{r,j+1}^{\prime}}{u_{j+2}u_{j+3}\cdots u_{j+r+1}}=\dfrac{t_{n-1,j+2}%
}{t_{n-1,j+r+2}}\cdot\dfrac{t_{r,j+1}}{x_{j+3}x_{j+4}\cdots x_{j+r+2}}.
\label{pf.lem.f.steps.i.short.IH3}%
\end{equation}

But Lemma \ref{lem.f.steps} \textbf{(c)} (applied to $j+1$ instead of $j$)
yields%
\[
x_{j+1}t_{r,j+1}+u_{j+1}u_{j+2}\cdots u_{j+r+1}=t_{r+1,j}.
\]
Hence,%
\[
t_{r+1,j}=x_{j+1}t_{r,j+1}+u_{j+1}u_{j+2}\cdots u_{j+r+1}=u_{j+1}u_{j+2}\cdots
u_{j+r+1}+x_{j+1}t_{r,j+1}.
\]
The same reasoning (applied to $y$ and $t_{r,j}^{\prime}$ instead of $x$ and
$t_{r,j}$) yields
\[
t_{r+1,j}^{\prime}=u_{j+1}u_{j+2}\cdots u_{j+r+1}+y_{j+1}t_{r,j+1}^{\prime}.
\]
Hence,%
\begin{align}
\dfrac{t_{r+1,j}^{\prime}}{u_{j+1}u_{j+2}\cdots u_{j+r+1}}  &  =\dfrac
{u_{j+1}u_{j+2}\cdots u_{j+r+1}+y_{j+1}t_{r,j+1}^{\prime}}{u_{j+1}%
u_{j+2}\cdots u_{j+r+1}}\nonumber\\
&  =1+\underbrace{\dfrac{y_{j+1}t_{r,j+1}^{\prime}}{u_{j+1}u_{j+2}\cdots
u_{j+r+1}}}_{\substack{=\dfrac{y_{j+1}t_{r,j+1}^{\prime}}{u_{j+1}\cdot
u_{j+2}u_{j+3}\cdots u_{j+r+1}}\\=\dfrac{y_{j+1}}{u_{j+1}}\cdot\dfrac
{t_{r,j+1}^{\prime}}{u_{j+2}u_{j+3}\cdots u_{j+r+1}}}}\nonumber\\
&  =1+\dfrac{y_{j+1}}{u_{j+1}}\cdot\underbrace{\dfrac{t_{r,j+1}^{\prime}%
}{u_{j+2}u_{j+3}\cdots u_{j+r+1}}}_{\substack{=\dfrac{t_{n-1,j+2}%
}{t_{n-1,j+r+2}}\cdot\dfrac{t_{r,j+1}}{x_{j+3}x_{j+4}\cdots x_{j+r+2}%
}\\\text{(by (\ref{pf.lem.f.steps.i.short.IH3}))}}}\nonumber\\
&  =1+\dfrac{y_{j+1}}{u_{j+1}}\cdot\dfrac{t_{n-1,j+2}}{t_{n-1,j+r+2}}%
\cdot\dfrac{t_{r,j+1}}{x_{j+3}x_{j+4}\cdots x_{j+r+2}}.
\label{pf.lem.f.steps.i.short.3}%
\end{align}

But Lemma \ref{lem.f.steps} \textbf{(h)} (applied to $i=j+1$) yields%
\[
y_{j+1}=u_{j+1}\cdot\dfrac{u_{j}t_{n-1,j}}{x_{j+2}t_{n-1,j+2}}.
\]
Dividing this equality by $u_{j+1}$, we find%
\[
\dfrac{y_{j+1}}{u_{j+1}}=\dfrac{u_{j}t_{n-1,j}}{x_{j+2}t_{n-1,j+2}}.
\]
Hence, (\ref{pf.lem.f.steps.i.short.3}) becomes%
\begin{align}
\dfrac{t_{r+1,j}^{\prime}}{u_{j+1}u_{j+2}\cdots u_{j+r+1}}  &
=1+\underbrace{\dfrac{y_{j+1}}{u_{j+1}}}_{=\dfrac{u_{j}t_{n-1,j}}%
{x_{j+2}t_{n-1,j+2}}}\cdot\dfrac{t_{n-1,j+2}}{t_{n-1,j+r+2}}\cdot
\dfrac{t_{r,j+1}}{x_{j+3}x_{j+4}\cdots x_{j+r+2}}\nonumber\\
&  =1+\dfrac{u_{j}t_{n-1,j}}{x_{j+2}t_{n-1,j+2}}\cdot\dfrac{t_{n-1,j+2}%
}{t_{n-1,j+r+2}}\cdot\dfrac{t_{r,j+1}}{x_{j+3}x_{j+4}\cdots x_{j+r+2}%
}\nonumber\\
&  =1+\dfrac{u_{j}t_{n-1,j}}{t_{n-1,j+r+2}}\cdot\dfrac{t_{r,j+1}}{x_{j+2}\cdot
x_{j+3}x_{j+4}\cdots x_{j+r+2}}\nonumber\\
&  =1+\dfrac{u_{j}t_{n-1,j}}{t_{n-1,j+r+2}}\cdot\dfrac{t_{r,j+1}}%
{x_{j+2}x_{j+3}\cdots x_{j+r+2}}\nonumber\\
&  \ \ \ \ \ \ \ \ \ \ \left(  \text{since }x_{j+2}\cdot x_{j+3}x_{j+4}\cdots
x_{j+r+2}=x_{j+2}x_{j+3}\cdots x_{j+r+2}\right) \nonumber\\
&  =\dfrac{t_{n-1,j+r+2}\cdot x_{j+2}x_{j+3}\cdots x_{j+r+2}+u_{j}%
t_{n-1,j}t_{r,j+1}}{t_{n-1,j+r+2}\cdot x_{j+2}x_{j+3}\cdots x_{j+r+2}}.
\label{pf.lem.f.steps.i.short.6}%
\end{align}

But Lemma \ref{lem.f.steps} \textbf{(g)} (applied to $q=r+1$) yields%
\begin{equation}
t_{n-1,j+r+2}\cdot x_{j+2}x_{j+3}\cdots x_{j+r+2}+u_{j}t_{n-1,j}%
t_{r,j+1}=t_{n-1,j+1}t_{r+1,j}. \label{pf.lem.f.steps.i.short.8}%
\end{equation}
Hence, (\ref{pf.lem.f.steps.i.short.6}) becomes%
\begin{align*}
\dfrac{t_{r+1,j}^{\prime}}{u_{j+1}u_{j+2}\cdots u_{j+r+1}}  &  =\dfrac
{t_{n-1,j+r+2}\cdot x_{j+2}x_{j+3}\cdots x_{j+r+2}+u_{j}t_{n-1,j}t_{r,j+1}%
}{t_{n-1,j+r+2}\cdot x_{j+2}x_{j+3}\cdots x_{j+r+2}}\\
&  =\dfrac{t_{n-1,j+1}t_{r+1,j}}{t_{n-1,j+r+2}\cdot x_{j+2}x_{j+3}\cdots
x_{j+r+2}}\ \ \ \ \ \ \ \ \ \ \left(  \text{by (\ref{pf.lem.f.steps.i.short.8}%
)}\right) \\
&  =\dfrac{t_{n-1,j+1}}{t_{n-1,j+r+2}}\cdot\dfrac{t_{r+1,j}}{x_{j+2}%
x_{j+3}\cdots x_{j+r+2}}.
\end{align*}

Forget that we fixed $j$. We thus have shown that each $j\in\mathbb{Z}$
satisfies%
\[
\dfrac{t_{r+1,j}^{\prime}}{u_{j+1}u_{j+2}\cdots u_{j+r+1}}=\dfrac{t_{n-1,j+1}%
}{t_{n-1,j+r+2}}\cdot\dfrac{t_{r+1,j}}{x_{j+2}x_{j+3}\cdots x_{j+r+2}}.
\]
In other words, Lemma \ref{lem.f.steps} \textbf{(i)} holds for $q=r+1$. This
completes the induction step. Hence, Lemma \ref{lem.f.steps} \textbf{(i)} is
proved by induction.
\end{vershort}

\begin{verlong}
Now, let $j\in\mathbb{Z}$ be arbitrary. Then, (\ref{pf.lem.f.steps.i.IH})
(applied to $j+1$ instead of $j$) yields%
\[
\dfrac{t_{r,j+1}^{\prime}}{u_{\left(  j+1\right)  +1}u_{\left(  j+1\right)
+2}\cdots u_{\left(  j+1\right)  +r}}=\dfrac{t_{n-1,\left(  j+1\right)  +1}%
}{t_{n-1,\left(  j+1\right)  +r+1}}\cdot\dfrac{t_{r,j+1}}{x_{\left(
j+1\right)  +2}x_{\left(  j+1\right)  +3}\cdots x_{\left(  j+1\right)  +r+1}%
}.
\]
In other words,%
\begin{equation}
\dfrac{t_{r,j+1}^{\prime}}{u_{j+2}u_{j+3}\cdots u_{j+r+1}}=\dfrac{t_{n-1,j+2}%
}{t_{n-1,j+r+2}}\cdot\dfrac{t_{r,j+1}}{x_{j+3}x_{j+4}\cdots x_{j+r+2}}
\label{pf.lem.f.steps.i.IH3}%
\end{equation}
(since $\left(  j+1\right)  +1=j+2$ and $\left(  j+1\right)  +2=j+3$ and
$\left(  j+1\right)  +r+1=j+r+2$ and $\left(  j+1\right)  +r=j+r+1$ and
$\left(  j+1\right)  +3=j+4$).

But Lemma \ref{lem.f.steps} \textbf{(c)} (applied to $j+1$ instead of $j$)
yields%
\[
x_{j+1}t_{r,j+1}+u_{j+1}u_{\left(  j+1\right)  +1}\cdots u_{\left(
j+1\right)  +r}=t_{r+1,\left(  j+1\right)  -1}=t_{r+1,j}%
\]
(since $\left(  j+1\right)  -1=j$). Hence,%
\[
t_{r+1,j}=x_{j+1}t_{r,j+1}+u_{j+1}u_{\left(  j+1\right)  +1}\cdots u_{\left(
j+1\right)  +r}=x_{j+1}t_{r,j+1}+u_{j+1}u_{j+2}\cdots u_{j+r+1}%
\]
(since $\left(  j+1\right)  +1=j+2$ and $\left(  j+1\right)  +r=j+r+1$). The
same reasoning (applied to $y$ and $t_{r,j}^{\prime}$ instead of $x$ and
$t_{r,j}$) yields
\[
t_{r+1,j}^{\prime}=y_{j+1}t_{r,j+1}^{\prime}+u_{j+1}u_{j+2}\cdots u_{j+r+1}.
\]
Hence,%
\begin{align}
\dfrac{t_{r+1,j}^{\prime}}{u_{j+1}u_{j+2}\cdots u_{j+r+1}}  &  =\dfrac
{y_{j+1}t_{r,j+1}^{\prime}+u_{j+1}u_{j+2}\cdots u_{j+r+1}}{u_{j+1}%
u_{j+2}\cdots u_{j+r+1}}\nonumber\\
&  =\underbrace{\dfrac{y_{j+1}t_{r,j+1}^{\prime}}{u_{j+1}u_{j+2}\cdots
u_{j+r+1}}}_{\substack{=\dfrac{y_{j+1}t_{r,j+1}^{\prime}}{u_{j+1}\cdot
u_{j+2}u_{j+3}\cdots u_{j+r+1}}\\\text{(since }u_{j+1}u_{j+2}\cdots
u_{j+r+1}=u_{j+1}\cdot u_{j+2}u_{j+3}\cdots u_{j+r+1}\text{)}}}+1\nonumber\\
&  =\underbrace{\dfrac{y_{j+1}t_{r,j+1}^{\prime}}{u_{j+1}\cdot u_{j+2}%
u_{j+3}\cdots u_{j+r+1}}}_{=\dfrac{y_{j+1}}{u_{j+1}}\cdot\dfrac{t_{r,j+1}%
^{\prime}}{u_{j+2}u_{j+3}\cdots u_{j+r+1}}}+1\nonumber\\
&  =\dfrac{y_{j+1}}{u_{j+1}}\cdot\underbrace{\dfrac{t_{r,j+1}^{\prime}%
}{u_{j+2}u_{j+3}\cdots u_{j+r+1}}}_{\substack{=\dfrac{t_{n-1,j+2}%
}{t_{n-1,j+r+2}}\cdot\dfrac{t_{r,j+1}}{x_{j+3}x_{j+4}\cdots x_{j+r+2}%
}\\\text{(by (\ref{pf.lem.f.steps.i.IH3}))}}}+1\nonumber\\
&  =\dfrac{y_{j+1}}{u_{j+1}}\cdot\dfrac{t_{n-1,j+2}}{t_{n-1,j+r+2}}\cdot
\dfrac{t_{r,j+1}}{x_{j+3}x_{j+4}\cdots x_{j+r+2}}+1.
\label{pf.lem.f.steps.i.3}%
\end{align}

But Lemma \ref{lem.f.steps} \textbf{(h)} (applied to $i=j+1$) yields%
\[
y_{j+1}=u_{j+1}\cdot\dfrac{u_{\left(  j+1\right)  -1}t_{n-1,\left(
j+1\right)  -1}}{x_{\left(  j+1\right)  +1}t_{n-1,\left(  j+1\right)  +1}%
}=u_{j+1}\cdot\dfrac{u_{j}t_{n-1,j}}{x_{j+2}t_{n-1,j+2}}%
\]
(since $\left(  j+1\right)  -1=j$ and $\left(  j+1\right)  +1=j+2$). Dividing
this equality by $u_{j+1}$, we find%
\[
\dfrac{y_{j+1}}{u_{j+1}}=\dfrac{u_{j}t_{n-1,j}}{x_{j+2}t_{n-1,j+2}}.
\]
Hence, (\ref{pf.lem.f.steps.i.3}) becomes%
\begin{align}
\dfrac{t_{r+1,j}^{\prime}}{u_{j+1}u_{j+2}\cdots u_{j+r+1}}  &
=\underbrace{\dfrac{y_{j+1}}{u_{j+1}}}_{=\dfrac{u_{j}t_{n-1,j}}{x_{j+2}%
t_{n-1,j+2}}}\cdot\dfrac{t_{n-1,j+2}}{t_{n-1,j+r+2}}\cdot\dfrac{t_{r,j+1}%
}{x_{j+3}x_{j+4}\cdots x_{j+r+2}}+1\nonumber\\
&  =\dfrac{u_{j}t_{n-1,j}}{x_{j+2}t_{n-1,j+2}}\cdot\dfrac{t_{n-1,j+2}%
}{t_{n-1,j+r+2}}\cdot\dfrac{t_{r,j+1}}{x_{j+3}x_{j+4}\cdots x_{j+r+2}%
}+1\nonumber\\
&  =\dfrac{u_{j}t_{n-1,j}t_{r,j+1}}{t_{n-1,j+r+2}\cdot x_{j+2}\cdot
x_{j+3}x_{j+4}\cdots x_{j+r+2}}+1\nonumber\\
&  =\dfrac{u_{j}t_{n-1,j}t_{r,j+1}}{t_{n-1,j+r+2}\cdot x_{j+2}x_{j+3}\cdots
x_{j+r+2}}+1\nonumber\\
&  \ \ \ \ \ \ \ \ \ \ \left(  \text{since }x_{j+2}\cdot x_{j+3}x_{j+4}\cdots
x_{j+r+2}=x_{j+2}x_{j+3}\cdots x_{j+r+2}\right) \nonumber\\
&  =\dfrac{u_{j}t_{n-1,j}t_{r,j+1}+t_{n-1,j+r+2}\cdot x_{j+2}x_{j+3}\cdots
x_{j+r+2}}{t_{n-1,j+r+2}\cdot x_{j+2}x_{j+3}\cdots x_{j+r+2}}\nonumber\\
&  =\dfrac{t_{n-1,j+r+2}\cdot x_{j+2}x_{j+3}\cdots x_{j+r+2}+u_{j}%
t_{n-1,j}t_{r,j+1}}{t_{n-1,j+r+2}\cdot x_{j+2}x_{j+3}\cdots x_{j+r+2}}.
\label{pf.lem.f.steps.i.6}%
\end{align}

But Lemma \ref{lem.f.steps} \textbf{(g)} (applied to $q=r+1$) yields%
\[
t_{n-1,j+\left(  r+1\right)  +1}\cdot x_{j+2}x_{j+3}\cdots x_{j+\left(
r+1\right)  +1}+u_{j}t_{n-1,j}t_{\left(  r+1\right)  -1,j+1}=t_{n-1,j+1}%
t_{r+1,j}.
\]
This rewrites as%
\begin{equation}
t_{n-1,j+r+2}\cdot x_{j+2}x_{j+3}\cdots x_{j+r+2}+u_{j}t_{n-1,j}%
t_{r,j+1}=t_{n-1,j+1}t_{r+1,j} \label{pf.lem.f.steps.i.8}%
\end{equation}
(since $j+\left(  r+1\right)  +1=j+r+2$ and $\left(  r+1\right)  -1=r$).
Hence, (\ref{pf.lem.f.steps.i.6}) becomes%
\begin{align*}
\dfrac{t_{r+1,j}^{\prime}}{u_{j+1}u_{j+2}\cdots u_{j+r+1}}  &  =\dfrac
{t_{n-1,j+r+2}\cdot x_{j+2}x_{j+3}\cdots x_{j+r+2}+u_{j}t_{n-1,j}t_{r,j+1}%
}{t_{n-1,j+r+2}\cdot x_{j+2}x_{j+3}\cdots x_{j+r+2}}\\
&  =\dfrac{t_{n-1,j+1}t_{r+1,j}}{t_{n-1,j+r+2}\cdot x_{j+2}x_{j+3}\cdots
x_{j+r+2}}\ \ \ \ \ \ \ \ \ \ \left(  \text{by (\ref{pf.lem.f.steps.i.8}%
)}\right) \\
&  =\dfrac{t_{n-1,j+1}}{t_{n-1,j+r+2}}\cdot\dfrac{t_{r+1,j}}{x_{j+2}%
x_{j+3}\cdots x_{j+r+2}}\\
&  =\dfrac{t_{n-1,j+1}}{t_{n-1,j+\left(  r+1\right)  +1}}\cdot\dfrac
{t_{r+1,j}}{x_{j+2}x_{j+3}\cdots x_{j+\left(  r+1\right)  +1}}%
\end{align*}
(since $j+r+2=j+\left(  r+1\right)  +1$).

Forget that we fixed $j$. We thus have shown that each $j\in\mathbb{Z}$
satisfies%
\[
\dfrac{t_{r+1,j}^{\prime}}{u_{j+1}u_{j+2}\cdots u_{j+r+1}}=\dfrac{t_{n-1,j+1}%
}{t_{n-1,j+\left(  r+1\right)  +1}}\cdot\dfrac{t_{r+1,j}}{x_{j+2}x_{j+3}\cdots
x_{j+\left(  r+1\right)  +1}}.
\]
In other words, Lemma \ref{lem.f.steps} \textbf{(i)} holds for $q=r+1$. This
completes the induction step. Hence, Lemma \ref{lem.f.steps} \textbf{(i)} is
proved by induction.
\end{verlong}

\textbf{(j)} Let $j\in\mathbb{Z}$. Then, $u_{j+n}=u_{j}$ (by Convention
\ref{conv.bir.peri}). Lemma \ref{lem.f.steps} \textbf{(i)} (applied to
$q=n-1$) yields%
\begin{align}
\dfrac{t_{n-1,j}^{\prime}}{u_{j+1}u_{j+2}\cdots u_{j+n-1}}  &  =\dfrac
{t_{n-1,j+1}}{t_{n-1,j+\left(  n-1\right)  +1}}\cdot\dfrac{t_{n-1,j}}%
{x_{j+2}x_{j+3}\cdots x_{j+\left(  n-1\right)  +1}}\nonumber\\
&  =\dfrac{t_{n-1,j+1}}{t_{n-1,j+n}}\cdot\dfrac{t_{n-1,j}}{x_{j+2}%
x_{j+3}\cdots x_{j+n}} \label{pf.lem.f.steps.j.1}%
\end{align}
(since $\left(  n-1\right)  +1=n$). But $j\equiv j+n\operatorname{mod}n$;
hence, Lemma \ref{lem.f.steps} \textbf{(a)} (applied to $r=n-1$ and
$j^{\prime}=j+n$) yields $t_{n-1,j}=t_{n-1,j+n}$. Hence,
(\ref{pf.lem.f.steps.j.1}) becomes%
\begin{align*}
\dfrac{t_{n-1,j}^{\prime}}{u_{j+1}u_{j+2}\cdots u_{j+n-1}}  &  =\dfrac
{t_{n-1,j+1}}{t_{n-1,j+n}}\cdot\dfrac{t_{n-1,j}}{x_{j+2}x_{j+3}\cdots x_{j+n}%
}=\dfrac{t_{n-1,j+1}}{t_{n-1,j+n}}\cdot\dfrac{t_{n-1,j+n}}{x_{j+2}%
x_{j+3}\cdots x_{j+n}}\\
&  \ \ \ \ \ \ \ \ \ \ \left(  \text{since }t_{n-1,j}=t_{n-1,j+n}\right) \\
&  =\dfrac{t_{n-1,j+1}}{x_{j+2}x_{j+3}\cdots x_{j+n}}=\dfrac{t_{n-1,j+1}%
x_{j+1}}{x_{j+2}x_{j+3}\cdots x_{j+n}\cdot x_{j+1}}=\dfrac{t_{n-1,j+1}x_{j+1}%
}{x_{1}x_{2}\cdots x_{n}}%
\end{align*}
(since $x_{j+2}x_{j+3}\cdots x_{j+n}\cdot x_{j+1}=x_{j+1}\cdot x_{j+2}%
x_{j+3}\cdots x_{j+n}=x_{j+1}x_{j+2}\cdots x_{j+n}=x_{1}x_{2}\cdots x_{n}$ (by
Lemma \ref{lem.aprod})). Hence,%
\begin{align*}
\dfrac{t_{n-1,j+1}x_{j+1}}{x_{1}x_{2}\cdots x_{n}}  &  =\dfrac{t_{n-1,j}%
^{\prime}}{u_{j+1}u_{j+2}\cdots u_{j+n-1}}=\dfrac{t_{n-1,j}^{\prime}u_{j+n}%
}{u_{j+1}u_{j+2}\cdots u_{j+n-1}\cdot u_{j+n}}=\dfrac{t_{n-1,j}^{\prime
}u_{j+n}}{u_{j+1}u_{j+2}\cdots u_{j+n}}\\
&  \ \ \ \ \ \ \ \ \ \ \left(  \text{since }u_{j+1}u_{j+2}\cdots
u_{j+n-1}\cdot u_{j+n}=u_{j+1}u_{j+2}\cdots u_{j+n}\right) \\
&  =\dfrac{t_{n-1,j}^{\prime}u_{j}}{u_{j+1}u_{j+2}\cdots u_{j+n}%
}\ \ \ \ \ \ \ \ \ \ \left(  \text{since }u_{j+n}=u_{j}\right) \\
&  =\dfrac{t_{n-1,j}^{\prime}u_{j}}{u_{1}u_{2}\cdots u_{n}}%
\end{align*}
(since $u_{j+1}u_{j+2}\cdots u_{j+n}=u_{1}u_{2}\cdots u_{n}$ (by Lemma
\ref{lem.aprod})). This proves Lemma \ref{lem.f.steps} \textbf{(j)}.

\begin{vershort}
\textbf{(k)} Let $i\in\mathbb{Z}$. Applying Lemma \ref{lem.f.steps}
\textbf{(j)} to $j=i-1$, we obtain%
\[
\dfrac{t_{n-1,i-1}^{\prime}u_{i-1}}{u_{1}u_{2}\cdots u_{n}}=\dfrac
{t_{n-1,i}x_{i}}{x_{1}x_{2}\cdots x_{n}}.
\]
Applying Lemma \ref{lem.f.steps} \textbf{(j)} to $j=i+1$, we obtain%
\[
\dfrac{t_{n-1,i+1}^{\prime}u_{i+1}}{u_{1}u_{2}\cdots u_{n}}=\dfrac
{t_{n-1,i+2}x_{i+2}}{x_{1}x_{2}\cdots x_{n}}.
\]
Dividing the former equality by the latter, we obtain%
\[
\dfrac{t_{n-1,i-1}^{\prime}u_{i-1}}{u_{1}u_{2}\cdots u_{n}}/\dfrac
{t_{n-1,i+1}^{\prime}u_{i+1}}{u_{1}u_{2}\cdots u_{n}}=\dfrac{t_{n-1,i}x_{i}%
}{x_{1}x_{2}\cdots x_{n}}/\dfrac{t_{n-1,i+2}x_{i+2}}{x_{1}x_{2}\cdots x_{n}}.
\]
This rewrites as%
\[
\dfrac{t_{n-1,i-1}^{\prime}u_{i-1}}{t_{n-1,i+1}^{\prime}u_{i+1}}%
=\dfrac{t_{n-1,i}x_{i}}{t_{n-1,i+2}x_{i+2}}.
\]
Multiplying both sides of this equality by $u_{i+1}$, we obtain%
\begin{equation}
\dfrac{t_{n-1,i-1}^{\prime}u_{i-1}}{t_{n-1,i+1}^{\prime}}=u_{i+1}\cdot
\dfrac{t_{n-1,i}x_{i}}{t_{n-1,i+2}x_{i+2}}. \label{pf.lem.f.steps.k.short.4}%
\end{equation}
Now,%
\begin{align*}
u_{i}\cdot\dfrac{u_{i-1}t_{n-1,i-1}^{\prime}}{y_{i+1}t_{n-1,i+1}^{\prime}}  &
=\underbrace{\dfrac{t_{n-1,i-1}^{\prime}u_{i-1}}{t_{n-1,i+1}^{\prime}}%
}_{\substack{=u_{i+1}\cdot\dfrac{t_{n-1,i}x_{i}}{t_{n-1,i+2}x_{i+2}%
}\\\text{(by (\ref{pf.lem.f.steps.k.short.4}))}}}\cdot u_{i}%
/\underbrace{y_{i+1}}_{\substack{=u_{i+1}\cdot\dfrac{u_{i}t_{n-1,i}}%
{x_{i+2}t_{n-1,i+2}}\\\text{(by Lemma \ref{lem.f.steps} \textbf{(h)}%
,}\\\text{applied to }i+1\text{ instead of }i\text{)}}}\\
&  =u_{i+1}\cdot\dfrac{t_{n-1,i}x_{i}}{t_{n-1,i+2}x_{i+2}}\cdot u_{i}/\left(
u_{i+1}\cdot\dfrac{u_{i}t_{n-1,i}}{x_{i+2}t_{n-1,i+2}}\right)  =x_{i}.
\end{align*}

\end{vershort}

\begin{verlong}
\textbf{(k)} Let $i\in\mathbb{Z}$. Applying Lemma \ref{lem.f.steps}
\textbf{(j)} to $j=i-1$, we obtain%
\[
\dfrac{t_{n-1,i-1}^{\prime}u_{i-1}}{u_{1}u_{2}\cdots u_{n}}=\dfrac
{t_{n-1,\left(  i-1\right)  +1}x_{\left(  i-1\right)  +1}}{x_{1}x_{2}\cdots
x_{n}}=\dfrac{t_{n-1,i}x_{i}}{x_{1}x_{2}\cdots x_{n}}%
\]
(since $\left(  i-1\right)  +1=i$). Applying Lemma \ref{lem.f.steps}
\textbf{(j)} to $j=i+1$, we obtain%
\[
\dfrac{t_{n-1,i+1}^{\prime}u_{i+1}}{u_{1}u_{2}\cdots u_{n}}=\dfrac
{t_{n-1,\left(  i+1\right)  +1}x_{\left(  i+1\right)  +1}}{x_{1}x_{2}\cdots
x_{n}}=\dfrac{t_{n-1,i+2}x_{i+2}}{x_{1}x_{2}\cdots x_{n}}%
\]
(since $\left(  i+1\right)  +1=i+2$). Dividing the former equality by the
latter, we obtain%
\[
\dfrac{t_{n-1,i-1}^{\prime}u_{i-1}}{u_{1}u_{2}\cdots u_{n}}/\dfrac
{t_{n-1,i+1}^{\prime}u_{i+1}}{u_{1}u_{2}\cdots u_{n}}=\dfrac{t_{n-1,i}x_{i}%
}{x_{1}x_{2}\cdots x_{n}}/\dfrac{t_{n-1,i+2}x_{i+2}}{x_{1}x_{2}\cdots x_{n}}.
\]
This rewrites as%
\[
\dfrac{t_{n-1,i-1}^{\prime}u_{i-1}}{t_{n-1,i+1}^{\prime}u_{i+1}}%
=\dfrac{t_{n-1,i}x_{i}}{t_{n-1,i+2}x_{i+2}}%
\]
(since $\dfrac{t_{n-1,i-1}^{\prime}u_{i-1}}{u_{1}u_{2}\cdots u_{n}}%
/\dfrac{t_{n-1,i+1}^{\prime}u_{i+1}}{u_{1}u_{2}\cdots u_{n}}=\dfrac
{t_{n-1,i-1}^{\prime}u_{i-1}}{t_{n-1,i+1}^{\prime}u_{i+1}}$ and\newline%
$\dfrac{t_{n-1,i}x_{i}}{x_{1}x_{2}\cdots x_{n}}/\dfrac{t_{n-1,i+2}x_{i+2}%
}{x_{1}x_{2}\cdots x_{n}}=\dfrac{t_{n-1,i}x_{i}}{t_{n-1,i+2}x_{i+2}}$).
Multiplying both sides of this equality by $u_{i+1}$, we obtain%
\begin{equation}
\dfrac{t_{n-1,i-1}^{\prime}u_{i-1}}{t_{n-1,i+1}^{\prime}}=u_{i+1}\cdot
\dfrac{t_{n-1,i}x_{i}}{t_{n-1,i+2}x_{i+2}}. \label{pf.lem.f.steps.k.4}%
\end{equation}
Now,%
\begin{align*}
u_{i}\cdot\dfrac{u_{i-1}t_{n-1,i-1}^{\prime}}{y_{i+1}t_{n-1,i+1}^{\prime}}  &
=\underbrace{\dfrac{t_{n-1,i-1}^{\prime}u_{i-1}}{t_{n-1,i+1}^{\prime}}%
}_{\substack{=u_{i+1}\cdot\dfrac{t_{n-1,i}x_{i}}{t_{n-1,i+2}x_{i+2}%
}\\\text{(by (\ref{pf.lem.f.steps.k.4}))}}}\cdot u_{i}/\underbrace{y_{i+1}%
}_{\substack{=u_{i+1}\cdot\dfrac{u_{\left(  i+1\right)  -1}t_{n-1,\left(
i+1\right)  -1}}{x_{\left(  i+1\right)  +1}t_{n-1,\left(  i+1\right)  +1}%
}\\\text{(by Lemma \ref{lem.f.steps} \textbf{(h)},}\\\text{applied to
}i+1\text{ instead of }i\text{)}}}\\
&  =u_{i+1}\cdot\dfrac{t_{n-1,i}x_{i}}{t_{n-1,i+2}x_{i+2}}\cdot u_{i}/\left(
u_{i+1}\cdot\dfrac{u_{\left(  i+1\right)  -1}t_{n-1,\left(  i+1\right)  -1}%
}{x_{\left(  i+1\right)  +1}t_{n-1,\left(  i+1\right)  +1}}\right) \\
&  =u_{i+1}\cdot\dfrac{t_{n-1,i}x_{i}}{t_{n-1,i+2}x_{i+2}}\cdot u_{i}/\left(
u_{i+1}\cdot\dfrac{u_{i}t_{n-1,i}}{x_{i+2}t_{n-1,i+2}}\right) \\
&  =x_{i}.
\end{align*}

\end{verlong}

\noindent This proves Lemma \ref{lem.f.steps} \textbf{(k)}.
\end{proof}

\begin{lemma}
\label{lem.f.back}Let $x\in\mathbb{K}^{n}$ be an $n$-tuple. For each
$j\in\mathbb{Z}$, let
\[
q_{j}=\sum_{k=0}^{n-1}\underbrace{x_{j+1}x_{j+2}\cdots x_{j+k}}_{=\prod
_{i=1}^{k}x_{j+i}}\cdot\underbrace{u_{j+k+1}u_{j+k+2}\cdots u_{j+n-1}}%
_{=\prod_{i=k+1}^{n-1}u_{j+i}}.
\]
Let $z\in\mathbb{K}^{n}$ be such that%
\[
z_{i}=u_{i}\cdot\dfrac{u_{i-1}q_{i-1}}{x_{i+1}q_{i+1}}%
\ \ \ \ \ \ \ \ \ \ \text{for each }i\in\left\{  1,2,\ldots,n\right\}  .
\]
Then, $\mathbf{f}_{u}\left(  x\right)  =z$.
\end{lemma}

\begin{vershort}

\begin{proof}
[Proof of Lemma \ref{lem.f.back}.]Let $t_{r,j}$ and $y$ be as in Definition
\ref{def.f}. Then, $t_{n-1,j}=q_{j}$ for each $j\in\mathbb{Z}$ (by comparing
the definitions of $t_{n-1,j}$ and $q_{j}$). Hence, $z_{i}=y_{i}$ for each
$i\in\left\{  1,2,\ldots,n\right\}  $ (by comparing the definitions of $z_{i}$
and $y_{i}$). Hence, $z=y=\mathbf{f}_{u}\left(  x\right)  $ (since
$\mathbf{f}_{u}\left(  x\right)  $ was defined to be $y$). This proves Lemma
\ref{lem.f.back}.
\end{proof}
\end{vershort}

\begin{verlong}

\begin{proof}
[Proof of Lemma \ref{lem.f.back}.]Let $t_{r,j}$ and $y$ be as in Definition
\ref{def.f}. Then, for each $j\in\mathbb{Z}$, we have%
\begin{align}
t_{n-1,j}  &  =\sum_{k=0}^{n-1}\underbrace{x_{j+1}x_{j+2}\cdots x_{j+k}%
}_{=\prod_{i=1}^{k}x_{j+i}}\cdot\underbrace{u_{j+k+1}u_{j+k+2}\cdots
u_{j+n-1}}_{=\prod_{i=k+1}^{n-1}u_{j+i}}\nonumber\\
&  \ \ \ \ \ \ \ \ \ \ \left(  \text{by the definition of }t_{n-1,j}\right)
\nonumber\\
&  =q_{j}\ \ \ \ \ \ \ \ \ \ \left(  \text{by the definition of }q_{j}\right)
. \label{pf.lem.f.back.t=q}%
\end{align}

Now, let $i\in\left\{  1,2,\ldots,n\right\}  $. Then, $t_{n-1,i-1}=q_{i-1}$
(by (\ref{pf.lem.f.back.t=q}), applied to $j=i-1$) and $t_{n-1,i+1}=q_{i+1}$
(by (\ref{pf.lem.f.back.t=q}), applied to $j=i+1$). Now, the definition of $y$
yields%
\begin{align*}
y_{i}  &  =u_{i}\cdot\dfrac{u_{i-1}t_{n-1,i-1}}{x_{i+1}t_{n-1,i+1}}=u_{i}%
\cdot\dfrac{u_{i-1}q_{i-1}}{x_{i+1}q_{i+1}}\ \ \ \ \ \ \ \ \ \ \left(
\text{since }t_{n-1,i-1}=q_{i-1}\text{ and }t_{n-1,i+1}=q_{i+1}\right) \\
&  =z_{i}%
\end{align*}
(since we assumed that $z_{i}=u_{i}\cdot\dfrac{u_{i-1}q_{i-1}}{x_{i+1}q_{i+1}%
}$).

Forget that we fixed $i$. We thus have shown that $y_{i}=z_{i}$ for each
$i\in\left\{  1,2,\ldots,n\right\}  $. In other words, $y=z$. Hence,
$z=y=\mathbf{f}_{u}\left(  x\right)  $ (since $\mathbf{f}_{u}\left(  x\right)
$ was defined to be $y$). This proves Lemma \ref{lem.f.back}.
\end{proof}
\end{verlong}

For future convenience, let us restate Lemma \ref{lem.f.back} with different labels:

\begin{lemma}
\label{lem.f.back-y}Let $y\in\mathbb{K}^{n}$ be an $n$-tuple. For each
$j\in\mathbb{Z}$, let
\[
r_{j}=\sum_{k=0}^{n-1}\underbrace{y_{j+1}y_{j+2}\cdots y_{j+k}}_{=\prod
_{i=1}^{k}y_{j+i}}\cdot\underbrace{u_{j+k+1}u_{j+k+2}\cdots u_{j+n-1}}%
_{=\prod_{i=k+1}^{n-1}u_{j+i}}.
\]
Let $x\in\mathbb{K}^{n}$ be such that%
\[
x_{i}=u_{i}\cdot\dfrac{u_{i-1}r_{i-1}}{y_{i+1}r_{i+1}}%
\ \ \ \ \ \ \ \ \ \ \text{for each }i\in\left\{  1,2,\ldots,n\right\}  .
\]
Then, $\mathbf{f}_{u}\left(  y\right)  =x$.
\end{lemma}

\begin{proof}
[Proof of Lemma \ref{lem.f.back-y}.]Lemma \ref{lem.f.back-y} is just Lemma
\ref{lem.f.back}, with $x$, $q_{j}$ and $z$ renamed as $y$, $r_{j}$ and $x$.
\end{proof}

We are now ready for the proof of Theorem \ref{thm.f.full}:

\begin{proof}
[Proof of Theorem \ref{thm.f.full}.]\textbf{(a)} Let $x\in\mathbb{K}^{n}$. We
shall prove that $\left(  \mathbf{f}_{u}\circ\mathbf{f}_{u}\right)  \left(
x\right)  =x$.

Let $t_{r,j}$ and $y$ be as in Definition \ref{def.f}. Then, $\mathbf{f}%
_{u}\left(  x\right)  =y$ (by the definition of $\mathbf{f}_{u}$), so that
$y=\mathbf{f}_{u}\left(  x\right)  $. Let $t_{r,j}^{\prime}$ (for each
$r\in\mathbb{N}$ and $j\in\mathbb{Z}$) be as in Lemma \ref{lem.f.steps}. The
definition of $t_{n-1,j}^{\prime}$ shows that%
\[
t_{n-1,j}^{\prime}=\sum_{k=0}^{n-1}\underbrace{y_{j+1}y_{j+2}\cdots y_{j+k}%
}_{=\prod_{i=1}^{k}y_{j+i}}\cdot\underbrace{u_{j+k+1}u_{j+k+2}\cdots
u_{j+n-1}}_{=\prod_{i=k+1}^{n-1}u_{j+i}}%
\]
for each $j\in\mathbb{Z}$. Lemma \ref{lem.f.steps} \textbf{(k)} shows that%
\[
x_{i}=u_{i}\cdot\dfrac{u_{i-1}t_{n-1,i-1}^{\prime}}{y_{i+1}t_{n-1,i+1}%
^{\prime}}\ \ \ \ \ \ \ \ \ \ \text{for each }i\in\left\{  1,2,\ldots
,n\right\}  .
\]
Thus, Lemma \ref{lem.f.back-y} (applied to $r_{j}=t_{n-1,j}^{\prime}$) yields
that $\mathbf{f}_{u}\left(  y\right)  =x$. Hence, $x=\mathbf{f}_{u}\left(
\underbrace{y}_{=\mathbf{f}_{u}\left(  x\right)  }\right)  =\mathbf{f}%
_{u}\left(  \mathbf{f}_{u}\left(  x\right)  \right)  =\left(  \mathbf{f}%
_{u}\circ\mathbf{f}_{u}\right)  \left(  x\right)  $. In other words, $\left(
\mathbf{f}_{u}\circ\mathbf{f}_{u}\right)  \left(  x\right)  =x$.

Forget that we fixed $x$. We thus have proved that $\left(  \mathbf{f}%
_{u}\circ\mathbf{f}_{u}\right)  \left(  x\right)  =x$ for each $x\in
\mathbb{K}^{n}$. In other words, $\mathbf{f}_{u}\circ\mathbf{f}_{u}%
=\operatorname*{id}$. In other words, $\mathbf{f}_{u}$ is an involution. This
proves Theorem \ref{thm.f.full} \textbf{(a)}.

\textbf{(b)} Let $t_{r,j}$ be as in Definition \ref{def.f}. Note that the $y$
from Definition \ref{def.f} is precisely the $y$ in Theorem \ref{thm.f.full}
\textbf{(b)} (because both $y$'s satisfy $\mathbf{f}_{u}\left(  x\right)  =y$).

Lemma \ref{lem.f.steps} \textbf{(a)} yields $t_{n-1,0}=t_{n-1,n}$ (since
$0\equiv n\operatorname{mod}n$) and $t_{n-1,1}=t_{n-1,n+1}$ (since $1\equiv
n+1\operatorname{mod}n$). Multiplying these two equalities, we obtain
$t_{n-1,0}t_{n-1,1}=t_{n-1,n}t_{n-1,n+1}$, whence
\begin{equation}
\dfrac{t_{n-1,0}t_{n-1,1}}{t_{n-1,n}t_{n-1,n+1}}=1. \label{pf.thm.f.full.b.1}%
\end{equation}

We have%
\begin{align*}
y_{1}y_{2}\cdots y_{n}  &  =\prod_{i=1}^{n}\underbrace{y_{i}}%
_{\substack{=u_{i}\cdot\dfrac{u_{i-1}t_{n-1,i-1}}{x_{i+1}t_{n-1,i+1}%
}\\\text{(by the definition of }y\\\text{in Definition \ref{def.f})}}%
}=\prod_{i=1}^{n}\left(  u_{i}\cdot\dfrac{u_{i-1}t_{n-1,i-1}}{x_{i+1}%
t_{n-1,i+1}}\right) \\
&  =\left(  \prod_{i=1}^{n}u_{i}\right)  \cdot\dfrac{\left(  \prod_{i=1}%
^{n}u_{i-1}\right)  \cdot\left(  \prod_{i=1}^{n}t_{n-1,i-1}\right)  }{\left(
\prod_{i=1}^{n}x_{i+1}\right)  \cdot\left(  \prod_{i=1}^{n}t_{n-1,i+1}\right)
}\\
&  =\underbrace{\left(  \prod_{i=1}^{n}u_{i}\right)  }_{=u_{1}u_{2}\cdots
u_{n}}\cdot\underbrace{\left(  \prod_{i=1}^{n}u_{i-1}\right)  }%
_{\substack{=u_{0}u_{1}\cdots u_{n-1}\\=u_{1}u_{2}\cdots u_{n}\\\text{(by
Lemma \ref{lem.aprod})}}}\cdot\underbrace{\left(  \prod_{i=1}^{n}%
t_{n-1,i-1}\right)  }_{\substack{=t_{n-1,0}t_{n-1,1}\cdots t_{n-1,n-1}%
\\=\dfrac{t_{n-1,0}t_{n-1,1}\cdots t_{n-1,n+1}}{t_{n-1,n}t_{n-1,n+1}}}}\\
&  \ \ \ \ \ \ \ \ \ \ /\left(  \underbrace{\left(  \prod_{i=1}^{n}%
x_{i+1}\right)  }_{\substack{=x_{2}x_{3}\cdots x_{n+1}\\=x_{1}x_{2}\cdots
x_{n}\\\text{(by Lemma \ref{lem.aprod})}}}\cdot\underbrace{\left(  \prod
_{i=1}^{n}t_{n-1,i+1}\right)  }_{\substack{=t_{n-1,2}t_{n-1,3}\cdots
t_{n-1,n+1}\\=\dfrac{t_{n-1,0}t_{n-1,1}\cdots t_{n-1,n+1}}{t_{n-1,0}t_{n-1,1}%
}}}\right) \\
&  =\left(  u_{1}u_{2}\cdots u_{n}\right)  \cdot\left(  u_{1}u_{2}\cdots
u_{n}\right)  \cdot\dfrac{t_{n-1,0}t_{n-1,1}\cdots t_{n-1,n+1}}{t_{n-1,n}%
t_{n-1,n+1}}\\
&  \ \ \ \ \ \ \ \ \ \ /\left(  \left(  x_{1}x_{2}\cdots x_{n}\right)
\cdot\dfrac{t_{n-1,0}t_{n-1,1}\cdots t_{n-1,n+1}}{t_{n-1,0}t_{n-1,1}}\right)
\\
&  =\dfrac{\left(  u_{1}u_{2}\cdots u_{n}\right)  ^{2}}{x_{1}x_{2}\cdots
x_{n}}\cdot\underbrace{\dfrac{t_{n-1,0}t_{n-1,1}}{t_{n-1,n}t_{n-1,n+1}}%
}_{\substack{=1\\\text{(by (\ref{pf.thm.f.full.b.1}))}}}=\dfrac{\left(
u_{1}u_{2}\cdots u_{n}\right)  ^{2}}{x_{1}x_{2}\cdots x_{n}},
\end{align*}
so that%
\[
y_{1}y_{2}\cdots y_{n}\cdot x_{1}x_{2}\cdots x_{n}=\left(  u_{1}u_{2}\cdots
u_{n}\right)  ^{2}.
\]
This proves Theorem \ref{thm.f.full} \textbf{(b)}.

\textbf{(c)} Let $t_{r,j}$ be as in Definition \ref{def.f}. Note that the $y$
from Definition \ref{def.f} is precisely the $y$ in Theorem \ref{thm.f.full}
\textbf{(c)} (because both $y$'s satisfy $\mathbf{f}_{u}\left(  x\right)  =y$).

Let $i\in\mathbb{Z}$. Then,%
\begin{align}
&  u_{i}+\underbrace{y_{i}}_{\substack{=u_{i}\cdot\dfrac{u_{i-1}t_{n-1,i-1}%
}{x_{i+1}t_{n-1,i+1}}\\\text{(by Lemma \ref{lem.f.steps} \textbf{(h)})}%
}}\nonumber\\
&  =u_{i}+u_{i}\cdot\dfrac{u_{i-1}t_{n-1,i-1}}{x_{i+1}t_{n-1,i+1}}%
=u_{i}\left(  1+\dfrac{u_{i-1}t_{n-1,i-1}}{x_{i+1}t_{n-1,i+1}}\right)
=u_{i}\cdot\dfrac{x_{i+1}t_{n-1,i+1}+u_{i-1}t_{n-1,i-1}}{x_{i+1}t_{n-1,i+1}%
}\nonumber\\
&  =\dfrac{u_{i}}{x_{i+1}t_{n-1,i+1}}\underbrace{\left(  x_{i+1}%
t_{n-1,i+1}+u_{i-1}t_{n-1,i-1}\right)  }_{\substack{=\left(  x_{i}%
+u_{i}\right)  t_{n-1,i}\\\text{(by Lemma \ref{lem.f.steps} \textbf{(f)})}%
}}\nonumber\\
&  =\dfrac{u_{i}}{x_{i+1}t_{n-1,i+1}}\left(  x_{i}+u_{i}\right)  t_{n-1,i}.
\label{pf.thm.f.full.c.1}%
\end{align}
Now,%
\begin{align*}
\dfrac{1}{u_{i}}+\dfrac{1}{y_{i}}  &  =\underbrace{\left(  u_{i}+y_{i}\right)
}_{\substack{=\dfrac{u_{i}}{x_{i+1}t_{n-1,i+1}}\left(  x_{i}+u_{i}\right)
t_{n-1,i}\\\text{(by (\ref{pf.thm.f.full.c.1}))}}}/\left(  u_{i}\cdot
y_{i}\right)  =\dfrac{u_{i}}{x_{i+1}t_{n-1,i+1}}\left(  x_{i}+u_{i}\right)
t_{n-1,i}/\left(  u_{i}\cdot y_{i}\right) \\
&  =\dfrac{1}{x_{i+1}t_{n-1,i+1}}\left(  x_{i}+u_{i}\right)  t_{n-1,i}%
/\underbrace{y_{i}}_{\substack{=u_{i}\cdot\dfrac{u_{i-1}t_{n-1,i-1}}%
{x_{i+1}t_{n-1,i+1}}\\\text{(by Lemma \ref{lem.f.steps} \textbf{(h)})}}}\\
&  =\dfrac{1}{x_{i+1}t_{n-1,i+1}}\left(  x_{i}+u_{i}\right)  t_{n-1,i}/\left(
u_{i}\cdot\dfrac{u_{i-1}t_{n-1,i-1}}{x_{i+1}t_{n-1,i+1}}\right)
=\dfrac{\left(  x_{i}+u_{i}\right)  t_{n-1,i}}{u_{i}u_{i-1}t_{n-1,i-1}}.
\end{align*}
The same argument (applied to $i+1$ instead of $i$) yields%
\[
\dfrac{1}{u_{i+1}}+\dfrac{1}{y_{i+1}}=\dfrac{\left(  x_{i+1}+u_{i+1}\right)
t_{n-1,i+1}}{u_{i+1}u_{\left(  i+1\right)  -1}t_{n-1,\left(  i+1\right)  -1}%
}=\dfrac{\left(  x_{i+1}+u_{i+1}\right)  t_{n-1,i+1}}{u_{i+1}u_{i}t_{n-1,i}}%
\]
(since $\left(  i+1\right)  -1=i$). Multiplying (\ref{pf.thm.f.full.c.1}) with
this equality, we obtain%
\begin{align*}
\left(  u_{i}+y_{i}\right)  \left(  \dfrac{1}{u_{i+1}}+\dfrac{1}{y_{i+1}%
}\right)   &  =\dfrac{u_{i}}{x_{i+1}t_{n-1,i+1}}\left(  x_{i}+u_{i}\right)
t_{n-1,i}\cdot\dfrac{\left(  x_{i+1}+u_{i+1}\right)  t_{n-1,i+1}}{u_{i+1}%
u_{i}t_{n-1,i}}\\
&  =\underbrace{\left(  x_{i}+u_{i}\right)  }_{=u_{i}+x_{i}}\cdot
\underbrace{\dfrac{x_{i+1}+u_{i+1}}{x_{i+1}u_{i+1}}}_{=\dfrac{1}{u_{i+1}%
}+\dfrac{1}{x_{i+1}}}=\left(  u_{i}+x_{i}\right)  \left(  \dfrac{1}{u_{i+1}%
}+\dfrac{1}{x_{i+1}}\right)  .
\end{align*}
This proves Theorem \ref{thm.f.full} \textbf{(c)}.

\textbf{(d)} Let $t_{r,j}$ be as in Definition \ref{def.f}. Note that the $y$
from Definition \ref{def.f} is precisely the $y$ in Theorem \ref{thm.f.full}
\textbf{(d)} (because both $y$'s satisfy $\mathbf{f}_{u}\left(  x\right)  =y$).

We have
\begin{align}
\prod_{i=1}^{n}x_{i+1}  &  =x_{2}x_{3}\cdots x_{n+1}=x_{1}x_{2}\cdots
x_{n}\ \ \ \ \ \ \ \ \ \ \left(  \text{by Lemma \ref{lem.aprod}}\right)
\nonumber\\
&  =\prod_{i=1}^{n}x_{i} \label{pf.thm.f.full.d.1}%
\end{align}
and%
\begin{align}
\prod_{i=1}^{n}t_{n-1,i+1}  &  =\prod_{i=2}^{n+1}t_{n-1,i}%
\ \ \ \ \ \ \ \ \ \ \left(  \text{here, we have substituted }i\text{ for
}i+1\text{ in the product}\right) \nonumber\\
&  =\underbrace{t_{n-1,n+1}}_{\substack{=t_{n-1,1}\\\text{(by Lemma
\ref{lem.f.steps} \textbf{(a)},}\\\text{since }n+1\equiv1\operatorname{mod}%
n\text{)}}}\prod_{i=2}^{n}t_{n-1,i}=t_{n-1,1}\prod_{i=2}^{n}t_{n-1,i}%
\nonumber\\
&  =\prod_{i=1}^{n}t_{n-1,i}. \label{pf.thm.f.full.d.2}%
\end{align}

Every $i\in\mathbb{Z}$ satisfies (\ref{pf.thm.f.full.c.1}) (as we have shown
in the proof of Theorem \ref{thm.f.full} \textbf{(c)} above). Hence,%
\begin{align*}
\prod_{i=1}^{n}\underbrace{\left(  u_{i}+y_{i}\right)  }_{\substack{=\dfrac
{u_{i}}{x_{i+1}t_{n-1,i+1}}\left(  x_{i}+u_{i}\right)  t_{n-1,i}\\\text{(by
(\ref{pf.thm.f.full.c.1}))}}}  &  =\prod_{i=1}^{n}\left(  \dfrac{u_{i}%
}{x_{i+1}t_{n-1,i+1}}\left(  x_{i}+u_{i}\right)  t_{n-1,i}\right) \\
&  =\dfrac{\prod_{i=1}^{n}u_{i}}{\left(  \prod_{i=1}^{n}x_{i+1}\right)
\left(  \prod_{i=1}^{n}t_{n-1,i+1}\right)  }\left(  \prod_{i=1}^{n}\left(
x_{i}+u_{i}\right)  \right)  \left(  \prod_{i=1}^{n}t_{n-1,i}\right) \\
&  =\dfrac{\prod_{i=1}^{n}u_{i}}{\left(  \prod_{i=1}^{n}x_{i}\right)  \left(
\prod_{i=1}^{n}t_{n-1,i}\right)  }\left(  \prod_{i=1}^{n}\left(  x_{i}%
+u_{i}\right)  \right)  \left(  \prod_{i=1}^{n}t_{n-1,i}\right) \\
&  \ \ \ \ \ \ \ \ \ \ \ \ \ \ \ \ \ \ \ \ \left(  \text{by
(\ref{pf.thm.f.full.d.1}) and (\ref{pf.thm.f.full.d.2})}\right) \\
&  =\left(  \prod_{i=1}^{n}u_{i}\right)  \underbrace{\dfrac{\prod_{i=1}%
^{n}\left(  x_{i}+u_{i}\right)  }{\prod_{i=1}^{n}x_{i}}}_{=\prod_{i=1}%
^{n}\dfrac{x_{i}+u_{i}}{x_{i}}}=\left(  \prod_{i=1}^{n}u_{i}\right)
\prod_{i=1}^{n}\dfrac{x_{i}+u_{i}}{x_{i}}.
\end{align*}
Thus,%
\[
\dfrac{\prod_{i=1}^{n}\left(  u_{i}+y_{i}\right)  }{\prod_{i=1}^{n}u_{i}%
}=\prod_{i=1}^{n}\underbrace{\dfrac{x_{i}+u_{i}}{x_{i}}}_{=\dfrac{u_{i}+x_{i}%
}{x_{i}}}=\prod_{i=1}^{n}\dfrac{u_{i}+x_{i}}{x_{i}},
\]
so that%
\[
\prod_{i=1}^{n}\dfrac{u_{i}+x_{i}}{x_{i}}=\dfrac{\prod_{i=1}^{n}\left(
u_{i}+y_{i}\right)  }{\prod_{i=1}^{n}u_{i}}=\prod_{i=1}^{n}\dfrac{u_{i}+y_{i}%
}{u_{i}}.
\]
This proves Theorem \ref{thm.f.full} \textbf{(d)}.
\end{proof}

Let us observe one more property of the involution $\mathbf{f}_{u}$ (even
though we will have no use for it):

\begin{proposition}
\label{prop.f.equal-case}Let $x\in\mathbb{K}^{n}$ be such that $x_{1}%
x_{2}\cdots x_{n}=u_{1}u_{2}\cdots u_{n}$. Then, $\mathbf{f}_{u}\left(
x\right)  =x$.
\end{proposition}

\begin{proof}
[Proof of Proposition \ref{prop.f.equal-case}.]Let $t_{r,j}$ and $y$ be as in
Definition \ref{def.f}. Then, $\mathbf{f}_{u}\left(  x\right)  =y$ (by the
definition of $\mathbf{f}_{u}$).

\begin{vershort}
Let $i\in\mathbb{Z}$. We shall first show that $u_{i-1}t_{n-1,i-1}%
=x_{i}t_{n-1,i}$.

Indeed, the definition of $t_{n-1,i}$ yields%
\[
t_{n-1,i}=\sum_{k=0}^{n-1}x_{i+1}x_{i+2}\cdots x_{i+k}\cdot u_{i+k+1}%
u_{i+k+2}\cdots u_{i+n-1}.
\]
Multiplying both sides of this equality by $x_{i}$, we find%
\begin{align}
x_{i}t_{n-1,i}  &  =x_{i}\sum_{k=0}^{n-1}x_{i+1}x_{i+2}\cdots x_{i+k}\cdot
u_{i+k+1}u_{i+k+2}\cdots u_{i+n-1}\nonumber\\
&  =\sum_{k=0}^{n-1}\underbrace{x_{i}\cdot x_{i+1}x_{i+2}\cdots x_{i+k}%
}_{=x_{i}x_{i+1}\cdots x_{i+k}}\cdot u_{i+k+1}u_{i+k+2}\cdots u_{i+n-1}%
\nonumber\\
&  =\sum_{k=0}^{n-1}x_{i}x_{i+1}\cdots x_{i+k}\cdot u_{i+k+1}u_{i+k+2}\cdots
u_{i+n-1}\nonumber\\
&  =\underbrace{x_{i}x_{i+1}\cdots x_{i+n-1}}_{\substack{=x_{1}x_{2}\cdots
x_{n}\\\text{(by Lemma \ref{lem.aprod})}}}\cdot\underbrace{u_{i+n}%
u_{i+n+1}\cdots u_{i+n-1}}_{=\left(  \text{empty product}\right)
=1}\nonumber\\
&  \ \ \ \ \ \ \ \ \ \ +\sum_{k=0}^{n-2}x_{i}x_{i+1}\cdots x_{i+k}\cdot
u_{i+k+1}u_{i+k+2}\cdots u_{i+n-1}\nonumber\\
&  \ \ \ \ \ \ \ \ \ \ \ \ \ \ \ \ \ \ \ \ \left(  \text{here, we have split
off the addend for }k=n-1\text{ from the sum}\right) \nonumber\\
&  =\underbrace{x_{1}x_{2}\cdots x_{n}}_{=u_{1}u_{2}\cdots u_{n}%
}+\underbrace{\sum_{k=0}^{n-2}x_{i}x_{i+1}\cdots x_{i+k}\cdot u_{i+k+1}%
u_{i+k+2}\cdots u_{i+n-1}}_{\substack{=\sum_{k=1}^{n-1}x_{i}x_{i+1}\cdots
x_{i+k-1}\cdot u_{i+k}u_{i+k+1}\cdots u_{i+n-1}\\\text{(here, we have
substituted }k-1\text{ for }k\text{ in the sum)}}}\nonumber\\
&  =u_{1}u_{2}\cdots u_{n}+\sum_{k=1}^{n-1}x_{i}x_{i+1}\cdots x_{i+k-1}\cdot
u_{i+k}u_{i+k+1}\cdots u_{i+n-1}. \label{pf.prop.f.equal-case.short.1}%
\end{align}

On the other hand, we have $u_{i-1}=u_{i+n-1}$ (since $i-1\equiv
i+n-1\operatorname{mod}n$) and%
\begin{align*}
t_{n-1,i-1}  &  =\sum_{k=0}^{n-1}\underbrace{x_{\left(  i-1\right)
+1}x_{\left(  i-1\right)  +2}\cdots x_{\left(  i-1\right)  +k}}_{=x_{i}%
x_{i+1}\cdots x_{i+k-1}}\cdot\underbrace{u_{\left(  i-1\right)  +k+1}%
u_{\left(  i-1\right)  +k+2}\cdots u_{\left(  i-1\right)  +n-1}}%
_{=u_{i+k}u_{i+k+1}\cdots u_{i+n-2}}\\
&  \ \ \ \ \ \ \ \ \ \ \left(  \text{by the definition of }t_{n-1,i-1}\right)
\\
&  =\sum_{k=0}^{n-1}x_{i}x_{i+1}\cdots x_{i+k-1}\cdot u_{i+k}u_{i+k+1}\cdots
u_{i+n-2}.
\end{align*}
Multiplying these two equalities, we obtain%
\begin{align*}
u_{i-1}t_{n-1,i-1}  &  =u_{i+n-1}\sum_{k=0}^{n-1}x_{i}x_{i+1}\cdots
x_{i+k-1}\cdot u_{i+k}u_{i+k+1}\cdots u_{i+n-2}\\
&  =\sum_{k=0}^{n-1}x_{i}x_{i+1}\cdots x_{i+k-1}\cdot\underbrace{u_{i+k}%
u_{i+k+1}\cdots u_{i+n-2}\cdot u_{i+n-1}}_{=u_{i+k}u_{i+k+1}\cdots u_{i+n-1}%
}\\
&  =\sum_{k=0}^{n-1}x_{i}x_{i+1}\cdots x_{i+k-1}\cdot u_{i+k}u_{i+k+1}\cdots
u_{i+n-1}\\
&  =\underbrace{x_{i}x_{i+1}\cdots x_{i+0-1}}_{=\left(  \text{empty
product}\right)  =1}\cdot\underbrace{u_{i+0}u_{i+0+1}\cdots u_{i+n-1}%
}_{\substack{=u_{i}u_{i+1}\cdots u_{i+n-1}\\=u_{1}u_{2}\cdots u_{n}\\\text{(by
Lemma \ref{lem.aprod})}}}\\
&  \ \ \ \ \ \ \ \ \ \ +\sum_{k=1}^{n-1}x_{i}x_{i+1}\cdots x_{i+k-1}\cdot
u_{i+k}u_{i+k+1}\cdots u_{i+n-1}\\
&  \ \ \ \ \ \ \ \ \ \ \ \ \ \ \ \ \ \ \ \ \left(  \text{here, we have split
off the addend for }k=0\text{ from the sum}\right) \\
&  =u_{1}u_{2}\cdots u_{n}+\sum_{k=1}^{n-1}x_{i}x_{i+1}\cdots x_{i+k-1}\cdot
u_{i+k}u_{i+k+1}\cdots u_{i+n-1}.
\end{align*}
Comparing this with (\ref{pf.prop.f.equal-case.short.1}), we obtain%
\begin{equation}
u_{i-1}t_{n-1,i-1}=x_{i}t_{n-1,i}. \label{pf.prop.f.equal-case.short.4}%
\end{equation}
The same argument (applied to $i+1$ instead of $i$) yields%
\begin{equation}
u_{i}t_{n-1,i}=x_{i+1}t_{n-1,i+1}. \label{pf.prop.f.equal-case.short.5}%
\end{equation}

Now, the definition of $y$ yields%
\begin{align*}
y_{i}  &  =u_{i}\cdot\dfrac{u_{i-1}t_{n-1,i-1}}{x_{i+1}t_{n-1,i+1}}%
=u_{i}\underbrace{u_{i-1}t_{n-1,i-1}}_{\substack{=x_{i}t_{n-1,i}\\\text{(by
(\ref{pf.prop.f.equal-case.short.4}))}}}/\underbrace{\left(  x_{i+1}%
t_{n-1,i+1}\right)  }_{\substack{=u_{i}t_{n-1,i}\\\text{(by
(\ref{pf.prop.f.equal-case.short.5}))}}}\\
&  =u_{i}x_{i}t_{n-1,i}/\left(  u_{i}t_{n-1,i}\right)  =x_{i}.
\end{align*}

\end{vershort}

\begin{verlong}
Let $i\in\mathbb{Z}$. We shall first show that $u_{i-1}t_{n-1,i-1}%
=x_{i}t_{n-1,i}$.

Indeed, the definition of $t_{n-1,i}$ yields%
\[
t_{n-1,i}=\sum_{k=0}^{n-1}x_{i+1}x_{i+2}\cdots x_{i+k}\cdot u_{i+k+1}%
u_{i+k+2}\cdots u_{i+n-1}.
\]
Multiplying both sides of this equality by $x_{i}$, we find%
\begin{align}
x_{i}t_{n-1,i}  &  =x_{i}\sum_{k=0}^{n-1}x_{i+1}x_{i+2}\cdots x_{i+k}\cdot
u_{i+k+1}u_{i+k+2}\cdots u_{i+n-1}\nonumber\\
&  =\sum_{k=0}^{n-1}\underbrace{x_{i}\cdot x_{i+1}x_{i+2}\cdots x_{i+k}%
}_{=x_{i}x_{i+1}\cdots x_{i+k}}\cdot u_{i+k+1}u_{i+k+2}\cdots u_{i+n-1}%
\nonumber\\
&  =\sum_{k=0}^{n-1}x_{i}x_{i+1}\cdots x_{i+k}\cdot u_{i+k+1}u_{i+k+2}\cdots
u_{i+n-1}\nonumber\\
&  =\underbrace{x_{i}x_{i+1}\cdots x_{i+n-1}}_{\substack{=x_{\left(
i-1\right)  +1}x_{\left(  i-1\right)  +2}\cdots x_{\left(  i-1\right)
+n}\\=x_{1}x_{2}\cdots x_{n}\\\text{(by Lemma \ref{lem.aprod})}}%
}\cdot\underbrace{u_{i+\left(  n-1\right)  +1}u_{i+\left(  n-1\right)
+2}\cdots u_{i+n-1}}_{\substack{=u_{i+n}u_{i+n+1}\cdots u_{i+n-1}\\=\left(
\text{empty product}\right)  =1}}\nonumber\\
&  \ \ \ \ \ \ \ \ \ \ +\sum_{k=0}^{n-2}x_{i}x_{i+1}\cdots x_{i+k}\cdot
u_{i+k+1}u_{i+k+2}\cdots u_{i+n-1}\nonumber\\
&  \ \ \ \ \ \ \ \ \ \ \ \ \ \ \ \ \ \ \ \ \left(  \text{here, we have split
off the addend for }k=n-1\text{ from the sum}\right) \nonumber\\
&  =\underbrace{x_{1}x_{2}\cdots x_{n}}_{=u_{1}u_{2}\cdots u_{n}%
}+\underbrace{\sum_{k=0}^{n-2}x_{i}x_{i+1}\cdots x_{i+k}\cdot u_{i+k+1}%
u_{i+k+2}\cdots u_{i+n-1}}_{\substack{=\sum_{k=1}^{n-1}x_{i}x_{i+1}\cdots
x_{i+\left(  k-1\right)  }\cdot u_{i+\left(  k-1\right)  +1}u_{i+\left(
k-1\right)  +2}\cdots u_{i+n-1}\\\text{(here, we have substituted }k-1\text{
for }k\text{ in the sum)}}}\nonumber\\
&  =u_{1}u_{2}\cdots u_{n}+\sum_{k=1}^{n-1}\underbrace{x_{i}x_{i+1}\cdots
x_{i+\left(  k-1\right)  }}_{=x_{i}x_{i+1}\cdots x_{i+k-1}}\cdot
\underbrace{u_{i+\left(  k-1\right)  +1}u_{i+\left(  k-1\right)  +2}\cdots
u_{i+n-1}}_{=u_{i+k}u_{i+k+1}\cdots u_{i+n-1}}\nonumber\\
&  =u_{1}u_{2}\cdots u_{n}+\sum_{k=1}^{n-1}x_{i}x_{i+1}\cdots x_{i+k-1}\cdot
u_{i+k}u_{i+k+1}\cdots u_{i+n-1}. \label{pf.prop.f.equal-case.1}%
\end{align}

On the other hand, Convention \ref{conv.bir.peri} yields that $u_{j}=u_{j+n}$
for each $j\in\mathbb{Z}$. Applying this to $j=i-1$, we obtain%
\[
u_{i-1}=u_{i-1+n}=u_{i+n-1}\ \ \ \ \ \ \ \ \ \ \left(  \text{since
}i-1+n=i+n-1\right)  .
\]
Furthermore,%
\begin{align*}
t_{n-1,i-1}  &  =\sum_{k=0}^{n-1}\underbrace{x_{\left(  i-1\right)
+1}x_{\left(  i-1\right)  +2}\cdots x_{\left(  i-1\right)  +k}}_{=x_{i}%
x_{i+1}\cdots x_{i+k-1}}\cdot\underbrace{u_{\left(  i-1\right)  +k+1}%
u_{\left(  i-1\right)  +k+2}\cdots u_{\left(  i-1\right)  +n-1}}%
_{=u_{i+k}u_{i+k+1}\cdots u_{i+n-2}}\\
&  \ \ \ \ \ \ \ \ \ \ \left(  \text{by the definition of }t_{n-1,i-1}\right)
\\
&  =\sum_{k=0}^{n-1}x_{i}x_{i+1}\cdots x_{i+k-1}\cdot u_{i+k}u_{i+k+1}\cdots
u_{i+n-2}.
\end{align*}
Multiplying these two equalities, we obtain%
\begin{align*}
u_{i-1}t_{n-1,i-1}  &  =u_{i+n-1}\sum_{k=0}^{n-1}x_{i}x_{i+1}\cdots
x_{i+k-1}\cdot u_{i+k}u_{i+k+1}\cdots u_{i+n-2}\\
&  =\sum_{k=0}^{n-1}x_{i}x_{i+1}\cdots x_{i+k-1}\cdot\underbrace{u_{i+k}%
u_{i+k+1}\cdots u_{i+n-2}\cdot u_{i+n-1}}_{=u_{i+k}u_{i+k+1}\cdots u_{i+n-1}%
}\\
&  =\sum_{k=0}^{n-1}x_{i}x_{i+1}\cdots x_{i+k-1}\cdot u_{i+k}u_{i+k+1}\cdots
u_{i+n-1}\\
&  =\underbrace{x_{i}x_{i+1}\cdots x_{i+0-1}}_{=\left(  \text{empty
product}\right)  =1}\cdot\underbrace{u_{i+0}u_{i+0+1}\cdots u_{i+n-1}%
}_{\substack{=u_{i}u_{i+1}\cdots u_{i+n-1}\\=u_{\left(  i-1\right)
+1}u_{\left(  i-1\right)  +2}\cdots u_{\left(  i-1\right)  +n}\\=u_{1}%
u_{2}\cdots u_{n}\\\text{(by Lemma \ref{lem.aprod})}}}\\
&  \ \ \ \ \ \ \ \ \ \ +\sum_{k=1}^{n-1}x_{i}x_{i+1}\cdots x_{i+k-1}\cdot
u_{i+k}u_{i+k+1}\cdots u_{i+n-1}\\
&  \ \ \ \ \ \ \ \ \ \ \ \ \ \ \ \ \ \ \ \ \left(  \text{here, we have split
off the addend for }k=0\text{ from the sum}\right) \\
&  =u_{1}u_{2}\cdots u_{n}+\sum_{k=1}^{n-1}x_{i}x_{i+1}\cdots x_{i+k-1}\cdot
u_{i+k}u_{i+k+1}\cdots u_{i+n-1}.
\end{align*}
Comparing this with (\ref{pf.prop.f.equal-case.1}), we obtain%
\begin{equation}
u_{i-1}t_{n-1,i-1}=x_{i}t_{n-1,i}. \label{pf.prop.f.equal-case.4}%
\end{equation}
The same argument (applied to $i+1$ instead of $i$) yields%
\[
u_{\left(  i+1\right)  -1}t_{n-1,\left(  i+1\right)  -1}=x_{i+1}t_{n-1,i+1}.
\]
In other words,%
\begin{equation}
u_{i}t_{n-1,i}=x_{i+1}t_{n-1,i+1} \label{pf.prop.f.equal-case.5}%
\end{equation}
(since $\left(  i+1\right)  -1=i$).

Now, the definition of $y$ yields%
\begin{align*}
y_{i}  &  =u_{i}\cdot\dfrac{u_{i-1}t_{n-1,i-1}}{x_{i+1}t_{n-1,i+1}}%
=u_{i}\underbrace{u_{i-1}t_{n-1,i-1}}_{\substack{=x_{i}t_{n-1,i}\\\text{(by
(\ref{pf.prop.f.equal-case.4}))}}}/\underbrace{\left(  x_{i+1}t_{n-1,i+1}%
\right)  }_{\substack{=u_{i}t_{n-1,i}\\\text{(by (\ref{pf.prop.f.equal-case.5}%
))}}}\\
&  =u_{i}x_{i}t_{n-1,i}/\left(  u_{i}t_{n-1,i}\right)  =x_{i}.
\end{align*}

\end{verlong}

Now, forget that we fixed $i$. We thus have proved that $y_{i}=x_{i}$ for each
$i\in\mathbb{Z}$. Thus, in particular, $y_{i}=x_{i}$ for each $i\in\left\{
1,2,\ldots,n\right\}  $. In other words, $y=x$. Hence, $\mathbf{f}_{u}\left(
x\right)  =y=x$. This proves Proposition \ref{prop.f.equal-case}.
\end{proof}

\begin{remark}
\label{rmk.f.invol-by-trick}There is an alternative proof of Theorem
\ref{thm.f.full} \textbf{(a)} that avoids the use of the more complicated
parts of Lemma \ref{lem.f.steps} (specifically, of parts \textbf{(g)},
\textbf{(i)}, \textbf{(j)} and \textbf{(k)}). Let us outline this proof:

The claim of Theorem \ref{thm.f.full} \textbf{(a)} can be restated as the
equality $\mathbf{f}_{u}\left(  \mathbf{f}_{u}\left(  x\right)  \right)  =x$
for each $x\in\mathbb{K}^{n}$ and each $u\in\mathbb{K}^{n}$ (we are not
regarding $u$ as fixed here). This equality boils down to a set of identities
between rational functions in the variables $u_{1},u_{2},\ldots,u_{n}%
,x_{1},x_{2},\ldots,x_{n}$ (since each entry of $\mathbf{f}_{u}\left(
x\right)  $ is a rational function in these variables, and each entry of
$\mathbf{f}_{u}\left(  \mathbf{f}_{u}\left(  x\right)  \right)  $ is a
rational function in the former entries as well as $u_{1},u_{2},\ldots,u_{n}%
$). These rational functions are subtraction-free (i.e., no subtraction signs
appear in them), and thus are defined over any semifield. But there is a
general principle saying that if we need to prove an identity between two
subtraction-free rational functions, it is sufficient to prove that it holds
over the semifield $\mathbb{Q}_{+}$ from Example \ref{exa.semifield.Qplus}.
(Indeed, this principle follows from the fact that any subtraction-free
rational function can be rewritten as a ratio of two polynomials with
nonnegative integer coefficients, and thus an identity between two
subtraction-free rational functions can be rewritten as an identity between
two such polynomials; but the latter kind of identity will necessarily be true
if it has been checked on all positive rational numbers.)

As a consequence of this discussion, in order to prove Theorem
\ref{thm.f.full} \textbf{(a)} in full generality, it suffices to prove Theorem
\ref{thm.f.full} \textbf{(a)} in the case when $\mathbb{K}=\mathbb{Q}_{+}$. So
let us restrict ourselves to this case. Let $x\in\mathbb{K}^{n}$. We must show
that $\mathbf{f}_{u}\left(  \mathbf{f}_{u}\left(  x\right)  \right)  =x$.

Let $y=\mathbf{f}_{u}\left(  x\right)  $, and let $z=\mathbf{f}_{u}\left(
y\right)  $. We will show that $z=x$.

Assume the contrary. Thus, $z\neq x$. Hence, there exists some $i\in\left\{
1,2,\ldots,n\right\}  $ such that $z_{i}\neq x_{i}$. Consider this $i$. Hence,
either $z_{i}>x_{i}$ or $z_{i}<x_{i}$. We WLOG assume that $z_{i}>x_{i}$
(since the proof in the case of $z_{i}<x_{i}$ is identical, except that all
inequality signs are reversed). But Theorem \ref{thm.f.full} \textbf{(c)}
yields%
\[
\left(  u_{i}+x_{i}\right)  \left(  \dfrac{1}{u_{i+1}}+\dfrac{1}{x_{i+1}%
}\right)  =\left(  u_{i}+y_{i}\right)  \left(  \dfrac{1}{u_{i+1}}+\dfrac
{1}{y_{i+1}}\right)  .
\]
Likewise, Theorem \ref{thm.f.full} \textbf{(c)} (applied to $y$ and $z$
instead of $x$ and $y$) yields%
\[
\left(  u_{i}+y_{i}\right)  \left(  \dfrac{1}{u_{i+1}}+\dfrac{1}{y_{i+1}%
}\right)  =\left(  u_{i}+z_{i}\right)  \left(  \dfrac{1}{u_{i+1}}+\dfrac
{1}{z_{i+1}}\right)
\]
(since $z=\mathbf{f}_{u}\left(  y\right)  $). Hence,%
\begin{align*}
\left(  u_{i}+x_{i}\right)  \left(  \dfrac{1}{u_{i+1}}+\dfrac{1}{x_{i+1}%
}\right)   &  =\left(  u_{i}+y_{i}\right)  \left(  \dfrac{1}{u_{i+1}}%
+\dfrac{1}{y_{i+1}}\right) \\
&  =\left(  u_{i}+\underbrace{z_{i}}_{>x_{i}}\right)  \left(  \dfrac
{1}{u_{i+1}}+\dfrac{1}{z_{i+1}}\right) \\
&  >\left(  u_{i}+x_{i}\right)  \left(  \dfrac{1}{u_{i+1}}+\dfrac{1}{z_{i+1}%
}\right)  .
\end{align*}
Cancelling the positive number $u_{i}+x_{i}$ from this inequality, we obtain
$\dfrac{1}{u_{i+1}}+\dfrac{1}{x_{i+1}}>\dfrac{1}{u_{i+1}}+\dfrac{1}{z_{i+1}}$.
Hence, $\dfrac{1}{x_{i+1}}>\dfrac{1}{z_{i+1}}$, so that $z_{i+1}>x_{i+1}$.
Thus, from $z_{i}>x_{i}$, we have obtained $z_{i+1}>x_{i+1}$. The same
reasoning (but applied to $i+1$ instead of $i$) now yields $z_{i+2}>x_{i+2}$
(since $z_{i+1}>x_{i+1}$). Proceeding in the same way, we successively obtain
$z_{i+3}>x_{i+3}$ and $z_{i+4}>x_{i+4}$ and $z_{i+5}>x_{i+5}$ and so on.
Hence,%
\begin{equation}
\underbrace{z_{i}}_{>x_{i}}\underbrace{z_{i+1}}_{>x_{i+1}}\cdots
\underbrace{z_{i+n-1}}_{>x_{i+n-1}}>x_{i}x_{i+1}\cdots x_{i+n-1}.
\label{eq.rmk.f.invol-by-trick.7}%
\end{equation}
But Theorem \ref{thm.f.full} \textbf{(b)} yields%
\[
y_{1}y_{2}\cdots y_{n}\cdot x_{1}x_{2}\cdots x_{n}=\left(  u_{1}u_{2}\cdots
u_{n}\right)  ^{2}.
\]
Also, Theorem \ref{thm.f.full} \textbf{(b)} (applied to $y$ and $z$ instead of
$x$ and $y$) yields%
\[
z_{1}z_{2}\cdots z_{n}\cdot y_{1}y_{2}\cdots y_{n}=\left(  u_{1}u_{2}\cdots
u_{n}\right)  ^{2}%
\]
(since $z=\mathbf{f}_{u}\left(  y\right)  $). Comparing these two equalities,
we find $y_{1}y_{2}\cdots y_{n}\cdot x_{1}x_{2}\cdots x_{n}=z_{1}z_{2}\cdots
z_{n}\cdot y_{1}y_{2}\cdots y_{n}$, so that%
\begin{equation}
x_{1}x_{2}\cdots x_{n}=z_{1}z_{2}\cdots z_{n}.
\label{eq.rmk.f.invol-by-trick.9}%
\end{equation}
But Lemma \ref{lem.aprod} yields $z_{i}z_{i+1}\cdots z_{i+n-1}=z_{1}%
z_{2}\cdots z_{n}$ and $x_{i}x_{i+1}\cdots x_{i+n-1}=x_{1}x_{2}\cdots x_{n}$.
In light of these two equalities, we can rewrite
(\ref{eq.rmk.f.invol-by-trick.7}) as $z_{1}z_{2}\cdots z_{n}>x_{1}x_{2}\cdots
x_{n}$. This, however, contradicts (\ref{eq.rmk.f.invol-by-trick.9}). This
contradiction shows that our assumption was false, thus concluding our proof
of $z=x$.

Now, $\mathbf{f}_{u}\left(  \underbrace{\mathbf{f}_{u}\left(  x\right)  }%
_{=y}\right)  =\mathbf{f}_{u}\left(  y\right)  =z=x$, as we wanted to prove.
Hence, Theorem \ref{thm.f.full} \textbf{(a)} is proved again.
\end{remark}

We shall take up the study of the birational involution $\mathbf{f}_{u}$ again
in Subsection \ref{subsect.fin.fu}, where we will pose several questions about
its meaning and uniqueness properties.

\section{\label{sect.pf}Proof of the main theorem}

We shall now slowly approach the proof of Theorem \ref{thm.main} through a
long sequence of auxiliary results, some of them easy, some well-known.

\subsection{From the life of snakes}

Recall the conventions introduced in Section \ref{sect.not} and in Convention
\ref{conv.main}. Let us next introduce some further notations.

\begin{definition}
\label{def.laurent}\ \ 

\begin{enumerate}
\item[\textbf{(a)}] Let $\mathcal{L}$ denote the ring $\mathbf{k}\left[
x_{1}^{\pm1},x_{2}^{\pm1},\ldots,x_{n}^{\pm1}\right]  $ of Laurent polynomials
in the $n$ indeterminates $x_{1},x_{2},\ldots,x_{n}$ over $\mathbf{k}$.
Clearly, the polynomial ring $\mathbf{k}\left[  x_{1},x_{2},\ldots
,x_{n}\right]  $ is a subring of $\mathcal{L}$.

\item[\textbf{(b)}] We let $x_{\Pi}$ denote the monomial $x_{1}x_{2}\cdots
x_{n}\in\mathbf{k}\left[  x_{1},x_{2},\ldots,x_{n}\right]  \subseteq
\mathcal{L}$.
\end{enumerate}
\end{definition}

If $f\in\Lambda$ is a symmetric function\footnote{or, more generally, any
formal power series in $\mathbf{k}\left[  \left[  x_{1},x_{2},x_{3}%
,\ldots\right]  \right]  $ that is of bounded degree}, and if $a_{1}%
,a_{2},\ldots,a_{n}$ are $n$ elements of a commutative $\mathbf{k}$-algebra
$A$, then $f\left(  a_{1},a_{2},\ldots,a_{n},0,0,0,\ldots\right)  $ means the
result of substituting $a_{1},a_{2},\ldots,a_{n},0,0,0,\ldots$ for
$x_{1},x_{2},\ldots,x_{n},x_{n+1},x_{n+2},x_{n+3},\ldots$ in $f$. This is a
well-defined element of $A$ (see \cite[Exercise 2.1.2]{GriRei} for the proof),
and is denoted by $f\left(  a_{1},a_{2},\ldots,a_{n}\right)  $. It is called
the \emph{evaluation} of $f$ at $a_{1},a_{2},\ldots,a_{n}$.

For any symmetric function $f\in\Lambda$, the evaluation%
\[
f\left(  x_{1},x_{2},\ldots,x_{n}\right)  =f\left(  x_{1},x_{2},\ldots
,x_{n},0,0,0,\ldots\right)
\]
is a polynomial in $\mathbf{k}\left[  x_{1},x_{2},\ldots,x_{n}\right]  $ and
thus a Laurent polynomial in $\mathcal{L}$. Moreover, for any symmetric
function $f\in\Lambda$, the evaluation
\[
f\left(  x_{1}^{-1},x_{2}^{-1},\ldots,x_{n}^{-1}\right)  =f\left(  x_{1}%
^{-1},x_{2}^{-1},\ldots,x_{n}^{-1},0,0,0,\ldots\right)
\]
is a Laurent polynomial in $\mathcal{L}$ as well.

\begin{convention}
\label{conv.gammai}For the rest of Section \ref{sect.pf}, let us agree to the
following notation: If $\gamma$ is an $n$-tuple (of any objects), then we let
$\gamma_{i}$ denote the $i$-th entry of $\gamma$ whenever $i\in\left\{
1,2,\ldots,n\right\}  $. Thus, each $n$-tuple $\gamma$ satisfies
$\gamma=\left(  \gamma_{1},\gamma_{2},\ldots,\gamma_{n}\right)  $.
\end{convention}

\begin{definition}
\label{def.snakes}\ \ 

\begin{enumerate}
\item[\textbf{(a)}] A \emph{snake} means an $n$-tuple $\lambda=\left(
\lambda_{1},\lambda_{2},\ldots,\lambda_{n}\right)  $ of integers (not
necessarily nonnegative) such that $\lambda_{1}\geq\lambda_{2}\geq\cdots
\geq\lambda_{n}$.

\item[\textbf{(b)}] A snake $\lambda$ is said to be \emph{nonnegative} if it
belongs to $\mathbb{N}^{n}$ (that is, if all its entries are nonnegative).
Thus, a nonnegative snake is the same as a partition having length $\leq n$.
In other words, a nonnegative snake is the same as a partition $\lambda
\in\operatorname*{Par}\left[  n\right]  $.

\item[\textbf{(c)}] If $\lambda\in\mathbb{Z}^{n}$ is an $n$-tuple, and $d$ is
an integer, then $\lambda+d$ denotes the $n$-tuple $\left(  \lambda
_{1}+d,\lambda_{2}+d,\ldots,\lambda_{n}+d\right)  \in\mathbb{Z}^{n}$ (which is
obtained from $\lambda$ by adding $d$ to each entry), whereas $\lambda-d$
denotes the $n$-tuple $\left(  \lambda_{1}-d,\lambda_{2}-d,\ldots,\lambda
_{n}-d\right)  \in\mathbb{Z}^{n}$. (Thus, $\lambda-d=\lambda+\left(
-d\right)  $.)

\item[\textbf{(d)}] If $\lambda\in\mathbb{Z}^{n}$, then $\lambda^{\vee}$
denotes the $n$-tuple $\left(  -\lambda_{n},-\lambda_{n-1},\ldots,-\lambda
_{1}\right)  \in\mathbb{Z}^{n}$.

\item[\textbf{(e)}] We regard $\mathbb{Z}^{n}$ as a $\mathbb{Z}$-module in the
obvious way. Thus, if $\lambda\in\mathbb{Z}^{n}$ and $\mu\in\mathbb{Z}^{n}$
are two $n$-tuples of integers, then
\begin{align*}
\lambda+\mu &  =\left(  \lambda_{1}+\mu_{1},\lambda_{2}+\mu_{2},\ldots
,\lambda_{n}+\mu_{n}\right)  ,\\
\lambda-\mu &  =\left(  \lambda_{1}-\mu_{1},\lambda_{2}-\mu_{2},\ldots
,\lambda_{n}-\mu_{n}\right)  .
\end{align*}

\item[\textbf{(f)}] We let $\rho$ denote the nonnegative snake $\left(
n-1,n-2,\ldots,2,1,0\right)  $. Thus,%
\begin{equation}
\rho_{i}=n-i\ \ \ \ \ \ \ \ \ \ \text{for each }i\in\left\{  1,2,\ldots
,n\right\}  . \label{eq.snakes.f.rhoi=}%
\end{equation}

\end{enumerate}
\end{definition}

\begin{example}
In this example, let $n=3$.

\begin{enumerate}
\item[\textbf{(a)}] The four $3$-tuples $\left(  3,1,0\right)  $, $\left(
2,2,1\right)  $, $\left(  1,0,-1\right)  $ and $\left(  -1,-2,-5\right)  $ are
examples of snakes.

\item[\textbf{(b)}] The first two of these four snakes (but not the last two)
are nonnegative.

\item[\textbf{(c)}] We have $\left(  5,3,1\right)  +3=\left(  8,6,4\right)  $
and $\left(  5,3,1\right)  -3=\left(  2,0,-2\right)  $.

\item[\textbf{(d)}] We have $\left(  5,2,2\right)  ^{\vee}=\left(
-2,-2,-5\right)  $.

\item[\textbf{(e)}] We have $\left(  2,1,2\right)  +\left(  3,4,5\right)
=\left(  5,5,7\right)  $.

\item[\textbf{(f)}] We have $\rho=\left(  2,1,0\right)  $.
\end{enumerate}
\end{example}

Note that what we call a \textquotedblleft snake\textquotedblright\ here is
called a \textquotedblleft staircase of height $n$\textquotedblright\ in
Stembridge's work \cite{Stembr87}, where he uses these snakes to index
finite-dimensional polynomial representations of the group $\operatorname*{GL}%
\nolimits_{n}\left(  \mathbb{C}\right)  $. We avoid calling them
\textquotedblleft staircases\textquotedblright, as that word has since been
used for other things (in particular, $\rho$ is often called \textquotedblleft
the $n$-staircase\textquotedblright\ in the jargon of combinatorialists).

The notations introduced in Definition \ref{def.snakes} have the following properties:

\begin{proposition}
\label{prop.snakes.trivia}\ \ 

\begin{enumerate}
\item[\textbf{(a)}] If $\lambda$ is a snake, and $d$ is an integer, then
$\lambda+d$ and $\lambda-d$ are snakes as well.

\item[\textbf{(b)}] If $\lambda$ is a snake, then $\lambda^{\vee}$ is a snake
as well.

\item[\textbf{(c)}] We have $\left(  \lambda+\mu\right)  +d=\left(
\lambda+d\right)  +\mu$ for any $\lambda\in\mathbb{Z}^{n}$, $\mu\in
\mathbb{Z}^{n}$ and $d\in\mathbb{Z}$.

\item[\textbf{(d)}] We have $\lambda+\left(  d+e\right)  =\left(
\lambda+d\right)  +e$ for any $\lambda\in\mathbb{Z}^{n}$, $d\in\mathbb{Z}$ and
$e\in\mathbb{Z}$.

\item[\textbf{(e)}] We have $\left(  \lambda+d\right)  -d=\left(
\lambda-d\right)  +d=\lambda$ for any $\lambda\in\mathbb{Z}^{n}$ and
$d\in\mathbb{Z}$.
\end{enumerate}
\end{proposition}

\begin{proof}
[Proof of Proposition \ref{prop.snakes.trivia}.]Completely straightforward.
\end{proof}

Let us now assign a Laurent polynomial $a_{\lambda}$ to each $\lambda
\in\mathbb{Z}^{n}$:

\begin{definition}
\label{def.alt}Let $\lambda\in\mathbb{Z}^{n}$ be any $n$-tuple. Then, we
define the Laurent polynomial%
\[
a_{\lambda}:=\sum_{w\in\mathfrak{S}_{n}}\left(  -1\right)  ^{w}x_{w\left(
1\right)  }^{\lambda_{1}}x_{w\left(  2\right)  }^{\lambda_{2}}\cdots
x_{w\left(  n\right)  }^{\lambda_{n}}\in\mathcal{L},
\]
where $\mathfrak{S}_{n}$ is the symmetric group of the set $\left\{
1,2,\ldots,n\right\}  $ (and where $\left(  -1\right)  ^{w}$ denotes the sign
of a permutation $w$). This Laurent polynomial $a_{\lambda}$ is called the
\emph{alternant} corresponding to the $n$-tuple $\lambda$.
\end{definition}

(The \textquotedblleft$a$\textquotedblright\ in the notation \textquotedblleft%
$a_{\lambda}$\textquotedblright\ has nothing to do with the $a$ in Theorem
\ref{thm.main}.)

\begin{example}
We have%
\begin{align*}
a_{\left(  5,3,2\right)  }  &  =\sum_{w\in\mathfrak{S}_{3}}\left(  -1\right)
^{w}x_{w\left(  1\right)  }^{5}x_{w\left(  2\right)  }^{3}x_{w\left(
3\right)  }^{2}\\
&  =x_{1}^{5}x_{2}^{3}x_{3}^{2}+x_{2}^{5}x_{3}^{3}x_{1}^{2}+x_{3}^{5}x_{1}%
^{3}x_{2}^{2}-x_{1}^{5}x_{3}^{3}x_{2}^{2}-x_{2}^{5}x_{1}^{3}x_{3}^{2}%
-x_{3}^{5}x_{2}^{3}x_{1}^{2}.
\end{align*}

\end{example}

The sum in Definition \ref{def.alt} is the same kind of sum that appears in
the definition of a determinant. Therefore, we can rewrite the alternant as follows:

\begin{proposition}
\label{prop.alt}Let $\lambda\in\mathbb{Z}^{n}$ be an $n$-tuple. Then, the
alternant $a_{\lambda}\in\mathcal{L}$ satisfies%
\[
a_{\lambda}=\det\left(  \left(  x_{j}^{\lambda_{i}}\right)  _{1\leq i\leq
n,\ 1\leq j\leq n}\right)  =\det\left(  \left(  x_{i}^{\lambda_{j}}\right)
_{1\leq i\leq n,\ 1\leq j\leq n}\right)  .
\]

\end{proposition}

\begin{verlong}

\begin{proof}
[Proof of Proposition \ref{prop.alt}.]The definition of a determinant yields
\[
\det\left(  \left(  x_{j}^{\lambda_{i}}\right)  _{1\leq i\leq n,\ 1\leq j\leq
n}\right)  =\sum_{w\in\mathfrak{S}_{n}}\left(  -1\right)  ^{w}%
\underbrace{\prod_{i=1}^{n}x_{w\left(  i\right)  }^{\lambda_{i}}%
}_{=x_{w\left(  1\right)  }^{\lambda_{1}}x_{w\left(  2\right)  }^{\lambda_{2}%
}\cdots x_{w\left(  n\right)  }^{\lambda_{n}}}=\sum_{w\in\mathfrak{S}_{n}%
}\left(  -1\right)  ^{w}x_{w\left(  1\right)  }^{\lambda_{1}}x_{w\left(
2\right)  }^{\lambda_{2}}\cdots x_{w\left(  n\right)  }^{\lambda_{n}}.
\]
Comparing this with%
\[
a_{\lambda}=\sum_{w\in\mathfrak{S}_{n}}\left(  -1\right)  ^{w}x_{w\left(
1\right)  }^{\lambda_{1}}x_{w\left(  2\right)  }^{\lambda_{2}}\cdots
x_{w\left(  n\right)  }^{\lambda_{n}}\ \ \ \ \ \ \ \ \ \ \left(  \text{by the
definition of }a_{\lambda}\right)  ,
\]
we obtain%
\begin{align*}
a_{\lambda}  &  =\det\left(  \left(  x_{j}^{\lambda_{i}}\right)  _{1\leq i\leq
n,\ 1\leq j\leq n}\right)  =\det\left(  \left(  \left(  x_{j}^{\lambda_{i}%
}\right)  _{1\leq i\leq n,\ 1\leq j\leq n}\right)  ^{T}\right) \\
&  \ \ \ \ \ \ \ \ \ \ \left(  \text{since }\det A=\det\left(  A^{T}\right)
\text{ for any square matrix }A\right) \\
&  =\det\left(  \left(  x_{i}^{\lambda_{j}}\right)  _{1\leq i\leq n,\ 1\leq
j\leq n}\right)  \ \ \ \ \ \ \ \ \ \ \left(  \text{since }\left(  \left(
x_{j}^{\lambda_{i}}\right)  _{1\leq i\leq n,\ 1\leq j\leq n}\right)
^{T}=\left(  x_{i}^{\lambda_{j}}\right)  _{1\leq i\leq n,\ 1\leq j\leq
n}\right)  .
\end{align*}
This proves Proposition \ref{prop.alt}.
\end{proof}
\end{verlong}

\begin{vershort}
Thus, in particular, the alternant $a_{\rho}$ corresponding to the snake
\[
\rho=\left(  n-1,n-2,\ldots,2,1,0\right)  =\left(  n-1,n-2,\ldots,n-n\right)
\]
satisfies%
\[
a_{\rho}=\det\left(  \left(  x_{i}^{n-j}\right)  _{1\leq i\leq n,\ 1\leq j\leq
n}\right)  =\prod_{1\leq i<j\leq n}\left(  x_{i}-x_{j}\right)
\]
(by the classical formula for the Vandermonde determinant).
\end{vershort}

\begin{verlong}
Thus, in particular, we can compute the alternant $a_{\rho}$ corresponding to
the snake
\[
\rho=\left(  n-1,n-2,\ldots,2,1,0\right)  =\left(  n-1,n-2,\ldots,n-n\right)
.
\]
Indeed, Proposition \ref{prop.alt} (applied to $\lambda=\rho$) yields%
\begin{align*}
a_{\rho}  &  =\det\left(  \left(  x_{j}^{\rho_{i}}\right)  _{1\leq i\leq
n,\ 1\leq j\leq n}\right)  =\det\left(  \left(  x_{i}^{\rho_{j}}\right)
_{1\leq i\leq n,\ 1\leq j\leq n}\right) \\
&  =\det\left(  \left(  x_{i}^{n-j}\right)  _{1\leq i\leq n,\ 1\leq j\leq
n}\right)  \ \ \ \ \ \ \ \ \ \ \left(
\begin{array}
[c]{c}%
\text{since }\rho_{j}=n-j\text{ for each }j\in\left\{  1,2,\ldots,n\right\} \\
\text{(because }\rho=\left(  n-1,n-2,\ldots,n-n\right)  \text{)}%
\end{array}
\right) \\
&  =\prod_{1\leq i<j\leq n}\left(  x_{i}-x_{j}\right)
\end{align*}
(by the classical formula for the Vandermonde determinant).
\end{verlong}

We recall a standard concept from commutative algebra: An element $a$ of a
commutative ring $A$ is said to be \textit{regular} if it has the property
that every $x\in A$ satisfying $ax=0$ must satisfy $x=0$. (Thus, regular
elements are the same as elements that are not zero-divisors, if one does not
require zero-divisors to be nonzero\footnote{Unfortunately, there is no
agreement in the literature on whether zero-divisors should be required to be
nonzero. This is one of the reasons why we are avoiding this notion.}.)

\begin{lemma}
\label{lem.alt.rho-reg}The alternant $a_{\rho}$ is a regular element of
$\mathcal{L}$.
\end{lemma}

\begin{proof}
[Proof of Lemma \ref{lem.alt.rho-reg}.]Let $b\in\mathcal{L}$ be such that
$a_{\rho}b=0$. We want to show that $b=0$.

We know that $b$ is a Laurent polynomial, and thus has the form $b=\dfrac
{c}{x_{1}^{u_{1}}x_{2}^{u_{2}}\cdots x_{n}^{u_{n}}}$ for some $u_{1}%
,u_{2},\ldots,u_{n}\in\mathbb{Z}$ and some polynomial $c\in\mathbf{k}\left[
x_{1},x_{2},\ldots,x_{n}\right]  $. Consider these $u_{1},u_{2},\ldots,u_{n}$
and this $c$. From $b=\dfrac{c}{x_{1}^{u_{1}}x_{2}^{u_{2}}\cdots x_{n}^{u_{n}%
}}$, we obtain $c=b\cdot x_{1}^{u_{1}}x_{2}^{u_{2}}\cdots x_{n}^{u_{n}}$.
Multiplying this equality by $a_{\rho}$, we obtain%
\[
a_{\rho}c=\underbrace{a_{\rho}b}_{=0}\cdot x_{1}^{u_{1}}x_{2}^{u_{2}}\cdots
x_{n}^{u_{n}}=0.
\]

But it is a well-known fact (see, e.g., \cite[Corollary 4.4]{Grinbe19}) that
the polynomial $\prod_{1\leq i<j\leq n}\left(  x_{i}-x_{j}\right)  $ is a
regular element of $\mathbf{k}\left[  x_{1},x_{2},\ldots,x_{n}\right]  $. In
other words, $a_{\rho}$ is a regular element of $\mathbf{k}\left[  x_{1}%
,x_{2},\ldots,x_{n}\right]  $ (since $a_{\rho}=\prod_{1\leq i<j\leq n}\left(
x_{i}-x_{j}\right)  $). Hence, from $a_{\rho}c=0$, we obtain $c=0$. Now,
$b=\dfrac{c}{x_{1}^{u_{1}}x_{2}^{u_{2}}\cdots x_{n}^{u_{n}}}=0$ (since $c=0$).

Forget that we fixed $b$. We thus have shown that each $b\in\mathcal{L}$
satisfying $a_{\rho}b=0$ satisfies $b=0$. In other words, $a_{\rho}$ is a
regular element of $\mathcal{L}$. This proves Lemma \ref{lem.alt.rho-reg}.
\end{proof}

Lemma \ref{lem.alt.rho-reg} shows that fractions of the form $\dfrac
{u}{a_{\rho}}$ (where $u\in\mathcal{L}$) are well-defined if $u$ is a multiple
of $a_{\rho}$. (That is, there is never more than one $b\in\mathcal{L}$ that
satisfies $a_{\rho}b=u$.)

We notice that the element $x_{\Pi}=x_{1}x_{2}\cdots x_{n}$ of $\mathcal{L}$
is invertible (since $x_{1},x_{2},\ldots,x_{n}$ are invertible in
$\mathcal{L}$).

\begin{lemma}
\label{lem.alt.move}Let $\lambda\in\mathbb{Z}^{n}$ be any $n$-tuple, and let
$d\in\mathbb{Z}$. Then, $a_{\lambda+d}=x_{\Pi}^{d}a_{\lambda}$.
\end{lemma}

\begin{vershort}

\begin{proof}
[Proof of Lemma \ref{lem.alt.move}.]The definition of $a_{\lambda}$ yields%
\begin{equation}
a_{\lambda}=\sum_{w\in\mathfrak{S}_{n}}\left(  -1\right)  ^{w}x_{w\left(
1\right)  }^{\lambda_{1}}x_{w\left(  2\right)  }^{\lambda_{2}}\cdots
x_{w\left(  n\right)  }^{\lambda_{n}}. \label{pf.lem.alt.move.short.alam=}%
\end{equation}
But the definition of $\lambda+d$ yields $\lambda+d=\left(  \lambda
_{1}+d,\lambda_{2}+d,\ldots,\lambda_{n}+d\right)  $. Hence, the definition of
$a_{\lambda+d}$ yields%
\begin{align*}
a_{\lambda+d}  &  =\sum_{w\in\mathfrak{S}_{n}}\left(  -1\right)
^{w}\underbrace{x_{w\left(  1\right)  }^{\lambda_{1}+d}x_{w\left(  2\right)
}^{\lambda_{2}+d}\cdots x_{w\left(  n\right)  }^{\lambda_{n}+d}}_{=\left(
x_{w\left(  1\right)  }^{\lambda_{1}}x_{w\left(  2\right)  }^{\lambda_{2}%
}\cdots x_{w\left(  n\right)  }^{\lambda_{n}}\right)  \left(  x_{w\left(
1\right)  }^{d}x_{w\left(  2\right)  }^{d}\cdots x_{w\left(  n\right)  }%
^{d}\right)  }\\
&  =\sum_{w\in\mathfrak{S}_{n}}\left(  -1\right)  ^{w}\left(  x_{w\left(
1\right)  }^{\lambda_{1}}x_{w\left(  2\right)  }^{\lambda_{2}}\cdots
x_{w\left(  n\right)  }^{\lambda_{n}}\right)  \underbrace{\left(  x_{w\left(
1\right)  }^{d}x_{w\left(  2\right)  }^{d}\cdots x_{w\left(  n\right)  }%
^{d}\right)  }_{\substack{=x_{1}^{d}x_{2}^{d}\cdots x_{n}^{d}\\\text{(since
}w:\left\{  1,2,\ldots,n\right\}  \rightarrow\left\{  1,2,\ldots,n\right\}
\\\text{is a bijection)}}}\\
&  =\sum_{w\in\mathfrak{S}_{n}}\left(  -1\right)  ^{w}\left(  x_{w\left(
1\right)  }^{\lambda_{1}}x_{w\left(  2\right)  }^{\lambda_{2}}\cdots
x_{w\left(  n\right)  }^{\lambda_{n}}\right)  \underbrace{\left(  x_{1}%
^{d}x_{2}^{d}\cdots x_{n}^{d}\right)  }_{\substack{=\left(  x_{1}x_{2}\cdots
x_{n}\right)  ^{d}=x_{\Pi}^{d}\\\text{(since }x_{1}x_{2}\cdots x_{n}=x_{\Pi
}\text{)}}}\\
&  =\sum_{w\in\mathfrak{S}_{n}}\left(  -1\right)  ^{w}\left(  x_{w\left(
1\right)  }^{\lambda_{1}}x_{w\left(  2\right)  }^{\lambda_{2}}\cdots
x_{w\left(  n\right)  }^{\lambda_{n}}\right)  x_{\Pi}^{d}\\
&  =x_{\Pi}^{d}\underbrace{\sum_{w\in\mathfrak{S}_{n}}\left(  -1\right)
^{w}x_{w\left(  1\right)  }^{\lambda_{1}}x_{w\left(  2\right)  }^{\lambda_{2}%
}\cdots x_{w\left(  n\right)  }^{\lambda_{n}}}_{\substack{=a_{\lambda
}\\\text{(by (\ref{pf.lem.alt.move.short.alam=}))}}}=x_{\Pi}^{d}a_{\lambda}.
\end{align*}
This proves Lemma \ref{lem.alt.move}.
\end{proof}
\end{vershort}

\begin{verlong}

\begin{proof}
[Proof of Lemma \ref{lem.alt.move}.]The definition of $a_{\lambda}$ yields%
\begin{equation}
a_{\lambda}=\sum_{w\in\mathfrak{S}_{n}}\left(  -1\right)  ^{w}x_{w\left(
1\right)  }^{\lambda_{1}}x_{w\left(  2\right)  }^{\lambda_{2}}\cdots
x_{w\left(  n\right)  }^{\lambda_{n}}. \label{pf.lem.alt.move.alam=}%
\end{equation}
But the definition of $\lambda+d$ yields $\lambda+d=\left(  \lambda
_{1}+d,\lambda_{2}+d,\ldots,\lambda_{n}+d\right)  $. Hence,
\begin{equation}
\left(  \lambda+d\right)  _{i}=\lambda_{i}+d\ \ \ \ \ \ \ \ \ \ \text{for each
}i\in\left\{  1,2,\ldots,n\right\}  . \label{pf.lem.alt.move.2}%
\end{equation}
Now, the definition of $a_{\lambda+d}$ yields%
\begin{align*}
a_{\lambda+d}  &  =\sum_{w\in\mathfrak{S}_{n}}\left(  -1\right)
^{w}\underbrace{x_{w\left(  1\right)  }^{\left(  \lambda+d\right)  _{1}%
}x_{w\left(  2\right)  }^{\left(  \lambda+d\right)  _{2}}\cdots x_{w\left(
n\right)  }^{\left(  \lambda+d\right)  _{n}}}_{=\prod_{i\in\left\{
1,2,\ldots,n\right\}  }x_{w\left(  i\right)  }^{\left(  \lambda+d\right)
_{i}}}=\sum_{w\in\mathfrak{S}_{n}}\left(  -1\right)  ^{w}\prod_{i\in\left\{
1,2,\ldots,n\right\}  }\underbrace{x_{w\left(  i\right)  }^{\left(
\lambda+d\right)  _{i}}}_{\substack{=x_{w\left(  i\right)  }^{\lambda_{i}%
+d}\\\text{(since (\ref{pf.lem.alt.move.2})}\\\text{yields }\left(
\lambda+d\right)  _{i}=\lambda_{i}+d\text{)}}}\\
&  =\sum_{w\in\mathfrak{S}_{n}}\left(  -1\right)  ^{w}\prod_{i\in\left\{
1,2,\ldots,n\right\}  }\underbrace{x_{w\left(  i\right)  }^{\lambda_{i}+d}%
}_{=x_{w\left(  i\right)  }^{\lambda_{i}}x_{w\left(  i\right)  }^{d}}\\
&  =\sum_{w\in\mathfrak{S}_{n}}\left(  -1\right)  ^{w}\underbrace{\prod
_{i\in\left\{  1,2,\ldots,n\right\}  }\left(  x_{w\left(  i\right)  }%
^{\lambda_{i}}x_{w\left(  i\right)  }^{d}\right)  }_{=\left(  \prod
_{i\in\left\{  1,2,\ldots,n\right\}  }x_{w\left(  i\right)  }^{\lambda_{i}%
}\right)  \left(  \prod_{i\in\left\{  1,2,\ldots,n\right\}  }x_{w\left(
i\right)  }^{d}\right)  }\\
&  =\sum_{w\in\mathfrak{S}_{n}}\left(  -1\right)  ^{w}\left(  \prod
_{i\in\left\{  1,2,\ldots,n\right\}  }x_{w\left(  i\right)  }^{\lambda_{i}%
}\right)  \underbrace{\left(  \prod_{i\in\left\{  1,2,\ldots,n\right\}
}x_{w\left(  i\right)  }^{d}\right)  }_{\substack{=\prod_{i\in\left\{
1,2,\ldots,n\right\}  }x_{i}^{d}\\\text{(here, we have substituted }i\text{
for }w\left(  i\right)  \\\text{in the product, since }w:\left\{
1,2,\ldots,n\right\}  \rightarrow\left\{  1,2,\ldots,n\right\}  \\\text{is a
bijection)}}}\\
&  =\sum_{w\in\mathfrak{S}_{n}}\left(  -1\right)  ^{w}\underbrace{\left(
\prod_{i\in\left\{  1,2,\ldots,n\right\}  }x_{w\left(  i\right)  }%
^{\lambda_{i}}\right)  }_{=x_{w\left(  1\right)  }^{\lambda_{1}}x_{w\left(
2\right)  }^{\lambda_{2}}\cdots x_{w\left(  n\right)  }^{\lambda_{n}}%
}\underbrace{\prod_{i\in\left\{  1,2,\ldots,n\right\}  }x_{i}^{d}%
}_{\substack{=x_{1}^{d}x_{2}^{d}\cdots x_{n}^{d}=\left(  x_{1}x_{2}\cdots
x_{n}\right)  ^{d}=x_{\Pi}^{d}\\\text{(since }x_{1}x_{2}\cdots x_{n}=x_{\Pi
}\text{)}}}\\
&  =\sum_{w\in\mathfrak{S}_{n}}\left(  -1\right)  ^{w}\left(  x_{w\left(
1\right)  }^{\lambda_{1}}x_{w\left(  2\right)  }^{\lambda_{2}}\cdots
x_{w\left(  n\right)  }^{\lambda_{n}}\right)  x_{\Pi}^{d}\\
&  =x_{\Pi}^{d}\underbrace{\sum_{w\in\mathfrak{S}_{n}}\left(  -1\right)
^{w}x_{w\left(  1\right)  }^{\lambda_{1}}x_{w\left(  2\right)  }^{\lambda_{2}%
}\cdots x_{w\left(  n\right)  }^{\lambda_{n}}}_{\substack{=a_{\lambda
}\\\text{(by (\ref{pf.lem.alt.move.alam=}))}}}=x_{\Pi}^{d}a_{\lambda}.
\end{align*}
This proves Lemma \ref{lem.alt.move}.
\end{proof}
\end{verlong}

\begin{verlong}
Clearly, $\mathbb{N}^{n}\subseteq\mathbb{Z}^{n}$. For any $\alpha\in
\mathbb{N}^{n}$, a polynomial $a_{\alpha}$ has been defined in
\cite[Definition 2.6.2]{GriRei}. This polynomial $a_{\alpha}$ is identical
with the alternant $a_{\alpha}$ we have defined in Definition \ref{def.alt},
since%
\begin{align*}
&  \left(  \text{the polynomial }a_{\alpha}\text{ as defined in
\cite[Definition 2.6.2]{GriRei}}\right) \\
&  =\sum_{w\in\mathfrak{S}_{n}}\underbrace{\operatorname*{sgn}\left(
w\right)  }_{=\left(  -1\right)  ^{w}}w\left(  \underbrace{\mathbf{x}^{\alpha
}}_{=x_{1}^{\alpha_{1}}x_{2}^{\alpha_{2}}\cdots x_{n}^{\alpha_{n}}=\prod
_{i=1}^{n}x_{i}^{\alpha_{i}}}\right) \\
&  \ \ \ \ \ \ \ \ \ \ \left(  \text{where we are using the notation of
\cite[Definition 2.6.2]{GriRei}}\right) \\
&  =\sum_{w\in\mathfrak{S}_{n}}\left(  -1\right)  ^{w}\underbrace{w\left(
\prod_{i=1}^{n}x_{i}^{\alpha_{i}}\right)  }_{\substack{=\prod_{i=1}^{n}\left(
w\left(  x_{i}\right)  \right)  ^{\alpha_{i}}\\\text{(since }\mathfrak{S}%
_{n}\text{ acts on }\mathbf{k}\left[  x_{1},x_{2},\ldots,x_{n}\right]
\\\text{by }\mathbf{k}\text{-algebra automorphisms)}}}=\sum_{w\in
\mathfrak{S}_{n}}\left(  -1\right)  ^{w}\prod_{i=1}^{n}\left(
\underbrace{w\left(  x_{i}\right)  }_{=x_{w\left(  i\right)  }}\right)
^{\alpha_{i}}\\
&  =\sum_{w\in\mathfrak{S}_{n}}\left(  -1\right)  ^{w}\underbrace{\prod
_{i=1}^{n}x_{w\left(  i\right)  }^{\alpha_{i}}}_{=x_{w\left(  1\right)
}^{\alpha_{1}}x_{w\left(  2\right)  }^{\alpha_{2}}\cdots x_{w\left(  n\right)
}^{\alpha_{n}}}=\sum_{w\in\mathfrak{S}_{n}}\left(  -1\right)  ^{w}x_{w\left(
1\right)  }^{\alpha_{1}}x_{w\left(  2\right)  }^{\alpha_{2}}\cdots x_{w\left(
n\right)  }^{\alpha_{n}}\\
&  =\left(  \text{the alternant }a_{\alpha}\text{ as defined in Definition
\ref{def.alt}}\right)  .
\end{align*}
Thus, we can freely use results from \cite[\S 2.6]{GriRei} without worrying
about conflicting definitions of $a_{\alpha}$. (But we need to keep in mind
that what is called $\mathbf{x}$ in \cite[\S 2.6]{GriRei} is $\left(
x_{1},x_{2},\ldots,x_{n}\right)  $ in our terminology, and that \cite[\S 2.6]%
{GriRei} only studies alternants $a_{\alpha}$ for $\alpha\in\mathbb{N}^{n}$,
while we are also interested in $a_{\lambda}$ with $\lambda\in\mathbb{Z}^{n}$.)
\end{verlong}

\begin{lemma}
\label{lem.alt.L}Let $\lambda$ be a snake. Then, $a_{\lambda+\rho}$ is a
multiple of $a_{\rho}$ in $\mathcal{L}$.
\end{lemma}

\begin{proof}
[Proof of Lemma \ref{lem.alt.L}.]Our proof will consist of two steps:

\begin{statement}
\textit{Step 1:} We will prove Lemma \ref{lem.alt.L} in the particular case
when $\lambda$ is nonnegative.
\end{statement}

\begin{statement}
\textit{Step 2:} We will use Lemma \ref{lem.alt.move} to derive the general
case of Lemma \ref{lem.alt.L} from this particular case.
\end{statement}

We will use this strategy again further on; we shall refer to it as the
\emph{right-shift strategy}.

Here are the details of the two steps:

\textit{Step 1:} Let us prove that Lemma \ref{lem.alt.L} holds in the
particular case when $\lambda$ is nonnegative.

Indeed, let us assume that $\lambda$ is nonnegative. We must show that
$a_{\lambda+\rho}$ is a multiple of $a_{\rho}$ in $\mathcal{L}$.

We know that $\lambda$ is a nonnegative snake, thus a partition of length
$\leq n$. Hence, \cite[Corollary 2.6.7]{GriRei} shows that $s_{\lambda}\left(
x_{1},x_{2},\ldots,x_{n}\right)  =\dfrac{a_{\lambda+\rho}}{a_{\rho}}$. Thus,
$a_{\lambda+\rho}=a_{\rho}\cdot s_{\lambda}\left(  x_{1},x_{2},\ldots
,x_{n}\right)  $. This shows that $a_{\lambda+\rho}$ is a multiple of
$a_{\rho}$ in $\mathcal{L}$ (since $s_{\lambda}\left(  x_{1},x_{2}%
,\ldots,x_{n}\right)  \in\mathbf{k}\left[  x_{1},x_{2},\ldots,x_{n}\right]
\subseteq\mathcal{L}$). Thus, Lemma \ref{lem.alt.L} is proved under the
assumption that $\lambda$ is nonnegative. This completes Step 1.

\textit{Step 2:} Let us now prove Lemma \ref{lem.alt.L} in the general case.

We know that $\lambda$ is a snake. Thus, $\lambda_{1}\geq\lambda_{2}\geq
\cdots\geq\lambda_{n}$. Hence, each $i\in\left\{  1,2,\ldots,n\right\}  $
satisfies $\lambda_{i}\geq\lambda_{n}$ and thus%
\begin{equation}
\lambda_{i}-\lambda_{n}\geq0. \label{pf.lem.alt.L.s2.1}%
\end{equation}

The snake $\lambda$ may or may not be nonnegative. However, there exists some
integer $d$ such that the snake $\lambda+d$ is nonnegative\footnote{Indeed,
for example, we can take $d=-\lambda_{n}$. Then, all entries of $\lambda+d$
have the form $\lambda_{i}+\underbrace{d}_{=-\lambda_{n}}=\lambda_{i}%
-\lambda_{n}$ for some $i\in\left\{  1,2,\ldots,n\right\}  $, and thus are
nonnegative (because of (\ref{pf.lem.alt.L.s2.1})); this shows that the snake
$\lambda+d$ is nonnegative.}. Consider this $d$. Proposition
\ref{prop.snakes.trivia} \textbf{(c)} (applied to $\mu=\rho$) yields $\left(
\lambda+\rho\right)  +d=\left(  \lambda+d\right)  +\rho$.

The snake $\lambda+d$ is nonnegative; thus, we can apply Lemma \ref{lem.alt.L}
to $\lambda+d$ instead of $\lambda$ (because in Step 1, we have proved that
Lemma \ref{lem.alt.L} holds in the particular case when $\lambda$ is
nonnegative). Thus we conclude that $a_{\left(  \lambda+d\right)  +\rho}$ is a
multiple of $a_{\rho}$ in $\mathcal{L}$. In other words, there exists some
$u\in\mathcal{L}$ such that $a_{\left(  \lambda+d\right)  +\rho}=a_{\rho}u$.
Consider this $u$. From $\left(  \lambda+\rho\right)  +d=\left(
\lambda+d\right)  +\rho$, we obtain $a_{\left(  \lambda+\rho\right)
+d}=a_{\left(  \lambda+d\right)  +\rho}=a_{\rho}u$.

Lemma \ref{lem.alt.move} (applied to $\lambda+\rho$ instead of $\lambda$)
yields $a_{\left(  \lambda+\rho\right)  +d}=x_{\Pi}^{d}a_{\lambda+\rho}$.
Since the element $x_{\Pi}$ of $\mathcal{L}$ is invertible, we thus obtain
\[
a_{\lambda+\rho}=\underbrace{\left(  x_{\Pi}^{d}\right)  ^{-1}}_{=x_{\Pi}%
^{-d}}\underbrace{a_{\left(  \lambda+\rho\right)  +d}}_{=a_{\rho}u}=x_{\Pi
}^{-d}a_{\rho}u=a_{\rho}\cdot x_{\Pi}^{-d}u.
\]
Hence, $a_{\lambda+\rho}$ is a multiple of $a_{\rho}$ (since $x_{\Pi}^{-d}%
u\in\mathcal{L}$). This proves Lemma \ref{lem.alt.L}. Thus, Step 2 is
complete, and Lemma \ref{lem.alt.L} is proven.
\end{proof}

\begin{definition}
\label{def.alt.sbar}Let $\lambda$ be a snake. We define an element
$\overline{s}_{\lambda}\in\mathcal{L}$ by $\overline{s}_{\lambda}%
=\dfrac{a_{\lambda+\rho}}{a_{\rho}}$. (This is well-defined, because Lemma
\ref{lem.alt.L} shows that $a_{\lambda+\rho}$ is a multiple of $a_{\rho}$ in
$\mathcal{L}$, and because Lemma \ref{lem.alt.rho-reg} shows that the fraction
$\dfrac{a_{\lambda+\rho}}{a_{\rho}}$ is uniquely defined.)
\end{definition}

It makes sense to refer to the elements $\overline{s}_{\lambda}$ just defined
as \textquotedblleft\textit{Schur Laurent polynomials}\textquotedblright. In
fact, as the following lemma shows, they are identical with the Schur
polynomials $s_{\lambda}\left(  x_{1},x_{2},\ldots,x_{n}\right)  $ when the
snake $\lambda$ is nonnegative:

\begin{lemma}
\label{lem.alt.sbar}Let $\lambda\in\operatorname*{Par}\left[  n\right]  $.
Then,%
\[
\overline{s}_{\lambda}=s_{\lambda}\left(  x_{1},x_{2},\ldots,x_{n}\right)  .
\]

\end{lemma}

\begin{proof}
[Proof of Lemma \ref{lem.alt.sbar}.]We know that $\lambda$ is a partition of
length $\leq n$ (since $\lambda\in\operatorname*{Par}\left[  n\right]  $).
Hence, $\lambda$ is a nonnegative snake. Furthermore, since $\lambda$ is a
partition of length $\leq n$, we can apply \cite[Corollary 2.6.7]{GriRei} and
obtain $s_{\lambda}\left(  x_{1},x_{2},\ldots,x_{n}\right)  =\dfrac
{a_{\lambda+\rho}}{a_{\rho}}=\overline{s}_{\lambda}$ (since $\overline
{s}_{\lambda}$ was defined to be $\dfrac{a_{\lambda+\rho}}{a_{\rho}}$). This
proves Lemma \ref{lem.alt.sbar}.
\end{proof}

The Schur Laurent polynomials $\overline{s}_{\lambda}$ appear in Stembridge's
\cite{Stembr87}, where they are named $s_{\lambda}$. (The equivalence of our
definition with his follows from \cite[Theorem 7.1]{Stembr87}.)

Furthermore, from Lemma \ref{lem.alt.move}, we can easily obtain an analogous
property for Schur Laurent polynomials:

\begin{lemma}
\label{lem.alt.smove}Let $\lambda\in\mathbb{Z}^{n}$ be any snake, and let
$d\in\mathbb{Z}$. Then, $\overline{s}_{\lambda+d}=x_{\Pi}^{d}\overline
{s}_{\lambda}$.
\end{lemma}

\begin{proof}
[Proof of Lemma \ref{lem.alt.smove}.]Proposition \ref{prop.snakes.trivia}
\textbf{(c)} (applied to $\mu=\rho$) yields $\left(  \lambda+\rho\right)
+d=\left(  \lambda+d\right)  +\rho$. But Lemma \ref{lem.alt.move} (applied to
$\lambda+\rho$ instead of $\lambda$) yields $a_{\left(  \lambda+\rho\right)
+d}=x_{\Pi}^{d}a_{\lambda+\rho}$. This rewrites as $a_{\left(  \lambda
+d\right)  +\rho}=x_{\Pi}^{d}a_{\lambda+\rho}$ (since $\left(  \lambda
+\rho\right)  +d=\left(  \lambda+d\right)  +\rho$).

The definition of $\overline{s}_{\lambda+d}$ yields%
\[
\overline{s}_{\lambda+d}=\dfrac{a_{\left(  \lambda+d\right)  +\rho}}{a_{\rho}%
}=\dfrac{x_{\Pi}^{d}a_{\lambda+\rho}}{a_{\rho}}\ \ \ \ \ \ \ \ \ \ \left(
\text{since }a_{\left(  \lambda+d\right)  +\rho}=x_{\Pi}^{d}a_{\lambda+\rho
}\right)  .
\]
Comparing this with%
\[
x_{\Pi}^{d}\underbrace{\overline{s}_{\lambda}}_{\substack{=\dfrac
{a_{\lambda+\rho}}{a_{\rho}}\\\text{(by the definition of }\overline
{s}_{\lambda}\text{)}}}=x_{\Pi}^{d}\cdot\dfrac{a_{\lambda+\rho}}{a_{\rho}%
}=\dfrac{x_{\Pi}^{d}a_{\lambda+\rho}}{a_{\rho}},
\]
we obtain $\overline{s}_{\lambda+d}=x_{\Pi}^{d}\overline{s}_{\lambda}$. This
proves Lemma \ref{lem.alt.smove}.
\end{proof}

\begin{lemma}
\label{lem.alt.lrr}Let $\mu,\nu\in\operatorname*{Par}\left[  n\right]  $.
Then,%
\[
\overline{s}_{\mu}\overline{s}_{\nu}=\sum_{\lambda\in\operatorname*{Par}%
\left[  n\right]  }c_{\mu,\nu}^{\lambda}\overline{s}_{\lambda}.
\]

\end{lemma}

\begin{proof}
[Proof of Lemma \ref{lem.alt.lrr}.]It is well-known (see, e.g., \cite[Exercise
2.3.8(b)]{GriRei}) that if $\lambda$ is a partition having length $>n$, then%
\begin{equation}
s_{\lambda}\left(  x_{1},x_{2},\ldots,x_{n}\right)  =0.
\label{pf.lem.alt.lrr.=0}%
\end{equation}

\begin{vershort}
We know that $\mu\in\operatorname*{Par}\left[  n\right]  $. Hence, Lemma
\ref{lem.alt.sbar} (applied to $\lambda=\mu$) yields $\overline{s}_{\mu
}=s_{\mu}\left(  x_{1},x_{2},\ldots,x_{n}\right)  $. Likewise, $\overline
{s}_{\nu}=s_{\nu}\left(  x_{1},x_{2},\ldots,x_{n}\right)  $. Multiplying these
two equalities, we obtain%
\begin{align*}
\overline{s}_{\mu}\overline{s}_{\nu}  &  =s_{\mu}\left(  x_{1},x_{2}%
,\ldots,x_{n}\right)  \cdot s_{\nu}\left(  x_{1},x_{2},\ldots,x_{n}\right) \\
&  =\underbrace{\left(  s_{\mu}s_{\nu}\right)  }_{\substack{=\sum_{\lambda
\in\operatorname*{Par}}c_{\mu,\nu}^{\lambda}s_{\lambda}\\\text{(by
(\ref{eq.lrcoeff.def}))}}}\left(  x_{1},x_{2},\ldots,x_{n}\right)
=\sum_{\lambda\in\operatorname*{Par}}c_{\mu,\nu}^{\lambda}s_{\lambda}\left(
x_{1},x_{2},\ldots,x_{n}\right) \\
&  =\underbrace{\sum_{\substack{\lambda\in\operatorname*{Par};\\\lambda\text{
has length }\leq n}}}_{\substack{=\sum_{\lambda\in\operatorname*{Par}\left[
n\right]  }\\\text{(by the definition of }\operatorname*{Par}\left[  n\right]
\text{)}}}c_{\mu,\nu}^{\lambda}s_{\lambda}\left(  x_{1},x_{2},\ldots
,x_{n}\right)  +\sum_{\substack{\lambda\in\operatorname*{Par};\\\lambda\text{
has length }>n}}c_{\mu,\nu}^{\lambda}\underbrace{s_{\lambda}\left(
x_{1},x_{2},\ldots,x_{n}\right)  }_{\substack{=0\\\text{(by
(\ref{pf.lem.alt.lrr.=0}))}}}\\
&  =\sum_{\lambda\in\operatorname*{Par}\left[  n\right]  }c_{\mu,\nu}%
^{\lambda}\underbrace{s_{\lambda}\left(  x_{1},x_{2},\ldots,x_{n}\right)
}_{\substack{=\overline{s}_{\lambda}\\\text{(by Lemma \ref{lem.alt.sbar})}%
}}+\underbrace{\sum_{\substack{\lambda\in\operatorname*{Par};\\\lambda\text{
has length }>n}}c_{\mu,\nu}^{\lambda}0}_{=0}=\sum_{\lambda\in
\operatorname*{Par}\left[  n\right]  }c_{\mu,\nu}^{\lambda}\overline
{s}_{\lambda}.
\end{align*}

\end{vershort}

\begin{verlong}
We know that $\mu\in\operatorname*{Par}\left[  n\right]  $. Hence, Lemma
\ref{lem.alt.sbar} (applied to $\lambda=\mu$) yields $\overline{s}_{\mu
}=s_{\mu}\left(  x_{1},x_{2},\ldots,x_{n}\right)  $. Likewise, $\overline
{s}_{\nu}=s_{\nu}\left(  x_{1},x_{2},\ldots,x_{n}\right)  $. Multiplying these
two equalities, we obtain%
\begin{align*}
\overline{s}_{\mu}\overline{s}_{\nu}  &  =s_{\mu}\left(  x_{1},x_{2}%
,\ldots,x_{n}\right)  \cdot s_{\nu}\left(  x_{1},x_{2},\ldots,x_{n}\right)
=\underbrace{\left(  s_{\mu}s_{\nu}\right)  }_{\substack{=\sum_{\lambda
\in\operatorname*{Par}}c_{\mu,\nu}^{\lambda}s_{\lambda}\\\text{(by
(\ref{eq.lrcoeff.def}))}}}\left(  x_{1},x_{2},\ldots,x_{n}\right) \\
&  =\left(  \sum_{\lambda\in\operatorname*{Par}}c_{\mu,\nu}^{\lambda
}s_{\lambda}\right)  \left(  x_{1},x_{2},\ldots,x_{n}\right)  =\sum
_{\lambda\in\operatorname*{Par}}c_{\mu,\nu}^{\lambda}s_{\lambda}\left(
x_{1},x_{2},\ldots,x_{n}\right) \\
&  =\underbrace{\sum_{\substack{\lambda\in\operatorname*{Par};\\\lambda\text{
has length }\leq n}}}_{\substack{=\sum_{\lambda\in\operatorname*{Par}\left[
n\right]  }\\\text{(by the definition of }\operatorname*{Par}\left[  n\right]
\text{)}}}c_{\mu,\nu}^{\lambda}s_{\lambda}\left(  x_{1},x_{2},\ldots
,x_{n}\right)  +\sum_{\substack{\lambda\in\operatorname*{Par};\\\lambda\text{
has length }>n}}c_{\mu,\nu}^{\lambda}\underbrace{s_{\lambda}\left(
x_{1},x_{2},\ldots,x_{n}\right)  }_{\substack{=0\\\text{(by
(\ref{pf.lem.alt.lrr.=0}))}}}\\
&  \ \ \ \ \ \ \ \ \ \ \ \ \ \ \ \ \ \ \ \ \left(
\begin{array}
[c]{c}%
\text{since each }\lambda\in\operatorname*{Par}\text{ either has length }\leq
n\text{ or has length }>n\\
\text{(but not both at the same time)}%
\end{array}
\right) \\
&  =\sum_{\lambda\in\operatorname*{Par}\left[  n\right]  }c_{\mu,\nu}%
^{\lambda}\underbrace{s_{\lambda}\left(  x_{1},x_{2},\ldots,x_{n}\right)
}_{\substack{=\overline{s}_{\lambda}\\\text{(by Lemma \ref{lem.alt.sbar})}%
}}+\underbrace{\sum_{\substack{\lambda\in\operatorname*{Par};\\\lambda\text{
has length }>n}}c_{\mu,\nu}^{\lambda}0}_{=0}=\sum_{\lambda\in
\operatorname*{Par}\left[  n\right]  }c_{\mu,\nu}^{\lambda}\overline
{s}_{\lambda}.
\end{align*}

\end{verlong}

\noindent This proves Lemma \ref{lem.alt.lrr}.
\end{proof}

\begin{lemma}
\label{lem.alt.linind}The family $\left(  \overline{s}_{\lambda}\right)
_{\lambda\in\left\{  \text{snakes}\right\}  }$ of elements of $\mathcal{L}$ is
$\mathbf{k}$-linearly independent.
\end{lemma}

\begin{proof}
[Proof of Lemma \ref{lem.alt.linind}.]Let us define a \emph{strict snake} to
be an $n$-tuple $\alpha\in\mathbb{Z}^{n}$ of integers satisfying $\alpha
_{1}>\alpha_{2}>\cdots>\alpha_{n}$. It is easy to see that the map
\begin{align}
\left\{  \text{snakes}\right\}   &  \rightarrow\left\{  \text{strict
snakes}\right\}  ,\nonumber\\
\lambda &  \mapsto\lambda+\rho\label{pf.lem.alt.linind.strict-bij}%
\end{align}
is a bijection.

It is also easy to see that if $\alpha$ and $\beta$ are two strict snakes,
then%
\begin{equation}
\left(  \text{the coefficient of }x_{1}^{\beta_{1}}x_{2}^{\beta_{2}}\cdots
x_{n}^{\beta_{n}}\text{ in }a_{\alpha}\right)  =\delta_{\alpha,\beta},
\label{pf.lem.alt.linind.coeff-kron}%
\end{equation}
where $\delta_{\alpha,\beta}$ is the Kronecker delta of $\alpha$ and $\beta$
(that is, the integer $%
\begin{cases}
1, & \text{if }\alpha=\beta;\\
0, & \text{if }\alpha\neq\beta
\end{cases}
$\ \ ).

\begin{verlong}
[\textit{Proof of (\ref{pf.lem.alt.linind.coeff-kron}):} Let $\alpha$ and
$\beta$ be two strict snakes. Thus, $\alpha_{1}>\alpha_{2}>\cdots>\alpha_{n}$
and $\beta_{1}>\beta_{2}>\cdots>\beta_{n}$ (by the definition of a strict snake).

The definition of $a_{\alpha}$ says that%
\begin{align}
a_{\alpha}  &  =\sum_{w\in\mathfrak{S}_{n}}\left(  -1\right)  ^{w}x_{w\left(
1\right)  }^{\alpha_{1}}x_{w\left(  2\right)  }^{\alpha_{2}}\cdots x_{w\left(
n\right)  }^{\alpha_{n}}=\sum_{w\in\mathfrak{S}_{n}}\left(  -1\right)
^{w^{-1}}x_{w^{-1}\left(  1\right)  }^{\alpha_{1}}x_{w^{-1}\left(  2\right)
}^{\alpha_{2}}\cdots x_{w^{-1}\left(  n\right)  }^{\alpha_{n}}\nonumber\\
&  \ \ \ \ \ \ \ \ \ \ \left(
\begin{array}
[c]{c}%
\text{here, we have substituted }w^{-1}\text{ for }w\text{ in the sum,}\\
\text{since the map }\mathfrak{S}_{n}\rightarrow\mathfrak{S}_{n},\ w\mapsto
w^{-1}\text{ is a bijection}%
\end{array}
\right) \nonumber\\
&  =\underbrace{\left(  -1\right)  ^{\operatorname*{id}\nolimits^{-1}}%
}_{\substack{=\left(  -1\right)  ^{\operatorname*{id}}\\=1}%
}\underbrace{x_{\operatorname*{id}\nolimits^{-1}\left(  1\right)  }%
^{\alpha_{1}}x_{\operatorname*{id}\nolimits^{-1}\left(  2\right)  }%
^{\alpha_{2}}\cdots x_{\operatorname*{id}\nolimits^{-1}\left(  n\right)
}^{\alpha_{n}}}_{\substack{=x_{1}^{\alpha_{1}}x_{2}^{\alpha_{2}}\cdots
x_{n}^{\alpha_{n}}\\\text{(since }\operatorname*{id}\nolimits^{-1}\left(
i\right)  =i\text{ for each }i\in\left\{  1,2,\ldots,n\right\}  \text{)}%
}}+\sum_{\substack{w\in\mathfrak{S}_{n};\\w\neq\operatorname*{id}%
}}\underbrace{\left(  -1\right)  ^{w^{-1}}}_{=\left(  -1\right)  ^{w}%
}x_{w^{-1}\left(  1\right)  }^{\alpha_{1}}x_{w^{-1}\left(  2\right)  }%
^{\alpha_{2}}\cdots x_{w^{-1}\left(  n\right)  }^{\alpha_{n}}\nonumber\\
&  \ \ \ \ \ \ \ \ \ \ \left(  \text{here, we have split off the addend for
}w=\operatorname*{id}\text{ from the sum}\right) \nonumber\\
&  =x_{1}^{\alpha_{1}}x_{2}^{\alpha_{2}}\cdots x_{n}^{\alpha_{n}}%
+\sum_{\substack{w\in\mathfrak{S}_{n};\\w\neq\operatorname*{id}}}\left(
-1\right)  ^{w}x_{w^{-1}\left(  1\right)  }^{\alpha_{1}}x_{w^{-1}\left(
2\right)  }^{\alpha_{2}}\cdots x_{w^{-1}\left(  n\right)  }^{\alpha_{n}}.
\label{pf.lem.alt.linind.coeff-kron.pf.1}%
\end{align}
But every permutation $w\in\mathfrak{S}_{n}$ satisfies%
\begin{align}
x_{w^{-1}\left(  1\right)  }^{\alpha_{1}}x_{w^{-1}\left(  2\right)  }%
^{\alpha_{2}}\cdots x_{w^{-1}\left(  n\right)  }^{\alpha_{n}}  &  =\prod
_{i\in\left\{  1,2,\ldots,n\right\}  }x_{w^{-1}\left(  i\right)  }^{\alpha
_{i}}=\prod_{i\in\left\{  1,2,\ldots,n\right\}  }\underbrace{x_{w^{-1}\left(
w\left(  i\right)  \right)  }^{\alpha_{w\left(  i\right)  }}}%
_{\substack{=x_{i}^{\alpha_{w\left(  i\right)  }}\\\text{(since }w^{-1}\left(
w\left(  i\right)  \right)  =i\text{)}}}\nonumber\\
&  \ \ \ \ \ \ \ \ \ \ \left(
\begin{array}
[c]{c}%
\text{here, we have substituted }w\left(  i\right)  \text{ for }i\text{ in the
product,}\\
\text{since }w:\left\{  1,2,\ldots,n\right\}  \rightarrow\left\{
1,2,\ldots,n\right\}  \text{ is a bijection}%
\end{array}
\right) \nonumber\\
&  =\prod_{i\in\left\{  1,2,\ldots,n\right\}  }x_{i}^{\alpha_{w\left(
i\right)  }}=x_{1}^{\alpha_{w\left(  1\right)  }}x_{2}^{\alpha_{w\left(
2\right)  }}\cdots x_{n}^{\alpha_{w\left(  n\right)  }}.
\label{pf.lem.alt.linind.coeff-kron.pf.w}%
\end{align}
Hence, (\ref{pf.lem.alt.linind.coeff-kron.pf.1}) becomes%
\begin{align}
a_{\alpha}  &  =x_{1}^{\alpha_{1}}x_{2}^{\alpha_{2}}\cdots x_{n}^{\alpha_{n}%
}+\sum_{\substack{w\in\mathfrak{S}_{n};\\w\neq\operatorname*{id}}}\left(
-1\right)  ^{w}\underbrace{x_{w^{-1}\left(  1\right)  }^{\alpha_{1}}%
x_{w^{-1}\left(  2\right)  }^{\alpha_{2}}\cdots x_{w^{-1}\left(  n\right)
}^{\alpha_{n}}}_{\substack{=x_{1}^{\alpha_{w\left(  1\right)  }}x_{2}%
^{\alpha_{w\left(  2\right)  }}\cdots x_{n}^{\alpha_{w\left(  n\right)  }%
}\\\text{(by (\ref{pf.lem.alt.linind.coeff-kron.pf.w}))}}}\nonumber\\
&  =x_{1}^{\alpha_{1}}x_{2}^{\alpha_{2}}\cdots x_{n}^{\alpha_{n}}%
+\sum_{\substack{w\in\mathfrak{S}_{n};\\w\neq\operatorname*{id}}}\left(
-1\right)  ^{w}x_{1}^{\alpha_{w\left(  1\right)  }}x_{2}^{\alpha_{w\left(
2\right)  }}\cdots x_{n}^{\alpha_{w\left(  n\right)  }}.
\label{pf.lem.alt.linind.coeff-kron.pf.2}%
\end{align}

Now, let us fix a permutation $w\in\mathfrak{S}_{n}$ satisfying $w\neq
\operatorname*{id}$. Then, the two $n$-tuples $\left(  \alpha_{w\left(
1\right)  },\alpha_{w\left(  2\right)  },\ldots,\alpha_{w\left(  n\right)
}\right)  $ and $\left(  \beta_{1},\beta_{2},\ldots,\beta_{n}\right)  $ are
distinct\footnote{\textit{Proof.} Assume the contrary. Thus, $\left(
\alpha_{w\left(  1\right)  },\alpha_{w\left(  2\right)  },\ldots
,\alpha_{w\left(  n\right)  }\right)  =\left(  \beta_{1},\beta_{2}%
,\ldots,\beta_{n}\right)  $.
\par
If we had $w\left(  1\right)  <w\left(  2\right)  <\cdots<w\left(  n\right)
$, then we would have $w=\operatorname*{id}$ (since $w$ is a permutation of
$\left\{  1,2,\ldots,n\right\}  $), which would contradict $w\neq
\operatorname*{id}$. Thus, we cannot have $w\left(  1\right)  <w\left(
2\right)  <\cdots<w\left(  n\right)  $. Hence, there exists some $i\in\left\{
1,2,\ldots,n-1\right\}  $ such that $w\left(  i\right)  \geq w\left(
i+1\right)  $. Consider this $i$.
\par
But if $u$ and $v$ are two elements of $\left\{  1,2,\ldots,n\right\}  $
satisfying $u\geq v$, then $\alpha_{u}\leq\alpha_{v}$ (since $\alpha
_{1}>\alpha_{2}>\cdots>\alpha_{n}$). Applying this to $u=w\left(  i\right)  $
and $v=w\left(  i+1\right)  $, we obtain $\alpha_{w\left(  i\right)  }%
\leq\alpha_{w\left(  i+1\right)  }$ (since $w\left(  i\right)  \geq w\left(
i+1\right)  $). But $\alpha_{w\left(  i\right)  }=\beta_{i}$ (since $\left(
\alpha_{w\left(  1\right)  },\alpha_{w\left(  2\right)  },\ldots
,\alpha_{w\left(  n\right)  }\right)  =\left(  \beta_{1},\beta_{2}%
,\ldots,\beta_{n}\right)  $) and $\alpha_{w\left(  i+1\right)  }=\beta_{i+1}$
(for the same reason). Hence, $\beta_{i}=\alpha_{w\left(  i\right)  }%
\leq\alpha_{w\left(  i+1\right)  }=\beta_{i+1}$. However, from $\beta
_{1}>\beta_{2}>\cdots>\beta_{n}$, we obtain $\beta_{i}>\beta_{i+1}$. This
contradicts $\beta_{i}\leq\beta_{i+1}$. This contradiction shows that our
assumption was false. Qed.}. Hence, $x_{1}^{\alpha_{w\left(  1\right)  }}%
x_{2}^{\alpha_{w\left(  2\right)  }}\cdots x_{n}^{\alpha_{w\left(  n\right)
}}$ and $x_{1}^{\beta_{1}}x_{2}^{\beta_{2}}\cdots x_{n}^{\beta_{n}}$ are two
distinct monomials. Thus,%
\begin{align}
&  \left(  \text{the coefficient of }x_{1}^{\beta_{1}}x_{2}^{\beta_{2}}\cdots
x_{n}^{\beta_{n}}\text{ in }x_{1}^{\alpha_{w\left(  1\right)  }}x_{2}%
^{\alpha_{w\left(  2\right)  }}\cdots x_{n}^{\alpha_{w\left(  n\right)  }%
}\right) \nonumber\\
&  =0. \label{pf.lem.alt.linind.coeff-kron.pf.4}%
\end{align}

Now, forget that we fixed $w$. We thus have proved
(\ref{pf.lem.alt.linind.coeff-kron.pf.4}) for every permutation $w\in
\mathfrak{S}_{n}$ satisfying $w\neq\operatorname*{id}$. Now,%
\begin{align*}
&  \left(  \text{the coefficient of }x_{1}^{\beta_{1}}x_{2}^{\beta_{2}}\cdots
x_{n}^{\beta_{n}}\text{ in }a_{\alpha}\right) \\
&  =\left(  \text{the coefficient of }x_{1}^{\beta_{1}}x_{2}^{\beta_{2}}\cdots
x_{n}^{\beta_{n}}\text{ in }x_{1}^{\alpha_{1}}x_{2}^{\alpha_{2}}\cdots
x_{n}^{\alpha_{n}}+\sum_{\substack{w\in\mathfrak{S}_{n};\\w\neq
\operatorname*{id}}}\left(  -1\right)  ^{w}x_{1}^{\alpha_{w\left(  1\right)
}}x_{2}^{\alpha_{w\left(  2\right)  }}\cdots x_{n}^{\alpha_{w\left(  n\right)
}}\right) \\
&  \ \ \ \ \ \ \ \ \ \ \ \ \ \ \ \ \ \ \ \ \left(  \text{by
(\ref{pf.lem.alt.linind.coeff-kron.pf.2})}\right) \\
&  =\underbrace{\left(  \text{the coefficient of }x_{1}^{\beta_{1}}%
x_{2}^{\beta_{2}}\cdots x_{n}^{\beta_{n}}\text{ in }x_{1}^{\alpha_{1}}%
x_{2}^{\alpha_{2}}\cdots x_{n}^{\alpha_{n}}\right)  }_{\substack{=\delta
_{\left(  \alpha_{1},\alpha_{2},\ldots,\alpha_{n}\right)  ,\left(  \beta
_{1},\beta_{2},\ldots,\beta_{n}\right)  }=\delta_{\alpha,\beta}\\\text{(since
}\left(  \alpha_{1},\alpha_{2},\ldots,\alpha_{n}\right)  =\alpha\text{ and
}\left(  \beta_{1},\beta_{2},\ldots,\beta_{n}\right)  =\beta\text{)}}}\\
&  \ \ \ \ \ \ \ \ \ \ +\sum_{\substack{w\in\mathfrak{S}_{n};\\w\neq
\operatorname*{id}}}\left(  -1\right)  ^{w}\underbrace{\left(  \text{the
coefficient of }x_{1}^{\beta_{1}}x_{2}^{\beta_{2}}\cdots x_{n}^{\beta_{n}%
}\text{ in }x_{1}^{\alpha_{w\left(  1\right)  }}x_{2}^{\alpha_{w\left(
2\right)  }}\cdots x_{n}^{\alpha_{w\left(  n\right)  }}\right)  }%
_{\substack{=0\\\text{(by (\ref{pf.lem.alt.linind.coeff-kron.pf.4}))}}}\\
&  =\delta_{\alpha,\beta}+\underbrace{\sum_{\substack{w\in\mathfrak{S}%
_{n};\\w\neq\operatorname*{id}}}\left(  -1\right)  ^{w}0}_{=0}=\delta
_{\alpha,\beta}.
\end{align*}
This proves (\ref{pf.lem.alt.linind.coeff-kron}).]
\end{verlong}

Now, assume that $\left(  u_{\lambda}\right)  _{\lambda\in\left\{
\text{snakes}\right\}  }\in\mathbf{k}^{\left\{  \text{snakes}\right\}  }$ be a
family of scalars with the property that $\left(  \text{all but finitely many
snakes }\lambda\text{ satisfy }u_{\lambda}=0\right)  $ and
\begin{equation}
\sum_{\lambda\in\left\{  \text{snakes}\right\}  }u_{\lambda}\overline
{s}_{\lambda}=0. \label{pf.lem.alt.linind.ass}%
\end{equation}
We shall show that $u_{\lambda}=0$ for all snakes $\lambda$.

Indeed, fix a snake $\mu$. Then, $\mu+\rho$ is a strict snake (since the map
(\ref{pf.lem.alt.linind.strict-bij}) is a bijection). Let us denote this
strict snake by $\beta$. Thus, $\beta=\mu+\rho$.

If $\lambda$ is any snake, then $\lambda+\rho$ is a strict snake (since the
map (\ref{pf.lem.alt.linind.strict-bij}) is a bijection), and thus satisfies%
\begin{align}
&  \left(  \text{the coefficient of }x_{1}^{\beta_{1}}x_{2}^{\beta_{2}}\cdots
x_{n}^{\beta_{n}}\text{ in }a_{\lambda+\rho}\right) \nonumber\\
&  =\delta_{\lambda+\rho,\beta}\ \ \ \ \ \ \ \ \ \ \left(  \text{by
(\ref{pf.lem.alt.linind.coeff-kron}), applied to }\alpha=\lambda+\rho\right)
\nonumber\\
&  =\delta_{\lambda+\rho,\mu+\rho}\ \ \ \ \ \ \ \ \ \ \left(  \text{since
}\beta=\mu+\rho\right) \nonumber\\
&  =\delta_{\lambda,\mu} \label{pf.lem.alt.linind.5}%
\end{align}
(since $\lambda+\rho=\mu+\rho$ holds if and only if $\lambda=\mu$ holds).

From (\ref{pf.lem.alt.linind.ass}), we obtain%
\[
0=\sum_{\lambda\in\left\{  \text{snakes}\right\}  }u_{\lambda}%
\underbrace{\overline{s}_{\lambda}}_{\substack{=\dfrac{a_{\lambda+\rho}%
}{a_{\rho}}\\\text{(by the definition of }\overline{s}_{\lambda}\text{)}%
}}=\sum_{\lambda\in\left\{  \text{snakes}\right\}  }u_{\lambda}\dfrac
{a_{\lambda+\rho}}{a_{\rho}}=\dfrac{1}{a_{\rho}}\sum_{\lambda\in\left\{
\text{snakes}\right\}  }u_{\lambda}a_{\lambda+\rho}.
\]
Multiplying both sides of this equality by $a_{\rho}$, we obtain%
\[
0=\sum_{\lambda\in\left\{  \text{snakes}\right\}  }u_{\lambda}a_{\lambda+\rho
}.
\]
Hence,%
\begin{align*}
&  \left(  \text{the coefficient of }x_{1}^{\beta_{1}}x_{2}^{\beta_{2}}\cdots
x_{n}^{\beta_{n}}\text{ in }0\right) \\
&  =\left(  \text{the coefficient of }x_{1}^{\beta_{1}}x_{2}^{\beta_{2}}\cdots
x_{n}^{\beta_{n}}\text{ in }\sum_{\lambda\in\left\{  \text{snakes}\right\}
}u_{\lambda}a_{\lambda+\rho}\right) \\
&  =\sum_{\lambda\in\left\{  \text{snakes}\right\}  }u_{\lambda}%
\underbrace{\left(  \text{the coefficient of }x_{1}^{\beta_{1}}x_{2}%
^{\beta_{2}}\cdots x_{n}^{\beta_{n}}\text{ in }a_{\lambda+\rho}\right)
}_{\substack{=\delta_{\lambda,\mu}\\\text{(by (\ref{pf.lem.alt.linind.5}))}%
}}\\
&  =\sum_{\lambda\in\left\{  \text{snakes}\right\}  }u_{\lambda}%
\delta_{\lambda,\mu}=u_{\mu},
\end{align*}
so that%
\[
u_{\mu}=\left(  \text{the coefficient of }x_{1}^{\beta_{1}}x_{2}^{\beta_{2}%
}\cdots x_{n}^{\beta_{n}}\text{ in }0\right)  =0.
\]

Now, forget that we fixed $\mu$. We thus have proved that $u_{\mu}=0$ for all
snakes $\mu$. In other words, $u_{\lambda}=0$ for all snakes $\lambda$.

Forget that we fixed $\left(  u_{\lambda}\right)  _{\lambda\in\left\{
\text{snakes}\right\}  }$. We thus have shown that if $\left(  u_{\lambda
}\right)  _{\lambda\in\left\{  \text{snakes}\right\}  }\in\mathbf{k}^{\left\{
\text{snakes}\right\}  }$ is a family of scalars with the property that
\newline$\left(  \text{all but finitely many snakes }\lambda\text{ satisfy
}u_{\lambda}=0\right)  $ and $\sum_{\lambda\in\left\{  \text{snakes}\right\}
}u_{\lambda}\overline{s}_{\lambda}=0$, then $u_{\lambda}=0$ for all snakes
$\lambda$. In other words, the family $\left(  \overline{s}_{\lambda}\right)
_{\lambda\in\left\{  \text{snakes}\right\}  }$ of elements of $\mathcal{L}$ is
$\mathbf{k}$-linearly independent. This proves Lemma \ref{lem.alt.linind}.
\end{proof}

Lemma \ref{lem.alt.linind} is actually part of a stronger claim: The family
$\left(  \overline{s}_{\lambda}\right)  _{\lambda\in\left\{  \text{snakes}%
\right\}  }$ is a basis of the $\mathbf{k}$-module of symmetric Laurent
polynomials in $x_{1},x_{2},\ldots,x_{n}$. We shall not need this, however, so
we omit the proof\footnote{Just in case: It follows easily from Lemma
\ref{lem.alt.smove} and \cite[Remark 2.3.9(d)]{GriRei}.}.

Recall Definition \ref{def.snakes} \textbf{(d)}. Our next lemma connects the
Laurent polynomials $\overline{s}_{\lambda}$ and $\overline{s}_{\lambda^{\vee
}}$ for every snake $\lambda$; it is folklore (see \cite[Exercise
2.9.15(d)]{GriRei} for an equivalent version), but we shall prove it for the
sake of completeness.

\begin{lemma}
\label{lem.alt.inverses}Let $\lambda$ be a snake. Then,%
\[
\overline{s}_{\lambda^{\vee}}=\overline{s}_{\lambda}\left(  x_{1}^{-1}%
,x_{2}^{-1},\ldots,x_{n}^{-1}\right)  .
\]

\end{lemma}

Here, of course, $\overline{s}_{\lambda}\left(  x_{1}^{-1},x_{2}^{-1}%
,\ldots,x_{n}^{-1}\right)  $ means the result of substituting $x_{1}%
^{-1},x_{2}^{-1},\ldots,x_{n}^{-1}$ for $x_{1},x_{2},\ldots,x_{n}$ in the
Laurent polynomial $\overline{s}_{\lambda}\in\mathcal{L}$.

\begin{proof}
[Proof of Lemma \ref{lem.alt.inverses}.]Let $w_{0}\in\mathfrak{S}_{n}$ be the
permutation of $\left\{  1,2,\ldots,n\right\}  $ that sends each $i\in\left\{
1,2,\ldots,n\right\}  $ to $n+1-i$. The map $\mathfrak{S}_{n}\rightarrow
\mathfrak{S}_{n},\ w\mapsto w\circ w_{0}$ is a bijection (since $\mathfrak{S}%
_{n}$ is a group).

For each $i\in\left\{  1,2,\ldots,n\right\}  $, we have
\begin{equation}
\left(  \lambda+\rho\right)  _{i}=\lambda_{i}+\underbrace{\rho_{i}%
}_{\substack{=n-i\\\text{(by (\ref{eq.snakes.f.rhoi=}))}}}=\lambda_{i}+n-i.
\label{pf.lem.alt.inverses.rho2}%
\end{equation}

Proposition \ref{prop.snakes.trivia} \textbf{(b)} shows that $\lambda^{\vee}$
is a snake. Thus, using Proposition \ref{prop.snakes.trivia} \textbf{(a)}, we
conclude that $\lambda^{\vee}+\left(  1-n\right)  $ is a snake. Let us denote
this snake by $\mu$. Thus, $\mu=\lambda^{\vee}+\left(  1-n\right)  $. Hence,
$\mu+\rho=\left(  \lambda^{\vee}+\left(  1-n\right)  \right)  +\rho=\left(
\lambda^{\vee}+\rho\right)  +\left(  1-n\right)  $ (this follows from
Proposition \ref{prop.snakes.trivia} \textbf{(c)}, applied to $\lambda^{\vee}%
$, $\rho$ and $1-n$ instead of $\lambda$, $\mu$ and $d$). Therefore,%
\begin{equation}
a_{\mu+\rho}=a_{\left(  \lambda^{\vee}+\rho\right)  +\left(  1-n\right)
}=x_{\Pi}^{1-n}a_{\lambda^{\vee}+\rho} \label{pf.lem.alt.inverses.0}%
\end{equation}
(by Lemma \ref{lem.alt.move}, applied to $\lambda^{\vee}+\rho$ and $1-n$
instead of $\lambda$ and $d$).

The definition of $\lambda^{\vee}$ yields $\lambda^{\vee}=\left(  -\lambda
_{n},-\lambda_{n-1},\ldots,-\lambda_{1}\right)  $. Hence, for each
$i\in\left\{  1,2,\ldots,n\right\}  $, we have%
\begin{equation}
\left(  \lambda^{\vee}\right)  _{i}=-\lambda_{n+1-i}.
\label{pf.lem.alt.inverses.check}%
\end{equation}
Thus, for each $i\in\left\{  1,2,\ldots,n\right\}  $, we have%
\begin{align}
\left(  \mu+\rho\right)  _{i}  &  =\underbrace{\mu_{i}}_{\substack{=\left(
\lambda^{\vee}+\left(  1-n\right)  \right)  _{i}\\\text{(since }\mu
=\lambda^{\vee}+\left(  1-n\right)  \text{)}}}+\underbrace{\rho_{i}%
}_{\substack{=n-i\\\text{(by (\ref{eq.snakes.f.rhoi=}))}}}=\underbrace{\left(
\lambda^{\vee}+\left(  1-n\right)  \right)  _{i}}_{\substack{=\left(
\lambda^{\vee}\right)  _{i}+\left(  1-n\right)  \\\text{(by the definition of
}\lambda^{\vee}+\left(  1-n\right)  \text{)}}}+n-i\nonumber\\
&  =\underbrace{\left(  \lambda^{\vee}\right)  _{i}}_{\substack{=-\lambda
_{n+1-i}\\\text{(by (\ref{pf.lem.alt.inverses.check}))}}}+\underbrace{\left(
1-n\right)  +n-i}_{=1-i}\nonumber\\
&  =-\lambda_{n+1-i}+1-i. \label{pf.lem.alt.inverses.chec2}%
\end{align}
Hence, the definition of $a_{\mu+\rho}$ yields%
\begin{align}
a_{\mu+\rho}  &  =\sum_{w\in\mathfrak{S}_{n}}\left(  -1\right)  ^{w}%
\underbrace{x_{w\left(  1\right)  }^{\left(  \mu+\rho\right)  _{1}}x_{w\left(
2\right)  }^{\left(  \mu+\rho\right)  _{2}}\cdots x_{w\left(  n\right)
}^{\left(  \mu+\rho\right)  _{n}}}_{=\prod_{i\in\left\{  1,2,\ldots,n\right\}
}x_{w\left(  i\right)  }^{\left(  \mu+\rho\right)  _{i}}}\nonumber\\
&  =\sum_{w\in\mathfrak{S}_{n}}\left(  -1\right)  ^{w}\prod_{i\in\left\{
1,2,\ldots,n\right\}  }\underbrace{x_{w\left(  i\right)  }^{\left(  \mu
+\rho\right)  _{i}}}_{\substack{=x_{w\left(  i\right)  }^{-\lambda
_{n+1-i}+1-i}\\\text{(since (\ref{pf.lem.alt.inverses.chec2}) yields }\left(
\mu+\rho\right)  _{i}=-\lambda_{n+1-i}+1-i\text{)}}}\nonumber\\
&  =\sum_{w\in\mathfrak{S}_{n}}\left(  -1\right)  ^{w}\prod_{i\in\left\{
1,2,\ldots,n\right\}  }x_{w\left(  i\right)  }^{-\lambda_{n+1-i}+1-i}.
\label{pf.lem.alt.inverses.ampr}%
\end{align}

The definition of $a_{\lambda+\rho}$ yields%
\begin{align*}
a_{\lambda+\rho}  &  =\sum_{w\in\mathfrak{S}_{n}}\left(  -1\right)
^{w}\underbrace{x_{w\left(  1\right)  }^{\left(  \lambda+\rho\right)  _{1}%
}x_{w\left(  2\right)  }^{\left(  \lambda+\rho\right)  _{2}}\cdots x_{w\left(
n\right)  }^{\left(  \lambda+\rho\right)  _{n}}}_{=\prod_{i\in\left\{
1,2,\ldots,n\right\}  }x_{w\left(  i\right)  }^{\left(  \lambda+\rho\right)
_{i}}}\\
&  =\sum_{w\in\mathfrak{S}_{n}}\left(  -1\right)  ^{w}\prod_{i\in\left\{
1,2,\ldots,n\right\}  }\underbrace{x_{w\left(  i\right)  }^{\left(
\lambda+\rho\right)  _{i}}}_{\substack{=x_{w\left(  i\right)  }^{\lambda
_{i}+n-i}\\\text{(since (\ref{pf.lem.alt.inverses.rho2}) yields }\left(
\lambda+\rho\right)  _{i}=\lambda_{i}+n-i\text{)}}}\\
&  =\sum_{w\in\mathfrak{S}_{n}}\left(  -1\right)  ^{w}\prod_{i\in\left\{
1,2,\ldots,n\right\}  }x_{w\left(  i\right)  }^{\lambda_{i}+n-i}.
\end{align*}
Substituting $x_{1}^{-1},x_{2}^{-1},\ldots,x_{n}^{-1}$ for $x_{1},x_{2}%
,\ldots,x_{n}$ on both sides of this equality, we obtain%
\begin{align}
&  a_{\lambda+\rho}\left(  x_{1}^{-1},x_{2}^{-1},\ldots,x_{n}^{-1}\right)
\nonumber\\
&  =\sum_{w\in\mathfrak{S}_{n}}\left(  -1\right)  ^{w}\prod_{i\in\left\{
1,2,\ldots,n\right\}  }\underbrace{\left(  x_{w\left(  i\right)  }%
^{-1}\right)  ^{\lambda_{i}+n-i}}_{=x_{w\left(  i\right)  }^{-\left(
\lambda_{i}+n-i\right)  }}=\sum_{w\in\mathfrak{S}_{n}}\left(  -1\right)
^{w}\prod_{i\in\left\{  1,2,\ldots,n\right\}  }x_{w\left(  i\right)
}^{-\left(  \lambda_{i}+n-i\right)  }\nonumber\\
&  =\sum_{w\in\mathfrak{S}_{n}}\underbrace{\left(  -1\right)  ^{w\circ w_{0}}%
}_{=\left(  -1\right)  ^{w}\left(  -1\right)  ^{w_{0}}}\prod_{i\in\left\{
1,2,\ldots,n\right\}  }\underbrace{x_{\left(  w\circ w_{0}\right)  \left(
i\right)  }^{-\left(  \lambda_{i}+n-i\right)  }}_{\substack{=x_{w\left(
n+1-i\right)  }^{-\lambda_{i}-n+i}\\\text{(since }-\left(  \lambda
_{i}+n-i\right)  =-\lambda_{i}-n+i\\\text{and }\left(  w\circ w_{0}\right)
\left(  i\right)  =w\left(  w_{0}\left(  i\right)  \right)  =w\left(
n+1-i\right)  \\\text{(because }w_{0}\left(  i\right)  =n+1-i\text{))}%
}}\nonumber\\
&  \ \ \ \ \ \ \ \ \ \ \left(
\begin{array}
[c]{c}%
\text{here, we have substituted }w\circ w_{0}\text{ for }w\text{ in the
sum,}\\
\text{since the map }\mathfrak{S}_{n}\rightarrow\mathfrak{S}_{n},\ w\mapsto
w\circ w_{0}\text{ is a bijection}%
\end{array}
\right) \nonumber\\
&  =\sum_{w\in\mathfrak{S}_{n}}\left(  -1\right)  ^{w}\left(  -1\right)
^{w_{0}}\prod_{i\in\left\{  1,2,\ldots,n\right\}  }x_{w\left(  n+1-i\right)
}^{-\lambda_{i}-n+i}\nonumber\\
&  =\left(  -1\right)  ^{w_{0}}\sum_{w\in\mathfrak{S}_{n}}\left(  -1\right)
^{w}\underbrace{\prod_{i\in\left\{  1,2,\ldots,n\right\}  }x_{w\left(
n+1-i\right)  }^{-\lambda_{i}-n+i}}_{\substack{=\prod_{i\in\left\{
1,2,\ldots,n\right\}  }x_{w\left(  n+1-\left(  n+1-i\right)  \right)
}^{-\lambda_{n+1-i}-n+\left(  n+1-i\right)  }\\\text{(here, we have
substituted }n+1-i\text{ for }i\\\text{in the product)}}}\nonumber\\
&  =\left(  -1\right)  ^{w_{0}}\sum_{w\in\mathfrak{S}_{n}}\left(  -1\right)
^{w}\prod_{i\in\left\{  1,2,\ldots,n\right\}  }\underbrace{x_{w\left(
n+1-\left(  n+1-i\right)  \right)  }^{-\lambda_{n+1-i}-n+\left(  n+1-i\right)
}}_{\substack{=x_{w\left(  i\right)  }^{-\lambda_{n+1-i}+1-i}\\\text{(since
}-\lambda_{n+1-i}-n+\left(  n+1-i\right)  =-\lambda_{n+1-i}+1-i\\\text{and
}n+1-\left(  n+1-i\right)  =i\text{)}}}\nonumber\\
&  =\left(  -1\right)  ^{w_{0}}\underbrace{\sum_{w\in\mathfrak{S}_{n}}\left(
-1\right)  ^{w}\prod_{i\in\left\{  1,2,\ldots,n\right\}  }x_{w\left(
i\right)  }^{-\lambda_{n+1-i}+1-i}}_{\substack{=a_{\mu+\rho}\\\text{(by
(\ref{pf.lem.alt.inverses.ampr}))}}}=\left(  -1\right)  ^{w_{0}}%
\underbrace{a_{\mu+\rho}}_{\substack{=x_{\Pi}^{1-n}a_{\lambda^{\vee}+\rho
}\\\text{(by (\ref{pf.lem.alt.inverses.0}))}}}\nonumber\\
&  =\left(  -1\right)  ^{w_{0}}x_{\Pi}^{1-n}a_{\lambda^{\vee}+\rho}.
\label{pf.lem.alt.inverses.a1}%
\end{align}

On the other hand, let us denote the snake $\left(  \underbrace{0,0,\ldots
,0}_{n\text{ times}}\right)  \in\mathbb{Z}^{n}$ by $\varnothing$; note that it
satisfies $\varnothing^{\vee}=\left(  \underbrace{-0,-0,\ldots,-0}_{n\text{
times}}\right)  =\left(  \underbrace{0,0,\ldots,0}_{n\text{ times}}\right)
=\varnothing$. We have proved (\ref{pf.lem.alt.inverses.a1}) for any snake
$\lambda$; thus, we can apply (\ref{pf.lem.alt.inverses.a1}) to $\varnothing$
instead of $\lambda$. We thus obtain%
\[
a_{\varnothing+\rho}\left(  x_{1}^{-1},x_{2}^{-1},\ldots,x_{n}^{-1}\right)
=\left(  -1\right)  ^{w_{0}}x_{\Pi}^{1-n}a_{\varnothing^{\vee}+\rho}.
\]
In view of $\varnothing+\rho=\rho$ and $\underbrace{\varnothing^{\vee}%
}_{=\varnothing}+\rho=\varnothing+\rho=\rho$, we can rewrite this as%
\begin{equation}
a_{\rho}\left(  x_{1}^{-1},x_{2}^{-1},\ldots,x_{n}^{-1}\right)  =\left(
-1\right)  ^{w_{0}}x_{\Pi}^{1-n}a_{\rho}. \label{pf.lem.alt.inverses.a2}%
\end{equation}

The definition of $\overline{s}_{\lambda}$ yields $\overline{s}_{\lambda
}=\dfrac{a_{\lambda+\rho}}{a_{\rho}}$. Hence, $a_{\lambda+\rho}=\overline
{s}_{\lambda}a_{\rho}$. Substituting $x_{1}^{-1},x_{2}^{-1},\ldots,x_{n}^{-1}$
for $x_{1},x_{2},\ldots,x_{n}$ on both sides of this equality, we obtain%
\begin{align*}
a_{\lambda+\rho}\left(  x_{1}^{-1},x_{2}^{-1},\ldots,x_{n}^{-1}\right)   &
=\left(  \overline{s}_{\lambda}a_{\rho}\right)  \left(  x_{1}^{-1},x_{2}%
^{-1},\ldots,x_{n}^{-1}\right) \\
&  =\overline{s}_{\lambda}\left(  x_{1}^{-1},x_{2}^{-1},\ldots,x_{n}%
^{-1}\right)  \cdot\underbrace{a_{\rho}\left(  x_{1}^{-1},x_{2}^{-1}%
,\ldots,x_{n}^{-1}\right)  }_{\substack{=\left(  -1\right)  ^{w_{0}}x_{\Pi
}^{1-n}a_{\rho}\\\text{(by (\ref{pf.lem.alt.inverses.a2}))}}}\\
&  =\overline{s}_{\lambda}\left(  x_{1}^{-1},x_{2}^{-1},\ldots,x_{n}%
^{-1}\right)  \cdot\left(  -1\right)  ^{w_{0}}x_{\Pi}^{1-n}a_{\rho}.
\end{align*}
Comparing this with (\ref{pf.lem.alt.inverses.a1}), we find%
\begin{equation}
\overline{s}_{\lambda}\left(  x_{1}^{-1},x_{2}^{-1},\ldots,x_{n}^{-1}\right)
\cdot\left(  -1\right)  ^{w_{0}}x_{\Pi}^{1-n}a_{\rho}=\left(  -1\right)
^{w_{0}}x_{\Pi}^{1-n}a_{\lambda^{\vee}+\rho}. \label{pf.lem.alt.inverses.at2}%
\end{equation}
The element $\left(  -1\right)  ^{w_{0}}x_{\Pi}^{1-n}$ of $\mathcal{L}$ is
invertible (since $x_{\Pi}$ is invertible), and thus we can cancel it from the
equality (\ref{pf.lem.alt.inverses.at2}). As a result, we obtain%
\begin{equation}
\overline{s}_{\lambda}\left(  x_{1}^{-1},x_{2}^{-1},\ldots,x_{n}^{-1}\right)
\cdot a_{\rho}=a_{\lambda^{\vee}+\rho}. \label{pf.lem.alt.inverses.at3}%
\end{equation}
But the definition of $\overline{s}_{\lambda^{\vee}}$ yields%
\[
\overline{s}_{\lambda^{\vee}}=\dfrac{a_{\lambda^{\vee}+\rho}}{a_{\rho}%
}=\overline{s}_{\lambda}\left(  x_{1}^{-1},x_{2}^{-1},\ldots,x_{n}%
^{-1}\right)
\]
(by (\ref{pf.lem.alt.inverses.at3})). This proves Lemma \ref{lem.alt.inverses}.
\end{proof}

\subsection{$h_{k}^{+}$, $h_{k}^{-}$ and the Pieri rule}

\begin{definition}
\label{def.hkpm}Let $k\in\mathbb{Z}$. Then, we define two Laurent polynomials
$h_{k}^{+}\in\mathcal{L}$ and $h_{k}^{-}\in\mathcal{L}$ by%
\begin{align*}
h_{k}^{+}  &  =h_{k}\left(  x_{1},x_{2},\ldots,x_{n}\right)
\ \ \ \ \ \ \ \ \ \ \text{and}\\
h_{k}^{-}  &  =h_{k}\left(  x_{1}^{-1},x_{2}^{-1},\ldots,x_{n}^{-1}\right)  .
\end{align*}

\end{definition}

Note that if $k\in\mathbb{Z}$ is negative, then $h_{k}^{+}=\underbrace{h_{k}%
}_{=0}\left(  x_{1},x_{2},\ldots,x_{n}\right)  =0$ and $h_{k}^{-}=0$ (similarly).

We begin by describing $h_{k}^{+}$ as a Schur Laurent polynomial:

\begin{lemma}
\label{lem.alt.h}Let $k\in\mathbb{N}$. Then, the partition $\left(  k\right)
$ is a nonnegative snake (when regarded as the $n$-tuple $\left(
k,0,0,\ldots,0\right)  $), and satisfies%
\[
\overline{s}_{\left(  k\right)  }=h_{k}^{+}.
\]

\end{lemma}

\begin{proof}
[Proof of Lemma \ref{lem.alt.h}.]The partition $\left(  k\right)  $ has length
$\leq n$ (since it has length $\leq1$, but we have $n\geq1$). Thus, it is a
nonnegative snake (since every partition having length $\leq n$ is a
nonnegative snake) and belongs to $\operatorname*{Par}\left[  n\right]  $.
Hence, Lemma \ref{lem.alt.sbar} (applied to $\lambda=\left(  k\right)  $)
yields $\overline{s}_{\left(  k\right)  }=s_{\left(  k\right)  }\left(
x_{1},x_{2},\ldots,x_{n}\right)  =h_{k}\left(  x_{1},x_{2},\ldots
,x_{n}\right)  $ (since $s_{\left(  k\right)  }=h_{k}$). Comparing this with
$h_{k}^{+}=h_{k}\left(  x_{1},x_{2},\ldots,x_{n}\right)  $, we obtain
$\overline{s}_{\left(  k\right)  }=h_{k}^{+}$. This proves Lemma
\ref{lem.alt.h}.
\end{proof}

Next, we need to know what happens when a Schur Laurent polynomial
$\overline{s}_{\lambda}$ is multiplied by some $h_{k}^{+}$. The answer to this
question is classically given by the \textit{first Pieri rule}; we shall state
it in a form that will be most convenient to us. To do so, we introduce some
more notation:

\begin{definition}
\label{def.size-snake}Let $\lambda\in\mathbb{Z}^{n}$. Then, we define the
\emph{size} $\left\vert \lambda\right\vert $ of $\lambda$ to be the integer
$\lambda_{1}+\lambda_{2}+\cdots+\lambda_{n}$.
\end{definition}

\begin{definition}
\label{def.horstr}Let $\lambda,\mu\in\mathbb{Z}^{n}$. Then, we write that
$\mu\rightharpoonup\lambda$ if and only if we have%
\begin{equation}
\mu_{1}\geq\lambda_{1}\geq\mu_{2}\geq\lambda_{2}\geq\cdots\geq\mu_{n}%
\geq\lambda_{n}. \label{eq.def.horstr.chain}%
\end{equation}
In other words, we write that $\mu\rightharpoonup\lambda$ if and only if we
have%
\begin{align*}
&  \left(  \mu_{i}\geq\lambda_{i}\text{ for each }i\in\left\{  1,2,\ldots
,n\right\}  \right)  \ \ \ \ \ \ \ \ \ \ \text{and}\\
&  \left(  \lambda_{i}\geq\mu_{i+1}\text{ for each }i\in\left\{
1,2,\ldots,n-1\right\}  \right)  .
\end{align*}

\end{definition}

The following properties of the sizes of $n$-tuples are obvious:

\begin{proposition}
\label{prop.size-snake.easies}\ \ 

\begin{enumerate}
\item[\textbf{(a)}] If $\lambda,\mu\in\mathbb{Z}^{n}$, then $\left\vert
\lambda+\mu\right\vert =\left\vert \lambda\right\vert +\left\vert
\mu\right\vert $.

\item[\textbf{(b)}] If $\lambda\in\mathbb{Z}^{n}$ and $d\in\mathbb{Z}$, then
$\left\vert \lambda+d\right\vert =\left\vert \lambda\right\vert +nd$.

\item[\textbf{(c)}] If $\lambda\in\mathbb{Z}^{n}$, then $\left\vert
\lambda^{\vee}\right\vert =-\left\vert \lambda\right\vert $.
\end{enumerate}
\end{proposition}

The relation $\rightharpoonup$ defined in Definition \ref{def.horstr} has the
following simple properties:

\begin{proposition}
\label{prop.horstr.easies}Let $\lambda,\mu\in\mathbb{Z}^{n}$.

\begin{enumerate}
\item[\textbf{(a)}] If $\mu\rightharpoonup\lambda$, then both $\lambda$ and
$\mu$ are snakes.

\item[\textbf{(b)}] We have $\mu\rightharpoonup\lambda$ if and only if
$\lambda^{\vee}\rightharpoonup\mu^{\vee}$.

\item[\textbf{(c)}] Let $d\in\mathbb{Z}$. Then, we have $\mu\rightharpoonup
\lambda$ if and only if $\mu+d\rightharpoonup\lambda+d$.
\end{enumerate}
\end{proposition}

\begin{proof}
[Proof of Proposition \ref{prop.horstr.easies}.]\textbf{(a)} Assume that
$\mu\rightharpoonup\lambda$. Thus, the chain of inequalities
(\ref{eq.def.horstr.chain}) holds (by Definition \ref{def.horstr}). But this
chain of inequalities implies both $\mu_{1}\geq\mu_{2}\geq\cdots\geq\mu_{n}$
and $\lambda_{1}\geq\lambda_{2}\geq\cdots\geq\lambda_{n}$. Thus, $\mu$ and
$\lambda$ are snakes. This proves Proposition \ref{prop.horstr.easies}
\textbf{(a)}.

\textbf{(b)} The definition of $\mu^{\vee}$ yields $\mu^{\vee}=\left(
-\mu_{n},-\mu_{n-1},\ldots,-\mu_{1}\right)  $. Similarly, $\lambda^{\vee
}=\left(  -\lambda_{n},-\lambda_{n-1},\ldots,-\lambda_{1}\right)  $. Hence, we
have $\lambda^{\vee}\rightharpoonup\mu^{\vee}$ if and only if we have%
\[
-\lambda_{n}\geq-\mu_{n}\geq-\lambda_{n-1}\geq-\mu_{n-1}\geq\cdots\geq
-\lambda_{1}\geq-\mu_{1}%
\]
(because of Definition \ref{def.horstr}).

Thus, we have the following chain of equivalences:%
\begin{align*}
\left(  \lambda^{\vee}\rightharpoonup\mu^{\vee}\right)  \  &
\Longleftrightarrow\ \left(  -\lambda_{n}\geq-\mu_{n}\geq-\lambda_{n-1}%
\geq-\mu_{n-1}\geq\cdots\geq-\lambda_{1}\geq-\mu_{1}\right) \\
&  \Longleftrightarrow\ \left(  \lambda_{n}\leq\mu_{n}\leq\lambda_{n-1}\leq
\mu_{n-1}\leq\cdots\leq\lambda_{1}\leq\mu_{1}\right) \\
&  \Longleftrightarrow\ \left(  \mu_{1}\geq\lambda_{1}\geq\mu_{2}\geq
\lambda_{2}\geq\cdots\geq\mu_{n}\geq\lambda_{n}\right) \\
&  \Longleftrightarrow\ \left(  \mu\rightharpoonup\lambda\right)
\ \ \ \ \ \ \ \ \ \ \left(  \text{by Definition \ref{def.horstr}}\right)  .
\end{align*}
In other words, we have $\mu\rightharpoonup\lambda$ if and only if
$\lambda^{\vee}\rightharpoonup\mu^{\vee}$. This proves Proposition
\ref{prop.horstr.easies} \textbf{(b)}.

\textbf{(c)} This follows easily from Definition \ref{def.horstr}.
\end{proof}

We can now state the Pieri rule in the form we need:

\begin{proposition}
\label{prop.alt.pieri1}Let $\lambda$ be a snake. Let $k\in\mathbb{Z}$. Then,%
\begin{equation}
h_{k}^{+}\cdot\overline{s}_{\lambda}=\sum_{\substack{\mu\text{ is a
snake;}\\\mu\rightharpoonup\lambda;\ \left\vert \mu\right\vert -\left\vert
\lambda\right\vert =k}}\overline{s}_{\mu}. \label{eq.prop.alt.pieri1.eq}%
\end{equation}

\end{proposition}

This can be proven directly using alternants; but let us give a proof based on
known theory:

\begin{proof}
[Proof of Proposition \ref{prop.alt.pieri1}.]We follow the same right-shift
strategy as we did in our proof of Lemma \ref{lem.alt.L}. Thus, our proof
shall consist of two steps:

\begin{statement}
\textit{Step 1:} We will prove Proposition \ref{prop.alt.pieri1} in the
particular case when $\lambda$ is nonnegative.
\end{statement}

\begin{statement}
\textit{Step 2:} We will use Lemma \ref{lem.alt.smove} to derive the general
case of Proposition \ref{prop.alt.pieri1} from this particular case.
\end{statement}

Here are the details of the two steps:

\textit{Step 1:} Let us prove that Proposition \ref{prop.alt.pieri1} holds in
the particular case when $\lambda$ is nonnegative.

Indeed, let us assume that $\lambda$ is nonnegative. We must prove the
equality (\ref{eq.prop.alt.pieri1.eq}).

If $k<0$, then both sides of this equality are $0$%
\ \ \ \ \footnote{\textit{Proof.} Assume that $k<0$. We must show that both
sides of (\ref{eq.prop.alt.pieri1.eq}) are $0$.
\par
Indeed, from $k<0$, we obtain $h_{k}=0$, thus $h_{k}^{+}=\underbrace{h_{k}%
}_{=0}\left(  x_{1},x_{2},\ldots,x_{n}\right)  =0$. Hence, $\underbrace{h_{k}%
^{+}}_{=0}\cdot\overline{s}_{\lambda}=0$. In other words, the left hand side
of (\ref{eq.prop.alt.pieri1.eq}) is $0$.
\par
It remains to show that the right hand side of (\ref{eq.prop.alt.pieri1.eq})
is $0$. This will follow if we can show that the sum on this right hand side
is empty, i.e., that there exists no snake $\mu$ such that $\mu\rightharpoonup
\lambda$ and $\left\vert \mu\right\vert -\left\vert \lambda\right\vert =k$. So
let us show this.
\par
Let $\mu$ be a snake such that $\mu\rightharpoonup\lambda$ and $\left\vert
\mu\right\vert -\left\vert \lambda\right\vert =k$. We shall derive a
contradiction.
\par
From $\mu\rightharpoonup\lambda$, we obtain $\mu_{1}\geq\lambda_{1}\geq\mu
_{2}\geq\lambda_{2}\geq\cdots\geq\mu_{n}\geq\lambda_{n}$ (by the definition of
\textquotedblleft$\mu\rightharpoonup\lambda$\textquotedblright). Hence,
$\mu_{i}\geq\lambda_{i}$ for each $i\in\left\{  1,2,\ldots,n\right\}  $. Thus,
$\sum_{i=1}^{n}\underbrace{\mu_{i}}_{\geq\lambda_{i}}\geq\sum_{i=1}^{n}%
\lambda_{i}$. In other words, $\left\vert \mu\right\vert \geq\left\vert
\lambda\right\vert $ (since $\left\vert \mu\right\vert =\mu_{1}+\mu_{2}%
+\cdots+\mu_{n}=\sum_{i=1}^{n}\mu_{i}$ and similarly $\left\vert
\lambda\right\vert =\sum_{i=1}^{n}\lambda_{i}$). Hence, $\left\vert
\mu\right\vert -\left\vert \lambda\right\vert \geq0$. But this contradicts
$\left\vert \mu\right\vert -\left\vert \lambda\right\vert =k<0$.
\par
Forget that we fixed $\mu$. We thus have found a contradiction whenever $\mu$
is a snake such that $\mu\rightharpoonup\lambda$ and $\left\vert
\mu\right\vert -\left\vert \lambda\right\vert =k$. Hence, there exists no such
snake $\mu$. In other words, the sum on the right hand side of
(\ref{eq.prop.alt.pieri1.eq}) is empty. Hence, the right hand side of
(\ref{eq.prop.alt.pieri1.eq}) is $0$. This completes our proof.}. Thus, the
equality (\ref{eq.prop.alt.pieri1.eq}) holds if $k<0$. Therefore, for the rest
of Step 1, we WLOG assume that $k\geq0$. In other words, $k\in\mathbb{N}$.

Note that $\lambda$ is a partition of length $\leq n$ (since $\lambda$ is a
nonnegative snake). In other words, $\lambda\in\operatorname*{Par}\left[
n\right]  $.

\begin{verlong}
We note that if $\mu$ is a partition of length $\leq n$, then we have defined
$\left\vert \mu\right\vert $ in two different ways: On the one hand,
$\left\vert \mu\right\vert $ was defined as the infinite sum $\mu_{1}+\mu
_{2}+\mu_{3}+\cdots$ (in the definition of the size of a partition); on the
other hand, $\left\vert \mu\right\vert $ was defined as the finite sum
$\mu_{1}+\mu_{2}+\cdots+\mu_{n}$ (because we can regard $\mu$ as a snake, and
then interpret $\left\vert \mu\right\vert $ according to Definition
\ref{def.size-snake}). Fortunately, these two definitions do not clash, since
the infinite sum $\mu_{1}+\mu_{2}+\mu_{3}+\cdots$ equals the finite sum
$\mu_{1}+\mu_{2}+\cdots+\mu_{n}$ whenever $\mu$ is a partition of length $\leq
n$. (In fact, if $\mu$ is a partition of length $\leq n$, then
\begin{align*}
&  \mu_{1}+\mu_{2}+\mu_{3}+\cdots\\
&  =\left(  \mu_{1}+\mu_{2}+\cdots+\mu_{n}\right)  +\underbrace{\left(
\mu_{n+1}+\mu_{n+2}+\mu_{n+3}+\cdots\right)  }_{\substack{=0+0+0+\cdots
\\\text{(since all of }\mu_{n+1},\mu_{n+2},\mu_{n+3},\ldots\text{ are
}0\\\text{(because }\mu\text{ has length }\leq n\text{))}}}\\
&  =\left(  \mu_{1}+\mu_{2}+\cdots+\mu_{n}\right)  +\underbrace{\left(
0+0+0+\cdots\right)  }_{=0}=\mu_{1}+\mu_{2}+\cdots+\mu_{n}.
\end{align*}
Thus, the infinite sum $\mu_{1}+\mu_{2}+\mu_{3}+\cdots$ equals the finite sum
$\mu_{1}+\mu_{2}+\cdots+\mu_{n}$.)
\end{verlong}

We will use some standard notations concerning partitions. Specifically:

\begin{itemize}
\item If $\alpha=\left(  \alpha_{1},\alpha_{2},\alpha_{3},\ldots\right)  $ and
$\beta=\left(  \beta_{1},\beta_{2},\beta_{3},\ldots\right)  $ are two
partitions, then we will write $\alpha\subseteq\beta$ if and only if each
$i\in\left\{  1,2,3,\ldots\right\}  $ satisfies $\alpha_{i}\leq\beta_{i}$.
(This is precisely the definition of $\alpha\subseteq\beta$ given in
\cite[Definition 2.3.1]{GriRei}.)

\item If $\alpha=\left(  \alpha_{1},\alpha_{2},\alpha_{3},\ldots\right)  $ and
$\beta=\left(  \beta_{1},\beta_{2},\beta_{3},\ldots\right)  $ are two
partitions, then we say that $\alpha/\beta$ \emph{is a horizontal strip} if
they satisfy
\[
\beta\subseteq\alpha\text{ and }\left(  \text{every }i\in\left\{
1,2,3,\ldots\right\}  \text{ satisfies }\beta_{i}\geq\alpha_{i+1}\right)  .
\]
(This is not literally how a \textquotedblleft horizontal
strip\textquotedblright\ is defined in \cite{GriRei}, but it is equivalent to
that definition; the equivalence follows from \cite[Exercise 2.7.5(a)]{GriRei}.)

\item If $\alpha$ and $\beta$ are two partitions, and if $k\in\mathbb{N}$,
then we say that $\alpha/\beta$ \emph{is a horizontal }$k$\emph{-strip} if
$\alpha/\beta$ is a horizontal strip and we have $\left\vert \alpha\right\vert
-\left\vert \beta\right\vert =k$. (This is equivalent to the definition of a
\textquotedblleft horizontal $k$-strip\textquotedblright\ in \cite[\S 2.7]%
{GriRei}).
\end{itemize}

We note the following claim:

\begin{statement}
\textit{Claim 1:} We have%
\begin{align*}
&  \left\{  \text{partitions }\mu\in\operatorname*{Par}\left[  n\right]
\text{ such that }\mu/\lambda\text{ is a horizontal }k\text{-strip}\right\} \\
&  =\left\{  \text{snakes }\mu\text{ such that }\mu\rightharpoonup
\lambda\text{ and}\ \left\vert \mu\right\vert -\left\vert \lambda\right\vert
=k\right\}  .
\end{align*}

\end{statement}

\begin{vershort}
[\textit{Proof of Claim 1:} This is an exercise in unraveling definitions and
recalling that partitions in $\operatorname*{Par}\left[  n\right]  $ are the
same as nonnegative snakes. We leave the details to the reader (who can also
find them expanded in the detailed version \cite{verlong} of this paper).]
\end{vershort}

\begin{verlong}
[\textit{Proof of Claim 1:} Let us notice that $\lambda_{n}\geq0$ (since the
snake $\lambda$ is nonnegative).

Any snake $\mu$ satisfying $\mu\rightharpoonup\lambda$ must be
nonnegative\footnote{\textit{Proof.} Let $\mu$ be a snake satisfying
$\mu\rightharpoonup\lambda$. Thus, we have the chain of inequalities $\mu
_{1}\geq\lambda_{1}\geq\mu_{2}\geq\lambda_{2}\geq\cdots\geq\mu_{n}\geq
\lambda_{n}$ (since $\mu\rightharpoonup\lambda$ was defined to be equivalent
to this chain of inequalities). Hence, in particular, we have $\mu_{i}%
\geq\lambda_{i}$ for each $i\in\left\{  1,2,\ldots,n\right\}  $. Thus, for
each $i\in\left\{  1,2,\ldots,n\right\}  $, we have $\mu_{i}\geq\lambda
_{i}\geq0$ (since $\lambda$ is nonnegative). In other words, $\mu$ is
nonnegative. Qed.}. Hence,%
\begin{align}
&  \left\{  \text{snakes }\mu\text{ such that }\mu\rightharpoonup\lambda\text{
and}\ \left\vert \mu\right\vert -\left\vert \lambda\right\vert =k\right\}
\nonumber\\
&  =\left\{  \text{nonnegative snakes }\mu\text{ such that }\mu\rightharpoonup
\lambda\text{ and}\ \left\vert \mu\right\vert -\left\vert \lambda\right\vert
=k\right\} \nonumber\\
&  =\left\{  \text{partitions }\mu\text{ of length }\leq n\text{ such that
}\mu\rightharpoonup\lambda\text{ and}\ \left\vert \mu\right\vert -\left\vert
\lambda\right\vert =k\right\} \nonumber\\
&  \ \ \ \ \ \ \ \ \ \ \left(  \text{since the nonnegative snakes are
precisely the partitions of length }\leq n\right) \nonumber\\
&  =\left\{  \text{partitions }\mu\in\operatorname*{Par}\left[  n\right]
\text{ such that }\mu\rightharpoonup\lambda\text{ and}\ \left\vert
\mu\right\vert -\left\vert \lambda\right\vert =k\right\}
\label{pf.prop.alt.pieri1.c1.pf.1}\\
&  \ \ \ \ \ \ \ \ \ \ \left(  \text{since the partitions of length }\leq
n\text{ are precisely the partitions }\mu\in\operatorname*{Par}\left[
n\right]  \right)  .\nonumber
\end{align}

However, if $\mu\in\operatorname*{Par}\left[  n\right]  $ is a partition, then
the statement \textquotedblleft$\mu\rightharpoonup\lambda$\textquotedblright%
\ is equivalent to the statement \textquotedblleft$\mu/\lambda$ is a
horizontal strip\textquotedblright\ \ \footnote{\textit{Proof.} Let $\mu
\in\operatorname*{Par}\left[  n\right]  $ be a partition. Let us write both
partitions $\mu\in\operatorname*{Par}\left[  n\right]  $ and $\lambda
\in\operatorname*{Par}\left[  n\right]  $ as infinite sequences%
\[
\mu=\left(  \mu_{1},\mu_{2},\mu_{3},\ldots\right)  \text{ and }\lambda=\left(
\lambda_{1},\lambda_{2},\lambda_{3},\ldots\right)  .
\]
Then, every integer $i>n$ satisfies $\mu_{i}=0$ (since $\mu\in
\operatorname*{Par}\left[  n\right]  $) and $\lambda_{i}=0$ (since $\lambda
\in\operatorname*{Par}\left[  n\right]  $). In other words, all the numbers
$\mu_{n+1},\mu_{n+2},\mu_{n+3},\ldots$ as well as $\lambda_{n+1},\lambda
_{n+2},\lambda_{n+3},\ldots$ equal $0$.
\par
Now, we have the following chain of equivalences:%
\begin{align*}
&  \ \left(  \mu/\lambda\text{ is a horizontal strip}\right) \\
&  \Longleftrightarrow\ \left(  \lambda\subseteq\mu\text{ and }\left(
\text{every }i\in\left\{  1,2,3,\ldots\right\}  \text{ satisfies }\lambda
_{i}\geq\mu_{i+1}\right)  \right) \\
&  \ \ \ \ \ \ \ \ \ \ \ \ \ \ \ \ \ \ \ \ \left(  \text{by the definition of
a \textquotedblleft horizontal strip\textquotedblright}\right) \\
&  \Longleftrightarrow\ \underbrace{\left(  \lambda\subseteq\mu\right)
}_{\substack{\Longleftrightarrow\ \left(  \text{each }i\in\left\{
1,2,3,\ldots\right\}  \text{ satisfies }\lambda_{i}\leq\mu_{i}\right)
\\\text{(by the definition of \textquotedblleft}\lambda\subseteq
\mu\text{\textquotedblright)}}}\wedge\left(  \text{every }i\in\left\{
1,2,3,\ldots\right\}  \text{ satisfies }\lambda_{i}\geq\mu_{i+1}\right) \\
&  \Longleftrightarrow\ \left(  \text{each }i\in\left\{  1,2,3,\ldots\right\}
\text{ satisfies }\underbrace{\lambda_{i}\leq\mu_{i}}_{\Longleftrightarrow
\ \left(  \mu_{i}\geq\lambda_{i}\right)  }\right)  \wedge\left(  \text{every
}i\in\left\{  1,2,3,\ldots\right\}  \text{ satisfies }\lambda_{i}\geq\mu
_{i+1}\right) \\
&  \Longleftrightarrow\ \left(  \text{each }i\in\left\{  1,2,3,\ldots\right\}
\text{ satisfies }\mu_{i}\geq\lambda_{i}\right)  \wedge\left(  \text{every
}i\in\left\{  1,2,3,\ldots\right\}  \text{ satisfies }\lambda_{i}\geq\mu
_{i+1}\right) \\
&  \Longleftrightarrow\ \left(  \mu_{1}\geq\lambda_{1}\geq\mu_{2}\geq
\lambda_{2}\geq\mu_{3}\geq\lambda_{3}\geq\cdots\right) \\
&  \Longleftrightarrow\ \left(  \mu_{1}\geq\lambda_{1}\geq\mu_{2}\geq
\lambda_{2}\geq\cdots\geq\mu_{n}\geq\lambda_{n}\right)  \wedge
\underbrace{\left(  \lambda_{n}\geq\mu_{n+1}\geq\lambda_{n+1}\geq\mu_{n+2}%
\geq\lambda_{n+2}\geq\mu_{n+3}\geq\lambda_{n+3}\geq\cdots\right)
}_{\substack{\Longleftrightarrow\ \left(  \lambda_{n}\geq0\geq0\geq0\geq
0\geq0\geq0\geq\cdots\right)  \\\text{(since all the numbers }\mu_{n+1}%
,\mu_{n+2},\mu_{n+3},\ldots\\\text{as well as }\lambda_{n+1},\lambda
_{n+2},\lambda_{n+3},\ldots\text{ equal }0\text{)}}}\\
&  \Longleftrightarrow\ \left(  \mu_{1}\geq\lambda_{1}\geq\mu_{2}\geq
\lambda_{2}\geq\cdots\geq\mu_{n}\geq\lambda_{n}\right)  \wedge
\underbrace{\left(  \lambda_{n}\geq0\geq0\geq0\geq0\geq0\geq0\geq
\cdots\right)  }_{\substack{\Longleftrightarrow\ \left(  \lambda_{n}%
\geq0\right)  \ \Longleftrightarrow\ \left(  \text{true}\right)
\\\text{(since }\lambda_{n}\geq0\text{)}}}\\
&  \Longleftrightarrow\ \left(  \mu_{1}\geq\lambda_{1}\geq\mu_{2}\geq
\lambda_{2}\geq\cdots\geq\mu_{n}\geq\lambda_{n}\right)  \ \Longleftrightarrow
\ \left(  \mu\rightharpoonup\lambda\right)
\end{align*}
(since the statement \textquotedblleft$\mu\rightharpoonup\lambda
$\textquotedblright\ was defined to mean \textquotedblleft$\mu_{1}\geq
\lambda_{1}\geq\mu_{2}\geq\lambda_{2}\geq\cdots\geq\mu_{n}\geq\lambda_{n}%
$\textquotedblright). In other words, the statement \textquotedblleft%
$\mu\rightharpoonup\lambda$\textquotedblright\ is equivalent to the statement
\textquotedblleft$\mu/\lambda$ is a horizontal strip\textquotedblright. Qed.}.
Hence,
\begin{align*}
&  \left\{  \text{partitions }\mu\in\operatorname*{Par}\left[  n\right]
\text{ such that }\mu\rightharpoonup\lambda\text{ and}\ \left\vert
\mu\right\vert -\left\vert \lambda\right\vert =k\right\} \\
&  =\left\{  \text{partitions }\mu\in\operatorname*{Par}\left[  n\right]
\text{ such that }\mu/\lambda\text{ is a horizontal strip and}\ \left\vert
\mu\right\vert -\left\vert \lambda\right\vert =k\right\} \\
&  =\left\{  \text{partitions }\mu\in\operatorname*{Par}\left[  n\right]
\text{ such that }\mu/\lambda\text{ is a horizontal }k\text{-strip}\right\}
\end{align*}
(because the statement \textquotedblleft$\mu/\lambda$ is a horizontal
$k$-strip\textquotedblright\ was defined to mean \textquotedblleft$\mu
/\lambda$ is a horizontal strip and$\ \left\vert \mu\right\vert -\left\vert
\lambda\right\vert =k$\textquotedblright). Thus,
(\ref{pf.prop.alt.pieri1.c1.pf.1}) becomes%
\begin{align*}
&  \left\{  \text{snakes }\mu\text{ such that }\mu\rightharpoonup\lambda\text{
and}\ \left\vert \mu\right\vert -\left\vert \lambda\right\vert =k\right\} \\
&  =\left\{  \text{partitions }\mu\in\operatorname*{Par}\left[  n\right]
\text{ such that }\mu\rightharpoonup\lambda\text{ and}\ \left\vert
\mu\right\vert -\left\vert \lambda\right\vert =k\right\} \\
&  =\left\{  \text{partitions }\mu\in\operatorname*{Par}\left[  n\right]
\text{ such that }\mu/\lambda\text{ is a horizontal }k\text{-strip}\right\}  .
\end{align*}
This proves Claim 1.]
\end{verlong}

From the first Pieri rule (\cite[(2.7.1)]{GriRei}\footnote{This also appears
in \cite[Theorem 5.3]{MenRem15}, in \cite[Theorem 7.15.7]{Stanley-EC2} and in
\cite[Theorem 9.3]{Egge19}.}, applied to $k$ instead of $n$), we obtain%
\[
s_{\lambda}h_{k}=\sum_{\substack{\lambda^{+}\in\operatorname*{Par}%
;\\\lambda^{+}/\lambda\text{ is a horizontal }k\text{-strip}}}s_{\lambda^{+}%
}=\sum_{\substack{\mu\in\operatorname*{Par};\\\mu/\lambda\text{ is a
horizontal }k\text{-strip}}}s_{\mu}%
\]
(here, we have renamed the summation index $\lambda^{+}$ as $\mu$).

\begin{vershort}
Evaluating both sides of this equality at $x_{1},x_{2},\ldots,x_{n}$, we find%
\begin{align*}
&  \left(  s_{\lambda}h_{k}\right)  \left(  x_{1},x_{2},\ldots,x_{n}\right) \\
&  =\sum_{\substack{\mu\in\operatorname*{Par};\\\mu/\lambda\text{ is a
horizontal }k\text{-strip}}}s_{\mu}\left(  x_{1},x_{2},\ldots,x_{n}\right) \\
&  =\sum_{\substack{\mu\in\operatorname*{Par};\\\mu/\lambda\text{ is a
horizontal }k\text{-strip;}\\\mu\text{ has length }\leq n}}s_{\mu}\left(
x_{1},x_{2},\ldots,x_{n}\right)  +\sum_{\substack{\mu\in\operatorname*{Par}%
;\\\mu/\lambda\text{ is a horizontal }k\text{-strip;}\\\mu\text{ has length
}>n}}\ \ \underbrace{s_{\mu}\left(  x_{1},x_{2},\ldots,x_{n}\right)
}_{\substack{=0\\\text{(by (\ref{pf.lem.alt.lrr.=0}),}\\\text{applied to }%
\mu\text{ instead of }\lambda\text{)}}}\\
&  =\sum_{\substack{\mu\in\operatorname*{Par};\\\mu/\lambda\text{ is a
horizontal }k\text{-strip;}\\\mu\text{ has length }\leq n}}s_{\mu}\left(
x_{1},x_{2},\ldots,x_{n}\right)  +\underbrace{\sum_{\substack{\mu
\in\operatorname*{Par};\\\mu/\lambda\text{ is a horizontal }k\text{-strip;}%
\\\mu\text{ has length }>n}}0}_{=0}\\
&  =\sum_{\substack{\mu\in\operatorname*{Par};\\\mu/\lambda\text{ is a
horizontal }k\text{-strip;}\\\mu\text{ has length }\leq n}}s_{\mu}\left(
x_{1},x_{2},\ldots,x_{n}\right)  .
\end{align*}

\end{vershort}

\begin{verlong}
Evaluating both sides of this equality at $x_{1},x_{2},\ldots,x_{n}$, we find%
\begin{align*}
&  \left(  s_{\lambda}h_{k}\right)  \left(  x_{1},x_{2},\ldots,x_{n}\right) \\
&  =\left(  \sum_{\substack{\mu\in\operatorname*{Par};\\\mu/\lambda\text{ is a
horizontal }k\text{-strip}}}s_{\mu}\right)  \left(  x_{1},x_{2},\ldots
,x_{n}\right) \\
&  =\sum_{\substack{\mu\in\operatorname*{Par};\\\mu/\lambda\text{ is a
horizontal }k\text{-strip}}}s_{\mu}\left(  x_{1},x_{2},\ldots,x_{n}\right) \\
&  =\sum_{\substack{\mu\in\operatorname*{Par};\\\mu/\lambda\text{ is a
horizontal }k\text{-strip;}\\\mu\text{ has length }\leq n}}s_{\mu}\left(
x_{1},x_{2},\ldots,x_{n}\right)  +\sum_{\substack{\mu\in\operatorname*{Par}%
;\\\mu/\lambda\text{ is a horizontal }k\text{-strip;}\\\mu\text{ has length
}>n}}\ \ \underbrace{s_{\mu}\left(  x_{1},x_{2},\ldots,x_{n}\right)
}_{\substack{=0\\\text{(by (\ref{pf.lem.alt.lrr.=0}),}\\\text{applied to }%
\mu\text{ instead of }\lambda\text{)}}}\\
&  \ \ \ \ \ \ \ \ \ \ \ \ \ \ \ \ \ \ \ \ \left(
\begin{array}
[c]{c}%
\text{since each }\mu\in\operatorname*{Par}\text{ either has length }\leq
n\text{ or has length }>n\\
\text{(but not both at the same time)}%
\end{array}
\right) \\
&  =\sum_{\substack{\mu\in\operatorname*{Par};\\\mu/\lambda\text{ is a
horizontal }k\text{-strip;}\\\mu\text{ has length }\leq n}}s_{\mu}\left(
x_{1},x_{2},\ldots,x_{n}\right)  +\underbrace{\sum_{\substack{\mu
\in\operatorname*{Par};\\\mu/\lambda\text{ is a horizontal }k\text{-strip;}%
\\\mu\text{ has length }>n}}0}_{=0}\\
&  =\sum_{\substack{\mu\in\operatorname*{Par};\\\mu/\lambda\text{ is a
horizontal }k\text{-strip;}\\\mu\text{ has length }\leq n}}s_{\mu}\left(
x_{1},x_{2},\ldots,x_{n}\right)  .
\end{align*}

\end{verlong}

\noindent In view of%
\begin{align*}
\left(  s_{\lambda}h_{k}\right)  \left(  x_{1},x_{2},\ldots,x_{n}\right)   &
=\underbrace{s_{\lambda}\left(  x_{1},x_{2},\ldots,x_{n}\right)
}_{\substack{=\overline{s}_{\lambda}\\\text{(by Lemma \ref{lem.alt.sbar})}%
}}\cdot\underbrace{h_{k}\left(  x_{1},x_{2},\ldots,x_{n}\right)
}_{\substack{=h_{k}^{+}\\\text{(since }h_{k}^{+}=h_{k}\left(  x_{1}%
,x_{2},\ldots,x_{n}\right)  \text{)}}}\\
&  =\overline{s}_{\lambda}\cdot h_{k}^{+}=h_{k}^{+}\cdot\overline{s}_{\lambda
},
\end{align*}
we can rewrite this as%
\begin{equation}
h_{k}^{+}\cdot\overline{s}_{\lambda}=\sum_{\substack{\mu\in\operatorname*{Par}%
;\\\mu/\lambda\text{ is a horizontal }k\text{-strip;}\\\mu\text{ has length
}\leq n}}s_{\mu}\left(  x_{1},x_{2},\ldots,x_{n}\right)  .
\label{pf.prop.alt.pieri1.s1.3}%
\end{equation}
But we have the following equality of summation signs:%
\begin{align*}
\sum_{\substack{\mu\in\operatorname*{Par};\\\mu/\lambda\text{ is a horizontal
}k\text{-strip;}\\\mu\text{ has length }\leq n}}  &  =\sum_{\substack{\mu
\in\operatorname*{Par};\\\mu\text{ has length }\leq n\text{;}\\\mu
/\lambda\text{ is a horizontal }k\text{-strip}}}=\sum_{\substack{\mu
\in\operatorname*{Par}\left[  n\right]  ;\\\mu/\lambda\text{ is a horizontal
}k\text{-strip}}}\\
&  \ \ \ \ \ \ \ \ \ \ \left(
\begin{array}
[c]{c}%
\text{since the partitions }\mu\in\operatorname*{Par}\text{ that have length
}\leq n\\
\text{are precisely the partitions }\mu\in\operatorname*{Par}\left[  n\right]
\end{array}
\right)  .
\end{align*}
Thus, we can rewrite (\ref{pf.prop.alt.pieri1.s1.3}) as%
\begin{align*}
h_{k}^{+}\cdot\overline{s}_{\lambda}  &  =\sum_{\substack{\mu\in
\operatorname*{Par}\left[  n\right]  ;\\\mu/\lambda\text{ is a horizontal
}k\text{-strip}}}\underbrace{s_{\mu}\left(  x_{1},x_{2},\ldots,x_{n}\right)
}_{\substack{=\overline{s}_{\mu}\\\text{(since Lemma \ref{lem.alt.sbar}%
}\\\text{(applied to }\mu\text{ instead of }\lambda\text{)}\\\text{yields
}\overline{s}_{\mu}=s_{\mu}\left(  x_{1},x_{2},\ldots,x_{n}\right)  \text{)}%
}}\\
&  =\underbrace{\sum_{\substack{\mu\in\operatorname*{Par}\left[  n\right]
;\\\mu/\lambda\text{ is a horizontal }k\text{-strip}}}}_{\substack{=\sum
_{\substack{\mu\text{ is a snake;}\\\mu\rightharpoonup\lambda;\ \left\vert
\mu\right\vert -\left\vert \lambda\right\vert =k}}\\\text{(by Claim 1)}%
}}\overline{s}_{\mu}=\sum_{\substack{\mu\text{ is a snake;}\\\mu
\rightharpoonup\lambda;\ \left\vert \mu\right\vert -\left\vert \lambda
\right\vert =k}}\overline{s}_{\mu}.
\end{align*}
This proves (\ref{eq.prop.alt.pieri1.eq}). Thus, Proposition
\ref{prop.alt.pieri1} is proved under the assumption that $\lambda$ is
nonnegative. This completes Step 1.

\textit{Step 2:} Let us now prove Proposition \ref{prop.alt.pieri1} in the
general case.

The snake $\lambda$ may or may not be nonnegative. However, there exists some
integer $d$ such that the snake $\lambda+d$ is nonnegative\footnote{Indeed,
this can be proved in the same way as it was proved during Step 2 of the proof
of Lemma \ref{lem.alt.L} above.}. Consider this $d$.

The map $\left\{  \text{snakes}\right\}  \rightarrow\left\{  \text{snakes}%
\right\}  ,\ \mu\mapsto\mu+d$ is a bijection. (Indeed, its inverse is the map
$\left\{  \text{snakes}\right\}  \rightarrow\left\{  \text{snakes}\right\}
,\ \mu\mapsto\mu-d$.) Moreover, every snake $\mu$ satisfies
\begin{align}
\underbrace{\left\vert \mu+d\right\vert }_{\substack{=\left\vert
\mu\right\vert +nd\\\text{(by Proposition \ref{prop.size-snake.easies}
\textbf{(b)})}}}-\underbrace{\left\vert \lambda+d\right\vert }%
_{\substack{=\left\vert \lambda\right\vert +nd\\\text{(by Proposition
\ref{prop.size-snake.easies} \textbf{(b)})}}}  &  =\left(  \left\vert
\mu\right\vert +nd\right)  -\left(  \left\vert \lambda\right\vert +nd\right)
\nonumber\\
&  =\left\vert \mu\right\vert -\left\vert \lambda\right\vert .
\label{pf.prop.alt.pieri1.s2.0}%
\end{align}

For any snake $\mu$, we have the logical equivalence%
\[
\left(  \mu\rightharpoonup\lambda\right)  \ \Longleftrightarrow\ \left(
\mu+d\rightharpoonup\lambda+d\right)
\]
(by Proposition \ref{prop.horstr.easies} \textbf{(c)}). In other words, for
any snake $\mu$, we have the logical equivalence%
\[
\left(  \mu+d\rightharpoonup\lambda+d\right)  \ \Longleftrightarrow\ \left(
\mu\rightharpoonup\lambda\right)  .
\]
Hence, we have the following equality of summation signs:%
\begin{equation}
\sum_{\substack{\mu\text{ is a snake;}\\\mu+d\rightharpoonup\lambda
+d;\ \left\vert \mu+d\right\vert -\left\vert \lambda+d\right\vert =k}%
}=\sum_{\substack{\mu\text{ is a snake;}\\\mu\rightharpoonup\lambda
;\ \left\vert \mu+d\right\vert -\left\vert \lambda+d\right\vert =k}%
}=\sum_{\substack{\mu\text{ is a snake;}\\\mu\rightharpoonup\lambda
;\ \left\vert \mu\right\vert -\left\vert \lambda\right\vert =k}}
\label{pf.prop.alt.pieri1.s2.sumeq}%
\end{equation}
(by (\ref{pf.prop.alt.pieri1.s2.0})).

The snake $\lambda+d$ is nonnegative; thus, we can apply Proposition
\ref{prop.alt.pieri1} to $\lambda+d$ instead of $\lambda$ (because in Step 1,
we have proved that Proposition \ref{prop.alt.pieri1} holds in the particular
case when $\lambda$ is nonnegative). Thus we conclude that
\begin{align*}
h_{k}^{+}\cdot\overline{s}_{\lambda+d}  &  =\sum_{\substack{\mu\text{ is a
snake;}\\\mu\rightharpoonup\lambda+d;\ \left\vert \mu\right\vert -\left\vert
\lambda+d\right\vert =k}}\overline{s}_{\mu}=\underbrace{\sum_{\substack{\mu
\text{ is a snake;}\\\mu+d\rightharpoonup\lambda+d;\ \left\vert \mu
+d\right\vert -\left\vert \lambda+d\right\vert =k}}}_{\substack{=\sum
_{\substack{\mu\text{ is a snake;}\\\mu\rightharpoonup\lambda;\ \left\vert
\mu\right\vert -\left\vert \lambda\right\vert =k}}\\\text{(by
(\ref{pf.prop.alt.pieri1.s2.sumeq}))}}}\underbrace{\overline{s}_{\mu+d}%
}_{\substack{=x_{\Pi}^{d}\overline{s}_{\mu}\\\text{(by Lemma
\ref{lem.alt.smove},}\\\text{applied to }\mu\text{ instead of }\lambda
\text{)}}}\\
&  \ \ \ \ \ \ \ \ \ \ \left(
\begin{array}
[c]{c}%
\text{here, we have substituted }\mu+d\text{ for }\mu\text{ in the sum,}\\
\text{since the map }\left\{  \text{snakes}\right\}  \rightarrow\left\{
\text{snakes}\right\}  ,\ \mu\mapsto\mu+d\\
\text{is a bijection}%
\end{array}
\right) \\
&  =\sum_{\substack{\mu\text{ is a snake;}\\\mu\rightharpoonup\lambda
;\ \left\vert \mu\right\vert -\left\vert \lambda\right\vert =k}}x_{\Pi}%
^{d}\overline{s}_{\mu}=x_{\Pi}^{d}\sum_{\substack{\mu\text{ is a snake;}%
\\\mu\rightharpoonup\lambda;\ \left\vert \mu\right\vert -\left\vert
\lambda\right\vert =k}}\overline{s}_{\mu}.
\end{align*}
Comparing this with%
\[
h_{k}^{+}\cdot\underbrace{\overline{s}_{\lambda+d}}_{\substack{=x_{\Pi}%
^{d}\overline{s}_{\lambda}\\\text{(by Lemma \ref{lem.alt.smove})}}}=h_{k}%
^{+}\cdot x_{\Pi}^{d}\overline{s}_{\lambda},
\]
we obtain%
\[
h_{k}^{+}\cdot x_{\Pi}^{d}\overline{s}_{\lambda}=x_{\Pi}^{d}\sum
_{\substack{\mu\text{ is a snake;}\\\mu\rightharpoonup\lambda;\ \left\vert
\mu\right\vert -\left\vert \lambda\right\vert =k}}\overline{s}_{\mu}.
\]
We can divide both sides of this equality by $x_{\Pi}^{d}$ (since $x_{\Pi}%
^{d}\in\mathcal{L}$ is invertible (because $x_{\Pi}\in\mathcal{L}$ is
invertible)), and thus obtain%
\[
h_{k}^{+}\cdot\overline{s}_{\lambda}=\sum_{\substack{\mu\text{ is a
snake;}\\\mu\rightharpoonup\lambda;\ \left\vert \mu\right\vert -\left\vert
\lambda\right\vert =k}}\overline{s}_{\mu}.
\]
This proves Proposition \ref{prop.alt.pieri1}. Thus, Step 2 is complete, and
Proposition \ref{prop.alt.pieri1} is proven.
\end{proof}

Using Lemma \ref{lem.alt.inverses}, we can \textquotedblleft turn Proposition
\ref{prop.alt.pieri1} upside down\textquotedblright, obtaining the following
analogous result for $h_{k}^{-}$ instead of $h_{k}^{+}$:

\begin{proposition}
\label{prop.alt.pieri2}Let $\lambda$ be a snake. Let $k\in\mathbb{Z}$. Then,%
\begin{equation}
h_{k}^{-}\cdot\overline{s}_{\lambda}=\sum_{\substack{\mu\text{ is a
snake;}\\\lambda\rightharpoonup\mu;\ \left\vert \lambda\right\vert -\left\vert
\mu\right\vert =k}}\overline{s}_{\mu}. \label{eq.prop.alt.pieri2.eq}%
\end{equation}

\end{proposition}

\begin{proof}
[Proof of Proposition \ref{prop.alt.pieri2}.]It is easy to see (from
Definition \ref{def.snakes} \textbf{(d)}) that $\left(  \lambda^{\vee}\right)
^{\vee}=\lambda$. Likewise, $\left(  \mu^{\vee}\right)  ^{\vee}=\mu$ for any
snake $\mu$. Hence, the map $\left\{  \text{snakes}\right\}  \rightarrow
\left\{  \text{snakes}\right\}  ,\ \mu\mapsto\mu^{\vee}$ (which is
well-defined because of Proposition \ref{prop.snakes.trivia} \textbf{(b)}) is
inverse to itself. Thus, this map is a bijection.

If $\mu$ is a snake, then
\begin{equation}
\text{we have }\mu\rightharpoonup\lambda\text{ if and only if }\lambda^{\vee
}\rightharpoonup\mu^{\vee} \label{pf.prop.alt.pieri2.1}%
\end{equation}
(by Proposition \ref{prop.horstr.easies} \textbf{(b)}). Moreover, if $\mu$ is
a snake, then%
\begin{align}
&  \underbrace{\left\vert \lambda^{\vee}\right\vert }_{\substack{=-\left\vert
\lambda\right\vert \\\text{(by Proposition \ref{prop.size-snake.easies}
\textbf{(c)})}}}-\underbrace{\left\vert \mu^{\vee}\right\vert }%
_{\substack{=-\left\vert \mu\right\vert \\\text{(by Proposition
\ref{prop.size-snake.easies} \textbf{(c)})}}}\nonumber\\
&  =\left(  -\left\vert \lambda\right\vert \right)  -\left(  -\left\vert
\mu\right\vert \right)  =\left\vert \mu\right\vert -\left\vert \lambda
\right\vert . \label{pf.prop.alt.pieri2.2}%
\end{align}

\begin{vershort}
But comparing the definitions of $h_{k}^{-}$ and $h_{k}^{+}$ easily yields%
\[
h_{k}^{-}=h_{k}^{+}\left(  x_{1}^{-1},x_{2}^{-1},\ldots,x_{n}^{-1}\right)  .
\]
Also, Lemma \ref{lem.alt.inverses} yields $\overline{s}_{\lambda^{\vee}%
}=\overline{s}_{\lambda}\left(  x_{1}^{-1},x_{2}^{-1},\ldots,x_{n}%
^{-1}\right)  $. Multiplying these two equalities, we obtain%
\begin{align*}
h_{k}^{-}\cdot\overline{s}_{\lambda^{\vee}}  &  =h_{k}^{+}\left(  x_{1}%
^{-1},x_{2}^{-1},\ldots,x_{n}^{-1}\right)  \cdot\overline{s}_{\lambda}\left(
x_{1}^{-1},x_{2}^{-1},\ldots,x_{n}^{-1}\right) \\
&  =\underbrace{\left(  h_{k}^{+}\cdot\overline{s}_{\lambda}\right)
}_{\substack{=\sum_{\substack{\mu\text{ is a snake;}\\\mu\rightharpoonup
\lambda;\ \left\vert \mu\right\vert -\left\vert \lambda\right\vert
=k}}\overline{s}_{\mu}\\\text{(by Proposition \ref{prop.alt.pieri1})}}}\left(
x_{1}^{-1},x_{2}^{-1},\ldots,x_{n}^{-1}\right) \\
&  =\underbrace{\sum_{\substack{\mu\text{ is a snake;}\\\mu\rightharpoonup
\lambda;\ \left\vert \mu\right\vert -\left\vert \lambda\right\vert =k}%
}}_{\substack{=\sum_{\substack{\mu\text{ is a snake;}\\\lambda^{\vee
}\rightharpoonup\mu^{\vee};\ \left\vert \mu\right\vert -\left\vert
\lambda\right\vert =k}}\\\text{(by (\ref{pf.prop.alt.pieri2.1}))}}%
}\overline{s}_{\mu}\left(  x_{1}^{-1},x_{2}^{-1},\ldots,x_{n}^{-1}\right) \\
&  =\sum_{\substack{\mu\text{ is a snake;}\\\lambda^{\vee}\rightharpoonup
\mu^{\vee};\ \left\vert \mu\right\vert -\left\vert \lambda\right\vert
=k}}\overline{s}_{\mu}\left(  x_{1}^{-1},x_{2}^{-1},\ldots,x_{n}^{-1}\right)
.
\end{align*}

\end{vershort}

\begin{verlong}
But we have
\begin{align*}
h_{k}^{-}  &  =h_{k}\left(  x_{1}^{-1},x_{2}^{-1},\ldots,x_{n}^{-1}\right)
=h_{k}^{+}\left(  x_{1}^{-1},x_{2}^{-1},\ldots,x_{n}^{-1}\right) \\
&  \ \ \ \ \ \ \ \ \ \ \left(
\begin{array}
[c]{c}%
\text{since }h_{k}^{+}=h_{k}\left(  x_{1},x_{2},\ldots,x_{n}\right)  \text{
and}\\
\text{thus }h_{k}^{+}\left(  x_{1}^{-1},x_{2}^{-1},\ldots,x_{n}^{-1}\right)
=h_{k}\left(  x_{1}^{-1},x_{2}^{-1},\ldots,x_{n}^{-1}\right)
\end{array}
\right)  .
\end{align*}
Also, Lemma \ref{lem.alt.inverses} yields $\overline{s}_{\lambda^{\vee}%
}=\overline{s}_{\lambda}\left(  x_{1}^{-1},x_{2}^{-1},\ldots,x_{n}%
^{-1}\right)  $. Multiplying these two equalities, we obtain%
\begin{align*}
h_{k}^{-}\cdot\overline{s}_{\lambda^{\vee}}  &  =h_{k}^{+}\left(  x_{1}%
^{-1},x_{2}^{-1},\ldots,x_{n}^{-1}\right)  \cdot\overline{s}_{\lambda}\left(
x_{1}^{-1},x_{2}^{-1},\ldots,x_{n}^{-1}\right) \\
&  =\underbrace{\left(  h_{k}^{+}\cdot\overline{s}_{\lambda}\right)
}_{\substack{=\sum_{\substack{\mu\text{ is a snake;}\\\mu\rightharpoonup
\lambda;\ \left\vert \mu\right\vert -\left\vert \lambda\right\vert
=k}}\overline{s}_{\mu}\\\text{(by Proposition \ref{prop.alt.pieri1})}}}\left(
x_{1}^{-1},x_{2}^{-1},\ldots,x_{n}^{-1}\right) \\
&  =\left(  \sum_{\substack{\mu\text{ is a snake;}\\\mu\rightharpoonup
\lambda;\ \left\vert \mu\right\vert -\left\vert \lambda\right\vert
=k}}\overline{s}_{\mu}\right)  \left(  x_{1}^{-1},x_{2}^{-1},\ldots,x_{n}%
^{-1}\right) \\
&  =\underbrace{\sum_{\substack{\mu\text{ is a snake;}\\\mu\rightharpoonup
\lambda;\ \left\vert \mu\right\vert -\left\vert \lambda\right\vert =k}%
}}_{\substack{=\sum_{\substack{\mu\text{ is a snake;}\\\lambda^{\vee
}\rightharpoonup\mu^{\vee};\ \left\vert \mu\right\vert -\left\vert
\lambda\right\vert =k}}\\\text{(since the statement \textquotedblleft}%
\mu\rightharpoonup\lambda\text{\textquotedblright}\\\text{is equivalent to
\textquotedblleft}\lambda^{\vee}\rightharpoonup\mu^{\vee}%
\text{\textquotedblright}\\\text{(by (\ref{pf.prop.alt.pieri2.1})))}%
}}\overline{s}_{\mu}\left(  x_{1}^{-1},x_{2}^{-1},\ldots,x_{n}^{-1}\right) \\
&  =\sum_{\substack{\mu\text{ is a snake;}\\\lambda^{\vee}\rightharpoonup
\mu^{\vee};\ \left\vert \mu\right\vert -\left\vert \lambda\right\vert
=k}}\overline{s}_{\mu}\left(  x_{1}^{-1},x_{2}^{-1},\ldots,x_{n}^{-1}\right)
.
\end{align*}

\end{verlong}

\noindent Comparing this with%
\begin{align*}
\sum_{\substack{\mu\text{ is a snake;}\\\lambda^{\vee}\rightharpoonup
\mu;\ \left\vert \lambda^{\vee}\right\vert -\left\vert \mu\right\vert
=k}}\overline{s}_{\mu}  &  =\underbrace{\sum_{\substack{\mu\text{ is a
snake;}\\\lambda^{\vee}\rightharpoonup\mu^{\vee};\ \left\vert \lambda^{\vee
}\right\vert -\left\vert \mu^{\vee}\right\vert =k}}}_{\substack{=\sum
_{\substack{\mu\text{ is a snake;}\\\lambda^{\vee}\rightharpoonup\mu^{\vee
};\ \left\vert \mu\right\vert -\left\vert \lambda\right\vert =k}}\\\text{(by
(\ref{pf.prop.alt.pieri2.2}))}}}\ \ \underbrace{\overline{s}_{\mu^{\vee}}%
}_{\substack{=\overline{s}_{\mu}\left(  x_{1}^{-1},x_{2}^{-1},\ldots
,x_{n}^{-1}\right)  \\\text{(by Lemma \ref{lem.alt.inverses},}\\\text{applied
to }\mu\text{ instead of }\lambda\text{)}}}\\
&  \ \ \ \ \ \ \ \ \ \ \left(
\begin{array}
[c]{c}%
\text{here, we have substituted }\mu^{\vee}\text{ for }\mu\text{ in the sum,
since}\\
\text{the map }\left\{  \text{snakes}\right\}  \rightarrow\left\{
\text{snakes}\right\}  ,\ \mu\mapsto\mu^{\vee}\text{ is a bijection}%
\end{array}
\right) \\
&  =\sum_{\substack{\mu\text{ is a snake;}\\\lambda^{\vee}\rightharpoonup
\mu^{\vee};\ \left\vert \mu\right\vert -\left\vert \lambda\right\vert
=k}}\overline{s}_{\mu}\left(  x_{1}^{-1},x_{2}^{-1},\ldots,x_{n}^{-1}\right)
,
\end{align*}
we obtain%
\[
h_{k}^{-}\cdot\overline{s}_{\lambda^{\vee}}=\sum_{\substack{\mu\text{ is a
snake;}\\\lambda^{\vee}\rightharpoonup\mu;\ \left\vert \lambda^{\vee
}\right\vert -\left\vert \mu\right\vert =k}}\overline{s}_{\mu}.
\]

We have proved this equality for any snake $\lambda$. Thus, we can apply it to
$\lambda^{\vee}$ instead of $\lambda$ (since Proposition
\ref{prop.snakes.trivia} \textbf{(b)} shows that $\lambda^{\vee}$ is a snake).
We thus obtain%
\[
h_{k}^{-}\cdot\overline{s}_{\left(  \lambda^{\vee}\right)  ^{\vee}}%
=\sum_{\substack{\mu\text{ is a snake;}\\\left(  \lambda^{\vee}\right)
^{\vee}\rightharpoonup\mu;\ \left\vert \left(  \lambda^{\vee}\right)  ^{\vee
}\right\vert -\left\vert \mu\right\vert =k}}\overline{s}_{\mu}.
\]
In view of $\left(  \lambda^{\vee}\right)  ^{\vee}=\lambda$, this can be
rewritten as follows:%
\[
h_{k}^{-}\cdot\overline{s}_{\lambda}=\sum_{\substack{\mu\text{ is a
snake;}\\\lambda\rightharpoonup\mu;\ \left\vert \lambda\right\vert -\left\vert
\mu\right\vert =k}}\overline{s}_{\mu}.
\]
This proves Proposition \ref{prop.alt.pieri2}.
\end{proof}

\subsection{Computing $\overline{s}_{\alpha}$}

\begin{convention}
From now on, for the rest of Section \ref{sect.pf}, we assume that $n\geq2$.
\end{convention}

Our next goal is to obtain a simple formula for the Schur polynomial
$\overline{s}_{\alpha}\left(  x_{1},x_{2},\ldots,x_{n}\right)  $, where
$\alpha$ is as in Theorem \ref{thm.main}. The first step is the following definition:

\begin{definition}
\label{def.ominus}Let $a,b\in\mathbb{N}$. Then, $b\ominus a$ will denote the
snake $\left(  b,0^{n-2},-a\right)  $. (This is indeed a well-defined snake,
since $n\geq2$ and since $b\geq0\geq-a$.)
\end{definition}

\begin{proposition}
\label{prop.ominus.h}Let $a,b\in\mathbb{Z}$. Then,%
\begin{equation}
h_{a}^{-}h_{b}^{+}=\sum_{k=0}^{\min\left\{  a,b\right\}  }\overline
{s}_{\left(  b-k\right)  \ominus\left(  a-k\right)  }.
\label{eq.prop.ominus.h.eq}%
\end{equation}

\end{proposition}

\begin{proof}
[Proof of Proposition \ref{prop.ominus.h}.]We must prove the equality
(\ref{eq.prop.ominus.h.eq}). If (at least) one of the integers $a$ and $b$ is
negative, then this equality boils down to $0=0$\ \ \ \ \footnote{Indeed, its
left hand side is $0$ in this case because every negative integer $k$
satisfies $h_{k}^{-}=0$ and $h_{k}^{+}=0$; but its right hand side is also $0$
in this case, because the negativity of $\min\left\{  a,b\right\}  $ causes
the sum to become empty.}. Hence, for the rest of this proof, we WLOG assume
that none of the integers $a$ and $b$ is negative. Hence, $a,b\in\mathbb{N}$.

Note that each $k\in\left\{  0,1,\ldots,\min\left\{  a,b\right\}  \right\}  $
satisfies $b-k\in\mathbb{N}$ (since $k\leq\min\left\{  a,b\right\}  \leq b$)
and $a-k\in\mathbb{N}$ (likewise). Hence, the snakes $\left(  b-k\right)
\ominus\left(  a-k\right)  $ on the right hand side of the equality
(\ref{eq.prop.ominus.h.eq}) are well-defined.

Lemma \ref{lem.alt.h} (applied to $k=b$) yields that the partition $\left(
b\right)  $ is a nonnegative snake (when regarded as the $n$-tuple $\left(
b,0,0,\ldots,0\right)  $), and satisfies $\overline{s}_{\left(  b\right)
}=h_{b}^{+}$.

Now, Proposition \ref{prop.alt.pieri2} (applied to $\lambda=\left(  b\right)
$ and $k=a$) yields
\[
h_{a}^{-}\cdot\overline{s}_{\left(  b\right)  }=\sum_{\substack{\mu\text{ is a
snake;}\\\left(  b\right)  \rightharpoonup\mu;\ \left\vert \left(  b\right)
\right\vert -\left\vert \mu\right\vert =a}}\overline{s}_{\mu}.
\]
In view of $\overline{s}_{\left(  b\right)  }=h_{b}^{+}$ and $\left\vert
\left(  b\right)  \right\vert =b$, we can rewrite this as%
\begin{equation}
h_{a}^{-}\cdot h_{b}^{+}=\sum_{\substack{\mu\text{ is a snake;}\\\left(
b\right)  \rightharpoonup\mu;\ b-\left\vert \mu\right\vert =a}}\overline
{s}_{\mu}. \label{pf.prop.ominus.h.1}%
\end{equation}

Now, we claim the following:

\begin{statement}
\textit{Claim 1:} The snakes $\mu$ satisfying $\left(  b\right)
\rightharpoonup\mu$ and $b-\left\vert \mu\right\vert =a$ are precisely the
snakes of the form $\left(  b-k\right)  \ominus\left(  a-k\right)  $ for
$k\in\left\{  0,1,\ldots,\min\left\{  a,b\right\}  \right\}  $.
\end{statement}

\begin{vershort}
[\textit{Proof of Claim 1:} The proof of this claim is straightforward (and
can be found elaborated in the detailed version \cite{verlong} of this paper).]
\end{vershort}

\begin{verlong}
[\textit{Proof of Claim 1:} Let $\mu$ be a snake satisfying $\left(  b\right)
\rightharpoonup\mu$ and $b-\left\vert \mu\right\vert =a$. We shall show that
$\mu=\left(  b-k\right)  \ominus\left(  a-k\right)  $ for some $k\in\left\{
0,1,\ldots,\min\left\{  a,b\right\}  \right\}  $.

Indeed, we have $\left(  b\right)  \rightharpoonup\mu$. In other words,
$\left(  b,\underbrace{0,0,\ldots,0}_{n-1\text{ times}}\right)
\rightharpoonup\mu$ (since we are identifying $\left(  b\right)  $ with the
snake $\left(  b,\underbrace{0,0,\ldots,0}_{n-1\text{ times}}\right)  $). But
Definition \ref{def.horstr} shows that we have $\left(
b,\underbrace{0,0,\ldots,0}_{n-1\text{ times}}\right)  \rightharpoonup\mu$ if
and only if we have%
\[
b\geq\mu_{1}\geq0\geq\mu_{2}\geq0\geq\mu_{3}\geq\cdots\geq0\geq\mu_{n}.
\]
Hence, we have%
\[
b\geq\mu_{1}\geq0\geq\mu_{2}\geq0\geq\mu_{3}\geq\cdots\geq0\geq\mu_{n}%
\]
(since we have $\left(  b,\underbrace{0,0,\ldots,0}_{n-1\text{ times}}\right)
\rightharpoonup\mu$). This chain of inequalities yields that each of the
numbers $\mu_{2},\mu_{3},\ldots,\mu_{n-1}$ equals $0$ (since it is sandwiched
between $0$ and $0$ in this chain). Hence, $\mu=\left(  \mu_{1}%
,\underbrace{0,0,\ldots,0}_{n-2\text{ times}},\mu_{n}\right)  =\left(  \mu
_{1},0^{n-2},\mu_{n}\right)  $. Therefore,
\[
\left\vert \mu\right\vert =\mu_{1}+\underbrace{0+0+\cdots+0}_{n-2\text{
times}}+\mu_{n}=\mu_{1}+\mu_{n},
\]
so that $\mu_{1}=\left\vert \mu\right\vert -\mu_{n}$. Moreover, from $0\geq
\mu_{n}$, we obtain $\mu_{n}\leq0$, so that $-\mu_{n}\in\mathbb{N}$ and
$\mu_{1}=\left\vert \mu\right\vert -\underbrace{\mu_{n}}_{\leq0}\geq\left\vert
\mu\right\vert $.

Set $g=b-\mu_{1}$. Thus, $\mu_{1}=b-g$. Also, $g=\underbrace{b}_{\geq\mu_{1}%
}-\mu_{1}\geq\mu_{1}-\mu_{1}=0$. Furthermore, from $\left\vert \mu\right\vert
=\mu_{1}+\mu_{n}$, we obtain%
\[
\mu_{n}=\left\vert \mu\right\vert -\mu_{1}=\underbrace{\left(  b-\mu
_{1}\right)  }_{=g}-\underbrace{\left(  b-\left\vert \mu\right\vert \right)
}_{=a}=g-a=-\left(  a-g\right)  .
\]

Combining the inequalities $g=b-\underbrace{\mu_{1}}_{\geq\left\vert
\mu\right\vert }\leq b-\left\vert \mu\right\vert =a$ and $g=b-\underbrace{\mu
_{1}}_{\geq0}\leq b$, we obtain $g\leq\min\left\{  a,b\right\}  $. Combining
this with $g\geq0$, we find $0\leq g\leq\min\left\{  a,b\right\}  $, so that
$g\in\left\{  0,1,\ldots,\min\left\{  a,b\right\}  \right\}  $. Furthermore,%
\[
\mu=\left(  \underbrace{\mu_{1}}_{=b-g},0^{n-2},\underbrace{\mu_{n}%
}_{=-\left(  a-g\right)  }\right)  =\left(  b-g,0^{n-2},-\left(  a-g\right)
\right)  =\left(  b-g\right)  \ominus\left(  a-g\right)
\]
(since $\left(  b-g\right)  \ominus\left(  a-g\right)  $ was defined to be the
snake $\left(  b-g,0^{n-2},-\left(  a-g\right)  \right)  $). Thus,
$\mu=\left(  b-k\right)  \ominus\left(  a-k\right)  $ for some $k\in\left\{
0,1,\ldots,\min\left\{  a,b\right\}  \right\}  $ (namely, for $k=g$).

Forget that we fixed $\mu$. We thus have shown that if $\mu$ is a snake
satisfying $\left(  b\right)  \rightharpoonup\mu$ and $b-\left\vert
\mu\right\vert =a$, then we have $\mu=\left(  b-k\right)  \ominus\left(
a-k\right)  $ for some $k\in\left\{  0,1,\ldots,\min\left\{  a,b\right\}
\right\}  $. It is straightforward to see that the converse holds as well
(i.e., that if $\mu=\left(  b-k\right)  \ominus\left(  a-k\right)  $ for some
$k\in\left\{  0,1,\ldots,\min\left\{  a,b\right\}  \right\}  $, then $\mu$ is
a snake satisfying $\left(  b\right)  \rightharpoonup\mu$ and $b-\left\vert
\mu\right\vert =a$). Combining these two facts, we conclude that the snakes
$\mu$ satisfying $\left(  b\right)  \rightharpoonup\mu$ and $b-\left\vert
\mu\right\vert =a$ are precisely the snakes of the form $\left(  b-k\right)
\ominus\left(  a-k\right)  $ for $k\in\left\{  0,1,\ldots,\min\left\{
a,b\right\}  \right\}  $. This proves Claim 1.]
\end{verlong}

Now, Claim 1 shows that%
\[
\sum_{\substack{\mu\text{ is a snake;}\\\left(  b\right)  \rightharpoonup
\mu;\ b-\left\vert \mu\right\vert =a}}\overline{s}_{\mu}=\sum_{k\in\left\{
0,1,\ldots,\min\left\{  a,b\right\}  \right\}  }\overline{s}_{\left(
b-k\right)  \ominus\left(  a-k\right)  }%
\]
(indeed, it is clear that the snakes $\left(  b-k\right)  \ominus\left(
a-k\right)  $ for $k\in\left\{  0,1,\ldots,\min\left\{  a,b\right\}  \right\}
$ are all distinct). Hence, (\ref{pf.prop.ominus.h.1}) becomes%
\[
h_{a}^{-}\cdot h_{b}^{+}=\sum_{\substack{\mu\text{ is a snake;}\\\left(
b\right)  \rightharpoonup\mu;\ b-\left\vert \mu\right\vert =a}}\overline
{s}_{\mu}=\sum_{k\in\left\{  0,1,\ldots,\min\left\{  a,b\right\}  \right\}
}\overline{s}_{\left(  b-k\right)  \ominus\left(  a-k\right)  }=\sum
_{k=0}^{\min\left\{  a,b\right\}  }\overline{s}_{\left(  b-k\right)
\ominus\left(  a-k\right)  }.
\]
This proves Proposition \ref{prop.ominus.h}.
\end{proof}

\begin{proposition}
\label{prop.ominus.s}Let $a,b\in\mathbb{N}$. Then,%
\[
\overline{s}_{b\ominus a}=h_{a}^{-}h_{b}^{+}-h_{a-1}^{-}h_{b-1}^{+}.
\]

\end{proposition}

(Recall that every negative integer $k$ satisfies $h_{k}^{-}=0$ and $h_{k}%
^{+}=0$.)

\begin{proof}
[Proof of Proposition \ref{prop.ominus.s}.]We have $\min\left\{  a,b\right\}
\in\mathbb{N}$ (since $a,b\in\mathbb{N}$), so that $\min\left\{  a,b\right\}
\geq0$.

Proposition \ref{prop.ominus.h} yields%
\begin{align}
h_{a}^{-}h_{b}^{+}  &  =\sum_{k=0}^{\min\left\{  a,b\right\}  }\overline
{s}_{\left(  b-k\right)  \ominus\left(  a-k\right)  }=\underbrace{\overline
{s}_{\left(  b-0\right)  \ominus\left(  a-0\right)  }}_{=\overline
{s}_{b\ominus a}}+\sum_{k=1}^{\min\left\{  a,b\right\}  }\overline{s}_{\left(
b-k\right)  \ominus\left(  a-k\right)  }\nonumber\\
&  \ \ \ \ \ \ \ \ \ \ \left(
\begin{array}
[c]{c}%
\text{here, we have split off the addend for }k=0\text{ from the sum,}\\
\text{since }\min\left\{  a,b\right\}  \geq0
\end{array}
\right) \nonumber\\
&  =\overline{s}_{b\ominus a}+\sum_{k=1}^{\min\left\{  a,b\right\}  }%
\overline{s}_{\left(  b-k\right)  \ominus\left(  a-k\right)  }.
\label{pf.prop.ominus.s.1}%
\end{align}
Proposition \ref{prop.ominus.h} (applied to $a-1$ and $b-1$ instead of $a$ and
$b$) yields%
\begin{align*}
h_{a-1}^{-}h_{b-1}^{+}  &  =\sum_{k=0}^{\min\left\{  a-1,b-1\right\}
}\overline{s}_{\left(  \left(  b-1\right)  -k\right)  \ominus\left(  \left(
a-1\right)  -k\right)  }=\sum_{k=0}^{\min\left\{  a,b\right\}  -1}\overline
{s}_{\left(  \left(  b-1\right)  -k\right)  \ominus\left(  \left(  a-1\right)
-k\right)  }\\
&  \ \ \ \ \ \ \ \ \ \ \left(  \text{since }\min\left\{  a-1,b-1\right\}
=\min\left\{  a,b\right\}  -1\right) \\
&  =\sum_{k=1}^{\min\left\{  a,b\right\}  }\overline{s}_{\left(  \left(
b-1\right)  -\left(  k-1\right)  \right)  \ominus\left(  \left(  a-1\right)
-\left(  k-1\right)  \right)  }\\
&  \ \ \ \ \ \ \ \ \ \ \left(  \text{here, we have substituted }k-1\text{ for
}k\text{ in the sum}\right) \\
&  =\sum_{k=1}^{\min\left\{  a,b\right\}  }\overline{s}_{\left(  b-k\right)
\ominus\left(  a-k\right)  }\\
&  \ \ \ \ \ \ \ \ \ \ \left(  \text{since }\left(  b-1\right)  -\left(
k-1\right)  =b-k\text{ and }\left(  a-1\right)  -\left(  k-1\right)
=a-k\right)  .
\end{align*}
Subtracting this equality from (\ref{pf.prop.ominus.s.1}), we obtain%
\[
h_{a}^{-}h_{b}^{+}-h_{a-1}^{-}h_{b-1}^{+}=\left(  \overline{s}_{b\ominus
a}+\sum_{k=1}^{\min\left\{  a,b\right\}  }\overline{s}_{\left(  b-k\right)
\ominus\left(  a-k\right)  }\right)  -\sum_{k=1}^{\min\left\{  a,b\right\}
}\overline{s}_{\left(  b-k\right)  \ominus\left(  a-k\right)  }=\overline
{s}_{b\ominus a}.
\]
This proves Proposition \ref{prop.ominus.s}.
\end{proof}

\begin{corollary}
\label{cor.ominus.salpha}Let $a,b\in\mathbb{N}$. Define the partition
$\alpha=\left(  a+b,a^{n-2}\right)  $. Then, $\alpha$ is a nonnegative snake
and satisfies
\[
\overline{s}_{\alpha}=x_{\Pi}^{a}\cdot\left(  h_{a}^{-}h_{b}^{+}-h_{a-1}%
^{-}h_{b-1}^{+}\right)  .
\]

\end{corollary}

\begin{proof}
[Proof of Corollary \ref{cor.ominus.salpha}.]From $a,b\in\mathbb{N}$, we
obtain $a+\underbrace{b}_{\geq0}\geq a\geq0$. Hence, $\alpha$ is a partition
of length $\leq n$ (since $n-1\leq n$). In other words, $\alpha$ is a
nonnegative snake.

\begin{vershort}
It is easy to see that $\alpha=\left(  b\ominus a\right)  +a$ (since
$\alpha=\left(  a+b,a^{n-2}\right)  =\left(  a+b,\underbrace{a,a,\ldots
,a}_{n-2\text{ times}},0\right)  $ and $b\ominus a=\left(  b,0^{n-2}%
,-a\right)  =\left(  b,\underbrace{0,0,\ldots,0}_{n-2\text{ times}},-a\right)
$). Hence,%
\begin{align*}
\overline{s}_{\alpha}  &  =\overline{s}_{\left(  b\ominus a\right)  +a}%
=x_{\Pi}^{a}\overline{s}_{b\ominus a}\ \ \ \ \ \ \ \ \ \ \left(  \text{by
Lemma \ref{lem.alt.smove}, applied to }\lambda=b\ominus a\text{ and
}d=a\right) \\
&  =x_{\Pi}^{a}\cdot\left(  h_{a}^{-}h_{b}^{+}-h_{a-1}^{-}h_{b-1}^{+}\right)
\ \ \ \ \ \ \ \ \ \ \left(  \text{by Proposition \ref{prop.ominus.s}}\right)
.
\end{align*}

\end{vershort}

\begin{verlong}
We have
\begin{align*}
&  \underbrace{\left(  b\ominus a\right)  }_{=\left(  b,0^{n-2},-a\right)
=\left(  b,\underbrace{0,0,\ldots,0}_{n-2\text{ times}},-a\right)  }+a\\
&  =\left(  b,\underbrace{0,0,\ldots,0}_{n-2\text{ times}},-a\right)
+a=\left(  b+a,\underbrace{0+a,0+a,\ldots,0+a}_{n-2\text{ times}},-a+a\right)
\\
&  =\left(  a+b,\underbrace{a,a,\ldots,a}_{n-2\text{ times}},0\right)
=\left(  a+b,a^{n-2}\right)  =\alpha
\end{align*}
(since $\alpha=\left(  a+b,a^{n-2}\right)  $). Hence, $\alpha=\left(  b\ominus
a\right)  +a$, so that%
\begin{align*}
\overline{s}_{\alpha}  &  =\overline{s}_{\left(  b\ominus a\right)  +a}%
=x_{\Pi}^{a}\underbrace{\overline{s}_{b\ominus a}}_{\substack{=h_{a}^{-}%
h_{b}^{+}-h_{a-1}^{-}h_{b-1}^{+}\\\text{(by Proposition \ref{prop.ominus.s})}%
}}\\
&  \ \ \ \ \ \ \ \ \ \ \left(  \text{by Lemma \ref{lem.alt.smove}, applied to
}\lambda=b\ominus a\text{ and }d=a\right) \\
&  =x_{\Pi}^{a}\cdot\left(  h_{a}^{-}h_{b}^{+}-h_{a-1}^{-}h_{b-1}^{+}\right)
.
\end{align*}

\end{verlong}

\noindent This proves Corollary \ref{cor.ominus.salpha}.
\end{proof}

We notice that the expression $h_{a}^{-}h_{b}^{+}-h_{a-1}^{-}h_{b-1}^{+}$ in
Proposition \ref{prop.ominus.s} can be rewritten as $\det\left(
\begin{array}
[c]{cc}%
h_{a}^{-} & h_{a-1}^{-}\\
h_{b-1}^{+} & h_{b}^{+}%
\end{array}
\right)  $. This suggests a way to generalize Proposition \ref{prop.ominus.s}
(as well as the first Jacobi--Trudi formula \cite[(2.4.16)]{GriRei}). Such a
generalization indeed exists, and has been proved by Koike as well as by Hamel
and King; see Proposition \ref{prop.jt} below.

\subsection{The sets $R_{\mu,a,b}\left(  \gamma\right)  $ and a formula for
$h_{a}^{-}h_{b}^{+}\overline{s}_{\mu}$}

We shall next aim for a formula for $h_{a}^{-}h_{b}^{+}\overline{s}_{\mu}$
(for a snake $\mu$ and integers $a,b\in\mathbb{Z}$), which will be obtained in
a straightforward way by applying Propositions \ref{prop.alt.pieri1} and
\ref{prop.alt.pieri2}. We will need the following definition:

\begin{definition}
\label{def.Rmab}Let $\mu,\gamma\in\mathbb{Z}^{n}$ and $a,b\in\mathbb{Z}$.
Then, $R_{\mu,a,b}\left(  \gamma\right)  $ shall denote the set of all snakes
$\nu$ satisfying the four conditions%
\[
\mu\rightharpoonup\nu\ \ \ \ \ \ \ \ \ \ \text{and}%
\ \ \ \ \ \ \ \ \ \ \left\vert \mu\right\vert -\left\vert \nu\right\vert
=a\ \ \ \ \ \ \ \ \ \ \text{and}\ \ \ \ \ \ \ \ \ \ \gamma\rightharpoonup
\nu\ \ \ \ \ \ \ \ \ \ \text{and}\ \ \ \ \ \ \ \ \ \ \left\vert \gamma
\right\vert -\left\vert \nu\right\vert =b.
\]

\end{definition}

\begin{lemma}
\label{lem.Rmab.0}Let $\mu,\gamma\in\mathbb{Z}^{n}$ and $a,b\in\mathbb{Z}$.
Assume that $\gamma$ is not a snake. Then, $\left\vert R_{\mu,a,b}\left(
\gamma\right)  \right\vert =0$.
\end{lemma}

\begin{proof}
[Proof of Lemma \ref{lem.Rmab.0}.]Let $\nu\in R_{\mu,a,b}\left(
\gamma\right)  $. We shall obtain a contradiction.

Indeed, we have $\nu\in R_{\mu,a,b}\left(  \gamma\right)  $. In other words,
$\nu$ is a snake satisfying the four conditions%
\[
\mu\rightharpoonup\nu\ \ \ \ \ \ \ \ \ \ \text{and}%
\ \ \ \ \ \ \ \ \ \ \left\vert \mu\right\vert -\left\vert \nu\right\vert
=a\ \ \ \ \ \ \ \ \ \ \text{and}\ \ \ \ \ \ \ \ \ \ \gamma\rightharpoonup
\nu\ \ \ \ \ \ \ \ \ \ \text{and}\ \ \ \ \ \ \ \ \ \ \left\vert \gamma
\right\vert -\left\vert \nu\right\vert =b
\]
(by the definition of $R_{\mu,a,b}\left(  \gamma\right)  $). Thus, in
particular, we have $\gamma\rightharpoonup\nu$. Hence, Proposition
\ref{prop.horstr.easies} \textbf{(a)} (applied to $\gamma$ and $\nu$ instead
of $\mu$ and $\lambda$) yields that both $\nu$ and $\gamma$ are snakes. Hence,
$\gamma$ is a snake. This contradicts the fact that $\gamma$ is not a snake.

Now, forget that we fixed $\nu$. We thus have obtained a contradiction for
each $\nu\in R_{\mu,a,b}\left(  \gamma\right)  $. Hence, there exists no
$\nu\in R_{\mu,a,b}\left(  \gamma\right)  $. In other words, the set
$R_{\mu,a,b}\left(  \gamma\right)  $ is empty. Thus, $\left\vert R_{\mu
,a,b}\left(  \gamma\right)  \right\vert =0$. This proves Lemma
\ref{lem.Rmab.0}.
\end{proof}

\begin{lemma}
\label{lem.Rmab.formula}Let $\mu$ be a snake. Let $a,b\in\mathbb{Z}$. Then,%
\[
h_{a}^{-}h_{b}^{+}\overline{s}_{\mu}=\sum_{\gamma\text{ is a snake}}\left\vert
R_{\mu,a,b}\left(  \gamma\right)  \right\vert \overline{s}_{\gamma}.
\]

\end{lemma}

\begin{proof}
[Proof of Lemma \ref{lem.Rmab.formula}.]Proposition \ref{prop.alt.pieri2}
(with the letters $\lambda$, $k$ and $\mu$ renamed as $\mu$, $a$ and $\nu$)
says that%
\begin{equation}
h_{a}^{-}\cdot\overline{s}_{\mu}=\sum_{\substack{\nu\text{ is a snake;}%
\\\mu\rightharpoonup\nu;\ \left\vert \mu\right\vert -\left\vert \nu\right\vert
=a}}\overline{s}_{\nu}. \label{pf.lem.Rmab.formula.1}%
\end{equation}
Proposition \ref{prop.alt.pieri1} (with the letters $\lambda$, $k$ and $\mu$
renamed as $\nu$, $b$ and $\gamma$) says that%
\begin{equation}
h_{b}^{+}\cdot\overline{s}_{\nu}=\sum_{\substack{\gamma\text{ is a
snake;}\\\gamma\rightharpoonup\nu;\ \left\vert \gamma\right\vert -\left\vert
\nu\right\vert =b}}\overline{s}_{\gamma} \label{pf.lem.Rmab.formula.2}%
\end{equation}
for each snake $\nu$.

Now,%
\begin{align*}
h_{a}^{-}h_{b}^{+}\overline{s}_{\mu}  &  =h_{b}^{+}\cdot\underbrace{h_{a}%
^{-}\cdot\overline{s}_{\mu}}_{\substack{=\sum_{\substack{\nu\text{ is a
snake;}\\\mu\rightharpoonup\nu;\ \left\vert \mu\right\vert -\left\vert
\nu\right\vert =a}}\overline{s}_{\nu}\\\text{(by (\ref{pf.lem.Rmab.formula.1}%
))}}}=h_{b}^{+}\cdot\sum_{\substack{\nu\text{ is a snake;}\\\mu\rightharpoonup
\nu;\ \left\vert \mu\right\vert -\left\vert \nu\right\vert =a}}\overline
{s}_{\nu}\\
&  =\sum_{\substack{\nu\text{ is a snake;}\\\mu\rightharpoonup\nu;\ \left\vert
\mu\right\vert -\left\vert \nu\right\vert =a}}\ \ \ \underbrace{h_{b}^{+}%
\cdot\overline{s}_{\nu}}_{\substack{=\sum_{\substack{\gamma\text{ is a
snake;}\\\gamma\rightharpoonup\nu;\ \left\vert \gamma\right\vert -\left\vert
\nu\right\vert =b}}\overline{s}_{\gamma}\\\text{(by
(\ref{pf.lem.Rmab.formula.2}))}}}=\underbrace{\sum_{\substack{\nu\text{ is a
snake;}\\\mu\rightharpoonup\nu;\ \left\vert \mu\right\vert -\left\vert
\nu\right\vert =a}}\ \ \ \sum_{\substack{\gamma\text{ is a snake;}%
\\\gamma\rightharpoonup\nu;\ \left\vert \gamma\right\vert -\left\vert
\nu\right\vert =b}}}_{=\sum_{\gamma\text{ is a snake}}\ \ \ \sum
_{\substack{\nu\text{ is a snake;}\\\mu\rightharpoonup\nu;\ \left\vert
\mu\right\vert -\left\vert \nu\right\vert =a;\\\gamma\rightharpoonup
\nu;\ \left\vert \gamma\right\vert -\left\vert \nu\right\vert =b}}}%
\overline{s}_{\gamma}\\
&  =\sum_{\gamma\text{ is a snake}}\ \ \ \underbrace{\sum_{\substack{\nu\text{
is a snake;}\\\mu\rightharpoonup\nu;\ \left\vert \mu\right\vert -\left\vert
\nu\right\vert =a;\\\gamma\rightharpoonup\nu;\ \left\vert \gamma\right\vert
-\left\vert \nu\right\vert =b}}}_{\substack{=\sum_{\nu\in R_{\mu,a,b}\left(
\gamma\right)  }\\\text{(by the definition of }R_{\mu,a,b}\left(
\gamma\right)  \text{)}}}\overline{s}_{\gamma}=\sum_{\gamma\text{ is a snake}%
}\ \ \ \underbrace{\sum_{\nu\in R_{\mu,a,b}\left(  \gamma\right)  }%
\overline{s}_{\gamma}}_{=\left\vert R_{\mu,a,b}\left(  \gamma\right)
\right\vert \overline{s}_{\gamma}}\\
&  =\sum_{\gamma\text{ is a snake}}\left\vert R_{\mu,a,b}\left(
\gamma\right)  \right\vert \overline{s}_{\gamma}.
\end{align*}
This proves Lemma \ref{lem.Rmab.formula}.
\end{proof}

\begin{corollary}
\label{cor.Rmab.lr}Let $\mu\in\operatorname*{Par}\left[  n\right]  $. Let
$a,b\in\mathbb{N}$. Define the partition $\alpha=\left(  a+b,a^{n-2}\right)
$. Then, every $\lambda\in\mathbb{Z}^{n}$ satisfies%
\begin{equation}
c_{\alpha,\mu}^{\lambda}=\left\vert R_{\mu,a,b}\left(  \lambda-a\right)
\right\vert -\left\vert R_{\mu,a-1,b-1}\left(  \lambda-a\right)  \right\vert .
\label{eq.cor.Rmab.lr.eq}%
\end{equation}
Here, we understand $c_{\alpha,\mu}^{\lambda}$ to mean $0$ if $\lambda$ is not
a partition (i.e., if $\lambda$ is not a nonnegative snake).
\end{corollary}

\begin{proof}
[Proof of Corollary \ref{cor.Rmab.lr}.]Corollary \ref{cor.ominus.salpha} shows
that $\alpha$ is a nonnegative snake and satisfies
\begin{equation}
\overline{s}_{\alpha}=x_{\Pi}^{a}\cdot\left(  h_{a}^{-}h_{b}^{+}-h_{a-1}%
^{-}h_{b-1}^{+}\right)  . \label{pf.cor.Rmab.lr.1}%
\end{equation}

Every snake $\gamma$ satisfies%
\begin{equation}
\overline{s}_{\gamma+a}=x_{\Pi}^{a}\overline{s}_{\gamma}
\label{pf.cor.Rmab.lr.2}%
\end{equation}
(by Lemma \ref{lem.alt.smove}, applied to $\gamma$ and $a$ instead of
$\lambda$ and $d$).

We have $\alpha\in\operatorname*{Par}\left[  n\right]  $ (since $\alpha$ is a
nonnegative snake). Hence, Lemma \ref{lem.alt.lrr} (applied to $\alpha$ and
$\mu$ instead of $\mu$ and $\nu$) yields%
\begin{align*}
\overline{s}_{\alpha}\overline{s}_{\mu}  &  =\underbrace{\sum_{\lambda
\in\operatorname*{Par}\left[  n\right]  }}_{\substack{=\sum_{\substack{\lambda
\text{ is a snake;}\\\lambda\text{ is nonnegative}}}\\\text{(since the
partitions }\lambda\in\operatorname*{Par}\left[  n\right]  \text{
are}\\\text{precisely the nonnegative snakes)}}}c_{\alpha,\mu}^{\lambda
}\overline{s}_{\lambda}=\sum_{\substack{\lambda\text{ is a snake;}%
\\\lambda\text{ is nonnegative}}}c_{\alpha,\mu}^{\lambda}\overline{s}%
_{\lambda}\\
&  =\sum_{\lambda\text{ is a snake}}c_{\alpha,\mu}^{\lambda}\overline
{s}_{\lambda}-\sum_{\substack{\lambda\text{ is a snake;}\\\lambda\text{ is not
nonnegative}}}\underbrace{c_{\alpha,\mu}^{\lambda}}%
_{\substack{=0\\\text{(since we understand }c_{\alpha,\mu}^{\lambda}\\\text{to
mean }0\text{ if }\lambda\text{ is not}\\\text{a nonnegative snake)}%
}}\overline{s}_{\lambda}\\
&  =\sum_{\lambda\text{ is a snake}}c_{\alpha,\mu}^{\lambda}\overline
{s}_{\lambda}-\underbrace{\sum_{\substack{\lambda\text{ is a snake;}%
\\\lambda\text{ is not nonnegative}}}0\overline{s}_{\lambda}}_{=0}%
=\sum_{\lambda\text{ is a snake}}c_{\alpha,\mu}^{\lambda}\overline{s}%
_{\lambda}.
\end{align*}

\begin{vershort}
\noindent Hence,%
\begin{align*}
\sum_{\lambda\text{ is a snake}}c_{\alpha,\mu}^{\lambda}\overline{s}%
_{\lambda}  &  =\overline{s}_{\alpha}\overline{s}_{\mu}=x_{\Pi}^{a}%
\cdot\left(  h_{a}^{-}h_{b}^{+}-h_{a-1}^{-}h_{b-1}^{+}\right)  \overline
{s}_{\mu}\ \ \ \ \ \ \ \ \ \ \left(  \text{by (\ref{pf.cor.Rmab.lr.1})}\right)
\\
&  =x_{\Pi}^{a}\cdot\underbrace{h_{a}^{-}h_{b}^{+}\overline{s}_{\mu}%
}_{\substack{=\sum_{\gamma\text{ is a snake}}\left\vert R_{\mu,a,b}\left(
\gamma\right)  \right\vert \overline{s}_{\gamma}\\\text{(by Lemma
\ref{lem.Rmab.formula})}}}-x_{\Pi}^{a}\cdot\underbrace{h_{a-1}^{-}h_{b-1}%
^{+}\overline{s}_{\mu}}_{\substack{=\sum_{\gamma\text{ is a snake}}\left\vert
R_{\mu,a-1,b-1}\left(  \gamma\right)  \right\vert \overline{s}_{\gamma
}\\\text{(by Lemma \ref{lem.Rmab.formula},}\\\text{applied to }a-1\text{ and
}b-1\\\text{instead of }a\text{ and }b\text{)}}}\\
&  =x_{\Pi}^{a}\cdot\sum_{\gamma\text{ is a snake}}\left\vert R_{\mu
,a,b}\left(  \gamma\right)  \right\vert \overline{s}_{\gamma}-x_{\Pi}^{a}%
\cdot\sum_{\gamma\text{ is a snake}}\left\vert R_{\mu,a-1,b-1}\left(
\gamma\right)  \right\vert \overline{s}_{\gamma}\\
&  =\sum_{\gamma\text{ is a snake}}\left(  \left\vert R_{\mu,a,b}\left(
\gamma\right)  \right\vert -\left\vert R_{\mu,a-1,b-1}\left(  \gamma\right)
\right\vert \right)  \underbrace{x_{\Pi}^{a}\overline{s}_{\gamma}%
}_{\substack{=\overline{s}_{\gamma+a}\\\text{(by (\ref{pf.cor.Rmab.lr.2}))}%
}}\\
&  =\sum_{\gamma\text{ is a snake}}\left(  \left\vert R_{\mu,a,b}\left(
\gamma\right)  \right\vert -\left\vert R_{\mu,a-1,b-1}\left(  \gamma\right)
\right\vert \right)  \overline{s}_{\gamma+a}\\
&  =\sum_{\lambda\text{ is a snake}}\left(  \left\vert R_{\mu,a,b}\left(
\lambda-a\right)  \right\vert -\left\vert R_{\mu,a-1,b-1}\left(
\lambda-a\right)  \right\vert \right)  \underbrace{\overline{s}_{\left(
\lambda-a\right)  +a}}_{=\overline{s}_{\lambda}}\\
&  \ \ \ \ \ \ \ \ \ \ \left(
\begin{array}
[c]{c}%
\text{here, we have substituted }\lambda-a\text{ for }\gamma\text{ in the
sum,}\\
\text{since the map }\left\{  \text{snakes}\right\}  \rightarrow\left\{
\text{snakes}\right\}  ,\ \lambda\mapsto\lambda-a\\
\text{is a bijection}%
\end{array}
\right) \\
&  =\sum_{\lambda\text{ is a snake}}\left(  \left\vert R_{\mu,a,b}\left(
\lambda-a\right)  \right\vert -\left\vert R_{\mu,a-1,b-1}\left(
\lambda-a\right)  \right\vert \right)  \overline{s}_{\lambda}.
\end{align*}

\end{vershort}

\begin{verlong}
\noindent Hence,%
\begin{align*}
\sum_{\lambda\text{ is a snake}}c_{\alpha,\mu}^{\lambda}\overline{s}%
_{\lambda}  &  =\underbrace{\overline{s}_{\alpha}}_{\substack{=x_{\Pi}%
^{a}\cdot\left(  h_{a}^{-}h_{b}^{+}-h_{a-1}^{-}h_{b-1}^{+}\right)  \\\text{(by
(\ref{pf.cor.Rmab.lr.1}))}}}\overline{s}_{\mu}=x_{\Pi}^{a}\cdot\left(
h_{a}^{-}h_{b}^{+}-h_{a-1}^{-}h_{b-1}^{+}\right)  \overline{s}_{\mu}\\
&  =x_{\Pi}^{a}\cdot\underbrace{h_{a}^{-}h_{b}^{+}\overline{s}_{\mu}%
}_{\substack{=\sum_{\gamma\text{ is a snake}}\left\vert R_{\mu,a,b}\left(
\gamma\right)  \right\vert \overline{s}_{\gamma}\\\text{(by Lemma
\ref{lem.Rmab.formula})}}}-x_{\Pi}^{a}\cdot\underbrace{h_{a-1}^{-}h_{b-1}%
^{+}\overline{s}_{\mu}}_{\substack{=\sum_{\gamma\text{ is a snake}}\left\vert
R_{\mu,a-1,b-1}\left(  \gamma\right)  \right\vert \overline{s}_{\gamma
}\\\text{(by Lemma \ref{lem.Rmab.formula},}\\\text{applied to }a-1\text{ and
}b-1\\\text{instead of }a\text{ and }b\text{)}}}\\
&  =x_{\Pi}^{a}\cdot\sum_{\gamma\text{ is a snake}}\left\vert R_{\mu
,a,b}\left(  \gamma\right)  \right\vert \overline{s}_{\gamma}-x_{\Pi}^{a}%
\cdot\sum_{\gamma\text{ is a snake}}\left\vert R_{\mu,a-1,b-1}\left(
\gamma\right)  \right\vert \overline{s}_{\gamma}\\
&  =\sum_{\gamma\text{ is a snake}}\left\vert R_{\mu,a,b}\left(
\gamma\right)  \right\vert x_{\Pi}^{a}\overline{s}_{\gamma}-\sum_{\gamma\text{
is a snake}}\left\vert R_{\mu,a-1,b-1}\left(  \gamma\right)  \right\vert
x_{\Pi}^{a}\overline{s}_{\gamma}\\
&  =\sum_{\gamma\text{ is a snake}}\left(  \left\vert R_{\mu,a,b}\left(
\gamma\right)  \right\vert -\left\vert R_{\mu,a-1,b-1}\left(  \gamma\right)
\right\vert \right)  \underbrace{x_{\Pi}^{a}\overline{s}_{\gamma}%
}_{\substack{=\overline{s}_{\gamma+a}\\\text{(by (\ref{pf.cor.Rmab.lr.2}))}%
}}\\
&  =\sum_{\gamma\text{ is a snake}}\left(  \left\vert R_{\mu,a,b}\left(
\gamma\right)  \right\vert -\left\vert R_{\mu,a-1,b-1}\left(  \gamma\right)
\right\vert \right)  \overline{s}_{\gamma+a}\\
&  =\sum_{\lambda\text{ is a snake}}\left(  \left\vert R_{\mu,a,b}\left(
\lambda-a\right)  \right\vert -\left\vert R_{\mu,a-1,b-1}\left(
\lambda-a\right)  \right\vert \right)  \underbrace{\overline{s}_{\left(
\lambda-a\right)  +a}}_{=\overline{s}_{\lambda}}\\
&  \ \ \ \ \ \ \ \ \ \ \left(
\begin{array}
[c]{c}%
\text{here, we have substituted }\lambda-a\text{ for }\gamma\text{ in the
sum,}\\
\text{since the map }\left\{  \text{snakes}\right\}  \rightarrow\left\{
\text{snakes}\right\}  ,\ \lambda\mapsto\lambda-a\\
\text{is a bijection}%
\end{array}
\right) \\
&  =\sum_{\lambda\text{ is a snake}}\left(  \left\vert R_{\mu,a,b}\left(
\lambda-a\right)  \right\vert -\left\vert R_{\mu,a-1,b-1}\left(
\lambda-a\right)  \right\vert \right)  \overline{s}_{\lambda}.
\end{align*}

\end{verlong}

\noindent We can compare coefficients on both sides of this equality (since
Lemma \ref{lem.alt.linind} shows that the family $\left(  \overline
{s}_{\lambda}\right)  _{\lambda\in\left\{  \text{snakes}\right\}  }$ of
elements of $\mathcal{L}$ is $\mathbf{k}$-linearly independent), and thus
conclude that%
\[
c_{\alpha,\mu}^{\lambda}=\left\vert R_{\mu,a,b}\left(  \lambda-a\right)
\right\vert -\left\vert R_{\mu,a-1,b-1}\left(  \lambda-a\right)  \right\vert
\ \ \ \ \ \ \ \ \ \ \text{for every snake }\lambda\text{.}%
\]
This proves (\ref{eq.cor.Rmab.lr.eq}) in the case when $\lambda$ is a snake.

However, it is easy to see that (\ref{eq.cor.Rmab.lr.eq}) also holds in the
case when $\lambda$ is not a snake\footnote{\textit{Proof.} Let $\lambda
\in\mathbb{Z}^{n}$ be such that $\lambda$ is not a snake. We must show that
(\ref{eq.cor.Rmab.lr.eq}) holds for this $\lambda$.
\par
We have assumed that $\lambda$ is not a snake. Hence, $\lambda-a$ is not a
snake (because it is easy to see from Definition \ref{def.snakes} that
$\lambda$ is a snake if and only if $\lambda-a$ is a snake). Thus, Lemma
\ref{lem.Rmab.0} (applied to $\gamma=\lambda-a$) yields $\left\vert
R_{\mu,a,b}\left(  \lambda-a\right)  \right\vert =0$. Also, Lemma
\ref{lem.Rmab.0} (applied to $\lambda-a$, $a-1$ and $b-1$ instead of $\gamma$,
$a$ and $b$) yields $\left\vert R_{\mu,a-1,b-1}\left(  \lambda-a\right)
\right\vert =0$. On the other hand, $\lambda$ is not a snake, and thus not a
nonnegative snake. Hence, $c_{\alpha,\mu}^{\lambda}=0$ (since we have defined
$c_{\alpha,\mu}^{\lambda}$ to be $0$ if $\lambda$ is not a nonnegative snake).
Comparing this with $\underbrace{\left\vert R_{\mu,a,b}\left(  \lambda
-a\right)  \right\vert }_{=0}-\underbrace{\left\vert R_{\mu,a-1,b-1}\left(
\lambda-a\right)  \right\vert }_{=0}=0$, we obtain $c_{\alpha,\mu}^{\lambda
}=\left\vert R_{\mu,a,b}\left(  \lambda-a\right)  \right\vert -\left\vert
R_{\mu,a-1,b-1}\left(  \lambda-a\right)  \right\vert $. In other words,
(\ref{eq.cor.Rmab.lr.eq}) holds.
\par
Thus, we have shown that (\ref{eq.cor.Rmab.lr.eq}) holds in the case when
$\lambda$ is not a snake.}. Thus, (\ref{eq.cor.Rmab.lr.eq}) always holds. This
proves Corollary \ref{cor.Rmab.lr}.
\end{proof}

\subsection{\label{subsect.fF}The map $\mathbf{f}_{\mu}$}

\begin{convention}
For the whole Subsection \ref{subsect.fF}, we shall use Convention
\ref{conv.bir.peri} (not only for $n$-tuples $a\in\mathbb{K}^{n}$, but for any
$n$-tuples $a$). This convention does not conflict with Convention
\ref{conv.gammai}, because both conventions define $\gamma_{i}$ in the same
way when $\gamma$ is an $n$-tuple and $i\in\left\{  1,2,\ldots,n\right\}  $
(whereas the latter convention does not define $\gamma_{i}$ for any other
values of $i$).

Convention \ref{conv.bir.peri} does conflict with our old convention (from
Section \ref{sect.not}) to identify partitions with finite tuples: Indeed, if
we let $\gamma$ be the $n$-tuple $\left(  \underbrace{1,1,\ldots,1}_{n\text{
times}}\right)  $, then Convention \ref{conv.bir.peri} yields $\gamma
_{n+1}=\gamma_{1}=1$ when we regard $\gamma$ as an $n$-tuple, but we get
$\gamma_{n+1}=0$ if we regard $\gamma$ as a partition. We shall resolve this
conflict by agreeing \textbf{not to identify partitions with finite tuples in
Subsection \ref{subsect.fF}}. (Thus, in particular, we will not identify a
nonnegative snake $\left(  \mu_{1},\mu_{2},\ldots,\mu_{n}\right)
\in\mathbb{Z}^{n}$ with its corresponding partition $\left(  \mu_{1},\mu
_{2},\ldots,\mu_{n},0,0,0,\ldots\right)  \in\operatorname*{Par}\left[
n\right]  $.)
\end{convention}

Let us now apply the results of Section \ref{sect.bir}. The abelian group
$\left(  \mathbb{Z},+,0\right)  $ of integers is totally ordered (in the usual
way). Thus, Example \ref{exa.semifield.mintrop} (applied to $\left(
\mathbb{A},\ast,e\right)  =\left(  \mathbb{Z},+,0\right)  $) shows that there
is a semifield $\left(  \mathbb{Z},\min,+,0\right)  $ (that is, a semifield
with ground set $\mathbb{Z}$, addition $\min$, multiplication $+$ and unity
$0$), called the \emph{min tropical semifield} of $\left(  \mathbb{Z}%
,+,0\right)  $. We have the following little dictionary between various
operations on this semifield $\left(  \mathbb{Z},\min,+,0\right)  $ and
familiar operations on integers:

\begin{itemize}
\item The addition operation of the semifield $\left(  \mathbb{Z}%
,\min,+,0\right)  $ is the binary operation $\min$ on $\mathbb{Z}$. That is,
for any $a,b\in\mathbb{Z}$, the sum $a+b$ understood with respect to the
semifield $\left(  \mathbb{Z},\min,+,0\right)  $ is precisely the integer
$\min\left\{  a,b\right\}  $.

\item Thus, (nonempty) finite sums in the semifield $\left(  \mathbb{Z}%
,\min,+,0\right)  $ are minima of finite sets of integers. That is, if
$r\in\mathbb{N}$, and if $a_{0},a_{1},\ldots,a_{r}$ are any $r+1$ integers,
then the sum $\sum_{k=0}^{r}a_{k}$ understood with respect to the semifield
$\left(  \mathbb{Z},\min,+,0\right)  $ is $\min\left\{  a_{0},a_{1}%
,\ldots,a_{r}\right\}  =\min\left\{  a_{k}\ \mid\ k\in\left\{  0,1,\ldots
,r\right\}  \right\}  $.

\item Furthermore, the multiplication operation of the semifield $\left(
\mathbb{Z},\min,+,0\right)  $ is the addition $+$ of integers. That is, for
any $a,b\in\mathbb{Z}$, the product $ab$ understood with respect to the
semifield $\left(  \mathbb{Z},\min,+,0\right)  $ is precisely the sum $a+b$
understood with respect to the integer ring $\mathbb{Z}$. Meanwhile, the unity
of the semifield $\left(  \mathbb{Z},\min,+,0\right)  $ is the integer $0$.

\item Thus, the division operation of the semifield $\left(  \mathbb{Z}%
,\min,+,0\right)  $ is the subtraction $-$ of integers. That is, for any
$a,b\in\mathbb{Z}$, the quotient $\dfrac{a}{b}$ understood with respect to the
semifield $\left(  \mathbb{Z},\min,+,0\right)  $ is precisely the difference
$a-b$ understood with respect to the integer ring $\mathbb{Z}$.

\item For the same reason, squaring an element of the semifield $\left(
\mathbb{Z},\min,+,0\right)  $ is tantamount to doubling it as an integer. That
is, for any $a\in\mathbb{Z}$, the square $a^{2}$ understood with respect to
the semifield $\left(  \mathbb{Z},\min,+,0\right)  $ is the product $2a$
understood with respect to the integer ring $\mathbb{Z}$.

\item For the same reason, taking reciprocals in the semifield $\left(
\mathbb{Z},\min,+,0\right)  $ is tantamount to negation of integers. That is,
for any $a\in\mathbb{Z}$, the reciprocal $\dfrac{1}{a}$ understood with
respect to the semifield $\left(  \mathbb{Z},\min,+,0\right)  $ is the integer
$-a$ understood with respect to the integer ring $\mathbb{Z}$.

\item For the same reason, finite products in the semifield $\left(
\mathbb{Z},\min,+,0\right)  $ are sums of integers. That is, if $r\in
\mathbb{N}$, and if $a_{1},a_{2},\ldots,a_{r}$ are any $r$ integers, then the
product $\prod_{k=1}^{r}a_{k}$ understood with respect to the semifield
$\left(  \mathbb{Z},\min,+,0\right)  $ is the sum $\sum_{k=1}^{r}a_{k}$
understood with respect to the integer ring $\mathbb{Z}$.
\end{itemize}

Thus, applying Definition \ref{def.f} to $\mathbb{K}=\left(  \mathbb{Z}%
,\min,+,0\right)  $ (and renaming everything\footnote{Namely, we are
\par
\begin{itemize}
\item renaming the (fixed) $n$-tuple $u$ as $\mu$;
\par
\item renaming the (variable) $n$-tuple $x$ as $\gamma$ (in order to avoid a
clash with the variables $x_{1},x_{2},\ldots,x_{n}$);
\par
\item renaming the elements $t_{r,j}$ as $\tau_{r,j}$;
\par
\item renaming the $n$-tuple $y$ as $\eta$.
\end{itemize}
\par
{}}), we obtain the following:

\begin{definition}
\label{def.fZ}Fix any $n$-tuple $\mu\in\mathbb{Z}^{n}$.

We define a map $\mathbf{f}_{\mu}:\mathbb{Z}^{n}\rightarrow\mathbb{Z}^{n}$ as follows:

Let $\gamma\in\mathbb{Z}^{n}$ be an $n$-tuple. For each $j\in\mathbb{Z}$ and
$r\in\mathbb{N}$, define an element $\tau_{r,j}\in\mathbb{Z}$ by
\begin{align*}
\tau_{r,j}  &  =\min\left\{  \underbrace{\gamma_{j+1}+\gamma_{j+2}%
+\cdots+\gamma_{j+k}}_{=\sum_{i=1}^{k}\gamma_{j+i}}+\underbrace{\mu
_{j+k+1}+\mu_{j+k+2}+\cdots+\mu_{j+r}}_{=\sum_{i=k+1}^{r}\mu_{j+i}}\right. \\
&  \ \ \ \ \ \ \ \ \ \ \ \ \ \ \ \ \ \ \ \ \left.  \ \mid\ k\in\left\{
0,1,\ldots,r\right\}  \right\}  .
\end{align*}
Define $\eta\in\mathbb{Z}^{n}$ by setting%
\[
\eta_{i}=\mu_{i}+\left(  \mu_{i-1}+\tau_{n-1,i-1}\right)  -\left(
\gamma_{i+1}+\tau_{n-1,i+1}\right)  \ \ \ \ \ \ \ \ \ \ \text{for each }%
i\in\left\{  1,2,\ldots,n\right\}  .
\]
Set $\mathbf{f}_{\mu}\left(  \gamma\right)  =\eta$.
\end{definition}

Applying Theorem \ref{thm.f.full} to $\mathbb{K}=\left(  \mathbb{Z}%
,\min,+,0\right)  $ (and renaming everything\footnote{Namely, we are
\par
\begin{itemize}
\item renaming the (fixed) $n$-tuple $u$ as $\mu$;
\par
\item renaming the (variable) $n$-tuple $x$ as $\gamma$;
\par
\item renaming the $n$-tuple $y$ as $\eta$.
\end{itemize}
\par
{}}, and using our above dictionary), we thus obtain the following:

\begin{theorem}
\label{thm.fZ.full}Fix any $n$-tuple $\mu\in\mathbb{Z}^{n}$.

\begin{enumerate}
\item[\textbf{(a)}] The map $\mathbf{f}_{\mu}$ is an involution (i.e., we have
$\mathbf{f}_{\mu}\circ\mathbf{f}_{\mu}=\operatorname*{id}$).

\item[\textbf{(b)}] Let $\gamma\in\mathbb{Z}^{n}$ and $\eta\in\mathbb{Z}^{n}$
be such that $\eta=\mathbf{f}_{\mu}\left(  \gamma\right)  $. Then,%
\[
\left(  \eta_{1}+\eta_{2}+\cdots+\eta_{n}\right)  +\left(  \gamma_{1}%
+\gamma_{2}+\cdots+\gamma_{n}\right)  =2\left(  \mu_{1}+\mu_{2}+\cdots+\mu
_{n}\right)  .
\]

\item[\textbf{(c)}] Let $\gamma\in\mathbb{Z}^{n}$ and $\eta\in\mathbb{Z}^{n}$
be such that $\eta=\mathbf{f}_{\mu}\left(  \gamma\right)  $. Then,%
\[
\min\left\{  \mu_{i},\gamma_{i}\right\}  +\min\left\{  -\mu_{i+1}%
,-\gamma_{i+1}\right\}  =\min\left\{  \mu_{i},\eta_{i}\right\}  +\min\left\{
-\mu_{i+1},-\eta_{i+1}\right\}
\]
for each $i\in\mathbb{Z}$.

\item[\textbf{(d)}] Let $\gamma\in\mathbb{Z}^{n}$ and $\eta\in\mathbb{Z}^{n}$
be such that $\eta=\mathbf{f}_{\mu}\left(  \gamma\right)  $. Then,%
\[
\sum_{i=1}^{n}\left(  \min\left\{  \mu_{i},\gamma_{i}\right\}  -\gamma
_{i}\right)  =\sum_{i=1}^{n}\left(  \min\left\{  \mu_{i},\eta_{i}\right\}
-\mu_{i}\right)  .
\]

\end{enumerate}
\end{theorem}

We obtain the following corollaries from Theorem \ref{thm.fZ.full}:

\begin{corollary}
\label{cor.thm.fZ.full2}Fix any $n$-tuple $\mu\in\mathbb{Z}^{n}$. Let
$\gamma\in\mathbb{Z}^{n}$ and $\eta\in\mathbb{Z}^{n}$ be such that
$\eta=\mathbf{f}_{\mu}\left(  \gamma\right)  $. Then:

\begin{enumerate}
\item[\textbf{(a)}] We have $\left\vert \eta\right\vert -\left\vert
\mu\right\vert =\left\vert \mu\right\vert -\left\vert \gamma\right\vert $.

\item[\textbf{(b)}] We have%
\[
\min\left\{  \mu_{i},\eta_{i}\right\}  -\min\left\{  \mu_{i},\gamma
_{i}\right\}  =\max\left\{  \mu_{i+1},\eta_{i+1}\right\}  -\max\left\{
\mu_{i+1},\gamma_{i+1}\right\}
\]
for each $i\in\left\{  1,2,\ldots,n-1\right\}  $.

\item[\textbf{(c)}] We have%
\[
\sum_{i=1}^{n}\left(  \mu_{i}-\min\left\{  \mu_{i},\eta_{i}\right\}
+\min\left\{  \mu_{i},\gamma_{i}\right\}  \right)  =\sum_{i=1}^{n}\gamma_{i}.
\]

\item[\textbf{(d)}] We have $\gamma=\mathbf{f}_{\mu}\left(  \eta\right)  $.
\end{enumerate}
\end{corollary}

\begin{proof}
[Proof of Corollary \ref{cor.thm.fZ.full2}.]\textbf{(a)} Theorem
\ref{thm.fZ.full} \textbf{(b)} yields%
\[
\left(  \eta_{1}+\eta_{2}+\cdots+\eta_{n}\right)  +\left(  \gamma_{1}%
+\gamma_{2}+\cdots+\gamma_{n}\right)  =2\left(  \mu_{1}+\mu_{2}+\cdots+\mu
_{n}\right)  .
\]
In view of the equalities%
\begin{align*}
\left\vert \eta\right\vert  &  =\eta_{1}+\eta_{2}+\cdots+\eta_{n}%
,\ \ \ \ \ \ \ \ \ \ \left\vert \gamma\right\vert =\gamma_{1}+\gamma
_{2}+\cdots+\gamma_{n}\\
\text{and}\ \ \ \ \ \ \ \ \ \ \left\vert \mu\right\vert  &  =\mu_{1}+\mu
_{2}+\cdots+\mu_{n},
\end{align*}
we can rewrite this as $\left\vert \eta\right\vert +\left\vert \gamma
\right\vert =2\left\vert \mu\right\vert $. Equivalently, $\left\vert
\eta\right\vert -\left\vert \mu\right\vert =\left\vert \mu\right\vert
-\left\vert \gamma\right\vert $. This proves Corollary \ref{cor.thm.fZ.full2}
\textbf{(a)}.

\textbf{(b)} Let $i\in\left\{  1,2,\ldots,n-1\right\}  $.\ Then, Theorem
\ref{thm.fZ.full} \textbf{(c)} yields%
\[
\min\left\{  \mu_{i},\gamma_{i}\right\}  +\min\left\{  -\mu_{i+1}%
,-\gamma_{i+1}\right\}  =\min\left\{  \mu_{i},\eta_{i}\right\}  +\min\left\{
-\mu_{i+1},-\eta_{i+1}\right\}  .
\]
In view of $\min\left\{  -\mu_{i+1},-\gamma_{i+1}\right\}  =-\max\left\{
\mu_{i+1},\gamma_{i+1}\right\}  $ and $\min\left\{  -\mu_{i+1},-\eta
_{i+1}\right\}  =-\max\left\{  \mu_{i+1},\eta_{i+1}\right\}  $, we can rewrite
this as%
\[
\min\left\{  \mu_{i},\gamma_{i}\right\}  +\left(  -\max\left\{  \mu
_{i+1},\gamma_{i+1}\right\}  \right)  =\min\left\{  \mu_{i},\eta_{i}\right\}
+\left(  -\max\left\{  \mu_{i+1},\eta_{i+1}\right\}  \right)  .
\]
In other words,%
\[
\min\left\{  \mu_{i},\gamma_{i}\right\}  -\max\left\{  \mu_{i+1},\gamma
_{i+1}\right\}  =\min\left\{  \mu_{i},\eta_{i}\right\}  -\max\left\{
\mu_{i+1},\eta_{i+1}\right\}  .
\]
Equivalently,%
\[
\min\left\{  \mu_{i},\eta_{i}\right\}  -\min\left\{  \mu_{i},\gamma
_{i}\right\}  =\max\left\{  \mu_{i+1},\eta_{i+1}\right\}  -\max\left\{
\mu_{i+1},\gamma_{i+1}\right\}  .
\]
This proves Corollary \ref{cor.thm.fZ.full2} \textbf{(b)}.

\textbf{(c)} Theorem \ref{thm.fZ.full} \textbf{(d)} yields%
\begin{equation}
\sum_{i=1}^{n}\left(  \min\left\{  \mu_{i},\gamma_{i}\right\}  -\gamma
_{i}\right)  =\sum_{i=1}^{n}\left(  \min\left\{  \mu_{i},\eta_{i}\right\}
-\mu_{i}\right)  . \label{pf.cor.thm.fZ.full2.c.1}%
\end{equation}
Now, we have
\begin{align*}
&  \sum_{i=1}^{n}\underbrace{\left(  \mu_{i}-\min\left\{  \mu_{i},\eta
_{i}\right\}  +\min\left\{  \mu_{i},\gamma_{i}\right\}  \right)  }%
_{=\min\left\{  \mu_{i},\gamma_{i}\right\}  -\left(  \min\left\{  \mu_{i}%
,\eta_{i}\right\}  -\mu_{i}\right)  }\\
&  =\sum_{i=1}^{n}\left(  \min\left\{  \mu_{i},\gamma_{i}\right\}  -\left(
\min\left\{  \mu_{i},\eta_{i}\right\}  -\mu_{i}\right)  \right) \\
&  =\sum_{i=1}^{n}\min\left\{  \mu_{i},\gamma_{i}\right\}  -\underbrace{\sum
_{i=1}^{n}\left(  \min\left\{  \mu_{i},\eta_{i}\right\}  -\mu_{i}\right)
}_{\substack{=\sum_{i=1}^{n}\left(  \min\left\{  \mu_{i},\gamma_{i}\right\}
-\gamma_{i}\right)  \\\text{(by (\ref{pf.cor.thm.fZ.full2.c.1}))}}}\\
&  =\sum_{i=1}^{n}\min\left\{  \mu_{i},\gamma_{i}\right\}  -\sum_{i=1}%
^{n}\left(  \min\left\{  \mu_{i},\gamma_{i}\right\}  -\gamma_{i}\right) \\
&  =\sum_{i=1}^{n}\underbrace{\left(  \min\left\{  \mu_{i},\gamma_{i}\right\}
-\left(  \min\left\{  \mu_{i},\gamma_{i}\right\}  -\gamma_{i}\right)  \right)
}_{=\gamma_{i}}=\sum_{i=1}^{n}\gamma_{i}.
\end{align*}
This proves Corollary \ref{cor.thm.fZ.full2} \textbf{(c)}.

\textbf{(d)} Theorem \ref{thm.fZ.full} \textbf{(a)} shows that $\mathbf{f}%
_{\mu}\circ\mathbf{f}_{\mu}=\operatorname*{id}$. But recall that
$\eta=\mathbf{f}_{\mu}\left(  \gamma\right)  $. Applying the map
$\mathbf{f}_{\mu}$ to both sides of this equality, we obtain%
\[
\mathbf{f}_{\mu}\left(  \eta\right)  =\mathbf{f}_{\mu}\left(  \mathbf{f}_{\mu
}\left(  \gamma\right)  \right)  =\underbrace{\left(  \mathbf{f}_{\mu}%
\circ\mathbf{f}_{\mu}\right)  }_{=\operatorname*{id}}\left(  \gamma\right)
=\gamma.
\]
This proves Corollary \ref{cor.thm.fZ.full2} \textbf{(d)}.
\end{proof}

We are now ready to prove the key lemma:

\begin{lemma}
\label{lem.fZ.lr}Fix any $n$-tuple $\mu\in\mathbb{Z}^{n}$. Let $\gamma
\in\mathbb{Z}^{n}$. Let $a,b\in\mathbb{Z}$. Then,%
\[
\left\vert R_{\mu,b,a}\left(  \mathbf{f}_{\mu}\left(  \gamma\right)  \right)
\right\vert =\left\vert R_{\mu,a,b}\left(  \gamma\right)  \right\vert .
\]

\end{lemma}

\begin{proof}
[Proof of Lemma \ref{lem.fZ.lr}.]Define $\eta\in\mathbb{Z}^{n}$ by
$\eta=\mathbf{f}_{\mu}\left(  \gamma\right)  $. We must then prove that
$\left\vert R_{\mu,b,a}\left(  \eta\right)  \right\vert =\left\vert
R_{\mu,a,b}\left(  \gamma\right)  \right\vert $.

We know that $R_{\mu,a,b}\left(  \gamma\right)  $ is the set of all snakes
$\nu$ satisfying the four conditions%
\[
\mu\rightharpoonup\nu\ \ \ \ \ \ \ \ \ \ \text{and}%
\ \ \ \ \ \ \ \ \ \ \left\vert \mu\right\vert -\left\vert \nu\right\vert
=a\ \ \ \ \ \ \ \ \ \ \text{and}\ \ \ \ \ \ \ \ \ \ \gamma\rightharpoonup
\nu\ \ \ \ \ \ \ \ \ \ \text{and}\ \ \ \ \ \ \ \ \ \ \left\vert \gamma
\right\vert -\left\vert \nu\right\vert =b.
\]
Likewise, $R_{\mu,b,a}\left(  \eta\right)  $ is the set of all snakes $\nu$
satisfying the four conditions%
\[
\mu\rightharpoonup\nu\ \ \ \ \ \ \ \ \ \ \text{and}%
\ \ \ \ \ \ \ \ \ \ \left\vert \mu\right\vert -\left\vert \nu\right\vert
=b\ \ \ \ \ \ \ \ \ \ \text{and}\ \ \ \ \ \ \ \ \ \ \eta\rightharpoonup
\nu\ \ \ \ \ \ \ \ \ \ \text{and}\ \ \ \ \ \ \ \ \ \ \left\vert \eta
\right\vert -\left\vert \nu\right\vert =a.
\]

Now, fix $\nu\in R_{\mu,a,b}\left(  \gamma\right)  $. Thus, $\nu$ is a snake
satisfying the four conditions%
\[
\mu\rightharpoonup\nu\ \ \ \ \ \ \ \ \ \ \text{and}%
\ \ \ \ \ \ \ \ \ \ \left\vert \mu\right\vert -\left\vert \nu\right\vert
=a\ \ \ \ \ \ \ \ \ \ \text{and}\ \ \ \ \ \ \ \ \ \ \gamma\rightharpoonup
\nu\ \ \ \ \ \ \ \ \ \ \text{and}\ \ \ \ \ \ \ \ \ \ \left\vert \gamma
\right\vert -\left\vert \nu\right\vert =b
\]
(by the definition of $R_{\mu,a,b}\left(  \gamma\right)  $). In particular, we
have $\mu\rightharpoonup\nu$. In other words, we have%
\begin{equation}
\mu_{1}\geq\nu_{1}\geq\mu_{2}\geq\nu_{2}\geq\cdots\geq\mu_{n}\geq\nu_{n}
\label{pf.lem.fZ.lr.3a}%
\end{equation}
(by the definition of \textquotedblleft$\mu\rightharpoonup\nu$%
\textquotedblright). Likewise, from $\gamma\rightharpoonup\nu$, we obtain%
\begin{equation}
\gamma_{1}\geq\nu_{1}\geq\gamma_{2}\geq\nu_{2}\geq\cdots\geq\gamma_{n}\geq
\nu_{n}. \label{pf.lem.fZ.lr.3b}%
\end{equation}

We define an $n$-tuple $\zeta\in\mathbb{Z}^{n}$ by setting
\[
\zeta_{i}=\min\left\{  \mu_{i},\eta_{i}\right\}  -\min\left\{  \mu_{i}%
,\gamma_{i}\right\}  +\nu_{i}\ \ \ \ \ \ \ \ \ \ \text{for each }i\in\left\{
1,2,\ldots,n\right\}  .
\]
We shall prove that $\zeta\in R_{\mu,b,a}\left(  \eta\right)  $.

We begin by proving several auxiliary claims:

\begin{statement}
\textit{Claim 1:} We have $\min\left\{  \mu_{i},\eta_{i}\right\}  \geq
\zeta_{i}$ for each $i\in\left\{  1,2,\ldots,n\right\}  $.
\end{statement}

[\textit{Proof of Claim 1:} Let $i\in\left\{  1,2,\ldots,n\right\}  $. From
(\ref{pf.lem.fZ.lr.3a}), we obtain $\mu_{i}\geq\nu_{i}$. From
(\ref{pf.lem.fZ.lr.3b}), we obtain $\gamma_{i}\geq\nu_{i}$. Combining $\mu
_{i}\geq\nu_{i}$ and $\gamma_{i}\geq\nu_{i}$, we obtain $\min\left\{  \mu
_{i},\gamma_{i}\right\}  \geq\nu_{i}$. Now, the definition of $\zeta$ yields%
\[
\zeta_{i}=\min\left\{  \mu_{i},\eta_{i}\right\}  -\underbrace{\min\left\{
\mu_{i},\gamma_{i}\right\}  }_{\geq\nu_{i}}+\nu_{i}\leq\min\left\{  \mu
_{i},\eta_{i}\right\}  -\nu_{i}+\nu_{i}=\min\left\{  \mu_{i},\eta_{i}\right\}
.
\]
In other words, $\min\left\{  \mu_{i},\eta_{i}\right\}  \geq\zeta_{i}$. This
proves Claim 1.]

\begin{statement}
\textit{Claim 2:} We have $\zeta_{i}\geq\max\left\{  \mu_{i+1},\eta
_{i+1}\right\}  $ for each $i\in\left\{  1,2,\ldots,n-1\right\}  $.
\end{statement}

[\textit{Proof of Claim 2:} Let $i\in\left\{  1,2,\ldots,n-1\right\}  $. From
(\ref{pf.lem.fZ.lr.3a}), we obtain $\nu_{i}\geq\mu_{i+1}$. From
(\ref{pf.lem.fZ.lr.3b}), we obtain $\nu_{i}\geq\gamma_{i+1}$. Combining
$\nu_{i}\geq\mu_{i+1}$ and $\nu_{i}\geq\gamma_{i+1}$, we obtain $\nu_{i}%
\geq\max\left\{  \mu_{i+1},\gamma_{i+1}\right\}  $. Now, the definition of
$\zeta$ yields%
\begin{align*}
\zeta_{i}  &  =\underbrace{\min\left\{  \mu_{i},\eta_{i}\right\}
-\min\left\{  \mu_{i},\gamma_{i}\right\}  }_{\substack{=\max\left\{  \mu
_{i+1},\eta_{i+1}\right\}  -\max\left\{  \mu_{i+1},\gamma_{i+1}\right\}
\\\text{(by Corollary \ref{cor.thm.fZ.full2} \textbf{(b)})}}}+\underbrace{\nu
_{i}}_{\geq\max\left\{  \mu_{i+1},\gamma_{i+1}\right\}  }\\
&  \geq\max\left\{  \mu_{i+1},\eta_{i+1}\right\}  -\max\left\{  \mu
_{i+1},\gamma_{i+1}\right\}  +\max\left\{  \mu_{i+1},\gamma_{i+1}\right\}
=\max\left\{  \mu_{i+1},\eta_{i+1}\right\}  .
\end{align*}
This proves Claim 2.]

\begin{statement}
\textit{Claim 3:} The $n$-tuple $\zeta$ is a snake and satisfies
$\mu\rightharpoonup\zeta$ and $\eta\rightharpoonup\zeta$.
\end{statement}

[\textit{Proof of Claim 3:} For each $i\in\left\{  1,2,\ldots,n\right\}  $, we
have $\mu_{i}\geq\zeta_{i}$ (since $\mu_{i}\geq\min\left\{  \mu_{i},\eta
_{i}\right\}  \geq\zeta_{i}$ (by Claim 1)). For each $i\in\left\{
1,2,\ldots,n-1\right\}  $, we have $\zeta_{i}\geq\mu_{i+1}$ (since Claim 2
yields $\zeta_{i}\geq\max\left\{  \mu_{i+1},\eta_{i+1}\right\}  \geq\mu_{i+1}%
$). Combining the preceding two sentences, we obtain%
\[
\mu_{1}\geq\zeta_{1}\geq\mu_{2}\geq\zeta_{2}\geq\cdots\geq\mu_{n}\geq\zeta
_{n}.
\]
In other words, $\mu\rightharpoonup\zeta$ (by the definition of
\textquotedblleft$\mu\rightharpoonup\zeta$\textquotedblright).

For each $i\in\left\{  1,2,\ldots,n\right\}  $, we have $\eta_{i}\geq\zeta
_{i}$ (since $\eta_{i}\geq\min\left\{  \mu_{i},\eta_{i}\right\}  \geq\zeta
_{i}$ (by Claim 1)). For each $i\in\left\{  1,2,\ldots,n-1\right\}  $, we have
$\zeta_{i}\geq\eta_{i+1}$ (since Claim 2 yields $\zeta_{i}\geq\max\left\{
\mu_{i+1},\eta_{i+1}\right\}  \geq\eta_{i+1}$). Combining the preceding two
sentences, we obtain%
\[
\eta_{1}\geq\zeta_{1}\geq\eta_{2}\geq\zeta_{2}\geq\cdots\geq\eta_{n}\geq
\zeta_{n}.
\]
In other words, $\eta\rightharpoonup\zeta$ (by the definition of
\textquotedblleft$\eta\rightharpoonup\zeta$\textquotedblright). Hence,
Proposition \ref{prop.horstr.easies} \textbf{(a)} (applied to $\eta$ and
$\zeta$ instead of $\mu$ and $\lambda$) yields that both $\zeta$ and $\eta$
are snakes. Hence, $\zeta$ is a snake. This completes the proof of Claim 3.]

\begin{statement}
\textit{Claim 4:} We have $\left\vert \mu\right\vert -\left\vert
\zeta\right\vert =b$ and $\left\vert \eta\right\vert -\left\vert
\zeta\right\vert =a$.
\end{statement}

[\textit{Proof of Claim 4:} We have%
\begin{align*}
\left\vert \mu\right\vert  &  =\mu_{1}+\mu_{2}+\cdots+\mu_{n}=\sum_{i=1}%
^{n}\mu_{i}\ \ \ \ \ \ \ \ \ \ \text{and}\\
\left\vert \zeta\right\vert  &  =\zeta_{1}+\zeta_{2}+\cdots+\zeta_{n}%
=\sum_{i=1}^{n}\zeta_{i}.
\end{align*}

\begin{vershort}
\noindent Subtracting these two equalities from one another, we find%
\begin{align*}
\left\vert \mu\right\vert -\left\vert \zeta\right\vert  &  =\sum_{i=1}^{n}%
\mu_{i}-\sum_{i=1}^{n}\underbrace{\zeta_{i}}_{\substack{=\min\left\{  \mu
_{i},\eta_{i}\right\}  -\min\left\{  \mu_{i},\gamma_{i}\right\}  +\nu
_{i}\\\text{(by the definition of }\zeta\text{)}}}\\
&  =\sum_{i=1}^{n}\mu_{i}-\sum_{i=1}^{n}\left(  \min\left\{  \mu_{i},\eta
_{i}\right\}  -\min\left\{  \mu_{i},\gamma_{i}\right\}  +\nu_{i}\right) \\
&  =\sum_{i=1}^{n}\underbrace{\left(  \mu_{i}-\left(  \min\left\{  \mu
_{i},\eta_{i}\right\}  -\min\left\{  \mu_{i},\gamma_{i}\right\}  +\nu
_{i}\right)  \right)  }_{=\left(  \mu_{i}-\min\left\{  \mu_{i},\eta
_{i}\right\}  +\min\left\{  \mu_{i},\gamma_{i}\right\}  \right)  -\nu_{i}}\\
&  =\sum_{i=1}^{n}\left(  \left(  \mu_{i}-\min\left\{  \mu_{i},\eta
_{i}\right\}  +\min\left\{  \mu_{i},\gamma_{i}\right\}  \right)  -\nu
_{i}\right) \\
&  =\underbrace{\sum_{i=1}^{n}\left(  \mu_{i}-\min\left\{  \mu_{i},\eta
_{i}\right\}  +\min\left\{  \mu_{i},\gamma_{i}\right\}  \right)
}_{\substack{=\sum_{i=1}^{n}\gamma_{i}\\\text{(by Corollary
\ref{cor.thm.fZ.full2} \textbf{(c)})}}}-\sum_{i=1}^{n}\nu_{i}\\
&  =\underbrace{\sum_{i=1}^{n}\gamma_{i}}_{=\gamma_{1}+\gamma_{2}%
+\cdots+\gamma_{n}=\left\vert \gamma\right\vert }-\underbrace{\sum_{i=1}%
^{n}\nu_{i}}_{=\nu_{1}+\nu_{2}+\cdots+\nu_{n}=\left\vert \nu\right\vert
}=\left\vert \gamma\right\vert -\left\vert \nu\right\vert =b.
\end{align*}

\end{vershort}

\begin{verlong}
\noindent Subtracting these two equalities from one another, we find%
\begin{align*}
\left\vert \mu\right\vert -\left\vert \zeta\right\vert  &  =\sum_{i=1}^{n}%
\mu_{i}-\sum_{i=1}^{n}\underbrace{\zeta_{i}}_{\substack{=\min\left\{  \mu
_{i},\eta_{i}\right\}  -\min\left\{  \mu_{i},\gamma_{i}\right\}  +\nu
_{i}\\\text{(by the definition of }\zeta\text{)}}}\\
&  =\sum_{i=1}^{n}\mu_{i}-\sum_{i=1}^{n}\left(  \min\left\{  \mu_{i},\eta
_{i}\right\}  -\min\left\{  \mu_{i},\gamma_{i}\right\}  +\nu_{i}\right) \\
&  =\sum_{i=1}^{n}\underbrace{\left(  \mu_{i}-\left(  \min\left\{  \mu
_{i},\eta_{i}\right\}  -\min\left\{  \mu_{i},\gamma_{i}\right\}  +\nu
_{i}\right)  \right)  }_{=\left(  \mu_{i}-\min\left\{  \mu_{i},\eta
_{i}\right\}  +\min\left\{  \mu_{i},\gamma_{i}\right\}  \right)  -\nu_{i}}\\
&  =\sum_{i=1}^{n}\left(  \left(  \mu_{i}-\min\left\{  \mu_{i},\eta
_{i}\right\}  +\min\left\{  \mu_{i},\gamma_{i}\right\}  \right)  -\nu
_{i}\right) \\
&  =\underbrace{\sum_{i=1}^{n}\left(  \mu_{i}-\min\left\{  \mu_{i},\eta
_{i}\right\}  +\min\left\{  \mu_{i},\gamma_{i}\right\}  \right)
}_{\substack{=\sum_{i=1}^{n}\gamma_{i}\\\text{(by Corollary
\ref{cor.thm.fZ.full2} \textbf{(c)})}}}-\sum_{i=1}^{n}\nu_{i}\\
&  =\underbrace{\sum_{i=1}^{n}\gamma_{i}}_{\substack{=\left\vert
\gamma\right\vert \\\text{(since }\left\vert \gamma\right\vert =\gamma
_{1}+\gamma_{2}+\cdots+\gamma_{n}=\sum_{i=1}^{n}\gamma_{i}\text{)}%
}}-\underbrace{\sum_{i=1}^{n}\nu_{i}}_{\substack{=\left\vert \nu\right\vert
\\\text{(since }\left\vert \nu\right\vert =\nu_{1}+\nu_{2}+\cdots+\nu_{n}%
=\sum_{i=1}^{n}\nu_{i}\text{)}}}=\left\vert \gamma\right\vert -\left\vert
\nu\right\vert =b.
\end{align*}

\end{verlong}

\noindent Furthermore,%
\[
\left\vert \mu\right\vert -\left\vert \gamma\right\vert =\underbrace{\left(
\left\vert \mu\right\vert -\left\vert \nu\right\vert \right)  }_{=a}%
-\underbrace{\left(  \left\vert \gamma\right\vert -\left\vert \nu\right\vert
\right)  }_{=b}=a-b
\]
and%
\[
\left\vert \eta\right\vert -\left\vert \zeta\right\vert =\underbrace{\left(
\left\vert \eta\right\vert -\left\vert \mu\right\vert \right)  }%
_{\substack{=\left\vert \mu\right\vert -\left\vert \gamma\right\vert
\\\text{(by Corollary \ref{cor.thm.fZ.full2} \textbf{(a)})}}%
}+\underbrace{\left(  \left\vert \mu\right\vert -\left\vert \zeta\right\vert
\right)  }_{=b}=\underbrace{\left\vert \mu\right\vert -\left\vert
\gamma\right\vert }_{=a-b}+b=\left(  a-b\right)  +b=a.
\]
Thus, Claim 4 is proven.]

We have now shown (in Claim 3 and Claim 4) that $\zeta$ is a snake satisfying
the four conditions%
\[
\mu\rightharpoonup\zeta\ \ \ \ \ \ \ \ \ \ \text{and}%
\ \ \ \ \ \ \ \ \ \ \left\vert \mu\right\vert -\left\vert \zeta\right\vert
=b\ \ \ \ \ \ \ \ \ \ \text{and}\ \ \ \ \ \ \ \ \ \ \eta\rightharpoonup
\zeta\ \ \ \ \ \ \ \ \ \ \text{and}\ \ \ \ \ \ \ \ \ \ \left\vert
\eta\right\vert -\left\vert \zeta\right\vert =a.
\]
In other words, $\zeta\in R_{\mu,b,a}\left(  \eta\right)  $ (by the definition
of $R_{\mu,b,a}\left(  \eta\right)  $).

Forget that we fixed $\nu$. Thus, for each $\nu\in R_{\mu,a,b}\left(
\gamma\right)  $, we have constructed a $\zeta\in R_{\mu,b,a}\left(
\eta\right)  $. Let us denote this $\zeta$ by $\widetilde{\nu}$. We thus have
defined a map%
\begin{align*}
R_{\mu,a,b}\left(  \gamma\right)   &  \rightarrow R_{\mu,b,a}\left(
\eta\right)  ,\\
\nu &  \mapsto\widetilde{\nu}.
\end{align*}
Let us denote this map by $\mathbf{g}_{\gamma,a,b}$. Its definition shows that%
\begin{align}
\left(  \underbrace{\mathbf{g}_{\gamma,a,b}\left(  \nu\right)  }%
_{=\widetilde{\nu}}\right)  _{i}  &  =\widetilde{\nu}_{i}=\min\left\{  \mu
_{i},\eta_{i}\right\}  -\min\left\{  \mu_{i},\gamma_{i}\right\}  +\nu
_{i}\label{pf.lem.fZ.lr.end1}\\
&  \ \ \ \ \ \ \ \ \ \ \left(
\begin{array}
[c]{c}%
\text{since }\widetilde{\nu}\text{ was defined as the }n\text{-tuple }%
\zeta\text{, whose entries}\\
\text{are given by }\zeta_{i}=\min\left\{  \mu_{i},\eta_{i}\right\}
-\min\left\{  \mu_{i},\gamma_{i}\right\}  +\nu_{i}%
\end{array}
\right) \nonumber
\end{align}
for each $\nu\in R_{\mu,a,b}\left(  \gamma\right)  $ and each $i\in\left\{
1,2,\ldots,n\right\}  $.

However, from $\eta=\mathbf{f}_{\mu}\left(  \gamma\right)  $, we obtain
$\gamma=\mathbf{f}_{\mu}\left(  \eta\right)  $ (by Corollary
\ref{cor.thm.fZ.full2} \textbf{(d)}). The relation between $\gamma$ and $\eta$
is thus symmetric. Hence, in the same way as we defined a map $\mathbf{g}%
_{\gamma,a,b}:R_{\mu,a,b}\left(  \gamma\right)  \rightarrow R_{\mu,b,a}\left(
\eta\right)  $, we can define a map $\mathbf{g}_{\eta,b,a}:R_{\mu,b,a}\left(
\eta\right)  \rightarrow R_{\mu,a,b}\left(  \gamma\right)  $ (by repeating the
above construction of $\mathbf{g}_{\gamma,a,b}$ with $b$, $a$, $\eta$ and
$\gamma$ taking the roles of $a$, $b$, $\gamma$ and $\eta$, respectively). The
resulting map $\mathbf{g}_{\eta,b,a}$ satisfies%
\begin{equation}
\left(  \mathbf{g}_{\eta,b,a}\left(  \nu\right)  \right)  _{i}=\min\left\{
\mu_{i},\gamma_{i}\right\}  -\min\left\{  \mu_{i},\eta_{i}\right\}  +\nu_{i}
\label{pf.lem.fZ.lr.end2}%
\end{equation}
for each $\nu\in R_{\mu,b,a}\left(  \eta\right)  $ and each $i\in\left\{
1,2,\ldots,n\right\}  $. (Indeed, this can be proved just as we proved
(\ref{pf.lem.fZ.lr.end1}), but with $b$, $a$, $\eta$ and $\gamma$ taking the
roles of $a$, $b$, $\gamma$ and $\eta$.)

\begin{vershort}
Now it is easy to see (using (\ref{pf.lem.fZ.lr.end1}) and
(\ref{pf.lem.fZ.lr.end2})) that the two maps $\mathbf{g}_{\gamma,a,b}$ and
$\mathbf{g}_{\eta,b,a}$ are mutually inverse. Hence, these two maps are
invertible, i.e., are bijections.
\end{vershort}

\begin{verlong}
Now, it is easy to see that $\mathbf{g}_{\gamma,a,b}\circ\mathbf{g}_{\eta
,b,a}=\operatorname*{id}$\ \ \ \ \footnote{\textit{Proof.} Let $\nu\in
R_{\mu,b,a}\left(  \eta\right)  $. Thus, $\mathbf{g}_{\eta,b,a}\left(
\nu\right)  \in R_{\mu,a,b}\left(  \gamma\right)  $. Now, both $\left(
\mathbf{g}_{\gamma,a,b}\circ\mathbf{g}_{\eta,b,a}\right)  \left(  \nu\right)
$ and $\nu$ are $n$-tuples. Furthermore, for each $i\in\left\{  1,2,\ldots
,n\right\}  $, we have%
\begin{align*}
\left(  \underbrace{\left(  \mathbf{g}_{\gamma,a,b}\circ\mathbf{g}_{\eta
,b,a}\right)  \left(  \nu\right)  }_{=\mathbf{g}_{\gamma,a,b}\left(
\mathbf{g}_{\eta,b,a}\left(  \nu\right)  \right)  }\right)  _{i}  &  =\left(
\mathbf{g}_{\gamma,a,b}\left(  \mathbf{g}_{\eta,b,a}\left(  \nu\right)
\right)  \right)  _{i}\\
&  =\min\left\{  \mu_{i},\eta_{i}\right\}  -\min\left\{  \mu_{i},\gamma
_{i}\right\}  +\underbrace{\left(  \mathbf{g}_{\eta,b,a}\left(  \nu\right)
\right)  _{i}}_{\substack{=\min\left\{  \mu_{i},\gamma_{i}\right\}
-\min\left\{  \mu_{i},\eta_{i}\right\}  +\nu_{i}\\\text{(by
(\ref{pf.lem.fZ.lr.end2}))}}}\\
&  \ \ \ \ \ \ \ \ \ \ \left(  \text{by (\ref{pf.lem.fZ.lr.end1}), applied to
}\mathbf{g}_{\eta,b,a}\left(  \nu\right)  \text{ instead of }\nu\right) \\
&  =\min\left\{  \mu_{i},\eta_{i}\right\}  -\min\left\{  \mu_{i},\gamma
_{i}\right\}  +\min\left\{  \mu_{i},\gamma_{i}\right\}  -\min\left\{  \mu
_{i},\eta_{i}\right\}  +\nu_{i}\\
&  =\nu_{i}.
\end{align*}
In other words, the two $n$-tuples $\left(  \mathbf{g}_{\gamma,a,b}%
\circ\mathbf{g}_{\eta,b,a}\right)  \left(  \nu\right)  $ and $\nu$ agree in
each entry. Hence, $\left(  \mathbf{g}_{\gamma,a,b}\circ\mathbf{g}_{\eta
,b,a}\right)  \left(  \nu\right)  =\nu=\operatorname*{id}\left(  \nu\right)
$.
\par
Forget that we fixed $\nu$. We thus have shown that $\left(  \mathbf{g}%
_{\gamma,a,b}\circ\mathbf{g}_{\eta,b,a}\right)  \left(  \nu\right)
=\operatorname*{id}\left(  \nu\right)  $ for each $\nu\in R_{\mu,b,a}\left(
\eta\right)  $. In other words, $\mathbf{g}_{\gamma,a,b}\circ\mathbf{g}%
_{\eta,b,a}=\operatorname*{id}$.} and $\mathbf{g}_{\eta,b,a}\circ
\mathbf{g}_{\gamma,a,b}=\operatorname*{id}$\ \ \ \ \footnote{for similar
reasons}. Thus, the two maps $\mathbf{g}_{\gamma,a,b}$ and $\mathbf{g}%
_{\eta,b,a}$ are mutually inverse. Hence, these two maps are invertible, i.e.,
are bijections.
\end{verlong}

Thus, there exists a bijection from $R_{\mu,a,b}\left(  \gamma\right)  $ to
$R_{\mu,b,a}\left(  \eta\right)  $ (namely, $\mathbf{g}_{\gamma,a,b}$). This
yields $\left\vert R_{\mu,a,b}\left(  \gamma\right)  \right\vert =\left\vert
R_{\mu,b,a}\left(  \eta\right)  \right\vert =\left\vert R_{\mu,b,a}\left(
\mathbf{f}_{\mu}\left(  \gamma\right)  \right)  \right\vert $ (since
$\eta=\mathbf{f}_{\mu}\left(  \gamma\right)  $). This proves Lemma
\ref{lem.fZ.lr}.
\end{proof}

Having learned a lot about the map $\mathbf{f}_{\mu}$, let us now connect it
to the map $\varphi$ defined in Theorem \ref{thm.main}. For this, we shall use
the following lemma:

\begin{lemma}
\label{lem.fZ.back}Fix any $n$-tuple $\mu\in\mathbb{Z}^{n}$.

Let $\nu\in\mathbb{Z}^{n}$ be an $n$-tuple. For each $j\in\mathbb{Z}$, let
\begin{align*}
\tau_{j}  &  =\min\left\{  \underbrace{\nu_{j+1}+\nu_{j+2}+\cdots+\nu_{j+k}%
}_{=\sum_{i=1}^{k}\nu_{j+i}}+\underbrace{\mu_{j+k+1}+\mu_{j+k+2}+\cdots
+\mu_{j+n-1}}_{=\sum_{i=k+1}^{n-1}\mu_{j+i}}\right. \\
&  \ \ \ \ \ \ \ \ \ \ \ \ \ \ \ \ \ \ \ \ \left.  \ \mid\ k\in\left\{
0,1,\ldots,n-1\right\}  \right\}  .
\end{align*}
Let $\eta\in\mathbb{Z}^{n}$ be such that%
\[
\eta_{i}=\mu_{i}+\left(  \mu_{i-1}+\tau_{i-1}\right)  -\left(  \nu_{i+1}%
+\tau_{i+1}\right)  \ \ \ \ \ \ \ \ \ \ \text{for each }i\in\left\{
1,2,\ldots,n\right\}  .
\]
Then, $\mathbf{f}_{\mu}\left(  \nu\right)  =\eta$.
\end{lemma}

\begin{proof}
[Proof of Lemma \ref{lem.fZ.back}.]Lemma \ref{lem.fZ.back} is obtained when we
apply Lemma \ref{lem.f.back} to \newline$\mathbb{K}=\left(  \mathbb{Z}%
,\min,+,0\right)  $ (and rename everything\footnote{Namely, we are
\par
\begin{itemize}
\item renaming the (fixed) $n$-tuple $u$ as $\mu$;
\par
\item renaming the $n$-tuple $x$ as $\nu$;
\par
\item renaming the elements $q_{j}$ as $\tau_{j}$;
\par
\item renaming the $n$-tuple $z$ as $\eta$.
\end{itemize}
\par
{}}, and use our above dictionary again).
\end{proof}

We can now connect the map $\mathbf{f}_{\mu}$ with the map $\varphi$ from
Theorem \ref{thm.main}:

\begin{lemma}
\label{lem.fZ.phi}Let $a,b\in\mathbb{N}$. Fix any $n$-tuple $\mu\in
\mathbb{Z}^{n}$. Define a map $\varphi:\mathbb{Z}^{n}\rightarrow\mathbb{Z}%
^{n}$ as in Theorem \ref{thm.main}. Then,%
\[
\varphi\left(  \omega\right)  =\mathbf{f}_{\mu}\left(  \omega-a\right)
+b\ \ \ \ \ \ \ \ \ \ \text{for each }\omega\in\mathbb{Z}^{n}.
\]

\end{lemma}

\begin{proof}
[Proof of Lemma \ref{lem.fZ.phi}.]Let $\omega\in\mathbb{Z}^{n}$.

\begin{vershort}
Define an $n$-tuple $\nu=\left(  \nu_{1},\nu_{2},\ldots,\nu_{n}\right)
\in\mathbb{Z}^{n}$ by
\[
\nu_{i}=\omega_{i}-a\ \ \ \ \ \ \ \ \ \ \text{for each }i\in\left\{
1,2,\ldots,n\right\}  .
\]
Thus, $\nu=\omega-a$.
\end{vershort}

\begin{verlong}
Define an $n$-tuple $\nu=\left(  \nu_{1},\nu_{2},\ldots,\nu_{n}\right)
\in\mathbb{Z}^{n}$ by%
\[
\nu_{i}=\omega_{i}-a\ \ \ \ \ \ \ \ \ \ \text{for each }i\in\left\{
1,2,\ldots,n\right\}  .
\]
Thus,%
\begin{align*}
\nu &  =\left(  \nu_{1},\nu_{2},\ldots,\nu_{n}\right)  =\left(  \omega
_{1}-a,\omega_{2}-a,\ldots,\omega_{n}-a\right)  \\
&  \ \ \ \ \ \ \ \ \ \ \left(  \text{since }\nu_{i}=\omega_{i}-a\text{ for
each }i\in\left\{  1,2,\ldots,n\right\}  \right)  \\
&  =\omega-a.
\end{align*}

\end{verlong}

For each $i\in\mathbb{Z}$, we let $i\#$ denote the unique element of $\left\{
1,2,\ldots,n\right\}  $ congruent to $i$ modulo $n$.

For each $j\in\mathbb{Z}$, set%
\begin{align}
\tau_{j}  &  =\min\left\{  \left(  \nu_{\left(  j+1\right)  \#}+\nu_{\left(
j+2\right)  \#}+\cdots+\nu_{\left(  j+k\right)  \#}\right)  \right.
\nonumber\\
&  \ \ \ \ \ \ \ \ \ \ \ \ \ \ \ \ \ \ \ \ \left.  +\left(  \mu_{\left(
j+k+1\right)  \#}+\mu_{\left(  j+k+2\right)  \#}+\cdots+\mu_{\left(
j+n-1\right)  \#}\right)  \right. \nonumber\\
&  \ \ \ \ \ \ \ \ \ \ \ \ \ \ \ \ \ \ \ \ \left.  \ \mid\ k\in\left\{
0,1,\ldots,n-1\right\}
\vphantom{\left(\nu_{\left(j+1\right)\#}\right)}\right\}  .
\label{pf.lem.fZ.phi.tauj=}%
\end{align}

Define an $n$-tuple $\eta=\left(  \eta_{1},\eta_{2},\ldots,\eta_{n}\right)
\in\mathbb{Z}^{n}$ by setting%
\[
\eta_{i}=\mu_{i\#}+\left(  \mu_{\left(  i-1\right)  \#}+\tau_{\left(
i-1\right)  \#}\right)  -\left(  \nu_{\left(  i+1\right)  \#}+\tau_{\left(
i+1\right)  \#}\right)  \ \ \ \ \ \ \ \ \ \ \text{for each }i\in\left\{
1,2,\ldots,n\right\}  .
\]

The definition of $\varphi$ then yields%
\begin{equation}
\varphi\left(  \omega\right)  =\left(  \eta_{1}+b,\eta_{2}+b,\ldots,\eta
_{n}+b\right)  =\eta+b. \label{pf.lem.fZ.phi.phi1}%
\end{equation}

Our plan is now to show that $\mathbf{f}_{\mu}\left(  \nu\right)  =\eta$. We
shall achieve this by applying Lemma \ref{lem.fZ.back}; but in order to do so,
we need to show that the assumptions of Lemma \ref{lem.fZ.back} are satisfied.

\begin{vershort}
We shall do this piece by piece. First, we make the following two claims,
which both follow from Convention \ref{conv.bir.peri}:
\end{vershort}

\begin{verlong}
We shall do this piece by piece. First we claim the following:
\end{verlong}

\begin{statement}
\textit{Claim 1:} We have $\nu_{p\#}=\nu_{p}$ for each $p\in\mathbb{Z}$.
\end{statement}

\begin{verlong}
[\textit{Proof of Claim 1:} Let $p\in\mathbb{Z}$. Then, $p\#\equiv
p\operatorname{mod}n$ (by the definition of $p\#$).

Convention \ref{conv.bir.peri} ensures that the family $\left(  \nu
_{i}\right)  _{i\in\mathbb{Z}}$ is $n$-periodic. In other words, if $j$ and
$j^{\prime}$ are two integers satisfying $j\equiv j^{\prime}\operatorname{mod}%
n$, then $\nu_{j}=\nu_{j^{\prime}}$. We can apply this to $j=p\#$ and
$j^{\prime}=p$, and thus obtain $\nu_{p\#}=\nu_{p}$. This proves Claim 1.]
\end{verlong}

\begin{statement}
\textit{Claim 2:} We have $\mu_{p\#}=\mu_{p}$ for each $p\in\mathbb{Z}$.
\end{statement}

\begin{verlong}
[\textit{Proof of Claim 2:} This is analogous to the proof of Claim 1.]
\end{verlong}

\begin{vershort}
The next claim is an easy consequence of Claims 1 and 2:
\end{vershort}

\begin{statement}
\textit{Claim 3:} For each $j\in\mathbb{Z}$, we have%
\begin{align*}
\tau_{j}  &  =\min\left\{  \underbrace{\nu_{j+1}+\nu_{j+2}+\cdots+\nu_{j+k}%
}_{=\sum_{i=1}^{k}\nu_{j+i}}+\underbrace{\mu_{j+k+1}+\mu_{j+k+2}+\cdots
+\mu_{j+n-1}}_{=\sum_{i=k+1}^{n-1}\mu_{j+i}}\right. \\
&  \ \ \ \ \ \ \ \ \ \ \ \ \ \ \ \ \ \ \ \ \left.  \ \mid\ k\in\left\{
0,1,\ldots,n-1\right\}  \right\}  .
\end{align*}

\end{statement}

\begin{verlong}
[\textit{Proof of Claim 3:} Let $j\in\mathbb{Z}$. Then, the definition of
$\tau_{j}$ yields%
\begin{align*}
\tau_{j}  &  =\min\left\{  \left(  \underbrace{\nu_{\left(  j+1\right)  \#}%
}_{\substack{=\nu_{j+1}\\\text{(by Claim 1)}}}+\underbrace{\nu_{\left(
j+2\right)  \#}}_{\substack{=\nu_{j+2}\\\text{(by Claim 1)}}}+\cdots
+\underbrace{\nu_{\left(  j+k\right)  \#}}_{\substack{=\nu_{j+k}\\\text{(by
Claim 1)}}}\right)  \right. \\
&  \ \ \ \ \ \ \ \ \ \ \ \ \ \ \ \ \ \ \ \ \left.  +\left(  \underbrace{\mu
_{\left(  j+k+1\right)  \#}}_{\substack{=\mu_{j+k+1}\\\text{(by Claim 2)}%
}}+\underbrace{\mu_{\left(  j+k+2\right)  \#}}_{\substack{=\mu_{j+k+2}%
\\\text{(by Claim 2)}}}+\cdots+\underbrace{\mu_{\left(  j+n-1\right)  \#}%
}_{\substack{=\mu_{j+n-1}\\\text{(by Claim 2)}}}\right)  \right. \\
&  \ \ \ \ \ \ \ \ \ \ \ \ \ \ \ \ \ \ \ \ \left.  \ \mid\ k\in\left\{
0,1,\ldots,n-1\right\}
\vphantom{\left(\nu_{\left(j+1\right)\#}\right)}\right\} \\
&  =\min\left\{  \underbrace{\nu_{j+1}+\nu_{j+2}+\cdots+\nu_{j+k}}%
_{=\sum_{i=1}^{k}\nu_{j+i}}+\underbrace{\mu_{j+k+1}+\mu_{j+k+2}+\cdots
+\mu_{j+n-1}}_{=\sum_{i=k+1}^{n-1}\mu_{j+i}}\right. \\
&  \ \ \ \ \ \ \ \ \ \ \ \ \ \ \ \ \ \ \ \ \left.  \ \mid\ k\in\left\{
0,1,\ldots,n-1\right\}  \right\}  .
\end{align*}
This proves Claim 3.]
\end{verlong}

The next claim is an easy fact from elementary number theory:

\begin{statement}
\textit{Claim 4:} We have $\left(  p\#+q\right)  \#=\left(  p+q\right)  \#$
for any $p\in\mathbb{Z}$ and $q\in\mathbb{Z}$.
\end{statement}

\begin{verlong}
[\textit{Proof of Claim 4:} Let $p\in\mathbb{Z}$ and $q\in\mathbb{Z}$. Recall
that $p\#$ is defined as the unique element of $\left\{  1,2,\ldots,n\right\}
$ congruent to $p$ modulo $n$. Hence, $p\#$ is congruent to $p$ modulo $n$. In
other words, $p\#\equiv p\operatorname{mod}n$. Thus, $\underbrace{p\#}_{\equiv
p\operatorname{mod}n}+q\equiv p+q\operatorname{mod}n$.

Recall that $\left(  p\#+q\right)  \#$ is defined as the unique element of
$\left\{  1,2,\ldots,n\right\}  $ congruent to $p\#+q$ modulo $n$. Hence,
$\left(  p\#+q\right)  \#$ is congruent to $p\#+q$ modulo $n$. In other words,
$\left(  p\#+q\right)  \#\equiv p\#+q\operatorname{mod}n$. Thus, $\left(
p\#+q\right)  \#\equiv p\#+q\equiv p+q\operatorname{mod}n$. In other words,
$\left(  p\#+q\right)  \#$ is congruent to $p+q$ modulo $n$.

Moreover, $\left(  p\#+q\right)  \#$ is an element of $\left\{  1,2,\ldots
,n\right\}  $ (since $\left(  p\#+q\right)  \#$ is the unique element of
$\left\{  1,2,\ldots,n\right\}  $ congruent to $p\#+q$ modulo $n$). Hence,
$\left(  p\#+q\right)  \#$ is an element of $\left\{  1,2,\ldots,n\right\}  $
congruent to $p+q$ modulo $n$.

But let us now recall that $\left(  p+q\right)  \#$ is defined as the unique
element of $\left\{  1,2,\ldots,n\right\}  $ congruent to $p+q$ modulo $n$.
Hence, $\left(  p+q\right)  \#$ is the only such element. In other words, if
$i$ is an element of $\left\{  1,2,\ldots,n\right\}  $ congruent to $p+q$
modulo $n$, then $i=\left(  p+q\right)  \#$. Applying this to $i=\left(
p\#+q\right)  \#$, we conclude that $\left(  p\#+q\right)  \#=\left(
p+q\right)  \#$ (since $\left(  p\#+q\right)  \#$ is an element of $\left\{
1,2,\ldots,n\right\}  $ congruent to $p+q$ modulo $n$). This proves Claim 4.]
\end{verlong}

Using Claim 4, we easily obtain the following:

\begin{statement}
\textit{Claim 5:} We have $\tau_{p\#}=\tau_{p}$ for each $p\in\mathbb{Z}$.
\end{statement}

\begin{verlong}
[\textit{Proof of Claim 5:} Let $p\in\mathbb{Z}$. Then, the definition of
$\tau_{p\#}$ yields%
\begin{align*}
\tau_{p\#}  &  =\min\left\{  \underbrace{\left(  \nu_{\left(  p\#+1\right)
\#}+\nu_{\left(  p\#+2\right)  \#}+\cdots+\nu_{\left(  p\#+k\right)
\#}\right)  }_{=\sum_{i=1}^{k}\nu_{\left(  p\#+i\right)  \#}}\right. \\
&  \ \ \ \ \ \ \ \ \ \ \ \ \ \ \ \ \ \ \ \ \left.  +\underbrace{\left(
\mu_{\left(  p\#+k+1\right)  \#}+\mu_{\left(  p\#+k+2\right)  \#}+\cdots
+\mu_{\left(  p\#+n-1\right)  \#}\right)  }_{=\sum_{i=k+1}^{n-1}\mu_{\left(
p\#+i\right)  \#}}\right. \\
&  \ \ \ \ \ \ \ \ \ \ \ \ \ \ \ \ \ \ \ \ \left.  \ \mid\ k\in\left\{
0,1,\ldots,n-1\right\}
\vphantom{\left(\nu_{\left(j+1\right)\#}\right)}\right\} \\
&  =\min\left\{  \sum_{i=1}^{k}\underbrace{\nu_{\left(  p\#+i\right)  \#}%
}_{\substack{=\nu_{\left(  p+i\right)  \#}\\\text{(since Claim 4}%
\\\text{(applied to }q=i\text{)}\\\text{yields }\left(  p\#+i\right)
\#=\left(  p+i\right)  \#\text{)}}}+\sum_{i=k+1}^{n-1}\underbrace{\mu_{\left(
p\#+i\right)  \#}}_{\substack{=\mu_{\left(  p+i\right)  \#}\\\text{(since
Claim 4}\\\text{(applied to }q=i\text{)}\\\text{yields }\left(  p\#+i\right)
\#=\left(  p+i\right)  \#\text{)}}}\ \mid\ k\in\left\{  0,1,\ldots
,n-1\right\}  \right\} \\
&  =\min\left\{  \sum_{i=1}^{k}\nu_{\left(  p+i\right)  \#}+\sum_{i=k+1}%
^{n-1}\mu_{\left(  p+i\right)  \#}\ \mid\ k\in\left\{  0,1,\ldots,n-1\right\}
\right\}  .
\end{align*}
Comparing this with%
\begin{align*}
\tau_{p}  &  =\min\left\{  \underbrace{\left(  \nu_{\left(  p+1\right)
\#}+\nu_{\left(  p+2\right)  \#}+\cdots+\nu_{\left(  p+k\right)  \#}\right)
}_{=\sum_{i=1}^{k}\nu_{\left(  p+i\right)  \#}}\right. \\
&  \ \ \ \ \ \ \ \ \ \ \ \ \ \ \ \ \ \ \ \ \left.  +\underbrace{\left(
\mu_{\left(  p+k+1\right)  \#}+\mu_{\left(  p+k+2\right)  \#}+\cdots
+\mu_{\left(  p+n-1\right)  \#}\right)  }_{=\sum_{i=k+1}^{n-1}\mu_{\left(
p+i\right)  \#}}\right. \\
&  \ \ \ \ \ \ \ \ \ \ \ \ \ \ \ \ \ \ \ \ \left.  \ \mid\ k\in\left\{
0,1,\ldots,n-1\right\}
\vphantom{\left(\nu_{\left(j+1\right)\#}\right)}\right\} \\
&  \ \ \ \ \ \ \ \ \ \ \ \ \ \ \ \ \ \ \ \ \left(  \text{by the definition of
}\tau_{p}\right) \\
&  =\min\left\{  \sum_{i=1}^{k}\nu_{\left(  p+i\right)  \#}+\sum_{i=k+1}%
^{n-1}\mu_{\left(  p+i\right)  \#}\ \mid\ k\in\left\{  0,1,\ldots,n-1\right\}
\right\}  ,
\end{align*}
we obtain $\tau_{p\#}=\tau_{p}$. This proves Claim 5.]
\end{verlong}

Now, let $i\in\left\{  1,2,\ldots,n\right\}  $. Then, the definition of $\eta$
yields%
\begin{align*}
\eta_{i}  &  =\underbrace{\mu_{i\#}}_{\substack{=\mu_{i}\\\text{(by Claim 2)}%
}}+\left(  \underbrace{\mu_{\left(  i-1\right)  \#}}_{\substack{=\mu
_{i-1}\\\text{(by Claim 2)}}}+\underbrace{\tau_{\left(  i-1\right)  \#}%
}_{\substack{=\tau_{i-1}\\\text{(by Claim 5)}}}\right)  -\left(
\underbrace{\nu_{\left(  i+1\right)  \#}}_{\substack{=\nu_{i+1}\\\text{(by
Claim 1)}}}+\underbrace{\tau_{\left(  i+1\right)  \#}}_{\substack{=\tau
_{i+1}\\\text{(by Claim 5)}}}\right) \\
&  =\mu_{i}+\left(  \mu_{i-1}+\tau_{i-1}\right)  -\left(  \nu_{i+1}+\tau
_{i+1}\right)  .
\end{align*}

Now, forget that we fixed $i$. We thus have proved that
\[
\eta_{i}=\mu_{i}+\left(  \mu_{i-1}+\tau_{i-1}\right)  -\left(  \nu_{i+1}%
+\tau_{i+1}\right)  \ \ \ \ \ \ \ \ \ \ \text{for each }i\in\left\{
1,2,\ldots,n\right\}  .
\]
Combining this with Claim 3, we conclude that the assumptions of Lemma
\ref{lem.fZ.back} are satisfied. Hence, Lemma \ref{lem.fZ.back} yields
$\mathbf{f}_{\mu}\left(  \nu\right)  =\eta$. In view of $\nu=\omega-a$, this
rewrites as $\mathbf{f}_{\mu}\left(  \omega-a\right)  =\eta$. Hence,
$\eta=\mathbf{f}_{\mu}\left(  \omega-a\right)  $, so that
(\ref{pf.lem.fZ.phi.phi1}) becomes%
\[
\varphi\left(  \omega\right)  =\underbrace{\eta}_{=\mathbf{f}_{\mu}\left(
\omega-a\right)  }+b=\mathbf{f}_{\mu}\left(  \omega-a\right)  +b.
\]
This proves Lemma \ref{lem.fZ.phi}.
\end{proof}

\subsection{The finale}

Now, let us again use the convention (from Section \ref{sect.not}) by which we
identify partitions with finite tuples (and therefore identify partitions in
$\operatorname*{Par}\left[  n\right]  $ with nonnegative snakes). This is no
longer problematic, since we are not using Convention \ref{conv.bir.peri} any more.

\begin{lemma}
\label{lem.finale.1}Let $a,b\in\mathbb{N}$. Define the two partitions
$\alpha=\left(  a+b,a^{n-2}\right)  $ and $\beta=\left(  a+b,b^{n-2}\right)  $.

Fix any partition $\mu\in\operatorname*{Par}\left[  n\right]  $. Consider the
map $\mathbf{f}_{\mu}:\mathbb{Z}^{n}\rightarrow\mathbb{Z}^{n}$ defined in
Definition \ref{def.fZ}.

Then, for any $\lambda\in\mathbb{Z}^{n}$, we have
\[
c_{\alpha,\mu}^{\lambda+a}=c_{\beta,\mu}^{\mathbf{f}_{\mu}\left(
\lambda\right)  +b}.
\]
Here, we understand $c_{\alpha,\mu}^{\lambda+a}$ to mean $0$ if $\lambda+a$ is
not a partition, and likewise we understand $c_{\beta,\mu}^{\mathbf{f}_{\mu
}\left(  \lambda\right)  +b}$ to mean $0$ if $\mathbf{f}_{\mu}\left(
\lambda\right)  +b$ is not a partition.
\end{lemma}

\begin{proof}
[Proof of Lemma \ref{lem.finale.1}.]Let $\lambda\in\mathbb{Z}^{n}$. Corollary
\ref{cor.Rmab.lr} (applied to $\lambda+a$ instead of $\lambda$) yields%
\begin{align}
c_{\alpha,\mu}^{\lambda+a}  &  =\left\vert R_{\mu,a,b}\left(  \left(
\lambda+a\right)  -a\right)  \right\vert -\left\vert R_{\mu,a-1,b-1}\left(
\left(  \lambda+a\right)  -a\right)  \right\vert \nonumber\\
&  =\left\vert R_{\mu,a,b}\left(  \lambda\right)  \right\vert -\left\vert
R_{\mu,a-1,b-1}\left(  \lambda\right)  \right\vert \label{pf.lem.finale.1.1}%
\end{align}
(since $\left(  \lambda+a\right)  -a=\lambda$). On the other hand,
$\beta=\left(  \underbrace{a+b}_{=b+a},b^{n-2}\right)  =\left(  b+a,b^{n-2}%
\right)  $. Hence, Corollary \ref{cor.Rmab.lr} (applied to $b$, $a$, $\beta$
and $\mathbf{f}_{\mu}\left(  \lambda\right)  +b$ instead of $a$, $b$, $\alpha$
and $\lambda$) yields%
\begin{align*}
c_{\beta,\mu}^{\mathbf{f}_{\mu}\left(  \lambda\right)  +b}  &  =\left\vert
R_{\mu,b,a}\left(  \left(  \mathbf{f}_{\mu}\left(  \lambda\right)  +b\right)
-b\right)  \right\vert -\left\vert R_{\mu,b-1,a-1}\left(  \left(
\mathbf{f}_{\mu}\left(  \lambda\right)  +b\right)  -b\right)  \right\vert \\
&  =\underbrace{\left\vert R_{\mu,b,a}\left(  \mathbf{f}_{\mu}\left(
\lambda\right)  \right)  \right\vert }_{\substack{=\left\vert R_{\mu
,a,b}\left(  \lambda\right)  \right\vert \\\text{(by Lemma \ref{lem.fZ.lr}%
,}\\\text{applied to }\gamma=\lambda\text{)}}}-\underbrace{\left\vert
R_{\mu,b-1,a-1}\left(  \mathbf{f}_{\mu}\left(  \lambda\right)  \right)
\right\vert }_{\substack{=\left\vert R_{\mu,a-1,b-1}\left(  \lambda\right)
\right\vert \\\text{(by Lemma \ref{lem.fZ.lr},}\\\text{applied to }%
\lambda\text{, }a-1\text{ and }b-1\\\text{instead of }\gamma\text{, }a\text{
and }b\text{)}}}\\
&  \ \ \ \ \ \ \ \ \ \ \left(  \text{since }\left(  \mathbf{f}_{\mu}\left(
\lambda\right)  +b\right)  -b=\mathbf{f}_{\mu}\left(  \lambda\right)  \right)
\\
&  =\left\vert R_{\mu,a,b}\left(  \lambda\right)  \right\vert -\left\vert
R_{\mu,a-1,b-1}\left(  \lambda\right)  \right\vert .
\end{align*}
Comparing this with (\ref{pf.lem.finale.1.1}), we find $c_{\alpha,\mu
}^{\lambda+a}=c_{\beta,\mu}^{\mathbf{f}_{\mu}\left(  \lambda\right)  +b}%
$.\ This proves Lemma \ref{lem.finale.1}.
\end{proof}

We are now ready to prove Theorem \ref{thm.main}:

\begin{proof}
[Proof of Theorem \ref{thm.main}.]The map $\mathbf{f}_{\mu}$ is an involution
(by Theorem \ref{thm.fZ.full} \textbf{(a)}), thus a bijection.

Let $\mathbf{a}^{-}:\mathbb{Z}^{n}\rightarrow\mathbb{Z}^{n}$ be the map that
sends each $\omega\in\mathbb{Z}^{n}$ to $\omega-a$. This map $\mathbf{a}^{-}$
is clearly a bijection.

Let $\mathbf{b}^{+}:\mathbb{Z}^{n}\rightarrow\mathbb{Z}^{n}$ be the map that
sends each $\omega\in\mathbb{Z}^{n}$ to $\omega+b$. This map $\mathbf{b}^{+}$
is clearly a bijection.

From Lemma \ref{lem.fZ.phi}, we can easily see that%
\[
\varphi=\mathbf{b}^{+}\circ\mathbf{f}_{\mu}\circ\mathbf{a}^{-}.
\]

\begin{verlong}
[\textit{Proof:} Let $\omega\in\mathbb{Z}^{n}$. Then, the definition of
$\mathbf{a}^{-}$ yields $\mathbf{a}^{-}\left(  \omega\right)  =\omega-a$.
Hence, $\omega-a=\mathbf{a}^{-}\left(  \omega\right)  $. But Lemma
\ref{lem.fZ.phi} yields
\begin{align*}
\varphi\left(  \omega\right)   &  =\mathbf{f}_{\mu}\left(  \omega-a\right)
+b=\mathbf{b}^{+}\left(  \mathbf{f}_{\mu}\left(  \underbrace{\omega
-a}_{=\mathbf{a}^{-}\left(  \omega\right)  }\right)  \right) \\
&  \ \ \ \ \ \ \ \ \ \ \left(  \text{since the definition of }\mathbf{b}%
^{+}\text{ yields }\mathbf{b}^{+}\left(  \mathbf{f}_{\mu}\left(
\omega-a\right)  \right)  =\mathbf{f}_{\mu}\left(  \omega-a\right)  +b\right)
\\
&  =\mathbf{b}^{+}\left(  \mathbf{f}_{\mu}\left(  \mathbf{a}^{-}\left(
\omega\right)  \right)  \right)  =\left(  \mathbf{b}^{+}\circ\mathbf{f}_{\mu
}\circ\mathbf{a}^{-}\right)  \left(  \omega\right)  .
\end{align*}

Forget that we fixed $\omega$. Thus we have shown that $\varphi\left(
\omega\right)  =\left(  \mathbf{b}^{+}\circ\mathbf{f}_{\mu}\circ\mathbf{a}%
^{-}\right)  \left(  \omega\right)  $ for each $\omega\in\mathbb{Z}^{n}$. In
other words, $\varphi=\mathbf{b}^{+}\circ\mathbf{f}_{\mu}\circ\mathbf{a}^{-}$, qed.]
\end{verlong}

Recall that the maps $\mathbf{b}^{+}$, $\mathbf{f}_{\mu}$ and $\mathbf{a}^{-}$
are bijections. Hence, their composition $\mathbf{b}^{+}\circ\mathbf{f}_{\mu
}\circ\mathbf{a}^{-}$ is a bijection as well. In other words, $\varphi$ is a
bijection (since $\varphi=\mathbf{b}^{+}\circ\mathbf{f}_{\mu}\circ
\mathbf{a}^{-}$). This proves Theorem \ref{thm.main} \textbf{(a)}.

\textbf{(b)} Let $\omega\in\mathbb{Z}^{n}$. Then, Lemma \ref{lem.fZ.phi}
yields $\varphi\left(  \omega\right)  =\mathbf{f}_{\mu}\left(  \omega
-a\right)  +b$. Hence, $\mathbf{f}_{\mu}\left(  \omega-a\right)
+b=\varphi\left(  \omega\right)  $. But Lemma \ref{lem.finale.1} (applied to
$\lambda=\omega-a$) yields%
\[
c_{\alpha,\mu}^{\left(  \omega-a\right)  +a}=c_{\beta,\mu}^{\mathbf{f}_{\mu
}\left(  \omega-a\right)  +b}=c_{\beta,\mu}^{\varphi\left(  \omega\right)
}\ \ \ \ \ \ \ \ \ \ \left(  \text{since }\mathbf{f}_{\mu}\left(
\omega-a\right)  +b=\varphi\left(  \omega\right)  \right)  .
\]
In view of $\left(  \omega-a\right)  +a=\omega$, this rewrites as
$c_{\alpha,\mu}^{\omega}=c_{\beta,\mu}^{\varphi\left(  \omega\right)  }$. This
proves Theorem \ref{thm.main} \textbf{(b)}.
\end{proof}

\section{\label{sect.fin}Final remarks}

\subsection{\label{subsect.fin.jt}Aside: A Jacobi--Trudi formula for Schur
Laurent polynomials}

As mentioned above, Proposition \ref{prop.ominus.s} has the following
generalization, which can be obtained from an identity of Koike
\cite[Proposition 2.8]{Koike89} (via the map $\widetilde{\pi}_{n}$ from
\cite{Koike89} and the correspondence between Schur Laurent polynomials and
rational representations of $\operatorname{GL}\left(  n\right)  $), which has
later been extended by Hamel and King \cite{HamKin11} (see \cite[(6) and
(10)]{HamKin11} for the connection):\footnote{We recall Definition
\ref{def.laurent}, Definition \ref{def.snakes} \textbf{(a)}, Definition
\ref{def.alt}, Definition \ref{def.alt.sbar} and Definition \ref{def.hkpm}, as
well as the conventions made in Section \ref{sect.not}, for the notations used
in this proposition.}

\begin{proposition}
\label{prop.jt}Let $p,q\in\mathbb{N}$ with $p+q\leq n$. Let $\mathbf{a}%
=\left(  a_{1},a_{2},\ldots,a_{p}\right)  $ and $\mathbf{b}=\left(
b_{1},b_{2},\ldots,b_{q}\right)  $ be two partitions. Let $\mathbf{b}%
\ominus\mathbf{a}$ denote the snake $\left(  b_{1},b_{2},\ldots,b_{q}%
,0^{n-p-q},-a_{p},-a_{p-1},\ldots,-a_{1}\right)  $. Let $M$ be the $\left(
p+q\right)  \times\left(  p+q\right)  $-matrix%
\begin{align*}
&  \left(
\begin{cases}
h_{a_{p-i+1}+i-j}^{-}, & \text{if }i\leq p;\\
h_{b_{i-p}-i+j}^{+}, & \text{if }i>p
\end{cases}
\right)  _{i,j\in\left\{  1,2,\ldots,p+q\right\}  }\\
&  =\left(
\begin{array}
[c]{cccccccc}%
h_{a_{p}}^{-} & h_{a_{p}-1}^{-} & \cdots & h_{a_{p}-p+1}^{-} & h_{a_{p}-p}^{-}
& h_{a_{p}-p-1}^{-} & \cdots & h_{a_{p}-p-q+1}^{-}\\
h_{a_{p-1}+1}^{-} & h_{a_{p-1}}^{-} & \cdots & h_{a_{p-1}-p+2}^{-} &
h_{a_{p-1}-p+1}^{-} & h_{a_{p-1}-p}^{-} & \cdots & h_{a_{p-1}-p-q+2}^{-}\\
\vdots & \vdots & \ddots & \vdots & \vdots & \vdots & \ddots & \vdots\\
h_{a_{1}+p-1}^{-} & h_{a_{1}+p-2}^{-} & \cdots & h_{a_{1}}^{-} & h_{a_{1}%
-1}^{-} & h_{a_{1}-2}^{-} & \cdots & h_{a_{1}-q}^{-}\\
h_{b_{1}-p}^{+} & h_{b_{1}-p+1}^{+} & \cdots & h_{b_{1}-1}^{+} & h_{b_{1}}^{+}
& h_{b_{1}+1}^{+} & \cdots & h_{b_{1}+q-1}^{+}\\
h_{b_{2}-p-1}^{+} & h_{b_{2}-p}^{+} & \cdots & h_{b_{2}-2}^{+} & h_{b_{2}%
-1}^{+} & h_{b_{2}}^{+} & \cdots & h_{b_{2}+q-2}^{+}\\
\vdots & \vdots & \ddots & \vdots & \vdots & \vdots & \ddots & \vdots\\
h_{b_{q}-p-q+1}^{+} & h_{b_{q}-p-q+2}^{+} & \cdots & h_{b_{q}-q}^{+} &
h_{b_{q}-q+1}^{+} & h_{b_{q}-q+2}^{+} & \cdots & h_{b_{q}}^{+}%
\end{array}
\right)  .
\end{align*}
Then,%
\[
\overline{s}_{\mathbf{b}\ominus\mathbf{a}}=\det M.
\]

\end{proposition}

\begin{remark}
The analogous generalization of the second Jacobi--Trudi formula
(\cite[(2.4.17)]{GriRei}) can easily be proved (although we leave both stating
and proving it to the reader). What makes it easy is the (fairly obvious) fact
that the elementary symmetric functions $e_{k}$ satisfy%
\[
e_{k}\left(  x_{1}^{-1},x_{2}^{-1},\ldots,x_{n}^{-1}\right)  =x_{\Pi}%
^{-1}e_{n-k}\left(  x_{1},x_{2},\ldots,x_{n}\right)
\]
for all $k\in\mathbb{Z}$. (See \cite[Definition 2.2.1]{GriRei} for the
definition of $e_{k}$.)
\end{remark}

Proposition \ref{prop.jt} generalizes Proposition \ref{prop.ominus.s} (which
corresponds to the particular case when $p=1$ and $q=1$) as well as the first
Jacobi--Trudi formula \cite[(2.4.16)]{GriRei} (which corresponds to the
particular case $p=0$).

We notice that what we called $\mathbf{b}\ominus\mathbf{a}$ in Proposition
\ref{prop.jt} has been called $\left[  \mathbf{b},\mathbf{a}\right]  $ in
\cite{Stembr87}.

We thank Grigori Olshanski for informing us of the provenance of Proposition
\ref{prop.jt}.

\subsection{\label{subsect.fin.fu}Questions on $\mathbf{f}_{u}$}

We shall now pose several questions about the birational involution
$\mathbf{f}_{u}$ studied in Section \ref{sect.bir}. Convention
\ref{conv.semifield.notations}, Convention \ref{conv.bir.1} and Convention
\ref{conv.bir.peri} will be used throughout Subsection \ref{subsect.fin.fu}.

\subsubsection{$\mathbf{f}_{u}$ as a composition?}

Most of our questions are meant to attempt seeing the involution
$\mathbf{f}_{u}$ from different directions. The first one is inspired by what
is now known as the \textquotedblleft\textit{toggle approach}%
\textquotedblright\ to dynamical combinatorics (see, e.g., \cite{Roby15}), but
is really an application of the age-old \textquotedblleft divide and
conquer\textquotedblright\ paradigm to complicated maps:

\begin{question}
Is there an equivalent definition of $\mathbf{f}_{u}$ as a composition of
toggles? (A \textit{toggle} here means a birational map $\mathbb{K}%
^{n}\rightarrow\mathbb{K}^{n}$ that changes only one entry of the $n$-tuple.
An example for a birational map that can be defined as a composition of
toggles is \textit{birational rowmotion} -- see, e.g., \cite{EinPro13}.
Cluster mutations, as in the theory of cluster algebras, are another example
of toggles.)
\end{question}

Another set of questions concern the \textit{uniqueness} of $\mathbf{f}_{u}$.
While we defined the map $\mathbf{f}_{u}$ explicitly, all we have then used
are the properties listed in Theorem \ref{thm.f.full}. Thus, it is a natural
question to ask whether these properties characterize $\mathbf{f}_{u}$
uniquely. A pointwise version of this question can be asked as well: Given
$x\in\mathbb{K}^{n}$ and $y\in\mathbb{K}^{n}$ satisfying some of the
equalities in parts \textbf{(b)}, \textbf{(c)} and \textbf{(d)} of Theorem
\ref{thm.f.full}, does it follow that $y=\mathbf{f}_{u}\left(  x\right)  $ ?
(Keep in mind that $u$ is fixed.)

Depending on which equalities we require, we may of course get different
answers. Let us first ask what happens if we require the equalities from
Theorem \ref{thm.f.full} \textbf{(c)} only:

\subsubsection{Characterizing $\mathbf{f}_{u}\left(  x\right)  $ via the
cyclic equations}

\begin{question}
\label{quest.unique.c}Given $x\in\mathbb{K}^{n}$ and $y\in\mathbb{K}^{n}$
satisfying
\begin{equation}
\left(  u_{i}+x_{i}\right)  \left(  \dfrac{1}{u_{i+1}}+\dfrac{1}{x_{i+1}%
}\right)  =\left(  u_{i}+y_{i}\right)  \left(  \dfrac{1}{u_{i+1}}+\dfrac
{1}{y_{i+1}}\right)  \label{eq.quest.f.unique.c}%
\end{equation}
for all $i\in\mathbb{Z}$. Does it follow that $y=\mathbf{f}_{u}\left(
x\right)  $ or $y=x$ ?
\end{question}

Note that the \textquotedblleft or $y=x$\textquotedblright\ part is needed
here, since $y=x$ is obviously a solution to the equations
(\ref{eq.quest.f.unique.c}).

The following example shows that the answer to Question \ref{quest.unique.c}
is \textquotedblleft no\textquotedblright\ if $\mathbb{K}$ is the min tropical
semifield $\left(  \mathbb{Z},\min,+,0\right)  $ of the totally ordered
abelian group $\mathbb{Z}$.

\begin{example}
\label{exa.unique.c.n=3-cex}Let $k,g\in\mathbb{N}$ with $g\geq k$. Let
$\mathbb{K}=\left(  \mathbb{Z},\min,+,0\right)  $ and $n=3$ and $u=\left(
0,0,g\right)  $ and $x=\left(  1,2,0\right)  $. Set $y=\left(  k+1,2,k\right)
$ (where the \textquotedblleft$+$\textquotedblright\ sign in \textquotedblleft%
$k+1$\textquotedblright\ stands for addition of integers, not addition in
$\mathbb{K}$). Then, the equations (\ref{eq.quest.f.unique.c}) hold in
$\mathbb{K}$ for all $i\in\mathbb{Z}$. (Restated in terms of standard
operations on integers, this is saying that%
\[
\min\left\{  u_{i},x_{i}\right\}  +\min\left\{  -u_{i+1},-x_{i+1}\right\}
=\min\left\{  u_{i},y_{i}\right\}  +\min\left\{  -u_{i+1},-y_{i+1}\right\}
\]
for all $i\in\mathbb{Z}$.) This is straightforward to verify, and shows that
for a given $x$ there can be an arbitrarily high (finite) number of
$y\in\mathbb{K}^{n}$ satisfying the equations (\ref{eq.quest.f.unique.c}) for
all $i\in\mathbb{Z}$.
\end{example}

\begin{vershort}
It is not hard to show (see \cite{verlong} for details) that this number
cannot be infinite.
\end{vershort}

\begin{verlong}
We note that this number cannot be infinite. In fact, this follows from the
following proposition\footnote{To be more precise, this claimed finiteness
follows from Proposition \ref{prop.unique.c.trop-finite} \textbf{(b)}, applied
to the $n$-tuple $z\in\mathbb{R}^{n}$ defined by%
\[
z_{i}=\min\left\{  u_{i},x_{i}\right\}  +\min\left\{  -u_{i+1},-x_{i+1}%
\right\}  \ \ \ \ \ \ \ \ \ \ \text{for all }i\in\mathbb{Z}.
\]
}:

\begin{proposition}
\label{prop.unique.c.trop-finite}Let $x\in\mathbb{R}^{n}$ and $z\in
\mathbb{R}^{n}$ be fixed. 

\begin{enumerate}
\item[\textbf{(a)}] Let $y\in\mathbb{R}^{n}$ be such that%
\begin{equation}
z_{i}=\min\left\{  u_{i},y_{i}\right\}  +\min\left\{  -u_{i+1},-y_{i+1}%
\right\}  \label{eq.prop.unique.c.trop-finite.ass}%
\end{equation}
for all $i\in\mathbb{Z}$. Then, we have $z_{i}+u_{i+1}\leq y_{i}\leq
u_{i-1}-z_{i-1}$ for each $i\in\mathbb{Z}$.

\item[\textbf{(b)}] There are only finitely many $y\in\mathbb{Z}^{n}$ such
that (\ref{eq.prop.unique.c.trop-finite.ass}) holds for all $i\in\mathbb{Z}$.
\end{enumerate}
\end{proposition}

\begin{proof}
[Proof of Proposition \ref{prop.unique.c.trop-finite} (sketched).]
\textbf{(a)} Let $i\in\mathbb{Z}$. Then,
(\ref{eq.prop.unique.c.trop-finite.ass}) yields%
\[
z_{i}=\underbrace{\min\left\{  u_{i},y_{i}\right\}  }_{\leq y_{i}%
}+\underbrace{\min\left\{  -u_{i+1},-y_{i+1}\right\}  }_{\leq-u_{i+1}}\leq
y_{i}-u_{i+1},
\]
so that $z_{i}+u_{i+1}\leq y_{i}$. Furthermore,
(\ref{eq.prop.unique.c.trop-finite.ass}) (applied to $i-1$ instead of $i$)
yields%
\[
z_{i-1}=\underbrace{\min\left\{  u_{i-1},y_{i-1}\right\}  }_{\leq u_{i-1}%
}+\underbrace{\min\left\{  -u_{i},-y_{i}\right\}  }_{\leq-y_{i}}\leq
u_{i-1}-y_{i},
\]
so that $y_{i}\leq u_{i-1}-z_{i-1}$. Combining this with $z_{i}+u_{i+1}\leq
y_{i}$, we obtain $z_{i}+u_{i+1}\leq y_{i}\leq u_{i-1}-z_{i-1}$. This proves
Proposition \ref{prop.unique.c.trop-finite} \textbf{(a)}.

\textbf{(b)} If $y\in\mathbb{Z}^{n}$ is such that
(\ref{eq.prop.unique.c.trop-finite.ass}) holds for all $i\in\mathbb{Z}$, then
Proposition \ref{prop.unique.c.trop-finite} \textbf{(a)} shows that each entry
$y_{i}$ of $y$ lies in the finite set
\[
F_{i}:=\left\{  \text{all integers between }z_{i}+u_{i+1}\text{ and }%
u_{i-1}-z_{i-1}\text{ (inclusive)}\right\}  .
\]
Thus, the whole $n$-tuple $y$ lies in the finite set $F_{1}\times F_{2}%
\times\cdots\times F_{n}$. Therefore, there are only finitely many such $y$'s.
This proves Proposition \ref{prop.unique.c.trop-finite} \textbf{(b)}.
\end{proof}
\end{verlong}

Example \ref{exa.unique.c.n=3-cex} has shown that the answer to Question
\ref{quest.unique.c} is \textquotedblleft no\textquotedblright\ when
$\mathbb{K}=\left(  \mathbb{Z},\min,+,0\right)  $. However, the answer to
Question \ref{quest.unique.c} is \textquotedblleft yes\textquotedblright\ if
$\mathbb{K}=\mathbb{Q}_{+}$ and, more generally, if the semifield $\mathbb{K}$
embeds into an integral domain:

\begin{proposition}
\label{prop.unique.c.Qplus}Assume that there is an integral domain
$\mathbb{L}$ such that the semifield $\mathbb{K}$ is a subsemifield of
$\mathbb{L}$ (in the sense that $\mathbb{K}\subseteq\mathbb{L}$ and that the
operations $+$ and $\cdot$ of $\mathbb{K}$ are restrictions of those of
$\mathbb{L}$, whereas the unity of $\mathbb{K}$ is the unity of $\mathbb{L}$).
Let $x\in\mathbb{K}^{n}$. Then, the only $n$-tuples $y\in\mathbb{K}^{n}$
satisfying the equations (\ref{eq.quest.f.unique.c}) for all $i\in\mathbb{Z}$
are $y=\mathbf{f}_{u}\left(  x\right)  $ and $y=x$.
\end{proposition}

\begin{vershort}
See the detailed version \cite{verlong} of this paper for a rough outline of
the proof of Proposition \ref{prop.unique.c.Qplus}.
\end{vershort}

\begin{verlong}
\begin{proof}
[Proof of Proposition \ref{prop.unique.c.Qplus} (sketched).]The following is a
rough outline, as we don't have any need for Proposition
\ref{prop.unique.c.Qplus}.

We define the elements $t_{r,j}\in\mathbb{K}$ for all $r\in\mathbb{N}$ and
$j\in\mathbb{Z}$ as in Definition \ref{def.f}. We further set $t_{-1,j}%
=0\in\mathbb{L}$ for each $j\in\mathbb{Z}$. (This is in line with our
definition of $t_{r,j}$, because an empty sum should be understood as $0$.)
Thus, $t_{r,j}\in\mathbb{L}$ is defined for each $r\in\mathbb{N}\cup\left\{
-1\right\}  $ and each $j\in\mathbb{Z}$.

For any $k\in\mathbb{N}$, we set $g_{k}=t_{k,-1}\in\mathbb{K}$.

For any integer $k\geq1$, we set $h_{k}=t_{k-2,1}\in\mathbb{L}$. Note that
$h_{k}\in\mathbb{K}$ if $k\geq2$; however, $h_{1}=t_{1-2,1}=t_{-1,1}=0$ (by
the definition of $t_{-1,1}$).

It is clear that $\mathbf{f}_{u}\left(  x\right)  $ and $x$ are two $n$-tuples
$y\in\mathbb{K}^{n}$ satisfying the equations (\ref{eq.quest.f.unique.c}) for
all $i\in\mathbb{Z}$. (Indeed, for $\mathbf{f}_{u}\left(  x\right)  $, this
follows from Theorem \ref{thm.f.full} \textbf{(c)}, while for $x$, it is
obvious.) We thus only need to prove the converse: Let $y\in\mathbb{K}^{n}$ be
an $n$-tuple satisfying the equations (\ref{eq.quest.f.unique.c}) for all
$i\in\mathbb{Z}$. We must show that $y=\mathbf{f}_{u}\left(  x\right)  $ or
$y=x$.

We WLOG assume that $n\geq2$, since otherwise (i.e., for $n=1$) this claim is
easily checked by hand.

For any integer $k\geq1$, we set $w_{k}=g_{k}-x_{1}y_{0}h_{k}\in\mathbb{L}$.
We shall see soon (as a consequence of Claim 3) that this $w_{k}$ actually
belongs to $\mathbb{K}$.

We begin with the following claim:

\begin{statement}
\textit{Claim 1:} \textbf{(a)} We have $u_{k-1}x_{k}g_{k-1}+g_{k+1}=\left(
u_{k}+x_{k}\right)  g_{k}$ for each integer $k\geq1$.

\textbf{(b)} We have $u_{k-1}x_{k}h_{k-1}+h_{k+1}=\left(  u_{k}+x_{k}\right)
h_{k}$ for each integer $k\geq2$.
\end{statement}

[\textit{Proof of Claim 1:} We shall prove something more general. Namely, for
any $p\in\mathbb{Z}$ and any integer $k\geq p-1$, we set%
\begin{equation}
z_{k,p}=t_{k-p,p-1}\in\mathbb{L}. \label{pf.prop.unique.c.Qplus.c1.pf.zkp=}%
\end{equation}
Thus, for each integer $k\geq-1$, we have
\begin{equation}
z_{k,0}=t_{k-0,0-1}=t_{k,-1}=g_{k}. \label{pf.prop.unique.c.Qplus.c1.pf.zk0=}%
\end{equation}
Moreover, for each integer $k\geq1$, we have%
\begin{equation}
z_{k,2}=t_{k-2,2-1}=t_{k-2,1}=h_{k}. \label{pf.prop.unique.c.Qplus.c1.pf.zk2=}%
\end{equation}

Now, we claim that%
\begin{equation}
u_{k-1}x_{k}z_{k-1,p}+z_{k+1,p}=\left(  u_{k}+x_{k}\right)  z_{k,p}
\label{pf.prop.unique.c.Qplus.c1.pf.main}%
\end{equation}
for any $p\in\mathbb{Z}$ and any integer $k\geq p$. Once this is proved, then
Claim 1 will easily follow. (Indeed, Claim 1 \textbf{(a)} will follow by
applying (\ref{pf.prop.unique.c.Qplus.c1.pf.main}) to $p=0$ and rewriting the
result using (\ref{pf.prop.unique.c.Qplus.c1.pf.zk0=}). Likewise, Claim 1
\textbf{(b)} will follow by applying (\ref{pf.prop.unique.c.Qplus.c1.pf.main})
to $p=2$ and rewriting the result using
(\ref{pf.prop.unique.c.Qplus.c1.pf.zk2=}).)

So it suffices to prove (\ref{pf.prop.unique.c.Qplus.c1.pf.main}). Let us do
this. Fix $p\in\mathbb{Z}$ and $k\geq p$. We must prove the equality
(\ref{pf.prop.unique.c.Qplus.c1.pf.main}). If $k=p$, then this equality boils
down to $u_{p-1}x_{p}\cdot0+\left(  u_{p}+x_{p}\right)  =\left(  u_{p}%
+x_{p}\right)  \cdot1$ (since it is easily seen that $z_{p-1,p}=t_{-1,p-1}=0$
and $z_{p+1,p}=t_{1,p-1}=u_{p}+x_{p}$ and $z_{p,p}=t_{0,p-1}=1$), which is
obvious. Thus, for the rest of this proof of
(\ref{pf.prop.unique.c.Qplus.c1.pf.main}), we WLOG assume that $k\neq p$.
Hence, $k\geq p+1$ (since $k\geq p$), so that $k-p-1\in\mathbb{N}$.

We can thus apply Lemma \ref{lem.f.steps} \textbf{(d)} to $r=k-p-1$ and
$j=p-1$\ \ \ \ \footnote{This might appear strange, since one of the
conditions in Lemma \ref{lem.f.steps} is not satisfied (namely, we do not have
a guarantee that $y$ is as in Definition \ref{def.f}). However, this does not
matter, since Lemma \ref{lem.f.steps} \textbf{(d)} does not depend on this
condition (as is clear from the proof).}. This results in%
\[
u_{k-1}t_{k-p-1,p-1}+x_{p}x_{p+1}\cdots x_{k-1}=t_{k-p,p-1}=z_{k,p}%
\]
(by (\ref{pf.prop.unique.c.Qplus.c1.pf.zkp=})). Multiplying both sides of this
equality by $x_{k}$, we find%
\[
x_{k}\left(  u_{k-1}t_{k-p-1,p-1}+x_{p}x_{p+1}\cdots x_{k-1}\right)
=x_{k}z_{k,p},
\]
so that%
\begin{align}
x_{k}z_{k,p}  &  =x_{k}\left(  u_{k-1}t_{k-p-1,p-1}+x_{p}x_{p+1}\cdots
x_{k-1}\right) \nonumber\\
&  =u_{k-1}x_{k}\underbrace{t_{k-p-1,p-1}}_{\substack{=t_{k-1-p,p-1}%
\\=z_{k-1,p}\\\text{(by (\ref{pf.prop.unique.c.Qplus.c1.pf.zkp=}), applied to
}k-1\\\text{instead of }k\text{)}}}+\underbrace{\left(  x_{p}x_{p+1}\cdots
x_{k-1}\right)  x_{k}}_{=x_{p}x_{p+1}\cdots x_{k}}\nonumber\\
&  =u_{k-1}x_{k}z_{k-1,p}+x_{p}x_{p+1}\cdots x_{k}.
\label{pf.prop.unique.c.Qplus.c1.pf.xz1}%
\end{align}

On the other hand, we can apply Lemma \ref{lem.f.steps} \textbf{(d)} to
$r=k-p$ and $j=p-1$. This results in%
\[
u_{k}t_{k-p,p-1}+x_{p}x_{p+1}\cdots x_{k}=t_{k-p+1,p-1}=t_{k+1-p,p-1}%
=z_{k+1,p}%
\]
(by (\ref{pf.prop.unique.c.Qplus.c1.pf.zkp=}), applied to $k+1$ instead of
$k$). Hence,%
\[
z_{k+1,p}=u_{k}t_{k-p,p-1}+x_{p}x_{p+1}\cdots x_{k}.
\]
Adding $u_{k-1}x_{k}z_{k-1,p}$ to both sides of this equality, we obtain%
\begin{align*}
&  u_{k-1}x_{k}z_{k-1,p}+z_{k+1,p}\\
&  =u_{k-1}x_{k}z_{k-1,p}+u_{k}t_{k-p,p-1}+x_{p}x_{p+1}\cdots x_{k}\\
&  =\underbrace{u_{k-1}x_{k}z_{k-1,p}+x_{p}x_{p+1}\cdots x_{k}}%
_{\substack{=x_{k}z_{k,p}\\\text{(by (\ref{pf.prop.unique.c.Qplus.c1.pf.xz1}%
))}}}+u_{k}\underbrace{t_{k-p,p-1}}_{\substack{=z_{k,p}\\\text{(by
(\ref{pf.prop.unique.c.Qplus.c1.pf.zkp=}))}}}\\
&  =x_{k}z_{k,p}+u_{k}z_{k,p}=\left(  u_{k}+x_{k}\right)  z_{k,p}.
\end{align*}
This proves (\ref{pf.prop.unique.c.Qplus.c1.pf.main}). As we said, this
completes the proof of Claim 1.]

As a consequence, we can easily conclude the following:

\begin{statement}
\textit{Claim 2:} We have
\begin{equation}
u_{k-1}x_{k}w_{k-1}+w_{k+1}=\left(  u_{k}+x_{k}\right)  w_{k}
\label{pf.prop.unique.c.Qplus.c3.1}%
\end{equation}
for each integer $k\geq2$.
\end{statement}

[\textit{Proof of Claim 2:} Let $k\geq2$ be an integer. Then, Claim 1
\textbf{(a)} yields $u_{k-1}x_{k}g_{k-1}+g_{k+1}=\left(  u_{k}+x_{k}\right)
g_{k}$, so that $g_{k+1}=\left(  u_{k}+x_{k}\right)  g_{k}-u_{k-1}x_{k}%
g_{k-1}$. Also, Claim 1 \textbf{(b)} yields $u_{k-1}x_{k}h_{k-1}%
+h_{k+1}=\left(  u_{k}+x_{k}\right)  h_{k}$, so that $h_{k+1}=\left(
u_{k}+x_{k}\right)  h_{k}-u_{k-1}x_{k}h_{k-1}$. Now, the definition of
$w_{k+1}$ yields%
\begin{align*}
w_{k+1}  &  =\underbrace{g_{k+1}}_{=\left(  u_{k}+x_{k}\right)  g_{k}%
-u_{k-1}x_{k}g_{k-1}}-x_{1}y_{0}\underbrace{h_{k+1}}_{=\left(  u_{k}%
+x_{k}\right)  h_{k}-u_{k-1}x_{k}h_{k-1}}\\
&  =\left(  \left(  u_{k}+x_{k}\right)  g_{k}-u_{k-1}x_{k}g_{k-1}\right)
-x_{1}y_{0}\left(  \left(  u_{k}+x_{k}\right)  h_{k}-u_{k-1}x_{k}%
h_{k-1}\right) \\
&  =\left(  u_{k}+x_{k}\right)  \underbrace{\left(  g_{k}-x_{1}y_{0}%
h_{k}\right)  }_{\substack{=w_{k}\\\text{(by the definition of }w_{k}\text{)}%
}}-u_{k-1}x_{k}\underbrace{\left(  g_{k-1}-x_{1}y_{0}h_{k-1}\right)
}_{\substack{=w_{k-1}\\\text{(by the definition of }w_{k-1}\text{)}}}\\
&  =\left(  u_{k}+x_{k}\right)  w_{k}-u_{k-1}x_{k}w_{k-1}.
\end{align*}
Hence, $u_{k-1}x_{k}w_{k-1}+w_{k+1}=\left(  u_{k}+x_{k}\right)  w_{k}$. This
proves Claim 2.]

We now claim the following:

\begin{statement}
\textit{Claim 3:} Let $k$ be a positive integer. Then, the elements
$w_{1},w_{2},\ldots,w_{k+1}$ belong to $\mathbb{K}$. Moreover, if $k\geq2$,
then%
\begin{equation}
y_{k}=\dfrac{u_{k-1}u_{k}x_{k}w_{k-1}}{w_{k+1}}.
\label{pf.prop.unique.c.Qplus.c3.2}%
\end{equation}

\end{statement}

[\textit{Proof of Claim 3:} Proceed by induction on $k$.

\textit{Induction base:} We shall use two base cases: the cases $k=1$ and
$k=2$. These require us to show that the elements $w_{1},w_{2},w_{3}$ belong
to $\mathbb{K}$ and that the equality (\ref{pf.prop.unique.c.Qplus.c3.2}) is
true for $k=2$.

The definition of $w_{1}$ says $w_{1}=g_{1}-x_{1}y_{0}\underbrace{h_{1}}%
_{=0}=g_{1}=t_{1,-1}=u_{0}+x_{0}$. Thus, $w_{1}$ belongs to $\mathbb{K}$.

The definition of $w_{2}$ yields%
\[
w_{2}=\underbrace{g_{2}}_{\substack{=t_{2,-1}\\=u_{0}u_{1}+x_{0}u_{1}%
+x_{0}x_{1}}}-x_{1}y_{0}\underbrace{h_{2}}_{=t_{0,1}=1}=u_{0}u_{1}+x_{0}%
u_{1}+x_{0}x_{1}-x_{1}y_{0}.
\]

The definition of $w_{3}$ yields%
\begin{align*}
w_{3}  &  =\underbrace{g_{3}}_{\substack{=t_{3,-1}\\=u_{0}u_{1}u_{2}%
+x_{0}u_{1}u_{2}+x_{0}x_{1}u_{2}+x_{0}x_{1}x_{2}}}-x_{1}y_{0}\underbrace{h_{3}%
}_{=t_{1,1}=u_{2}+x_{2}}\\
&  =u_{0}u_{1}u_{2}+x_{0}u_{1}u_{2}+x_{0}x_{1}u_{2}+x_{0}x_{1}x_{2}-x_{1}%
y_{0}\left(  u_{2}+x_{2}\right)  .
\end{align*}

The equations (\ref{eq.quest.f.unique.c}) are satisfied for all $i\in
\mathbb{Z}$, and thus are satisfied for $i=0$. In other words, we have%
\[
\left(  u_{0}+x_{0}\right)  \left(  \dfrac{1}{u_{1}}+\dfrac{1}{x_{1}}\right)
=\left(  u_{0}+y_{0}\right)  \left(  \dfrac{1}{u_{1}}+\dfrac{1}{y_{1}}\right)
.
\]
Solving this equation for $\dfrac{1}{y_{1}}$, we obtain%
\begin{align}
\dfrac{1}{y_{1}}  &  =\dfrac{\left(  u_{0}+x_{0}\right)  \left(  \dfrac
{1}{u_{1}}+\dfrac{1}{x_{1}}\right)  }{u_{0}+y_{0}}-\dfrac{1}{u_{1}}%
=\dfrac{u_{0}u_{1}+x_{0}u_{1}+x_{0}x_{1}+x_{1}u_{0}}{u_{1}x_{1}\left(
u_{0}+y_{0}\right)  }-\dfrac{1}{u_{1}}\nonumber\\
&  =\dfrac{u_{0}u_{1}+x_{0}u_{1}+x_{0}x_{1}+x_{1}u_{0}-x_{1}\left(
u_{0}+y_{0}\right)  }{u_{1}x_{1}\left(  u_{0}+y_{0}\right)  }\nonumber\\
&  =\dfrac{w_{2}}{u_{1}x_{1}\left(  u_{0}+y_{0}\right)  },
\label{pf.prop.unique.c.Qplus.c3.pf.1}%
\end{align}
since $u_{0}u_{1}+x_{0}u_{1}+x_{0}x_{1}+x_{1}u_{0}-x_{1}\left(  u_{0}%
+y_{0}\right)  =u_{0}u_{1}+x_{0}u_{1}+x_{0}x_{1}-x_{1}y_{0}=w_{2}$. This shows
that $w_{2}=u_{1}x_{1}\left(  u_{0}+y_{0}\right)  \cdot\dfrac{1}{y_{1}}%
\in\mathbb{K}$ (since $u_{1},x_{1},u_{0},y_{0},y_{1}\in\mathbb{K}$). Thus, we
have shown that $w_{2}$ belongs to $\mathbb{K}$.

Furthermore, by taking reciprocals on both sides of
(\ref{pf.prop.unique.c.Qplus.c3.pf.1}), we obtain%
\[
y_{1}=\dfrac{u_{1}x_{1}\left(  u_{0}+y_{0}\right)  }{w_{2}}.
\]
Hence,%
\begin{align}
u_{1}+y_{1}  &  =u_{1}+\dfrac{u_{1}x_{1}\left(  u_{0}+y_{0}\right)  }{w_{2}%
}=u_{1}\cdot\dfrac{w_{2}+x_{1}\left(  u_{0}+y_{0}\right)  }{w_{2}}\nonumber\\
&  =u_{1}\cdot\dfrac{\left(  u_{0}+x_{0}\right)  \left(  u_{1}+x_{1}\right)
}{w_{2}} \label{pf.prop.unique.c.Qplus.c3.pf.2}%
\end{align}
(since a straightforward computation yields $w_{2}+x_{1}\left(  u_{0}%
+y_{0}\right)  =\left(  u_{0}+x_{0}\right)  \left(  u_{1}+x_{1}\right)  $).

The equations (\ref{eq.quest.f.unique.c}) are satisfied for all $i\in
\mathbb{Z}$, and thus are satisfied for $i=1$. In other words, we have%
\[
\left(  u_{1}+x_{1}\right)  \left(  \dfrac{1}{u_{2}}+\dfrac{1}{x_{2}}\right)
=\left(  u_{1}+y_{1}\right)  \left(  \dfrac{1}{u_{2}}+\dfrac{1}{y_{2}}\right)
.
\]
Solving this equation for $\dfrac{1}{y_{2}}$, we obtain%
\begin{align}
\dfrac{1}{y_{2}}  &  =\dfrac{\left(  u_{1}+x_{1}\right)  \left(  \dfrac
{1}{u_{2}}+\dfrac{1}{x_{2}}\right)  }{u_{1}+y_{1}}-\dfrac{1}{u_{2}}%
=\dfrac{\left(  u_{1}+x_{1}\right)  \left(  \dfrac{1}{u_{2}}+\dfrac{1}{x_{2}%
}\right)  }{u_{1}\cdot\dfrac{\left(  u_{0}+x_{0}\right)  \left(  u_{1}%
+x_{1}\right)  }{w_{2}}}-\dfrac{1}{u_{2}}\nonumber\\
&  \ \ \ \ \ \ \ \ \ \ \left(  \text{by (\ref{pf.prop.unique.c.Qplus.c3.pf.2}%
)}\right) \nonumber\\
&  =\dfrac{\dfrac{1}{u_{2}}+\dfrac{1}{x_{2}}}{u_{1}\cdot\dfrac{u_{0}+x_{0}%
}{w_{2}}}-\dfrac{1}{u_{2}}=\dfrac{w_{2}\left(  u_{2}+x_{2}\right)  }%
{u_{1}u_{2}x_{2}\left(  u_{0}+x_{0}\right)  }-\dfrac{1}{u_{2}}\nonumber\\
&  =\dfrac{w_{2}\left(  u_{2}+x_{2}\right)  -u_{1}x_{2}\left(  u_{0}%
+x_{0}\right)  }{u_{1}u_{2}x_{2}\left(  u_{0}+x_{0}\right)  }\nonumber\\
&  =\dfrac{w_{3}}{u_{1}u_{2}x_{2}w_{1}} \label{pf.prop.unique.c.Qplus.c3.pf.3}%
\end{align}
(since straightforward computations yield $w_{2}\left(  u_{2}+x_{2}\right)
-u_{1}x_{2}\left(  u_{0}+x_{0}\right)  =w_{3}$ and $u_{0}+x_{0}=w_{1}$). This
shows that $w_{3}=u_{1}u_{2}x_{2}w_{1}\cdot\dfrac{1}{y_{2}}\in\mathbb{K}$
(since $u_{1},u_{2},x_{2},w_{1},y_{2}\in\mathbb{K}$). Thus, we have shown that
$w_{3}$ belongs to $\mathbb{K}$.

Taking reciprocals on both sides of (\ref{pf.prop.unique.c.Qplus.c3.pf.3}), we
find%
\[
y_{2}=\dfrac{u_{1}u_{2}x_{2}w_{1}}{w_{3}}.
\]
In other words, (\ref{pf.prop.unique.c.Qplus.c3.2}) is true for $k=2$. Thus,
the induction base is complete.

\textit{Induction step:} Let $i\geq2$ be an integer. Assume (as the induction
hypothesis) that Claim 3 holds for $k=i$. We shall now show that Claim 3 holds
for $k=i+1$.

Claim 2 (applied to $k=i$) yields%
\begin{equation}
u_{i-1}x_{i}w_{i-1}+w_{i+1}=\left(  u_{i}+x_{i}\right)  w_{i}.
\label{pf.prop.unique.c.Qplus.c3.pf.IS.IH1}%
\end{equation}

We have assumed that Claim 3 holds for $k=i$. In other words, the elements
$w_{1},w_{2},\ldots,w_{i+1}$ belong to $\mathbb{K}$, and the equality
(\ref{pf.prop.unique.c.Qplus.c3.2}) is true for $k=i$.

Thus, in particular, (\ref{pf.prop.unique.c.Qplus.c3.2}) is true for $k=i$. In
other words, we have%
\begin{equation}
y_{i}=\dfrac{u_{i-1}u_{i}x_{i}w_{i-1}}{w_{i+1}}.
\label{pf.prop.unique.c.Qplus.c3.pf.IS.IH2}%
\end{equation}
Adding $u_{i}$ to both sides of this equality, we obtain%
\begin{align}
u_{i}+y_{i}  &  =u_{i}+\dfrac{u_{i-1}u_{i}x_{i}w_{i-1}}{w_{i+1}}=\dfrac
{u_{i}\left(  u_{i-1}x_{i}w_{i-1}+w_{i+1}\right)  }{w_{i+1}}\nonumber\\
&  =\dfrac{u_{i}\left(  u_{i}+x_{i}\right)  w_{i}}{w_{i+1}}
\label{pf.prop.unique.c.Qplus.c3.pf.IS.3}%
\end{align}
(by (\ref{pf.prop.unique.c.Qplus.c3.pf.IS.IH1})).

But Claim 2 (applied to $k=i+1$) yields $u_{i}x_{i+1}w_{i}+w_{i+2}=\left(
u_{i+1}+x_{i+1}\right)  w_{i+1}$, so that
\begin{equation}
\left(  u_{i+1}+x_{i+1}\right)  w_{i+1}-u_{i}x_{i+1}w_{i}=w_{i+2}.
\label{pf.prop.unique.c.Qplus.c3.pf.IS.5}%
\end{equation}

Now, recall that the equation (\ref{eq.quest.f.unique.c}) is true. Solving
this equation for $\dfrac{1}{y_{i+1}}$, we find%
\begin{align}
\dfrac{1}{y_{i+1}}  &  =\dfrac{\left(  u_{i}+x_{i}\right)  \left(  \dfrac
{1}{u_{i+1}}+\dfrac{1}{x_{i+1}}\right)  }{u_{i}+y_{i}}-\dfrac{1}{u_{i+1}%
}=\dfrac{\left(  u_{i}+x_{i}\right)  \left(  u_{i+1}+x_{i+1}\right)  }%
{u_{i+1}x_{i+1}\left(  u_{i}+y_{i}\right)  }-\dfrac{1}{u_{i+1}}\nonumber\\
&  =\dfrac{\left(  u_{i}+x_{i}\right)  \left(  u_{i+1}+x_{i+1}\right)
}{u_{i+1}x_{i+1}\cdot\dfrac{u_{i}\left(  u_{i}+x_{i}\right)  w_{i}}{w_{i+1}}%
}-\dfrac{1}{u_{i+1}}\ \ \ \ \ \ \ \ \ \ \left(  \text{by
(\ref{pf.prop.unique.c.Qplus.c3.pf.IS.3})}\right) \nonumber\\
&  =\dfrac{\left(  u_{i+1}+x_{i+1}\right)  w_{i+1}}{u_{i}u_{i+1}x_{i+1}w_{i}%
}-\dfrac{1}{u_{i+1}}=\dfrac{\left(  u_{i+1}+x_{i+1}\right)  w_{i+1}%
-u_{i}x_{i+1}w_{i}}{u_{i}u_{i+1}x_{i+1}w_{i}}\nonumber\\
&  =\dfrac{w_{i+2}}{u_{i}u_{i+1}x_{i+1}w_{i}}\ \ \ \ \ \ \ \ \ \ \left(
\text{by (\ref{pf.prop.unique.c.Qplus.c3.pf.IS.5})}\right)  .
\label{pf.prop.unique.c.Qplus.c3.pf.IS.4}%
\end{align}
Hence,%
\begin{equation}
w_{i+2}=u_{i}u_{i+1}x_{i+1}w_{i}\cdot\dfrac{1}{y_{i+1}}\in\mathbb{K}\nonumber
\end{equation}
(since $u_{i+1},x_{i+1},w_{i+1},u_{i},x_{i+1},w_{i}\in\mathbb{K}$). Hence,
$w_{i+2}$ belongs to $\mathbb{K}$.

Taking reciprocals on both sides of the equality
(\ref{pf.prop.unique.c.Qplus.c3.pf.IS.4}), we obtain%
\[
y_{i+1}=\dfrac{u_{i}u_{i+1}x_{i+1}w_{i}}{w_{i+2}}.
\]
In other words, (\ref{pf.prop.unique.c.Qplus.c3.2}) is true for $k=i+1$.

We have now proved that the elements $w_{1},w_{2},\ldots,w_{i+2}$ belong to
$\mathbb{K}$ (since we already know that $w_{1},w_{2},\ldots,w_{i+1}$ belong
to $\mathbb{K}$, and since $w_{i+2}$ belongs to $\mathbb{K}$), and that the
equality (\ref{pf.prop.unique.c.Qplus.c3.2}) is true for $k=i+1$. In other
words, Claim 3 holds for $k=i+1$. This completes the induction step, and
therefore Claim 3 is proved.]

\begin{statement}
\textit{Claim 4:} For each $k\geq1$, we have
\begin{equation}
g_{k}=\left(  u_{0}+x_{0}\right)  \cdot u_{1}u_{2}\cdots u_{k-1}+x_{0}%
x_{1}h_{k} \label{pf.prop.unique.c.Qplus.c4.1}%
\end{equation}
and%
\begin{equation}
h_{k}=\dfrac{g_{k}-\left(  u_{0}+x_{0}\right)  \cdot u_{1}u_{2}\cdots u_{k-1}%
}{x_{0}x_{1}}. \label{pf.prop.unique.c.Qplus.c4.2}%
\end{equation}

\end{statement}

[\textit{Proof of Claim 4:} The equality (\ref{pf.prop.unique.c.Qplus.c4.1})
is easily checked directly. The equality (\ref{pf.prop.unique.c.Qplus.c4.2})
follows by solving (\ref{pf.prop.unique.c.Qplus.c4.1}) for $h_{k}$.]

Now, Claim 3 (applied to $k=n$) shows that the elements $w_{1},w_{2}%
,\ldots,w_{n+1}$ belong to $\mathbb{K}$ and satisfy
\begin{equation}
y_{n}=\dfrac{u_{n-1}u_{n}x_{n}w_{n-1}}{w_{n+1}}.
\label{pf.prop.unique.c.Qplus.yn}%
\end{equation}

The definition of $w_{n-1}$ yields%
\begin{align*}
w_{n-1}  &  =g_{n-1}-x_{1}y_{0}\underbrace{h_{n-1}}_{\substack{=\dfrac
{g_{n-1}-\left(  u_{0}+x_{0}\right)  \cdot u_{1}u_{2}\cdots u_{n-2}}%
{x_{0}x_{1}}\\\text{(by (\ref{pf.prop.unique.c.Qplus.c4.2}), applied to
}k=n-1\text{)}}}\\
&  =g_{n-1}-x_{1}y_{0}\cdot\dfrac{g_{n-1}-\left(  u_{0}+x_{0}\right)  \cdot
u_{1}u_{2}\cdots u_{n-2}}{x_{0}x_{1}}\\
&  =\left(  1-\dfrac{y_{0}}{x_{0}}\right)  \underbrace{g_{n-1}}%
_{\substack{=t_{n-1,-1}\\\text{(by the definition}\\\text{of }g_{n-1}\text{)}%
}}+\dfrac{y_{0}\cdot\left(  u_{0}+x_{0}\right)  \cdot u_{1}u_{2}\cdots
u_{n-2}}{x_{0}}\\
&  =\left(  1-\dfrac{y_{0}}{x_{0}}\right)  t_{n-1,-1}+\dfrac{y_{0}\cdot\left(
u_{0}+x_{0}\right)  \cdot u_{1}u_{2}\cdots u_{n-2}}{x_{0}}\\
&  =\dfrac{1}{x_{0}}\left(  \left(  x_{0}-y_{0}\right)  t_{n-1,-1}+y_{0}%
\cdot\left(  u_{0}+x_{0}\right)  \cdot u_{1}u_{2}\cdots u_{n-2}\right)  .
\end{align*}
Thus,%
\begin{align*}
x_{n}w_{n-1}  &  =\underbrace{x_{n}\cdot\dfrac{1}{x_{0}}}%
_{\substack{=1\\\text{(since }x_{n}=x_{0}\text{)}}}\left(  \left(  x_{0}%
-y_{0}\right)  t_{n-1,-1}+y_{0}\cdot\left(  u_{0}+x_{0}\right)  \cdot
u_{1}u_{2}\cdots u_{n-2}\right) \\
&  =\left(  x_{0}-y_{0}\right)  t_{n-1,-1}+y_{0}\cdot\left(  u_{0}%
+x_{0}\right)  \cdot u_{1}u_{2}\cdots u_{n-2}.
\end{align*}
Therefore,%
\begin{align}
&  u_{n-1}u_{n}x_{n}w_{n-1}\nonumber\\
&  =u_{n-1}u_{n}\left(  \left(  x_{0}-y_{0}\right)  t_{n-1,-1}+y_{0}%
\cdot\left(  u_{0}+x_{0}\right)  \cdot u_{1}u_{2}\cdots u_{n-2}\right)
\nonumber\\
&  =u_{n-1}\underbrace{u_{n}}_{=u_{0}}\left(  x_{0}-y_{0}\right)
t_{n-1,-1}+\underbrace{u_{n-1}u_{n}\cdot y_{0}\cdot\left(  u_{0}+x_{0}\right)
\cdot u_{1}u_{2}\cdots u_{n-2}}_{=y_{0}\cdot\left(  u_{0}+x_{0}\right)  \cdot
u_{1}u_{2}\cdots u_{n}}\nonumber\\
&  =u_{n-1}u_{0}\left(  x_{0}-y_{0}\right)  t_{n-1,-1}+y_{0}\cdot\left(
u_{0}+x_{0}\right)  \cdot u_{1}u_{2}\cdots u_{n}.
\label{pf.prop.unique.c.Qplus.xnwn-1}%
\end{align}

The definition of $w_{n+1}$ yields%
\begin{align}
w_{n+1}  &  =\underbrace{g_{n+1}}_{\substack{=\left(  u_{0}+x_{0}\right)
\cdot u_{1}u_{2}\cdots u_{n}+x_{0}x_{1}h_{n+1}\\\text{(by
(\ref{pf.prop.unique.c.Qplus.c4.1}), applied to }k=n+1\text{)}}}-x_{1}%
y_{0}h_{n+1}\nonumber\\
&  =\left(  u_{0}+x_{0}\right)  \cdot u_{1}u_{2}\cdots u_{n}+x_{0}x_{1}%
h_{n+1}-x_{1}y_{0}h_{n+1}\nonumber\\
&  =\left(  u_{0}+x_{0}\right)  \cdot u_{1}u_{2}\cdots u_{n}+x_{1}\left(
x_{0}-y_{0}\right)  \underbrace{h_{n+1}}_{\substack{=t_{n-1,1}\\\text{(by the
definition}\\\text{of }h_{n+1}\text{)}}}\nonumber\\
&  =\left(  u_{0}+x_{0}\right)  \cdot u_{1}u_{2}\cdots u_{n}+x_{1}\left(
x_{0}-y_{0}\right)  t_{n-1,1}. \label{pf.prop.unique.c.Qplus.wnp1}%
\end{align}

Now, (\ref{pf.prop.unique.c.Qplus.yn}) becomes%
\[
y_{n}=\dfrac{u_{n-1}u_{n}x_{n}w_{n-1}}{w_{n+1}}=\dfrac{u_{n-1}u_{0}\left(
x_{0}-y_{0}\right)  t_{n-1,-1}+y_{0}\cdot\left(  u_{0}+x_{0}\right)  \cdot
u_{1}u_{2}\cdots u_{n}}{\left(  u_{0}+x_{0}\right)  \cdot u_{1}u_{2}\cdots
u_{n}+x_{1}\left(  x_{0}-y_{0}\right)  t_{n-1,1}}%
\]
(by (\ref{pf.prop.unique.c.Qplus.xnwn-1}) and
(\ref{pf.prop.unique.c.Qplus.wnp1})). Comparing this with $y_{n}=y_{0}$, we
obtain%
\[
y_{0}=\dfrac{u_{n-1}u_{0}\left(  x_{0}-y_{0}\right)  t_{n-1,-1}+y_{0}%
\cdot\left(  u_{0}+x_{0}\right)  \cdot u_{1}u_{2}\cdots u_{n}}{\left(
u_{0}+x_{0}\right)  \cdot u_{1}u_{2}\cdots u_{n}+x_{1}\left(  x_{0}%
-y_{0}\right)  t_{n-1,1}}.
\]
In other words,%
\begin{align*}
&  y_{0}\cdot\left(  \left(  u_{0}+x_{0}\right)  \cdot u_{1}u_{2}\cdots
u_{n}+x_{1}\left(  x_{0}-y_{0}\right)  t_{n-1,1}\right) \\
&  =u_{n-1}u_{0}\left(  x_{0}-y_{0}\right)  t_{n-1,-1}+y_{0}\cdot\left(
u_{0}+x_{0}\right)  \cdot u_{1}u_{2}\cdots u_{n}.
\end{align*}
Thus,%
\begin{align*}
0  &  =u_{n-1}u_{0}\left(  x_{0}-y_{0}\right)  t_{n-1,-1}+y_{0}\cdot\left(
u_{0}+x_{0}\right)  \cdot u_{1}u_{2}\cdots u_{n}\\
&  \ \ \ \ \ \ \ \ \ \ -y_{0}\cdot\left(  \left(  u_{0}+x_{0}\right)  \cdot
u_{1}u_{2}\cdots u_{n}+x_{1}\left(  x_{0}-y_{0}\right)  t_{n-1,1}\right) \\
&  =u_{n-1}u_{0}\left(  x_{0}-y_{0}\right)  t_{n-1,-1}-y_{0}x_{1}\left(
x_{0}-y_{0}\right)  t_{n-1,1}\\
&  =\left(  x_{0}-y_{0}\right)  \left(  u_{n-1}u_{0}t_{n-1,-1}-y_{0}%
x_{1}t_{n-1,1}\right)  .
\end{align*}
Since $\mathbb{L}$ is an integral domain, we thus conclude that either
$0=x_{0}-y_{0}$ or $0=u_{n-1}u_{0}t_{n-1,-1}-y_{0}x_{1}t_{n-1,1}$. In the
former case, we obtain $y_{0}=x_{0}$; in the latter, we find
\begin{align*}
y_{0}  &  =\dfrac{u_{n-1}u_{0}t_{n-1,-1}}{x_{1}t_{n-1,1}}=u_{0}\cdot
\dfrac{u_{n-1}t_{n-1,-1}}{x_{1}t_{n-1,1}}=u_{0}\cdot\dfrac{u_{-1}t_{n-1,-1}%
}{x_{1}t_{n-1,1}}\ \ \ \ \ \ \ \ \ \ \left(  \text{since }u_{n-1}%
=u_{-1}\right) \\
&  =\left(  \mathbf{f}_{u}\left(  x\right)  \right)  _{0}%
\end{align*}
(by Lemma \ref{lem.f.steps} \textbf{(h)} (applied to $\mathbf{f}_{u}\left(
x\right)  $ and $0$ instead of $y$ and $i$), because the $y$ in Lemma
\ref{lem.f.steps} \textbf{(h)} equals $\mathbf{f}_{u}\left(  x\right)  $).
Thus, we have shown that either $y_{0}=x_{0}$ or $y_{0}=\left(  \mathbf{f}%
_{u}\left(  x\right)  \right)  _{0}$.

However, the equalities (\ref{pf.prop.unique.c.Qplus.c3.2}) holding for all
$k\geq2$ show that $y_{k}$ is uniquely determined by $y_{0}$ for all $k\geq2$.
The same formula determines $x_{k}$ in terms of $x_{0}$ and determines
$\left(  \mathbf{f}_{u}\left(  x\right)  \right)  _{k}$ in terms of $\left(
\mathbf{f}_{u}\left(  x\right)  \right)  _{0}$ (since both $n$-tuples $x$ and
$\mathbf{f}_{u}\left(  x\right)  $ satisfy the same equalities
(\ref{eq.quest.f.unique.c}) as $y$ does). Thus, if $y_{0}=x_{0}$, then
$y_{k}=x_{k}$ for all $k\geq2$, and therefore $y=x$; likewise, if
$y_{0}=\left(  \mathbf{f}_{u}\left(  x\right)  \right)  _{0}$, then
$y_{k}=\left(  \mathbf{f}_{u}\left(  x\right)  \right)  _{k}$ for all $k\geq
2$, and therefore $y=\mathbf{f}_{u}\left(  x\right)  $. Hence, we conclude
that either $y=x$ or $y=\mathbf{f}_{u}\left(  x\right)  $ (since we know that
either $y_{0}=x_{0}$ or $y_{0}=\left(  \mathbf{f}_{u}\left(  x\right)
\right)  _{0}$). This completes the proof of Proposition
\ref{prop.unique.c.Qplus}.
\end{proof}
\end{verlong}

\subsubsection{Characterizing $\mathbf{f}_{u}\left(  x\right)  $ via the
cyclic equations and the product equation}

Another avatar of the uniqueness question is the following:

\begin{question}
\label{quest.f.unique}Given $x\in\mathbb{K}^{n}$ and $y\in\mathbb{K}^{n}$
satisfying both (\ref{eq.quest.f.unique.c}) for all $i\in\mathbb{Z}$ and
\begin{equation}
y_{1}y_{2}\cdots y_{n}\cdot x_{1}x_{2}\cdots x_{n}=\left(  u_{1}u_{2}\cdots
u_{n}\right)  ^{2}. \label{eq.quest.f.unique.b}%
\end{equation}
Does it follow that $y=\mathbf{f}_{u}\left(  x\right)  $ ?
\end{question}

The answer to this question is definitely \textquotedblleft
yes\textquotedblright\ when $\mathbb{K}=\mathbb{Q}_{+}$, by essentially the
same argument that was used in Remark \ref{rmk.f.invol-by-trick}. Again,
however, the answer is \textquotedblleft no\textquotedblright\ when
$\mathbb{K}=\left(  \mathbb{Z},\min,+,0\right)  $. For example, if
$\mathbb{K}=\left(  \mathbb{Z},\min,+,0\right)  $ and $n=4$ and $u=\left(
2,1,1,0\right)  $ and $x=\left(  1,1,1,1\right)  $, then the two $n$-tuples
$\left(  1,1,1,1\right)  $ and $\left(  2,2,0,0\right)  $ both can be taken as
$y$ in Question \ref{quest.f.unique}, but clearly cannot both equal
$\mathbf{f}_{u}\left(  x\right)  $. (On the other hand, if
$\mathbb{K}=\left(  \mathbb{Z},\min,+,0\right)  $ and $n=3$, then
the answer is \textquotedblleft yes\textquotedblright\ again; this can
be shown by an unenlightening yet not particularly arduous case analysis.)

An even stronger version of Question \ref{quest.f.unique} holds when
$\mathbb{K}=\mathbb{Q}_{+}$:

\begin{proposition}
\label{prop.unique.bc.Qplus}Assume that $\mathbb{K}=\mathbb{Q}_{+}$. Let
$x\in\mathbb{K}^{n}$ and $y\in\mathbb{K}^{n}$. Assume that
(\ref{eq.quest.f.unique.c}) holds for all $i\in\left\{  1,2,\ldots
,n-1\right\}  $, and assume that (\ref{eq.quest.f.unique.b}) holds. Then,
$y=\mathbf{f}_{u}\left(  x\right)  $.
\end{proposition}

\begin{vershort}
The proof of Proposition \ref{prop.unique.bc.Qplus} is sketched in the
detailed version \cite{verlong} of this paper.
\end{vershort}

\begin{verlong}
\begin{proof}
[Proof of Proposition \ref{prop.unique.bc.Qplus} (sketched).]Let
$z=\mathbf{f}_{u}\left(  x\right)  $. We have either $y_{1}\geq z_{1}$ or
$y_{1}\leq z_{1}$. Assume WLOG that $y_{1}\geq z_{1}$ holds (since in the
other case, we can use the same argument with all inequality signs reversed).

Now, we notice the following:

\begin{statement}
\textit{Claim 1:} We have $y_{k}\geq z_{k}$ for each $k\in\left\{
1,2,\ldots,n\right\}  $.
\end{statement}

[\textit{Proof of Claim 1:} We shall prove Claim 1 by induction on $k$:

\textit{Induction base:} We have assumed that $y_{1}\geq z_{1}$. In other
words, Claim 1 holds for $k=1$. This completes the induction base.

\textit{Induction step:} Let $i\in\left\{  1,2,\ldots,n-1\right\}  $. Assume
that Claim 1 holds for $k=i$. We must prove that Claim 1 holds for $k=i+1$.

We have assumed that Claim 1 holds for $k=i$. In other words, we have
$y_{i}\geq z_{i}$.

We have $i\in\left\{  1,2,\ldots,n-1\right\}  $; thus,
(\ref{eq.quest.f.unique.c}) holds (by our assumption). In other words,%
\[
\left(  u_{i}+x_{i}\right)  \left(  \dfrac{1}{u_{i+1}}+\dfrac{1}{x_{i+1}%
}\right)  =\left(  u_{i}+y_{i}\right)  \left(  \dfrac{1}{u_{i+1}}+\dfrac
{1}{y_{i+1}}\right)  .
\]
On the other hand, Theorem \ref{thm.f.full} \textbf{(c)} (applied to $z$
instead of $y$) yields%
\[
\left(  u_{i}+x_{i}\right)  \left(  \dfrac{1}{u_{i+1}}+\dfrac{1}{x_{i+1}%
}\right)  =\left(  u_{i}+z_{i}\right)  \left(  \dfrac{1}{u_{i+1}}+\dfrac
{1}{z_{i+1}}\right)
\]
(since $z=\mathbf{f}_{u}\left(  x\right)  $). Comparing these two equalities,
we obtain%
\[
\left(  u_{i}+y_{i}\right)  \left(  \dfrac{1}{u_{i+1}}+\dfrac{1}{y_{i+1}%
}\right)  =\left(  u_{i}+z_{i}\right)  \left(  \dfrac{1}{u_{i+1}}+\dfrac
{1}{z_{i+1}}\right)  .
\]
Hence,%
\[
\left(  u_{i}+z_{i}\right)  \left(  \dfrac{1}{u_{i+1}}+\dfrac{1}{z_{i+1}%
}\right)  =\left(  u_{i}+\underbrace{y_{i}}_{\geq z_{i}}\right)  \left(
\dfrac{1}{u_{i+1}}+\dfrac{1}{y_{i+1}}\right)  \geq\left(  u_{i}+z_{i}\right)
\left(  \dfrac{1}{u_{i+1}}+\dfrac{1}{y_{i+1}}\right)  .
\]
We can cancel $u_{i}+z_{i}$ from this inequality (since $u_{i}+z_{i}$ is a
positive rational number), and obtain $\dfrac{1}{u_{i+1}}+\dfrac{1}{z_{i+1}%
}\geq\dfrac{1}{u_{i+1}}+\dfrac{1}{y_{i+1}}$. In other words, $\dfrac
{1}{z_{i+1}}\geq\dfrac{1}{y_{i+1}}$. Thus, $y_{i+1}\geq z_{i+1}$. In other
words, Claim 1 holds for $k=i+1$. This completes the induction step. Thus,
Claim 1 is proven by induction.]

On the other hand, Theorem \ref{thm.f.full} \textbf{(b)} (applied to $z$
instead of $y$) yields%
\begin{align*}
z_{1}z_{2}\cdots z_{n}\cdot x_{1}x_{2}\cdots x_{n}  &  =\left(  u_{1}%
u_{2}\cdots u_{n}\right)  ^{2}\ \ \ \ \ \ \ \ \ \ \left(  \text{since
}z=\mathbf{f}_{u}\left(  x\right)  \right) \\
&  =y_{1}y_{2}\cdots y_{n}\cdot x_{1}x_{2}\cdots x_{n}%
\end{align*}
(by (\ref{eq.quest.f.unique.b})). We can cancel $x_{1}x_{2}\cdots x_{n}$ from
this equality (since $x_{1}x_{2}\cdots x_{n}$ is a positive rational number),
and obtain $z_{1}z_{2}\cdots z_{n}=y_{1}y_{2}\cdots y_{n}$.

Claim 1 shows that $y_{1}\geq z_{1}$ and $y_{2}\geq z_{2}$ and $\ldots$ and
$y_{n}\geq z_{n}$. Multiplying these $n$ inequalities yields%
\[
y_{1}y_{2}\cdots y_{n}\geq z_{1}z_{2}\cdots z_{n}.
\]
But this inequality must be an equality (since $z_{1}z_{2}\cdots z_{n}%
=y_{1}y_{2}\cdots y_{n}$). Hence, all the $n$ inequalities $y_{1}\geq z_{1}$
and $y_{2}\geq z_{2}$ and $\ldots$ and $y_{n}\geq z_{n}$ (which we multiplied
to obtain it) must be equalities (indeed, since we are working with positive
rational numbers, we will always obtain a strict inequality if we multiply a
strict inequality with a weak inequality). In other words, we have
$y_{1}=z_{1}$ and $y_{2}=z_{2}$ and $\ldots$ and $y_{n}=z_{n}$. In other
words, $y=z$. In other words, $y=\mathbf{f}_{u}\left(  x\right)  $ (since
$z=\mathbf{f}_{u}\left(  x\right)  $). This proves Proposition
\ref{prop.unique.bc.Qplus}.
\end{proof}
\end{verlong}

\subsubsection{Understanding Lemma \ref{lem.f.steps}}

Another question concerns Lemma \ref{lem.f.steps}:

\begin{question}
What is the \textquotedblleft real meaning\textquotedblright\ of some of the
more complicated parts of Lemma \ref{lem.f.steps}? In particular, Lemma
\ref{lem.f.steps} \textbf{(g)} reminds of the Pl\"{u}cker relation for minors
of a $2\times m$-matrix; can it be viewed that way? (Such a proof would not be
superior to the one given above, as it wouldn't be subtraction-free and thus
wouldn't work natively over arbitrary semifields. But it would shine more
light on the lemma.)
\end{question}

\subsection{On the genesis of $\varphi$ (and $\mathbf{f}_{u}$)}

As we mentioned in the introduction to this paper, Pelletier and Ressayre did
not conjecture Theorem \ref{thm.main} in this exact form; instead, they
conjectured the existence of a mysterious bijection $\varphi$ that satisfies
Theorem \ref{thm.main} \textbf{(b)}. Our definition of $\varphi$ appears
\textit{ex caelis oblatus}; while we have seen that our $\varphi$ duly plays
its part, it is far from clear how we have found it in the first place. The
following few paragraphs are meant to demystify this process.

We were looking for a bijection $\varphi:\mathbb{Z}^{n}\rightarrow
\mathbb{Z}^{n}$ satisfying Theorem \ref{thm.main} \textbf{(b)}. In other
words, we were looking for a way to match\footnote{\textquotedblleft
Matching\textquotedblright\ means \textquotedblleft perfect
matching\textquotedblright\ here -- i.e., every coefficient on either side
should get a unique partner.} the nonzero coefficients in the product
$s_{\alpha}\left(  x_{1},x_{2},\ldots,x_{n}\right)  \cdot s_{\mu}\left(
x_{1},x_{2},\ldots,x_{n}\right)  $ (when expanded in the basis $\left(
s_{\lambda}\left(  x_{1},x_{2},\ldots,x_{n}\right)  \right)  _{\lambda
\in\operatorname*{Par}\left[  n\right]  }$ of the $\mathbf{k}$-module of
symmetric polynomials in $x_{1},x_{2},\ldots,x_{n}$) with the nonzero
coefficients in the product $s_{\beta}\left(  x_{1},x_{2},\ldots,x_{n}\right)
\cdot s_{\mu}\left(  x_{1},x_{2},\ldots,x_{n}\right)  $ in such a way that
matching coefficients are equal.

The first step towards this goal was the discovery of the formula $s_{\alpha
}\left(  x_{1},x_{2},\ldots,x_{n}\right)  =x_{\Pi}^{a}\cdot\left(  h_{a}%
^{-}h_{b}^{+}-h_{a-1}^{-}h_{b-1}^{+}\right)  $: our Corollary
\ref{cor.ominus.salpha}. We originally proved this formula combinatorially, by
analyzing the structure of semistandard tableaux of shape $\alpha
$.\ \ \ \ \footnote{Each of the first $a$ columns of such a tableau would have
the form $\left(  1,2,\ldots,i-1,i+1,\ldots,n\right)  $ for some $i\in\left\{
1,2,\ldots,n\right\}  $, and these numbers $i$ would weakly increase as one
moves right.} The proof of Corollary \ref{cor.ominus.salpha} given above
(using the Pieri rule) was an afterthought.

Corollary \ref{cor.ominus.salpha} was a visible step in the right direction,
as it moved the problem from the world of Littlewood--Richardson coefficients
into the simpler world of Pieri rules. Indeed, instead of expanding
$s_{\alpha}\left(  x_{1},x_{2},\ldots,x_{n}\right)  \cdot s_{\mu}\left(
x_{1},x_{2},\ldots,x_{n}\right)  $, we now only had to expand $x_{\Pi}%
^{a}\cdot\left(  h_{a}^{-}h_{b}^{+}-h_{a-1}^{-}h_{b-1}^{+}\right)  \cdot
s_{\mu}\left(  x_{1},x_{2},\ldots,x_{n}\right)  $, which looked like an
expansion that the Pieri rule could help with (to be fully honest, we only
knew the Pieri rule for multiplying by $h_{k}^{+}$; but we soon would find one
for multiplying by $h_{k}^{-}$). The $x_{\Pi}^{a}$ factor was clearly a mere
distraction, but in order to get rid of it, we had to extend our polynomial
ring to the ring $\mathcal{L}$ of Laurent polynomials (since $\left(
h_{a}^{-}h_{b}^{+}-h_{a-1}^{-}h_{b-1}^{+}\right)  \cdot s_{\mu}\left(
x_{1},x_{2},\ldots,x_{n}\right)  $ is, in general, not a polynomial). This
extension had already been done by Stembridge in \cite{Stembr87}, and all we
had to do was rename \textquotedblleft staircases\textquotedblright\ as
\textquotedblleft snakes\textquotedblright, define Schur Laurent polynomials
(by generalizing the alternant formula for Schur polynomials in the most
obvious way), extend some basic properties of Schur polynomials to Schur
Laurent polynomials, and find the \textquotedblleft
upside-down\textquotedblright\ Pieri rule (Proposition \ref{prop.alt.pieri2}).
None of this was difficult; in particular, the \textquotedblleft
upside-down\textquotedblright\ Pieri rule followed easily from the usual Pieri
rule using Lemma \ref{lem.alt.inverses} (which is our main device for turning
things \textquotedblleft upside down\textquotedblright). The Schur polynomial
$s_{\nu}\left(  x_{1},x_{2},\ldots,x_{n}\right)  $ was generalized to the
Schur Laurent polynomial $\overline{s}_{\nu}$.

Thus our problem was reduced to matching the nonzero coefficients in the
product $\left(  h_{a}^{-}h_{b}^{+}-h_{a-1}^{-}h_{b-1}^{+}\right)
\cdot\overline{s}_{\mu}$ with the nonzero coefficients in the product $\left(
h_{b}^{-}h_{a}^{+}-h_{b-1}^{-}h_{a-1}^{+}\right)  \cdot\overline{s}_{\mu}$.
The products could both be expressed using the Pieri rules, but the
differences were still a distraction. At this point, we made a fortunate
guess: We hoped it would suffice to match the nonzero coefficients in the
product $h_{a}^{-}h_{b}^{+}\cdot\overline{s}_{\mu}$ with the nonzero
coefficients in the product $h_{b}^{-}h_{a}^{+}\cdot\overline{s}_{\mu}$. More
precisely, we hoped to find such a matching that would not depend on $a$ and
$b$; then it would also provide a matching between the nonzero coefficients in
the product $h_{a-1}^{-}h_{b-1}^{+}\cdot\overline{s}_{\mu}$ and the nonzero
coefficients in the product $h_{b-1}^{-}h_{a-1}^{+}\cdot\overline{s}_{\mu}$,
and therefore by taking differences we would obtain a matching between the
nonzero coefficients in the product $\left(  h_{a}^{-}h_{b}^{+}-h_{a-1}%
^{-}h_{b-1}^{+}\right)  \cdot\overline{s}_{\mu}$ and the nonzero coefficients
in the product $\left(  h_{b}^{-}h_{a}^{+}-h_{b-1}^{-}h_{a-1}^{+}\right)
\cdot\overline{s}_{\mu}$.

Thus we needed to expand $h_{a}^{-}h_{b}^{+}\cdot\overline{s}_{\mu}$. Using
the Pieri rules, this was straightforward -- the answer is in Lemma
\ref{lem.Rmab.formula}. Our problem was to connect this result with what we
would similarly obtain from expanding $h_{b}^{-}h_{a}^{+}\cdot\overline
{s}_{\mu}$. In other words, we wanted to construct a bijection $\mathbf{f}%
_{\mu}:\mathbb{Z}^{n}\rightarrow\mathbb{Z}^{n}$ that would satisfy%
\[
\left\vert R_{\mu,b,a}\left(  \mathbf{f}_{\mu}\left(  \gamma\right)  \right)
\right\vert =\left\vert R_{\mu,a,b}\left(  \gamma\right)  \right\vert
\ \ \ \ \ \ \ \ \ \ \text{for any }a,b\in\mathbb{Z}\text{ and }\gamma
\in\mathbb{Z}^{n}.
\]
Fixing $\gamma\in\mathbb{Z}^{n}$, we thus were looking for an $n$-tuple $\eta$
(our $\mathbf{f}_{\mu}\left(  \gamma\right)  $-to-be) that would satisfy
$\left\vert R_{\mu,b,a}\left(  \eta\right)  \right\vert =\left\vert
R_{\mu,a,b}\left(  \gamma\right)  \right\vert $.

In the case when $b=0$ (in which case this equality would be equivalent to
saying \textquotedblleft$\eta\rightharpoonup\mu$ if and only if $\mu
\rightharpoonup\gamma$\textquotedblright), we found such an $\eta$ directly,
by setting%
\[
\eta_{i}=\mu_{i}+\mu_{i+1}-\gamma_{i+1}\ \ \ \ \ \ \ \ \ \ \text{for each
}i\in\left\{  1,2,\ldots,n\right\}  ,
\]
where indices are cyclic modulo $n$ (so that $\mu_{0}=\mu_{n}$ and $\nu
_{0}=\nu_{n}$). This formula surprised us with its cyclic symmetry (which was
not expected from the original problem, and which foreshadowed the usefulness
of Convention \ref{conv.bir.peri}, although we thought nothing of it at that
point). Nevertheless, the formula failed in various examples for $b>0$, and we
could not easily fix it.

We tried to be more systematic. It was easy to rewrite the definition of
$R_{\mu,a,b}\left(  \gamma\right)  $ as%
\begin{align*}
&  R_{\mu,a,b}\left(  \gamma\right) \\
&  =\left\{  \nu\in\mathbb{Z}^{n}\ \mid\ \left(  \min\left\{  \mu_{i}%
,\gamma_{i}\right\}  \geq\nu_{i}\geq\max\left\{  \mu_{i+1},\gamma
_{i+1}\right\}  \text{ for each }i\in\left\{  1,2,\ldots,n-1\right\}  \right)
\right. \\
&  \ \ \ \ \ \ \ \ \ \ \ \ \ \ \ \ \ \ \ \ \ \ \ \ \ \ \ \ \ \ \left.
\text{and }\left\vert \mu\right\vert -\left\vert \nu\right\vert =a\text{ and
}\left\vert \gamma\right\vert -\left\vert \nu\right\vert =b\right\}  .
\end{align*}
Thus, the size of this set would depend only

\begin{itemize}
\item on the differences $\min\left\{  \mu_{i},\gamma_{i}\right\}
-\max\left\{  \mu_{i+1},\gamma_{i+1}\right\}  $ for $i\in\left\{
1,2,\ldots,n-1\right\}  $ (each of which differences would determine the
\textquotedblleft breathing space\textquotedblright\ for the corresponding
$\nu_{i}$),

\item on the difference $\left\vert \mu\right\vert -\left\vert \gamma
\right\vert $ (which would have to equal $a-b$ in order for the two conditions
$\left\vert \mu\right\vert -\left\vert \nu\right\vert =a$ and $\left\vert
\gamma\right\vert -\left\vert \nu\right\vert =b$ to be satisfiable simultaneously),

\item as well as on something else we could not quite pinpoint (in order for
$\left\vert \mu\right\vert -\left\vert \nu\right\vert =a$ and $\left\vert
\gamma\right\vert -\left\vert \nu\right\vert =b$ to actually hold, as opposed
to merely $\left\vert \mu\right\vert -\left\vert \gamma\right\vert =a-b$).
\end{itemize}

\noindent Analogous observations held for $R_{\mu,b,a}\left(  \eta\right)  $.
With Occam's razor in hand, we suspected that $\left\vert R_{\mu,b,a}\left(
\eta\right)  \right\vert =\left\vert R_{\mu,a,b}\left(  \gamma\right)
\right\vert $ could best be achieved by requiring these differences to be the
same for $\left(  \mu,a,b,\gamma\right)  $ as for $\left(  \mu,b,a,\eta
\right)  $. Thus, in particular, we hoped to have%
\begin{align*}
\min\left\{  \mu_{i},\gamma_{i}\right\}  -\max\left\{  \mu_{i+1},\gamma
_{i+1}\right\}   &  =\min\left\{  \mu_{i},\eta_{i}\right\}  -\max\left\{
\mu_{i+1},\eta_{i+1}\right\} \\
&  \ \ \ \ \ \ \ \ \ \ \text{for each }i\in\left\{  1,2,\ldots,n-1\right\}
\end{align*}
and%
\[
\left\vert \mu\right\vert -\left\vert \gamma\right\vert =\left\vert
\eta\right\vert -\left\vert \mu\right\vert .
\]
(Due to the \textquotedblleft mystery ingredient\textquotedblright, this would
likely neither be necessary nor sufficient for $\left\vert R_{\mu,b,a}\left(
\eta\right)  \right\vert =\left\vert R_{\mu,a,b}\left(  \gamma\right)
\right\vert $, but it looked like the right tree to bark up.) This is a system
of equations that whose solution is neither unique nor straightforward.
However, the system was a beacon rather than a destination to us, so we merely
needed something like a good solution.

Systems of equations involving sums, differences, minima and maxima belong to
\textit{tropical geometry} -- a discipline we were not expert in and could not
hope to master quickly. However, we were aware of a surprisingly successful
strategy for taming such systems: \textit{detropicalization}. The mainstay of
this strategy is the observation (made above in Example
\ref{exa.semifield.mintrop}) that the binary operations $\min$, $\max$, $+$
and $-$ are the addition, the \textquotedblleft harmonic
addition\textquotedblright\footnote{\textit{Harmonic addition} is a binary
operation defined on any semifield. It sends any pair $\left(  a,b\right)  $
of elements of the semifield to $\dfrac{1}{\dfrac{1}{a}+\dfrac{1}{b}}%
=\dfrac{ab}{a+b}$.}, the multiplication and the division of a certain
semifield (the min tropical semifield of $\left(  \mathbb{Z},+,0\right)  $, or
of whatever totally ordered abelian group our numbers belong to). Thus, even
if we could not solve our system, we could generalize it to arbitrary
semifields by replacing $\min$, $\max$, $+$ and $-$ by addition,
\textquotedblleft harmonic addition\textquotedblright, multiplication and
division, respectively. Thus our system would become%
\begin{align*}
\left(  \mu_{i}+\gamma_{i}\right)  /\dfrac{1}{\dfrac{1}{\mu_{i+1}}+\dfrac
{1}{\gamma_{i+1}}}  &  =\left(  \mu_{i}+\eta_{i}\right)  /\dfrac{1}{\dfrac
{1}{\mu_{i+1}}+\dfrac{1}{\eta_{i+1}}}\\
&  \ \ \ \ \ \ \ \ \ \ \text{for each }i\in\left\{  1,2,\ldots,n-1\right\}
\end{align*}
and%
\[
\dfrac{\mu_{1}\mu_{2}\cdots\mu_{n}}{\gamma_{1}\gamma_{2}\cdots\gamma_{n}%
}=\dfrac{\eta_{1}\eta_{2}\cdots\eta_{n}}{\mu_{1}\mu_{2}\cdots\mu_{n}}.
\]
Renaming $\mu$, $\gamma$ and $\eta$ as $u$, $x$ and $y$, and simplifying the
fractions somewhat, we rewrote this as%
\begin{align*}
\left(  u_{i}+x_{i}\right)  \left(  \dfrac{1}{u_{i+1}}+\dfrac{1}{x_{i+1}%
}\right)   &  =\left(  u_{i}+y_{i}\right)  \left(  \dfrac{1}{u_{i+1}}%
+\dfrac{1}{y_{i+1}}\right) \\
&  \ \ \ \ \ \ \ \ \ \ \text{for each }i\in\left\{  1,2,\ldots,n-1\right\}
\end{align*}
and%
\[
y_{1}y_{2}\cdots y_{n}\cdot x_{1}x_{2}\cdots x_{n}=\left(  u_{1}u_{2}\cdots
u_{n}\right)  ^{2}.
\]

This new system was a system of polynomial equations (at least after clearing
denominators), so we did the obvious thing: We left it to the computer for
small values of $n$ (specifically, $n=2$, $n=3$ and $n=4$) and looked at the
results. For $n=3$, the computer (SageMath's \texttt{solve} function, to be
precise) laid out the following two solutions:

\begin{itemize}
\item \textit{Solution 1:}%
\begin{align*}
y_{1}  &  =\dfrac{u_{1}\left(  u_{1}u_{2}u_{3}+x_{1}u_{2}u_{3}+x_{1}x_{2}%
u_{3}+x_{1}x_{2}x_{3}\right)  }{u_{1}x_{2}u_{3}-x_{1}x_{2}x_{3}},\\
y_{2}  &  =\dfrac{-u_{1}u_{2}u_{3}}{x_{1}x_{3}},\\
y_{3}  &  =\dfrac{u_{2}u_{3}\left(  x_{1}x_{3}-u_{1}u_{3}\right)  }{u_{1}%
u_{2}u_{3}+x_{1}u_{2}u_{3}+x_{1}x_{2}u_{3}+x_{1}x_{2}x_{3}}.
\end{align*}

\item \textit{Solution 2:}%
\begin{align*}
y_{1}  &  =\dfrac{u_{1}u_{3}\left(  u_{1}u_{2}+x_{1}u_{2}+x_{1}x_{2}\right)
}{x_{2}\left(  u_{1}u_{3}+u_{1}x_{3}+x_{1}x_{3}\right)  },\\
y_{2}  &  =\dfrac{u_{1}u_{2}\left(  u_{2}u_{3}+x_{2}u_{3}+x_{2}x_{3}\right)
}{x_{3}\left(  u_{1}u_{2}+x_{1}u_{2}+x_{1}x_{2}\right)  },\\
y_{3}  &  =\dfrac{u_{2}u_{3}\left(  u_{1}u_{3}+u_{1}x_{3}+x_{1}x_{3}\right)
}{x_{1}\left(  u_{2}u_{3}+x_{2}u_{3}+x_{2}x_{3}\right)  }.
\end{align*}

\end{itemize}

The computer did not know that we were trying to work over a semifield (which
had no subtraction), but we did, so we immediately discarded Solution 1 as
useless due to the minus signs. The question was whether Solution 2 would be
of any use. The omens were favorable: There were no minus signs; the
(unexpected, but not unwelcome) cyclic symmetry reared its head again;
finally, the nontrivial factors (such as $u_{1}u_{2}+x_{1}u_{2}+x_{1}x_{2}$)
had a structure that appeared in the definition of the geometric crystal
R-matrix (see, e.g., \cite[(5)]{Etingo03} or \cite[(4.19)]{NoumiYamada}) -- a
known successful case of detropicalization.

Solution 2 turned out to be generalizable indeed. Proving that the general
formula indeed produced a solution to our system (parts \textbf{(b)} and
\textbf{(c)} of Theorem \ref{thm.f.full}) was not completely trivial, but not
hard either. (The first few parts of Lemma \ref{lem.f.steps} were discovered
along the way.) Thus we had a candidate for the map $\mathbf{f}_{\mu}$ (and
thus for the map $\varphi$, which was obtained from $\mathbf{f}_{\mu}$ by
shifting by $a$ and $b$, corresponding to the $x_{\Pi}^{a}$ factor that we had dropped).

Why was this map $\mathbf{f}_{\mu}$ a bijection? Again, we believed that the
easiest way lay through the birational realm (i.e., we had to detropicalize).
Computer experiments suggested that $\mathbf{f}_{\mu}$ was not only a
bijection but actually an involution (part \textbf{(a)} of Theorem
\ref{thm.f.full}). The first proof of this we found was the one sketched in
Remark \ref{rmk.f.invol-by-trick}; the alternative, computational proof that
we gave first was found afterwards.

Having found our bijection $\mathbf{f}_{\mu}$, we had to retrace our steps.
Most of this was straightforward. The equality $\left\vert R_{\mu,b,a}\left(
\mathbf{f}_{\mu}\left(  \gamma\right)  \right)  \right\vert =\left\vert
R_{\mu,a,b}\left(  \gamma\right)  \right\vert $ still had to be proved, but
this turned out to be rather easy (part \textbf{(d)} of Theorem
\ref{thm.f.full} was discovered along the way, as the missing ingredient from
our previous analysis of the size of $R_{\mu,a,b}\left(  \gamma\right)  $).
The way the proof was written up in the end was mostly decided by concerns of
readability rather than authenticity; we believe that, had we followed the
logic of its discovery in our writeup, we would have lost more in clarity than
would be gained in motivation. We placed the study of the birational map
$\mathbf{f}_{u}$ (Section \ref{sect.bir}) in front due to its self-contained
nature and possible applicability to different problems; likewise, Section
\ref{sect.pf} begins with general properties of Schur Laurent polynomials and
slowly progresses towards more technical lemmas tailored for the proof of
Theorem \ref{thm.main}. We would not be too surprised if some unnecessary
detours were made along our way (Lemma \ref{lem.f.steps} appears a
particularly likely place for such), for which we apologize in advance (any
simplifications are appreciated).

\subsection{\label{subsect.fin.Rmat}The birational $R$-matrix connection}

In this section, we shall connect the map $\mathbf{f}_{u}$ from our Definition
\ref{def.f} with the \textit{birational R-matrix }$\eta$ defined in
\cite[\S 6]{LamPyl12} and studied further (e.g.) in \cite{CheLin20}.

We fix a positive integer $n$ and a semifield $\mathbb{K}$. We shall use
Convention \ref{conv.semifield.notations} and Convention \ref{conv.bir.peri}.
Let us recall the definition of the birational R-matrix $\eta$ (no relation to
the $\eta$ in Theorem \ref{thm.main}):

\begin{definition}
\label{def.eta}We define a map $\eta:\mathbb{K}^{n}\times\mathbb{K}%
^{n}\rightarrow\mathbb{K}^{n}\times\mathbb{K}^{n}$ as follows:

Let $a\in\mathbb{K}^{n}$ and $b\in\mathbb{K}^{n}$ be two $n$-tuples. For any
$i\in\mathbb{Z}$, define an element $\kappa_{i}\left(  a,b\right)
\in\mathbb{K}$ by%
\[
\kappa_{i}\left(  a,b\right)  =\sum_{j=i}^{i+n-1}\underbrace{b_{i+1}%
b_{i+2}\cdots b_{j}}_{=\prod_{p=i+1}^{j}b_{p}}\cdot\underbrace{a_{j+1}%
a_{j+2}\cdots a_{i+n-1}}_{=\prod_{p=j+1}^{i+n-1}a_{p}}.
\]
Define $a^{\prime}\in\mathbb{K}^{n}$ and $b^{\prime}\in\mathbb{K}^{n}$ by
setting%
\[
a_{i}^{\prime}=\dfrac{a_{i-1}\kappa_{i-1}\left(  a,b\right)  }{\kappa
_{i}\left(  a,b\right)  }\ \ \ \ \ \ \ \ \ \ \text{for each }i\in\left\{
1,2,\ldots,n\right\}
\]
and%
\[
b_{i}^{\prime}=\dfrac{b_{i+1}\kappa_{i+1}\left(  a,b\right)  }{\kappa
_{i}\left(  a,b\right)  }\ \ \ \ \ \ \ \ \ \ \text{for each }i\in\left\{
1,2,\ldots,n\right\}  .
\]
Set $\eta\left(  a,b\right)  =\left(  a^{\prime},b^{\prime}\right)  $.
\end{definition}

The map $\eta$ we just defined is known as a \textit{birational R-matrix};
related maps have previously appeared in the literature (\cite[Lemma
8.6]{BraKaz00}, \cite[Definition 2.1]{Yamada01}, \cite[Proposition
3.1]{Etingo03}). In particular, the map $R$ from \cite[Proposition
3.1]{Etingo03} is equivalent to $\eta$ (at least up to technical issues of
where it is defined\footnote{Namely: We have defined our map $\eta$ as a
literal map $\mathbb{K}^{n}\times\mathbb{K}^{n}\rightarrow\mathbb{K}^{n}%
\times\mathbb{K}^{n}$ for any semifield $\mathbb{K}$, whereas
\cite[Proposition 3.1]{Etingo03} defines $R$ as a birational map $\left(
\mathbb{C}^{\times}\right)  ^{n}\times\left(  \mathbb{C}^{\times}\right)
^{n}\dashrightarrow\left(  \mathbb{C}^{\times}\right)  ^{n}\times\left(
\mathbb{C}^{\times}\right)  ^{n}$. Neither of these two settings generalizes
the other, but it is not hard to transfer identities from one to the other (as
long as they are \textit{subtraction-free}, i.e., no minus signs appear in
them).}). Indeed, it is not hard to see that the map $\eta$ from Definition
\ref{def.eta} becomes the map $R$ from \cite[Proposition 3.1]{Etingo03} if we
set $x_{i}=b_{i+1}$ and $y_{i}=a_{i}$ and $x_{i}^{\prime}=b_{i}^{\prime}$ and
$y_{i}^{\prime}=a_{i+1}^{\prime}$ (that is, if we define $x_{i},y_{i}%
,x_{i}^{\prime},y_{i}^{\prime}$ this way, then the equalities \cite[(8), (9)
and (10)]{Etingo03} are satisfied, so that we have $R\left(  x,y\right)
=\left(  x^{\prime},y^{\prime}\right)  $ where $R$ is as defined in
\cite[Proposition 3.1]{Etingo03}). This birational R-matrix $R$ has its
origins in the theory of geometric crystals and total positivity. A related
map is the transformation $\left(  x,a\right)  \mapsto\left(  y,b\right)  $ in
\cite[\S 2.2]{NoumiYamada} (see also \cite{Zygour18}).

Now, we shall see that the map $\eta$ is intimately related to our map
$\mathbf{f}_{u}$ (even though $\mathbf{f}_{u}$ transforms a single $n$-tuple
$x$ into a single $n$-tuple $y$ using the fixed $n$-tuple $u$, while $\eta$
takes a pair of two $n$-tuples to another such pair). In order to state this
relation, we define some more notation:

\begin{definition}
\label{def.n-tup-frac}If $a\in\mathbb{K}^{n}$ and $b\in\mathbb{K}^{n}$ are two
$n$-tuples, then we define two new $n$-tuples $ab\in\mathbb{K}^{n}$ and
$\dfrac{a}{b}\in\mathbb{K}^{n}$ by setting%
\[
\left(  ab\right)  _{i}=a_{i}b_{i}\ \ \ \ \ \ \ \ \ \ \text{and}%
\ \ \ \ \ \ \ \ \ \ \left(  \dfrac{a}{b}\right)  _{i}=\dfrac{a_{i}}{b_{i}%
}\ \ \ \ \ \ \ \ \ \ \text{for each }i\in\left\{  1,2,\ldots,n\right\}  .
\]

\end{definition}

We can now express the map $\mathbf{f}_{u}$ from Definition \ref{def.f}
through the map $\eta$ from Definition \ref{def.eta} as follows:

\begin{theorem}
\label{thm.eta-f1}Let $u\in\mathbb{K}^{n}$ and $x\in\mathbb{K}^{n}$ be two
$n$-tuples. Let $\left(  a^{\prime},b^{\prime}\right)  =\eta\left(
u,x\right)  $. Then,%
\[
\mathbf{f}_{u}\left(  x\right)  =u\dfrac{a^{\prime}}{b^{\prime}}.
\]

\end{theorem}

\begin{proof}
[Proof of Theorem \ref{thm.eta-f1}.]Set $a=u$ and $b=x$. We shall use the
notations $\kappa_{i}\left(  a,b\right)  $ from Definition \ref{def.eta} and
the notations $t_{r,j}$ and $y$ from Definition \ref{def.f}. Then,
$\mathbf{f}_{u}\left(  x\right)  =y$ (by Definition \ref{def.f}).

For each $i\in\mathbb{Z}$, we have%
\begin{align}
\kappa_{i}\left(  a,b\right)   &  =\sum_{j=i}^{i+n-1}b_{i+1}b_{i+2}\cdots
b_{j}\cdot a_{j+1}a_{j+2}\cdots a_{i+n-1}\ \ \ \ \ \ \ \ \ \ \left(  \text{by
the definition of }\kappa_{i}\left(  a,b\right)  \right) \nonumber\\
&  =\sum_{j=i}^{i+n-1}x_{i+1}x_{i+2}\cdots x_{j}\cdot u_{j+1}u_{j+2}\cdots
u_{i+n-1}\ \ \ \ \ \ \ \ \ \ \left(  \text{since }a=u\text{ and }b=x\right)
\nonumber\\
&  =\sum_{k=0}^{n-1}x_{i+1}x_{i+2}\cdots x_{i+k}\cdot u_{i+k+1}u_{i+k+2}\cdots
u_{i+n-1}\nonumber\\
&  \ \ \ \ \ \ \ \ \ \ \ \ \ \ \ \ \ \ \ \ \left(  \text{here, we have
substituted }i+k\text{ for }j\text{ in the sum}\right) \nonumber\\
&  =t_{n-1,i} \label{pf.thm.eta-f1.k=t}%
\end{align}
(since the definition of $t_{n-1,i}$ yields $t_{n-1,i}=\sum_{k=0}^{n-1}%
x_{i+1}x_{i+2}\cdots x_{i+k}\cdot u_{i+k+1}u_{i+k+2}\cdots u_{i+n-1}$).

However, $\left(  a^{\prime},b^{\prime}\right)  =\eta\left(  u,x\right)
=\eta\left(  a,b\right)  $ (since $u=a$ and $x=b$). Hence, Definition
\ref{def.eta} yields that
\[
a_{i}^{\prime}=\dfrac{a_{i-1}\kappa_{i-1}\left(  a,b\right)  }{\kappa
_{i}\left(  a,b\right)  }\ \ \ \ \ \ \ \ \ \ \text{for each }i\in\left\{
1,2,\ldots,n\right\}
\]
and%
\[
b_{i}^{\prime}=\dfrac{b_{i+1}\kappa_{i+1}\left(  a,b\right)  }{\kappa
_{i}\left(  a,b\right)  }\ \ \ \ \ \ \ \ \ \ \text{for each }i\in\left\{
1,2,\ldots,n\right\}  .
\]
Hence, for each $i\in\left\{  1,2,\ldots,n\right\}  $, we have%
\begin{align}
a_{i}^{\prime}/b_{i}^{\prime} &  =\dfrac{a_{i-1}\kappa_{i-1}\left(
a,b\right)  }{\kappa_{i}\left(  a,b\right)  }/\dfrac{b_{i+1}\kappa
_{i+1}\left(  a,b\right)  }{\kappa_{i}\left(  a,b\right)  }=\dfrac
{a_{i-1}\kappa_{i-1}\left(  a,b\right)  }{b_{i+1}\kappa_{i+1}\left(
a,b\right)  }=\dfrac{a_{i-1}t_{n-1,i-1}}{b_{i+1}t_{n-1,i+1}}\nonumber\\
&  \ \ \ \ \ \ \ \ \ \ \ \ \ \ \ \ \ \ \ \ \left(
\begin{array}
[c]{c}%
\text{since }\kappa_{i-1}\left(  a,b\right)  =t_{n-1,i-1}\text{ (by
(\ref{pf.thm.eta-f1.k=t}), applied to }i-1\text{ instead of }i\text{)}\\
\text{and }\kappa_{i+1}\left(  a,b\right)  =t_{n-1,i+1}\text{ (by
(\ref{pf.thm.eta-f1.k=t}), applied to }i+1\text{ instead of }i\text{)}%
\end{array}
\right)  \nonumber\\
&  =\dfrac{u_{i-1}t_{n-1,i-1}}{x_{i+1}t_{n-1,i+1}}\ \ \ \ \ \ \ \ \ \ \left(
\text{since }a=u\text{ and }b=x\right)  \nonumber\\
&  =y_{i}/u_{i}\label{pf.thm.eta-f1.fracs}%
\end{align}
(since the definition of $y$ yields $y_{i}=u_{i}\cdot\dfrac{u_{i-1}%
t_{n-1,i-1}}{x_{i+1}t_{n-1,i+1}}$). Now, for each $i\in\left\{  1,2,\ldots
,n\right\}  $, we have%
\begin{align*}
\left(  u\dfrac{a^{\prime}}{b^{\prime}}\right)  _{i} &  =u_{i}\dfrac
{a_{i}^{\prime}}{b_{i}^{\prime}}\ \ \ \ \ \ \ \ \ \ \left(  \text{by
Definition \ref{def.n-tup-frac}}\right)  \\
&  =u_{i}\cdot\underbrace{a_{i}^{\prime}/b_{i}^{\prime}}_{\substack{=y_{i}%
/u_{i}\\\text{(by (\ref{pf.thm.eta-f1.fracs}))}}}=u_{i}\cdot y_{i}/u_{i}%
=y_{i}.
\end{align*}
In other words, $u\dfrac{a^{\prime}}{b^{\prime}}=y$. Comparing this with
$\mathbf{f}_{u}\left(  x\right)  =y$, we obtain $\mathbf{f}_{u}\left(
x\right)  =u\dfrac{a^{\prime}}{b^{\prime}}$. This proves Theorem
\ref{thm.eta-f1}.
\end{proof}

We finish by stating some \textquotedblleft gauge-invariance\textquotedblright%
\ properties for $\mathbf{f}_{u}$ and $\eta$:

\begin{proposition}
\label{prop.fu-gauge}Let $g,u,x\in\mathbb{K}^{n}$. Then, $\mathbf{f}%
_{gu}\left(  gx\right)  =g\mathbf{f}_{u}\left(  x\right)  $.
\end{proposition}

\begin{proposition}
\label{prop.eta-gauge}Let $g,a,b\in\mathbb{K}^{n}$. Let $\left(  a^{\prime
},b^{\prime}\right)  =\eta\left(  a,b\right)  $. Then, $\left(  ga^{\prime
},gb^{\prime}\right)  =\eta\left(  ga,gb\right)  $.
\end{proposition}

\begin{vershort}
Both of these propositions can be proved by fairly simple computations, which
are left to the reader.
\end{vershort}

\begin{verlong}
\begin{proof}
[Proof of Proposition \ref{prop.fu-gauge}.] Clearly, we have%
\begin{equation}
\left(  ab\right)  _{i}=a_{i}b_{i}\label{pf.prop.fu-gauge.abi=}%
\end{equation}
for any $a,b\in\mathbb{K}^{n}$ and any $i\in\mathbb{Z}$. (Indeed, because of
Convention \ref{conv.bir.peri}, it suffices to prove this in the case when
$i\in\left\{  1,2,\ldots,n\right\}  $. However, in this case, this follows
from Definition \ref{def.n-tup-frac}.)

We shall use the notations $t_{r,j}$ and $y$ from Definition \ref{def.f}.
Then, $\mathbf{f}_{u}\left(  x\right)  =y$ (by Definition \ref{def.f}).

For each $j\in\mathbb{Z}$, define an element $q_{j}\in\mathbb{K}$ by
\[
q_{j}=\sum_{k=0}^{n-1}\underbrace{\left(  gx\right)  _{j+1}\left(  gx\right)
_{j+2}\cdots\left(  gx\right)  _{j+k}}_{=\prod_{i=1}^{k}\left(  gx\right)
_{j+i}}\cdot\underbrace{\left(  gu\right)  _{j+k+1}\left(  gu\right)
_{j+k+2}\cdots\left(  gu\right)  _{j+n-1}}_{=\prod_{i=k+1}^{n-1}\left(
gu\right)  _{j+i}}.
\]
Thus, for each $j\in\mathbb{Z}$, we have%
\begin{align}
q_{j}  & =\sum_{k=0}^{n-1}\underbrace{\left(  gx\right)  _{j+1}\left(
gx\right)  _{j+2}\cdots\left(  gx\right)  _{j+k}}_{=\prod_{i=1}^{k}\left(
gx\right)  _{j+i}}\cdot\underbrace{\left(  gu\right)  _{j+k+1}\left(
gu\right)  _{j+k+2}\cdots\left(  gu\right)  _{j+n-1}}_{=\prod_{i=k+1}%
^{n-1}\left(  gu\right)  _{j+i}}\nonumber\\
& =\sum_{k=0}^{n-1}\left(  \prod_{i=1}^{k}\underbrace{\left(  gx\right)
_{j+i}}_{\substack{=g_{j+i}x_{j+i}\\\text{(by (\ref{pf.prop.fu-gauge.abi=}))}%
}}\right)  \cdot\left(  \prod_{i=k+1}^{n-1}\underbrace{\left(  gu\right)
_{j+i}}_{\substack{=g_{j+i}u_{j+i}\\\text{(by (\ref{pf.prop.fu-gauge.abi=}))}%
}}\right)  \nonumber\\
& =\sum_{k=0}^{n-1}\underbrace{\left(  \prod_{i=1}^{k}\left(  g_{j+i}%
x_{j+i}\right)  \right)  }_{=\left(  \prod_{i=1}^{k}g_{j+i}\right)
\cdot\left(  \prod_{i=1}^{k}x_{j+i}\right)  }\cdot\underbrace{\left(
\prod_{i=k+1}^{n-1}\left(  g_{j+i}u_{j+i}\right)  \right)  }_{=\left(
\prod_{i=k+1}^{n-1}g_{j+i}\right)  \cdot\left(  \prod_{i=k+1}^{n-1}%
u_{j+i}\right)  }\nonumber\\
& =\sum_{k=0}^{n-1}\underbrace{\left(  \prod_{i=1}^{k}\left(  g_{j+i}%
x_{j+i}\right)  \right)  }_{=\left(  \prod_{i=1}^{k}g_{j+i}\right)
\cdot\left(  \prod_{i=1}^{k}x_{j+i}\right)  }\cdot\underbrace{\left(
\prod_{i=k+1}^{n-1}\left(  g_{j+i}u_{j+i}\right)  \right)  }_{=\left(
\prod_{i=k+1}^{n-1}g_{j+i}\right)  \cdot\left(  \prod_{i=k+1}^{n-1}%
u_{j+i}\right)  }\nonumber\\
& =\sum_{k=0}^{n-1}\left(  \prod_{i=1}^{k}g_{j+i}\right)  \cdot
\underbrace{\left(  \prod_{i=1}^{k}x_{j+i}\right)  \cdot\left(  \prod
_{i=k+1}^{n-1}g_{j+i}\right)  }_{=\left(  \prod_{i=k+1}^{n-1}g_{j+i}\right)
\cdot\left(  \prod_{i=1}^{k}x_{j+i}\right)  }\cdot\left(  \prod_{i=k+1}%
^{n-1}u_{j+i}\right)  \nonumber\\
& =\sum_{k=0}^{n-1}\underbrace{\left(  \prod_{i=1}^{k}g_{j+i}\right)
\cdot\left(  \prod_{i=k+1}^{n-1}g_{j+i}\right)  }_{=\prod_{i=1}^{n-1}g_{j+i}%
}\cdot\left(  \prod_{i=1}^{k}x_{j+i}\right)  \cdot\left(  \prod_{i=k+1}%
^{n-1}u_{j+i}\right)  \nonumber\\
& =\sum_{k=0}^{n-1}\left(  \prod_{i=1}^{n-1}g_{j+i}\right)  \cdot\left(
\prod_{i=1}^{k}x_{j+i}\right)  \cdot\left(  \prod_{i=k+1}^{n-1}u_{j+i}\right)
\nonumber
\end{align}%
\begin{align}
& =\left(  \prod_{i=1}^{n-1}g_{j+i}\right)  \cdot\sum_{k=0}^{n-1}%
\underbrace{\left(  \prod_{i=1}^{k}x_{j+i}\right)  }_{=x_{j+1}x_{j+2}\cdots
x_{j+k}}\cdot\underbrace{\left(  \prod_{i=k+1}^{n-1}u_{j+i}\right)
}_{=u_{j+k+1}u_{j+k+2}\cdots u_{j+n-1}}\nonumber\\
& =\left(  \prod_{i=1}^{n-1}g_{j+i}\right)  \cdot\underbrace{\sum_{k=0}%
^{n-1}x_{j+1}x_{j+2}\cdots x_{j+k}\cdot u_{j+k+1}u_{j+k+2}\cdots u_{j+n-1}%
}_{\substack{=t_{n-1,j}\\\text{(by the definition of }t_{n-1,j}\text{)}%
}}\nonumber\\
& =\left(  \prod_{i=1}^{n-1}g_{j+i}\right)  \cdot t_{n-1,j}%
.\label{pf.prop.fu-gauge.1}%
\end{align}

However, for each $j\in\mathbb{Z}$, we have%
\[
\left(  \prod_{i=1}^{n-1}g_{j+i}\right)  \cdot g_{j+n}=\prod_{i=1}^{n}%
g_{j+i}=g_{j+1}g_{j+2}\cdots g_{j+n}=g_{1}g_{2}\cdots g_{n}%
\]
(by Lemma \ref{lem.aprod}, applied to $a=g$ and $k=j$) and therefore%
\begin{equation}
\prod_{i=1}^{n-1}g_{j+i}=\dfrac{g_{1}g_{2}\cdots g_{n}}{g_{j+n}}=\dfrac
{g_{1}g_{2}\cdots g_{n}}{g_{j}}\label{pf.prop.fu-gauge.2}%
\end{equation}
(since Convention \ref{conv.bir.peri} yields $g_{j+n}=g_{j}$). 

Thus, for each $j\in\mathbb{Z}$, we have
\begin{align}
q_{j}  & =\underbrace{\left(  \prod_{i=1}^{n-1}g_{j+i}\right)  }%
_{\substack{=\dfrac{g_{1}g_{2}\cdots g_{n}}{g_{j}}\\\text{(by
(\ref{pf.prop.fu-gauge.2}))}}}\cdot t_{n-1,j}\ \ \ \ \ \ \ \ \ \ \left(
\text{by (\ref{pf.prop.fu-gauge.1})}\right)  \nonumber\\
& =\dfrac{g_{1}g_{2}\cdots g_{n}}{g_{j}}\cdot t_{n-1,j}%
.\label{pf.prop.fu-gauge.3}%
\end{align}

Now, let $i\in\left\{  1,2,\ldots,n\right\}  $. Applying
(\ref{pf.prop.fu-gauge.3}) to $j=i-1$, we obtain%
\begin{equation}
q_{i-1}=\dfrac{g_{1}g_{2}\cdots g_{n}}{g_{i-1}}\cdot t_{n-1,i-1}%
.\label{pf.prop.fu-gauge.3a}%
\end{equation}
Applying (\ref{pf.prop.fu-gauge.3}) to $j=i+1$, we obtain%
\begin{equation}
q_{i+1}=\dfrac{g_{1}g_{2}\cdots g_{n}}{g_{i+1}}\cdot t_{n-1,i+1}%
.\label{pf.prop.fu-gauge.3b}%
\end{equation}
Dividing the equality (\ref{pf.prop.fu-gauge.3a}) by the equality
(\ref{pf.prop.fu-gauge.3b}), we obtain%
\begin{equation}
\dfrac{q_{i-1}}{q_{i+1}}=\dfrac{\dfrac{g_{1}g_{2}\cdots g_{n}}{g_{i-1}}\cdot
t_{n-1,i-1}}{\dfrac{g_{1}g_{2}\cdots g_{n}}{g_{i+1}}\cdot t_{n-1,i+1}}%
=\dfrac{t_{n-1,i-1}}{t_{n-1,i+1}}\cdot\dfrac{g_{i+1}}{g_{i-1}}%
.\label{pf.prop.fu-gauge.3ab}%
\end{equation}

However, the definition of $y$ yields%
\begin{equation}
y_{i}=u_{i}\cdot\dfrac{u_{i-1}t_{n-1,i-1}}{x_{i+1}t_{n-1,i+1}}%
.\label{pf.prop.fu-gauge.4}%
\end{equation}
On the other hand, from (\ref{pf.prop.fu-gauge.abi=}), we obtain the
equalities $\left(  gu\right)  _{i}=g_{i}u_{i}$ and $\left(  gu\right)
_{i-1}=g_{i-1}u_{i-1}$ and $\left(  gx\right)  _{i+1}=g_{i+1}x_{i+1}$. Thus,%
\begin{align*}
\left(  gu\right)  _{i}\cdot\dfrac{\left(  gu\right)  _{i-1}q_{i-1}}{\left(
gx\right)  _{i+1}q_{i+1}}  & =g_{i}u_{i}\cdot\dfrac{g_{i-1}u_{i-1}q_{i-1}%
}{g_{i+1}x_{i+1}q_{i+1}}=g_{i}u_{i}\cdot\dfrac{g_{i-1}u_{i-1}}{g_{i+1}x_{i+1}%
}\cdot\dfrac{q_{i-1}}{q_{i+1}}\\
& =g_{i}u_{i}\cdot\dfrac{g_{i-1}u_{i-1}}{g_{i+1}x_{i+1}}\cdot\dfrac
{t_{n-1,i-1}}{t_{n-1,i+1}}\cdot\dfrac{g_{i+1}}{g_{i-1}}%
\ \ \ \ \ \ \ \ \ \ \left(  \text{by (\ref{pf.prop.fu-gauge.3ab})}\right)  \\
& =g_{i}\underbrace{u_{i}\cdot\dfrac{u_{i-1}t_{n-1,i-1}}{x_{i+1}t_{n-1,i+1}}%
}_{\substack{=y_{i}\\\text{(by (\ref{pf.prop.fu-gauge.4}))}}}=g_{i}%
y_{i}=\left(  gy\right)  _{i}%
\end{align*}
(since the definition of $gy$ yields $\left(  gy\right)  _{i}=g_{i}y_{i}$).
Therefore,
\[
\left(  gy\right)  _{i}=\left(  gu\right)  _{i}\cdot\dfrac{\left(  gu\right)
_{i-1}q_{i-1}}{\left(  gx\right)  _{i+1}q_{i+1}}.
\]

Now, forget that we fixed $i$. We thus have proved that
\[
\left(  gy\right)  _{i}=\left(  gu\right)  _{i}\cdot\dfrac{\left(  gu\right)
_{i-1}q_{i-1}}{\left(  gx\right)  _{i+1}q_{i+1}}\ \ \ \ \ \ \ \ \ \ \text{for
each }i\in\left\{  1,2,\ldots,n\right\}  .
\]
Hence, Lemma \ref{lem.f.back} (applied to $gu$, $gx$ and $gy$ instead of $u$,
$x$ and $z$) yields that $\mathbf{f}_{gu}\left(  gx\right)  =g\underbrace{y}%
_{=\mathbf{f}_{u}\left(  x\right)  }=g\mathbf{f}_{u}\left(  x\right)  $.
Proposition \ref{prop.fu-gauge} is thus proved.
\end{proof}

\begin{proof}
[Proof of Proposition \ref{prop.eta-gauge}.] For any $i\in\mathbb{Z}$, we
consider the element $\kappa_{i}\left(  a,b\right)  $ defined in Definition
\ref{def.eta}, and we also consider the element $\kappa_{i}\left(
ga,gb\right)  $ defined in the same way as $\kappa_{i}\left(  a,b\right)  $
(but using $ga$ and $gb$ instead of $a$ and $b$).

Define two $n$-tuples $a^{\prime}\in\mathbb{K}^{n}$ and $b^{\prime}%
\in\mathbb{K}^{n}$ as in Definition \ref{def.eta}. Furthermore, let $\left(
ga\right)  ^{\prime}\in\mathbb{K}^{n}$ and $\left(  gb\right)  ^{\prime}%
\in\mathbb{K}^{n}$ be the two $n$-tuples defined in the same way (but using
$ga$ and $gb$ instead of $a$ and $b$). Definition \ref{def.eta} yields that
$\eta\left(  a,b\right)  =\left(  a^{\prime},b^{\prime}\right)  $. The same
argument (applied to $ga$ and $gb$ instead of $a$ and $b$) yields
\begin{equation}
\eta\left(  ga,gb\right)  =\left(  \left(  ga\right)  ^{\prime},\left(
gb\right)  ^{\prime}\right)  .\label{pf.prop.eta-gauge.etag1}%
\end{equation}

We shall now show that $\left(  ga\right)  ^{\prime}=ga^{\prime}$ and $\left(
gb\right)  ^{\prime}=gb^{\prime}$.

Indeed, let $i\in\mathbb{Z}$ be arbitrary. Then, the definition of $\kappa
_{i}\left(  ga,gb\right)  $ yields%
\begin{align}
\kappa_{i}\left(  ga,gb\right)    & =\sum_{j=i}^{i+n-1}\underbrace{\left(
gb\right)  _{i+1}\left(  gb\right)  _{i+2}\cdots\left(  gb\right)  _{j}%
}_{=\prod_{p=i+1}^{j}\left(  gb\right)  _{p}}\cdot\underbrace{\left(
ga\right)  _{j+1}\left(  ga\right)  _{j+2}\cdots\left(  ga\right)  _{i+n-1}%
}_{=\prod_{p=j+1}^{i+n-1}\left(  ga\right)  _{p}}\nonumber\\
& =\sum_{j=i}^{i+n-1}\left(  \prod_{p=i+1}^{j}\underbrace{\left(  gb\right)
_{p}}_{\substack{=g_{p}b_{p}\\\text{(by (\ref{pf.prop.fu-gauge.abi=}))}%
}}\right)  \cdot\left(  \prod_{p=j+1}^{i+n-1}\underbrace{\left(  ga\right)
_{p}}_{\substack{=g_{p}a_{p}\\\text{(by (\ref{pf.prop.fu-gauge.abi=}))}%
}}\right)  \nonumber\\
& =\sum_{j=i}^{i+n-1}\underbrace{\left(  \prod_{p=i+1}^{j}\left(  g_{p}%
b_{p}\right)  \right)  }_{=\left(  \prod_{p=i+1}^{j}g_{p}\right)  \cdot\left(
\prod_{p=i+1}^{j}b_{p}\right)  }\cdot\underbrace{\left(  \prod_{p=j+1}%
^{i+n-1}\left(  g_{p}a_{p}\right)  \right)  }_{=\left(  \prod_{p=j+1}%
^{i+n-1}g_{p}\right)  \cdot\left(  \prod_{p=j+1}^{i+n-1}a_{p}\right)
}\nonumber\\
& =\sum_{j=i}^{i+n-1}\left(  \prod_{p=i+1}^{j}g_{p}\right)  \cdot
\underbrace{\left(  \prod_{p=i+1}^{j}b_{p}\right)  \cdot\left(  \prod
_{p=j+1}^{i+n-1}g_{p}\right)  }_{=\left(  \prod_{p=j+1}^{i+n-1}g_{p}\right)
\cdot\left(  \prod_{p=i+1}^{j}b_{p}\right)  }\cdot\left(  \prod_{p=j+1}%
^{i+n-1}a_{p}\right)  \nonumber\\
& =\sum_{j=i}^{i+n-1}\underbrace{\left(  \prod_{p=i+1}^{j}g_{p}\right)
\cdot\left(  \prod_{p=j+1}^{i+n-1}g_{p}\right)  }_{=\prod_{p=i+1}^{i+n-1}%
g_{p}}\cdot\left(  \prod_{p=i+1}^{j}b_{p}\right)  \cdot\left(  \prod
_{p=j+1}^{i+n-1}a_{p}\right)  \nonumber\\
& =\sum_{j=i}^{i+n-1}\left(  \prod_{p=i+1}^{i+n-1}g_{p}\right)  \cdot\left(
\prod_{p=i+1}^{j}b_{p}\right)  \cdot\left(  \prod_{p=j+1}^{i+n-1}a_{p}\right)
\nonumber\\
& =\left(  \prod_{p=i+1}^{i+n-1}g_{p}\right)  \cdot\sum_{j=i}^{i+n-1}%
\underbrace{\left(  \prod_{p=i+1}^{j}b_{p}\right)  }_{=b_{i+1}b_{i+2}\cdots
b_{j}}\cdot\underbrace{\left(  \prod_{p=j+1}^{i+n-1}a_{p}\right)  }%
_{=a_{j+1}a_{j+2}\cdots a_{i+n-1}}\nonumber\\
& =\left(  \prod_{p=i+1}^{i+n-1}g_{p}\right)  \cdot\underbrace{\sum
_{j=i}^{i+n-1}b_{i+1}b_{i+2}\cdots b_{j}\cdot a_{j+1}a_{j+2}\cdots a_{i+n-1}%
}_{\substack{=\kappa_{i}\left(  a,b\right)  \\\text{(by the definition of
}\kappa_{i}\left(  a,b\right)  \text{)}}}\nonumber\\
& =\left(  \prod_{p=i+1}^{i+n-1}g_{p}\right)  \cdot\kappa_{i}\left(
a,b\right)  .\label{pf.prop.eta-gauge.1}%
\end{align}
However, for each $i\in\mathbb{Z}$, we have%
\[
\left(  \prod_{p=i+1}^{i+n-1}g_{p}\right)  \cdot g_{i+n}=\prod_{p=i+1}%
^{i+n}g_{p}=g_{i+1}g_{i+2}\cdots g_{i+n}=g_{1}g_{2}\cdots g_{n}%
\]
(by Lemma \ref{lem.aprod}, applied to $g$ and $i$ instead of $a$ and $k$) and
therefore%
\begin{equation}
\prod_{p=i+1}^{i+n-1}g_{p}=\dfrac{g_{1}g_{2}\cdots g_{n}}{g_{i+n}}%
=\dfrac{g_{1}g_{2}\cdots g_{n}}{g_{i}}\label{pf.prop.eta-gauge.2}%
\end{equation}
(since Convention \ref{conv.bir.peri} yields $g_{i+n}=g_{i}$). 

Thus, for each $i\in\mathbb{Z}$, we have
\begin{align}
\kappa_{i}\left(  ga,gb\right)    & =\underbrace{\left(  \prod_{p=i+1}%
^{i+n-1}g_{p}\right)  }_{\substack{=\dfrac{g_{1}g_{2}\cdots g_{n}}{g_{i}%
}\\\text{(by (\ref{pf.prop.eta-gauge.2}))}}}\cdot\kappa_{i}\left(  a,b\right)
\ \ \ \ \ \ \ \ \ \ \left(  \text{by (\ref{pf.prop.eta-gauge.1})}\right)
\nonumber\\
& =\dfrac{g_{1}g_{2}\cdots g_{n}}{g_{i}}\cdot\kappa_{i}\left(  a,b\right)
.\label{pf.prop.eta-gauge.3}%
\end{align}

Now, let $i\in\left\{  1,2,\ldots,n\right\}  $. Then, the definition of
$a_{i}^{\prime}$ yields%
\begin{equation}
a_{i}^{\prime}=\dfrac{a_{i-1}\kappa_{i-1}\left(  a,b\right)  }{\kappa
_{i}\left(  a,b\right)  }.\label{pf.prop.eta-gauge.4}%
\end{equation}
Likewise, the definition of $\left(  ga\right)  _{i}^{\prime}$ yields%
\begin{align*}
\left(  ga\right)  _{i}^{\prime}  & =\dfrac{\left(  ga\right)  _{i-1}%
\kappa_{i-1}\left(  ga,gb\right)  }{\kappa_{i}\left(  ga,gb\right)
}=\underbrace{\left(  ga\right)  _{i-1}}_{\substack{=g_{i-1}a_{i-1}\\\text{(by
(\ref{pf.prop.fu-gauge.abi=}))}}}\cdot\underbrace{\kappa_{i-1}\left(
ga,gb\right)  }_{\substack{=\dfrac{g_{1}g_{2}\cdots g_{n}}{g_{i-1}}\cdot
\kappa_{i-1}\left(  a,b\right)  \\\text{(by (\ref{pf.prop.eta-gauge.3}%
),}\\\text{applied to }i-1\\\text{instead of }i\text{)}}}\diagup
\underbrace{\kappa_{i}\left(  ga,gb\right)  }_{\substack{=\dfrac{g_{1}%
g_{2}\cdots g_{n}}{g_{i}}\cdot\kappa_{i}\left(  a,b\right)  \\\text{(by
(\ref{pf.prop.eta-gauge.3}))}}}\\
& =g_{i-1}a_{i-1}\cdot\dfrac{g_{1}g_{2}\cdots g_{n}}{g_{i-1}}\cdot\kappa
_{i-1}\left(  a,b\right)  \diagup\left(  \dfrac{g_{1}g_{2}\cdots g_{n}}{g_{i}%
}\cdot\kappa_{i}\left(  a,b\right)  \right)  \\
& =g_{i}\cdot\underbrace{\dfrac{a_{i-1}\kappa_{i-1}\left(  a,b\right)
}{\kappa_{i}\left(  a,b\right)  }}_{\substack{=a_{i}^{\prime}\\\text{(by
(\ref{pf.prop.eta-gauge.4}))}}}=g_{i}a_{i}^{\prime}=\left(  ga^{\prime
}\right)  _{i}%
\end{align*}
(since the definition of $ga^{\prime}$ yields $\left(  ga^{\prime}\right)
_{i}=g_{i}a_{i}^{\prime}$).

Forget that we fixed $i$. We thus have shown that $\left(  ga\right)
_{i}^{\prime}=\left(  ga^{\prime}\right)  _{i}$ for each $i\in\left\{
1,2,\ldots,n\right\}  $. In other words, $\left(  ga\right)  ^{\prime
}=ga^{\prime}$. A similar argument shows that $\left(  gb\right)  ^{\prime
}=gb^{\prime}$. Thus, (\ref{pf.prop.eta-gauge.etag1}) becomes%
\[
\eta\left(  ga,gb\right)  =\left(  \underbrace{\left(  ga\right)  ^{\prime}%
}_{=ga^{\prime}},\underbrace{\left(  gb\right)  ^{\prime}}_{=gb^{\prime}%
}\right)  =\left(  ga^{\prime},gb^{\prime}\right)  .
\]
This proves Proposition \ref{prop.eta-gauge}.
\end{proof}
\end{verlong}

\end{document}